\documentclass[11pt]{amsart}

\usepackage{amsmath}
\usepackage{xcolor}
\usepackage{amsxtra}
\usepackage{amscd}
\usepackage{amsthm}
\usepackage{amsfonts}
\usepackage{amssymb}
\usepackage{eucal}
\usepackage[matrix,arrow,curve]{xy}
\usepackage[dvips]{graphicx}
\usepackage[dvips]{graphics}
\usepackage[T2A]{fontenc}
\usepackage[cp866]{inputenc}
\usepackage{mathtools}
\usepackage{mathrsfs}

\usepackage{hyperref}

\usepackage[margin=1in]{geometry}

\newtheorem*{claim}{Claim}

\newcommand{\Aut}{\mathop{\sf Aut}\nolimits}

\newcommand{\End}{\mathop{\sf End}\nolimits}
\newcommand{\Sym}{\mathop{\textrm{Sym}}\nolimits}
\newcommand{\Hom}{\mathop{\sf Hom}\nolimits}

\newcommand{\im}{\mathop{\sf Im}\nolimits}

\def \Z {{\mathbb Z}}

\def \id {{\rm id}}

\theoremstyle{plain}
\newtheorem{Thm}[subsection]{Theorem}
\newtheorem{Cor}[subsection]{Corollary}
\newtheorem{Lem}[subsection]{Lemma}
\newtheorem{Prop}[subsection]{Proposition}
\newtheorem{Conj}[subsection]{Conjecture}
\newtheorem{Ex}[subsection]{Example}

\theoremstyle{definition}
\newtheorem{Def}[subsection]{Definition}

\theoremstyle{remark}

\newtheorem{Rem}[subsection]{Remark}

\newtheorem{conj}{Conjecture}

\numberwithin{equation}{section}

\newif\ifShowLabels
\ShowLabelstrue
\newdimen\theight
\def\TeXref#1{%
    \leavevmode\vadjust{\setbox0=\hbox{{\tt
        \quad\quad  {\small \rm #1}}}%
    \theight=\ht0
    \advance\theight by \lineskip
    \kern -\theight \vbox to
    \theight{\rightline{\rlap{\box0}}%
    \vss}%
    }}%

\ShowLabelsfalse

\renewcommand{\sec}[2]{\section{#2}\label{S:#1}%
    \ifShowLabels \TeXref{{S:#1}} \fi}
\newcommand{\ssec}[2]{\subsection{#2}\label{SS:#1}%
    \ifShowLabels \TeXref{{SS:#1}} \fi}

\newcommand{\refs}[1]{Section ~\ref{S:#1}}
\newcommand{\refss}[1]{Section ~\ref{SS:#1}}

\newcommand{\reft}[1]{Theorem ~\ref{T:#1}}
\newcommand{\refl}[1]{Lemma ~\ref{L:#1}}
\newcommand{\refp}[1]{Proposition ~\ref{P:#1}}
\newcommand{\refc}[1]{Corollary ~\ref{C:#1}}
\newcommand{\refd}[1]{Definition ~\ref{D:#1}}
\newcommand{\refr}[1]{Remark ~\ref{R:#1}}
\newcommand{\refe}[1]{\eqref{E:#1}}

\newenvironment{thm}[1]%
    { \begin{Thm} \label{T:#1}  \ifShowLabels \TeXref{T:#1} \fi }%
    { \end{Thm} }

\renewcommand{\th}[1]{\begin{thm}{#1} \sl }
\renewcommand{\eth}{\end{thm} }

\newenvironment{lemma}[1]%
    { \begin{Lem} \label{L:#1}  \ifShowLabels \TeXref{L:#1} \fi }%
    { \end{Lem} }
\newcommand{\lem}[1]{\begin{lemma}{#1} \sl}
\newcommand{\elem}{\end{lemma}}

\newenvironment{propos}[1]%
    { \begin{Prop} \label{P:#1}  \ifShowLabels \TeXref{P:#1} \fi }%
    { \end{Prop} }
\newcommand{\prop}[1]{\begin{propos}{#1}\sl }
\newcommand{\eprop}{\end{propos}}

\newenvironment{corol}[1]%
    { \begin{Cor} \label{C:#1}  \ifShowLabels \TeXref{C:#1} \fi }%
    { \end{Cor} }
\newcommand{\cor}[1]{\begin{corol}{#1} \sl }
\newcommand{\ecor}{\end{corol}}

\newenvironment{defeni}[1]%
    { \begin{Def} \label{D:#1}  \ifShowLabels \TeXref{D:#1} \fi }%
    { \end{Def} }
\newcommand{\defe}[1]{\begin{defeni}{#1} \sl }
\newcommand{\edefe}{\end{defeni}}

\newenvironment{remark}[1]%
    { \begin{Rem} \label{R:#1}  \ifShowLabels \TeXref{R:#1} \fi }%
    { \end{Rem} }
\newcommand{\rem}[1]{\begin{remark}{#1}}
\newcommand{\erem}{\end{remark}}

\newenvironment{conjec}[1]%
    { \begin{Conj} \label{Co:#1}  \ifShowLabels \TeXref{Co:#1} \fi }%
    { \end{Conj} }
\renewcommand{\conj}[1]{\begin{conjec}{#1} \sl }
\newcommand{\econj}{\end{conjec}}

\newenvironment{example}[1]%
    { \begin{Ex} \label{Exx:#1}  \ifShowLabels \TeXref{Exx:#1} \fi }%
    { \end{Ex} }
\newcommand{\ex}[1]{\begin{example}{#1} \sl }
\newcommand{\eex}{\end{example}}

\newcommand{\eq}[1]%
    { \ifShowLabels \TeXref{E:#1} \fi
       \begin{equation} \label{E:#1} }
\newcommand{\eeq}{ \end{equation} }

\newcommand{\prf}{ \begin{proof} }
\newcommand{\epr}{ \end{proof} }

\newcommand\alp{\alpha}     

     \newcommand\Gam{\Gamma}
\newcommand\del{\delta}     \newcommand\Del{\Delta}

\newcommand\tet{\theta}

\newcommand\lam{\lambda}        \newcommand\Lam{\Lambda}
\newcommand\sig{\sigma}     

\newcommand\ome{\omega}     

\newcommand\calC{{\mathcal{C}}}

\newcommand\calW{{\mathcal{W}}}

\newcommand\bfk{{\mathbf k}}

        \newcommand\bfT{{\mathbf T}}

\newcommand\QQ{\mathbb{Q}}

\newcommand\RR{\mathbb{R}}

\newcommand\UU{\mathbb{U}}

\newcommand\PP{\mathbb{P}}
\renewcommand\AA{\mathbb{A}}

\newcommand\GG{\mathbb{G}}
\newcommand\HH{\mathbb{H}}

\newcommand\LL{\mathbb{L}}
\newcommand\ZZ{\mathbb{Z}}
\newcommand\XX{\mathbb{X}}
\newcommand\CC{\mathbb{C}}
\newcommand\VV{\mathbb{V}}
\newcommand\BB{\mathbb{B}}

 \newcommand\gra{{\mathfrak{a}}}
\newcommand\grB{{\mathfrak{B}}} \newcommand\grb{{\mathfrak{b}}}

 \newcommand\grg{{\mathfrak{g}}}

 \newcommand\grk{{\mathfrak{k}}}
 \newcommand\grl{{\mathfrak{l}}}
 
 \newcommand\grn{{\mathfrak{n}}}
 
 \newcommand\grp{{\mathfrak{p}}}

\newcommand\grS{{\mathfrak{S}}} 
 \newcommand\grt{{\mathfrak{t}}}

\newcommand\sdp{\times \hskip -0.3em {\raise 0.3ex
\hbox{$\scriptscriptstyle |$}}} 

\newcommand\ad{\operatorname{ad}}

\newcommand\Ad{\operatorname{Ad}}

\newcommand\Coker{\operatorname{Coker}}
\newcommand\codim{\operatorname{codim}}

\newcommand\Ext{\operatorname{Ext}}

\newcommand\Gr{\operatorname{Gr}}

\newcommand\RHom{\operatorname {RHom}}

\newcommand\Int{\operatorname{Int}}

\newcommand\PGL{{\rm PGL}}

\newcommand\Proj{\operatorname{Proj}}

\newcommand\rk{\operatorname{rk}}

\newcommand\Perv{\operatorname{Perv}}

\newcommand\pr{\operatorname{pr}}

\newcommand\sgn{\operatorname{sgn}}
\newcommand\SL{{\rm SL}}

\newcommand\SO{{\rm SO}}

\newcommand\Spec{\operatorname{Spec}}

\newcommand\supp{\operatorname{supp}}
\newcommand\Supp{\operatorname{Supp}}

\newcommand\op{{\overline{p}}}

\newcommand\ot{{\overline{t}}}

\newcommand\tilt{{\widetilde{t}}}

\newcommand\tilDel{{\widetilde{\Delta}}}

\newcommand\tilna{{\widetilde{\Na}}}

\newcommand\x{\times}
\newcommand\ten{\otimes}

\newcommand\wt{\widetilde}

\newcommand\ch{\text{ch}}

\renewcommand\Spec{\operatorname{Spec}}

\newcommand\nc{\newcommand}

\newcommand{\IC}{{\operatorname{IC}}}

\newcommand{\iso}{{\stackrel{\sim}{\longrightarrow}}}

\nc\aff{\operatorname{aff}}
\nc\oGr{\overline{\Gr}}
\nc\Bun{\operatorname{Bun}}
\nc\hgrg{\widehat{\grg}}
\renewcommand\Int{\operatorname{Int}}
\nc\bInt{\overline{\Int}}
\nc\hatLam{\widehat{\Lam}}
\nc\bmu{\overline{\mu}}
\nc\bnu{\overline{\nu}}
\nc\blambda{\overline{\lam}}
\renewcommand\SL{\operatorname{SL}}
\nc\ocalW{\overline{\calW}}
\nc\pos{\operatorname{pos}}
\nc\IH{\operatorname{IH}}
\nc\Rep{\operatorname{Rep}}
\nc\Gal{\operatorname{Gal}}
\nc{\tilGr}{\widetilde{\Gr}}

\nc\Pic{\operatorname{Pic}}
\nc\Prym{\operatorname{Prym}}
\nc\pa{\partial}
\nc\Na{\nabla}

\nc{\HC}{{\mathcal{HC}}}
\nc{\on}{\operatorname}
\nc{\BA}{{\mathbb{A}}}
\nc{\BC}{{\mathbb{C}}}
\nc{\BG}{{\mathbb{G}}}
\nc{\BM}{{\mathbb{M}}}
\nc{\BN}{{\mathbb{N}}}
\nc{\BQ}{{\mathbb{Q}}}
\nc{\BP}{{\mathbb{P}}}
\nc{\BR}{{\mathbb{R}}}
\nc{\BZ}{{\mathbb{Z}}}
\nc{\BS}{{\mathbb{S}}}

\nc{\CA}{{\mathcal{A}}}
\nc{\CB}{{\mathcal{B}}}
\nc{\CalC}{{\mathcal C}}
\nc{\CalD}{{\mathcal D}}
\nc{\CE}{{\mathcal{E}}}
\nc{\CF}{{\mathcal{F}}}
\nc{\CG}{{\mathcal{G}}}
\nc{\CH}{{\mathcal{H}}}
\nc{\CK}{{\mathcal{K}}}
\nc{\CL}{{\mathcal{L}}}
\nc{\CM}{{\mathcal{M}}}
\nc{\CMM}{{\mathcal{M}^{\operatorname{gen}}_\hbar(-\rho)}}
\nc{\CN}{{\mathcal{N}}}
\nc{\CO}{{\mathcal{O}}}
\nc{\CP}{{\mathcal{P}}}
\nc{\CQ}{{\mathcal{Q}}}
\nc{\CR}{{\mathcal{R}}}
\nc{\CS}{{\mathcal{S}}}
\nc{\CT}{{\mathcal{T}}}
\nc{\CU}{{\mathcal{U}}}
\nc{\CV}{{\mathcal{V}}}
\nc{\CW}{{\mathcal{W}}}
\nc{\CX}{{\mathcal{X}}}
\nc{\CY}{{\mathcal{Y}}}
\nc{\CZ}{{\mathcal{Z}}}

\nc{\gen}{{\operatorname{gen}}}
\nc{\cM}{{\check{\mathcal M}}{}}
\nc{\csM}{{\check{\mathcal A}}{}}
\nc{\obM}{{\overset{\circ}{\mathbf M}}{}}
\nc{\oCA}{{\overset{\circ}{\mathcal A}}{}}
\nc{\obA}{{\overset{\circ}{\mathbf A}}{}}
\nc{\ooM}{{\overset{\circ}{M}}{}}
\nc{\osM}{{\overset{\circ}{\mathsf M}}{}}
\nc{\vM}{{\overset{\bullet}{\mathcal M}}{}}
\nc{\nM}{{\underset{\bullet}{\mathcal M}}{}}
\nc{\obD}{{\overset{\circ}{\mathbf D}}{}}
\nc{\cp}{{\overset{\circ}{\mathbf p}}{}}
\nc{\ofZ}{{\overset{\circ}{\mathfrak Z}}{}}

\nc{\fa}{{\mathfrak{a}}}
\nc{\fb}{{\mathfrak{b}}}
\nc{\fg}{{\mathfrak{g}}}
\nc{\fgl}{{\mathfrak{gl}}}
\nc{\fh}{{\mathfrak{h}}}
\nc{\fj}{{\mathfrak{j}}}
\nc{\fm}{{\mathfrak{m}}}
\nc{\fn}{{\mathfrak{n}}}
\nc{\fu}{{\mathfrak{u}}}
\nc{\fp}{{\mathfrak{p}}}
\nc{\frr}{{\mathfrak{r}}}
\nc{\fs}{{\mathfrak{s}}}
\nc{\ft}{{\mathfrak{t}}}
\nc{\fT}{{\mathfrak{T}}}
\nc{\ofT}{{\overline{\mathfrak T}}}
\nc{\ofS}{{\overline{\mathfrak S}}}
\nc{\fsl}{{\mathfrak{sl}}}
\nc{\hsl}{{\widehat{\mathfrak{sl}}}}
\nc{\hgl}{{\widehat{\mathfrak{gl}}}}
\nc{\hg}{{\widehat{\mathfrak{g}}}}
\nc{\chg}{{\widehat{\mathfrak{g}}}{}^\vee}
\nc{\hn}{{\widehat{\mathfrak{n}}}}
\nc{\chn}{{\widehat{\mathfrak{n}}}{}^\vee}

\nc{\fA}{{\mathfrak{A}}}
\nc{\fB}{{\mathfrak{B}}}
\nc{\fD}{{\mathfrak{D}}}
\nc{\fE}{{\mathfrak{E}}}
\nc{\fF}{{\mathfrak{F}}}
\nc{\fG}{{\mathfrak{G}}}
\nc{\fI}{{\mathfrak{I}}}
\nc{\fJ}{{\mathfrak{J}}}
\nc{\fK}{{\mathfrak{K}}}
\nc{\fL}{{\mathfrak{L}}}
\nc{\fM}{{\mathfrak{M}}}
\nc{\fN}{{\mathfrak{N}}}
\nc{\frP}{{\mathfrak{P}}}
\nc{\fS}{{\mathfrak S}}
\nc{\fU}{{\mathfrak{U}}}
\nc{\fZ}{{\mathfrak{Z}}}

\nc{\bb}{{\mathbf{b}}}
\nc{\bc}{{\mathbf{c}}}
\nc{\be}{{\mathbf{e}}}
\nc{\bj}{{\mathbf{j}}}
\nc{\bn}{{\mathbf{n}}}
\nc{\bp}{{\mathbf{p}}}
\nc{\bq}{{\mathbf{q}}}
\nc{\bv}{{\mathbf{v}}}
\nc{\bx}{{\mathbf{x}}}
\nc{\by}{{\mathbf{y}}}
\nc{\bw}{{\mathbf{w}}}
\nc{\bA}{{\mathbf{A}}}
\nc{\bB}{{\mathbf{B}}}
\nc{\bC}{{\mathbf{C}}}
\nc{\bK}{{\mathbf{K}}}
\nc{\bL}{{\mathbf{L}}}
\nc{\bD}{{\mathbf{D}}}
\nc{\bH}{{\mathbf{H}}}
\nc{\bM}{{\mathbf{M}}}
\nc{\bN}{{\mathbf{N}}}
\nc{\bS}{{\mathbf{S}}}
\nc{\bT}{{\mathbf{T}}}
\nc{\bV}{{\mathbf{V}}}
\nc{\bW}{{\mathbf{W}}}
\nc{\bX}{{\mathbf{X}}}
\nc{\bP}{{\mathbf{P}}}
\nc{\bZ}{{\mathbf{Z}}}

\nc{\sA}{{\mathsf{A}}}
\nc{\sB}{{\mathsf{B}}}
\nc{\sC}{{\mathsf{C}}}
\nc{\sD}{{\mathsf{D}}}
\nc{\sF}{{\mathsf{F}}}
\nc{\sK}{{\mathsf{K}}}
\nc{\sM}{{\mathsf{M}}}
\nc{\sO}{{\mathsf{O}}}
\nc{\sQ}{{\mathsf{Q}}}
\nc{\sP}{{\mathsf{P}}}
\nc{\sV}{{\mathsf{V}}}
\nc{\sW}{{\mathsf{W}}}
\nc{\sZ}{{\mathsf{Z}}}
\nc{\sfp}{{\mathsf{p}}}
\nc{\sr}{{\mathsf{r}}}
\nc{\sfb}{{\mathsf{b}}}
\nc{\sfc}{{\mathsf{c}}}
\nc{\sd}{{\mathsf{d}}}
\nc{\sg}{{\mathsf{g}}}
\nc{\sfl}{{\mathsf{l}}}

\nc{\BK}{{\bar{K}}}

\nc{\tA}{{\widetilde{\mathbf{A}}}}
\nc{\tB}{{\widetilde{\mathcal{B}}}}
\nc{\tg}{{\widetilde{\mathfrak{g}}}}
\nc{\tG}{{\widetilde{G}}}
\nc{\TM}{{\widetilde{\mathbb{M}}}{}}
\nc{\tO}{{\widetilde{\mathsf{O}}}{}}
\nc{\tU}{{\widetilde{\mathfrak{U}}}{}}
\nc{\TZ}{{\tilde{Z}}}
\nc{\tZ}{\widetilde{Z}{}}
\nc{\tx}{{\tilde{x}}}
\nc{\tbv}{{\tilde{\bv}}}
\nc{\tfP}{{\widetilde{\mathfrak{P}}}{}}
\nc{\tz}{{\tilde{\zeta}}}
\nc{\tmu}{{\tilde{\mu}}}

\nc{\td}{\ddot{\underline{d}}{}}
\nc{\tzeta}{\widetilde{\zeta}{}}
\nc{\hd}{{\widehat{\underline{d}}}}
\nc{\hG}{{\widehat{G}}}
\nc{\hBP}{\widehat{\mathbb P}{}}
\nc{\hQ}{{\widehat{Q}}}
\nc{\hsM}{\widehat{\mathsf M}{}}
\nc{\hfM}{\widehat{\mathfrak M}{}}
\nc{\hCP}{\widehat{\mathcal P}{}}
\nc{\hCR}{\widehat{\mathcal R}{}}
\nc{\hCS}{{\widehat{\mathcal S}}}
\nc{\hfZ}{\widehat{\mathfrak Z}{}}

\nc{\urho}{\underline{\rho}}
\nc{\uB}{\underline{B}}
\nc{\uC}{{\underline{\mathbb{C}}}}
\nc{\uk}{{\underline{\bfk}}}
\nc{\ui}{\underline{i}}
\nc{\ofP}{{\overline{\mathfrak{P}}}}

\nc{\hrho}{{\hat{\rho}}}

\nc{\unl}{\underline}
\nc{\ol}{\overline}
\nc{\one}{{\mathbf{1}}}
\nc{\two}{{\mathbf{t}}}

\nc{\Tot}{{\mathop{\operatorname{\rm Tot}}}}
\nc{\Hilb}{{\mathop{\operatorname{\rm Hilb}}}}
\nc{\CHom}{{\mathop{\operatorname{{\mathcal{H}}\it om}}}}
\nc{\defi}{{\mathop{\operatorname{\rm def}}}}
\nc{\length}{{\mathop{\operatorname{\rm length}}}}

\nc{\Cliff}{{\mathsf{Cliff}}}
\nc{\Fl}{{\mathsf{Fl}}}
\nc{\Fib}{{\mathsf{Fib}}}
\nc{\Coh}{{\mathsf{Coh}}}
\nc{\FCoh}{{\mathsf{FCoh}}}

\nc{\reg}{{\text{\rm reg}}}

\nc{\cplus}{{\mathbf{C}_+}}
\nc{\cminus}{{\mathbf{C}_-}}
\nc{\cthree}{{\mathbf{C}_*}}
\nc{\Qbar}{{\bar{Q}}}
    
\nc{\bh}{{\bar{h}}}
\nc{\bOmega}{{\overline{\Omega}}}
\nc\tGr{\widetilde{\Gr}}

\nc{\seq}[1]{\stackrel{#1}{\sim}}
\nc\ogu{\overline{G/U}}
\nc\chlam{\check{\lam}}

\nc\St{\operatorname{St}}

\nc\uS{\underline{S}}
\nc\QM{\mathcal{QM}}
\nc\FT{\mathsf{FT}}
\nc{\Ca}{\underline{C_a}}
\nc{\SCa}{\underline{S^mC_a}}
\nc{\sCa}{\Ca\x_\CA\underline{S^{m-1}C_a}}

\nc\Shv{\operatorname{Shv}}
\nc\cmod{\textup{-comod}}
\nc\lmod{\textup{-mod}}
\nc\rmod{\textup{mod-}}
\nc\bimod{\textup{-mod-}}
\nc\grrmod{\textup{grmod-}}
\nc\grlmod{\textup{-grmod}}
\nc\mon{\textup{-mon}}
\nc\pt{\textup{pt}}
\nc\nil{\textup{nil}}
\nc\RG{\textup{R}\Gamma}
\nc\RGc{\textup{R}\Gamma_{c}}
\nc\fmo{free-monodromic }

\newcommand{\scT}{\mathscr{T}}

\nc\fra{\mathfrak{a}}
\nc\frg{\mathfrak{g}}
\nc\frb{\mathfrak{b}}
\nc\frn{\mathfrak{n}}
\nc\frh{\mathfrak{h}}
\nc\frp{\mathfrak{p}}
\nc\frk{\mathfrak{k}}

\newcommand{\incl}{\hookrightarrow}
\newcommand{\isom}{\stackrel{\sim}{\to}}
\newcommand{\bij}{\leftrightarrow}
\newcommand{\surj}{\twoheadrightarrow}

\nc\Lot{\stackrel{\LL}{\otimes}}

\renewcommand{\j}[1]{\langle{#1}\rangle}
\newcommand{\wh}[1]{\widehat{#1}}
\newcommand\quash[1]{}
\newcommand{\bu}{\bullet}
\newcommand{\bs}{\backslash}

\newcommand\xr{\xrightarrow}
\renewcommand\op{\oplus}
\renewcommand\ot{\otimes}
\renewcommand\c\circ

\renewcommand\a\alpha
\renewcommand\b\beta
\newcommand\g\gamma
\newcommand\G\Gamma
\renewcommand\d\delta
\newcommand\D\Delta
\newcommand{\e}{\epsilon}

\nc\io{\iota}
\nc\ka{\kappa}
\newcommand{\ph}{\varphi}
\renewcommand\r\rho
\newcommand{\s}{\sigma}
\renewcommand{\t}{\tau}
\newcommand{\y}{\eta}

\newcommand{\ep}{\epsilon}

\renewcommand{\l}{\lambda}
\renewcommand{\L}{\Lambda}

\newcommand\hs{\heartsuit}
\newcommand\na{\natural}
\newcommand\sh{\sharp}
\newcommand\da{\dagger}

\newcommand\cA{\mathcal{A}}
\newcommand\cB{\mathcal{B}}
\newcommand\cC{\mathcal{C}}
\newcommand\cD{\mathcal{D}}
\newcommand\cE{\mathcal{E}}
\newcommand\cF{\mathcal{F}}
\newcommand\cG{\mathcal{G}}
\newcommand\cH{\mathcal{H}}

\newcommand\cK{\mathcal{K}}
\newcommand\cL{\mathcal{L}}
\newcommand\cN{\mathcal{N}}
\newcommand\cO{\mathcal{O}}
\newcommand\cP{\mathcal{P}}
\newcommand\cQ{\mathcal{Q}}
\newcommand\cR{\mathcal{R}}
\newcommand\cS{\mathcal{S}}
\newcommand\cT{\mathcal{T}}

\nc\coker{\textup{coker}}
\nc\ev{\textup{ev}}
\nc\pro{\textup{pro}}
\nc\perf{\textup{perf}}
\nc\Perf{\textup{Perf}}
\nc\Tilt{\textup{Tilt}}
\nc\triv{\textup{triv}}
\nc\Lie{\textup{Lie}}
\nc\bt{\boxtimes}
\nc\SU{\textup{SU}}
\nc\PU{\textup{PU}}
\nc\SBim{\textup{SBim}}
\nc\TGR{\Tilt(\CM_{G_{\BR}})}
\nc\THG{\Tilt(\CH_{G})}
\nc\Frac{\textup{Frac}}
\nc\red{\textup{red}}
\nc\Spf{\textup{Spf}\ }
\nc\Inv{\textup{Inv}}
\nc\Stab{\textup{Stab}}
\newcommand{\sslash}{\mathbin{/\mkern-6mu/}}
\nc\bstar{\overline{\star}}
\nc\LS{\textup{LS}}
\nc\act{\textup{act}}
\nc\av{\textup{av}}
\nc\Gm{\GG_{m}}
\nc\dG{\check{G}}
\nc\dB{\check{B}}
\nc\dT{\check{T}}
\nc\dK{\check{K}}
\nc\dX{\check{X}}
\nc\dua{\check{a}}
\nc\dL{\check{L}}
\nc\dcD{\check{\cD}}
\nc\dcH{\check{\cH}}
\nc\dE{\check{E}}
\nc\dph{\check{\ph}}
\nc\dl{\check{\lambda}}
\nc\dcB{\check{\cB}}
\nc\dBB{\check{\BB}}
\nc\dgS{\check{\mathfrak{S}}}
\nc\dcK{\check{\cK}}
\nc\dI{\check{I}}
\nc\dell{\check{\ell}}
\nc\dth{\check{\theta}}
\nc\dbT{\check{\mathbf{T}}}
\nc\dmu{\check\mu}

\nc\SSC{\textup{SSC}}
\nc\HOM{\textup{HOM}}

\nc\Yun[1]{{\color{red}  {#1}}}

\title{Tilting sheaves for real groups and Koszul duality}

\author{Andrei Ionov}

\address{
Boston College, Department of Mathematics, Maloney Hall, Fifth
Floor, Chestnut Hill, MA 02467-3806, United States}
\email{ionov@bc.edu}

\author{Zhiwei Yun}

\address{Department of Mathematics, Massachusetts Institute of Technology, 77 Massachusetts Ave, Cambridge, MA 02139}
\email{zyun@mit.edu}

\begin{document}

\maketitle

\begin{abstract}
For a certain class of real analytic varieties with Lie group actions we develop a theory of (free-monodromic) tilting sheaves, and apply it to flag varieties stratified by real group orbits.
For quasi-split real groups, we construct a fully faithful embedding of the category of tilting sheaves to a real analog of the category of Soergel bimodules, establishing real group analogs of Soergel's Structure Theorem and Endomorphism Theorem.
We apply these results to give a purely geometric proof of the main result of \cite{BV} which proves Soergel's conjecture \cite{S2} for quasi-split groups.
\end{abstract}

\tableofcontents

\sec{}{Introduction}

\ssec{}{Koszul duality and tilting sheaves for complex groups}
The formalism of Koszul duality in representation theory was introduced in \cite{S} and further developed in \cite{BGS} and subsequent works. 

For a complex reductive group $G$, it is proved in \cite{BGS} that its category $\CO$ is equivalent to modules over a quadratic Koszul algebra. Moreover, this Koszul algebra is self-dual, or more canonically Koszul dual to its counterpart for the Langlands dual $\dG$. The Koszul duality in this case can be stated as an equivalence between the graded version of the derived category $D^{b}(\CO)$ for $G$ and the graded version of $D^{b}(\check{\CO})$ for $\dG$. It sends irreducible objects in $\check{\CO}$ to indecomposable projective objects in $\CO$. It is well-known that the category $\CO$ is equivalent to the category of Harish-Chandra bimodules for $G$, as well as to the category of perverse sheaves on the flag variety of $G$ constructible with respect to the Schubert stratification.

Another important duality, namely the Ringel duality, was introduced in \cite{R}. In the context of the highest weight categories this is an equivalence of derived categories that sends standard objects to costandard objects and projective objects to tilting objects. 
The composition of Koszul duality and Ringel duality was first studied in \cite{BG}. It sends irreducible objects in $\check{\CO}$ to indecomposable tilting objects in $\CO$. Compared to the original Koszul duality, the version composed with Ringel duality has the advantage that it preserves the monoidal structure on modified versions of both sides (using tensor product of Harish-Chandra bimodules). 

In the geometric setting the theory of tilting perverse sheaves was developed in \cite{BBM}. In particular, it is proved that the long intertwining functor provides a Ringel self-duality of the category of the perverse sheaves on the flag variety of $G$. The monoidal version of the Koszul duality (really the composition of Koszul duality with Ringel duality) was proved in \cite{BY} for arbitrary Kac-Moody groups. It sends irreducible perverse sheaves on $\dB\bs \dG/\dB$ to indecomposable {\em free-monodromic} tilting sheaves on $U\bs G/U$ (where $\dB\subset \dG$ is a Borel subgroup; $B\subset G$ is a Borel subgroup with unipotent radical $U$).

In the above versions of Koszul duality, the authors prove the equivalences of categories by showing that both categories are equivalent to graded modules over isomorphic algebras. To construct these algebras, the key ingredients are to construct {\em Soergel functors}  to the appropriate category of Soergel bimodules, and to prove that the Soergel functors are fully faithful on certain classes of objects (irreducible, projective or tilting objects).  For example, the Soergel functor maps category $\CO$ to coherent sheaves on the {\em complex block variety} $\grt^{*}\times_{\grt^{*}\sslash W}\grt^{*}$, where $\grt$ is the abstract Cartan algebra of $\grg$.





\ssec{}{Koszul duality for real groups}

In \cite{S2} W.\,Soergel formulated the Koszul duality conjecture for real groups as a categorification of Vogan's character duality \cite{V2}. Let $G_\BR$ be a real form of a semisimple group $G$ and let $\CM$ be a block of its representations with regular integral infinitesimal character. Vogan's duality associates to $\CM$ a block $\check{\CM}$ for a real form $\dG_{\RR}$ of the Langlands dual group $\dG$, which can be geometrically realized as a full subcategory $\dcD$ in the $\dK$-equivariant derived category (with $\CC$-coefficients) of the flag variety $\dX$ of $\dG$ (where $\dK\subset\dG$ is the fixed points of the Cartan involution defined using $\dG_{\RR}$). 

A breakthrough in this direction appears in  \cite{BV} where the authors prove Soergel's conjecture for quasi-split groups.  

\th{BV}(\cite[Theorem 5.1]{BV}) Let $G_{\RR}$ be a quasi-split semisimple real group.
Then there is an equivalence between graded versions of triangulated categories $D^b(\CM)$ and $\dcD$. It maps indecomposable projective objects in $\CM$ to irreducible perverse sheaves in $\dcD$.
\eth

Once again the analog of the Soergel functors played an important role in the proof of the above theorem. For the category $\CM$, the authors of \cite{BV} constructed a Soergel functor from $\CM$ to coherent sheaves on the {\em real block variety} $\gra^*\sslash W'_\CM\x_{\grt^*\sslash W}\grt^*$ where $\gra$ is the complexification of the Lie algebra of the maximal split torus of $G_\BR$, and $W'_{\CM}$ is a certain subgroup of the real Weyl group. The construction of the Soergel functor in \cite{BV} is algebraic: it is given by the translation functor to the most singular infinitesimal character.


\reft{BV} can be stated geometrically: the category $D^{b}(\CM)$ can be identified with the $K$-equivariant derived category of sheaves on the enhanced flag variety $\wt X=G/U$.  The goal of this paper is to give a geometric proof of this geometric reformulation. Two new ingredients leading to the proof are the theory of tilting sheaves on spaces stratified by real group orbits, and an analog of Soergel's Structure Theorem for quasi-split real groups. We give an overview of these ingredients below.

\ssec{}{Main results: Tilting sheaves for real group orbits}

We work with sheaves in $\bfk$-vector spaces, where $\bfk$ is an algebraically closed field of characteristic zero \footnote{We believe that with a more careful study, the results should hold for mod $p$ coefficients when $p$ is not too small.}.

The Matsuki correspondence is a bijection between $K$-orbits and $G(\RR)$-orbits on the flag variety $X$ of $G$. In \cite{MUV} the authors upgraded this set-theoretic bijection to an equivalence of categories between $K$-equivariant and $G(\RR)$-equivariant derived categories of $X$. When $G_{\RR}$ is a complex group, this equivalence can be identified with the long intertwining endo-functor of the category $\cO$. This suggests that the Matsuki equivalence in \cite{MUV} might be a Ringel duality. 

To justify this expectation, we first need to define tilting sheaves on $X$ (or rather free-monodromic tilting sheaves on $\wt X=G/U$ \footnote{In the main body of the paper we work with the compact manifold $(G/U)/\bT^{>0}$ homotopy equivalent to $G/U$.}) with respect to the $G(\RR)$-orbits. In this paper we develop such a theory of (free-monodromic) tilting sheaves in the general setting of a real analytic variety $\wt X$ stratified by Lie group orbits satisfying certain conditions. The development here follows that of \cite{BBM} and \cite{BY}. One point worth mentioning is that the definition of tilting sheaves requires working with a new $t$-structure on $\wt X$ coming from a specific perversity function (see \refss{t}). The perversity function takes into account not only the dimension of the stratum but also the size of its equivariant fundamental group.

Applying the general theory to the case of $G(\RR)$ orbits on $\wt X=G/U$, we show:
\th{intro Mat}(\reft{Mat})
The Matsuki correspondence in \cite{MUV} is a Ringel duality in the sense that it sends indecomposable projective pro-objects in $\Perv_{K}(\wt X)$ to indecomposable free-monodromic tilting sheaves in $\CM_{G_\BR}:=\wh D^{b}_{G(\RR)}(\wt X)_{\bT\mon}$.
\eth
Here, $\wh D^{b}_{G(\RR)}(\wt X)_{\bT\mon}$ is a completed version of the derived category of $G(\RR)$-equivariant complexes on $\wt X$ that are monodromic along the right $\bT$-action (where $\bT=B/U$ is the abstract Cartan) with unipotent monodromy. More details of the completion procedure in the real analytic setting is developed in \refs{compmoncat}.



It is worth noting that the common approach to the representation theory of $G(\BR)$ passes through the Harish-Chandra $(\grg,K)$-modules, which are then studied by algebraic methods. Similarly, in geometry more attention has been put on $K$-equivariant sheaves on $X$ rather than $G(\RR)$-equivariant sheaves, perhaps to stay within the context of algebraic geometry.  On the other hand, the category $\CM_{G_\BR}$ is directly related to the category of $G(\RR)$-representations by globalization functors of \cite{KSc} (see also \cite{Vi}). We wonder whether the $t$-structure we introduced on $\CM_{G_\BR}$ and tilting sheaves therein have any significance for $G(\RR)$-representations.



\ssec{}{Main results: Soergel's Structure Theorem for real groups}

Now we assume that $G_\BR$ is quasi-split. Let $\cR$ be the completion of the $\bfk$-group ring of $\pi_{1}(T)$ at the augmentation ideal. In this case we define a real Soergel functor $\VV_\BR: \CM_{G_{\RR}}\to \rmod\cR$ as a generic vanishing cycles functor to the closed $G(\RR)$-orbit\footnote{In the complex case the Soergel functor was identified with the generic vanishing cycles functor to the closed orbit in \cite[Section 11]{B}, which motivates our definition.}. 

We prove a generalization of Soergel's Struktursatz and Endomorphismensatz of \cite{S} to quasi-split groups. To state it, we need some notations. Let $T_{\RR}$ be a maximally split real Cartan subgroup of $G_{\RR}$ and $W(\RR)=W(G(\RR), T(\RR))$ be the real Weyl group. Let $\cS$ be the completion of the $\bfk$-group ring of $\pi_{1}(T(\CC)/T(\RR)^{\c})$. Let $\cB_{0}=(\prod_{\chi}\cS)^{W(\RR)}$ where the product is over characters $\chi: \pi_{0}(T(\RR))\to \bfk^{\times}$ and $W(\RR)$ permutes the factors under Vogan's {\em cross action}. Finally let $\CB$ be the commutative $\bfk$-algebra $\cB_{0}\ot_{\CR^{\bW}}\cR$, where $\bW$ is the abstract Weyl group acting on $\bT$, hence on $\CR$. The algebra $\cB$ is the real analog of $\cR\ot_{\cR^{\bW}}\cR$ that appear in the theory of Soergel bimodules, and it can be identified with formal functions on a disjoint of union of real block varieties of \cite{BV}.


\th{main0}(\refc{maincor} and \reft{bigtilting}) Assume the quasi-split group $G_{\RR}$ is nice (see \refd{Nice quasi-split groups}, in particular, $G_{\RR}$ is nice if $G$ is adjoint).  Let $\TGR$ be the full subcategory of $\CM_{G_{\RR}}$ consisting of \fmo tilting sheaves.
\begin{enumerate}
\item The real Soergel functor $\VV_{\RR}$ on \fmo tilting sheaves can be lifted to a functor
\begin{equation*}
\VV^{\sh}_\BR\colon\TGR\to\rmod\CB
\end{equation*}
and is compatible with the Hecke action on $\TGR$ and the Soergel bimodule action on $\rmod\cB$ (via the classical Soergel functor $\VV$ for the monodromic Hecke category).
\item (Struktursatz) The functor $\VV^{\sh}_\BR$ is fully-faithful.  
\item (Endomorphismensatz) The functor $\VV^{\sh}_\BR$  is corepresented by a ``big tilting'' object $\op_{a}\cT^{a}\in \TGR$ (where the direct sum is over blocks of $\CM_{G_{\RR}}$), and $\End(\op_{a}\cT^{a})\cong \cB$.
\end{enumerate}
\eth

Here is our proof strategy for \reft{main0}. Tautologically, we can lift $\VV_{\RR}|_{\Tilt}$ to be valued in $\rmod\cA$, where $\cA$ is the ring opposite to the endomorphism ring of the functor $\VV_{\RR}|_{\Tilt}$. By the localization technique developed in \refs{loc}, which is analogous to  the equivariant localization, we show that $\cA$ and $\cB$ are isomorphic at the generic points. Moreover, by localizing along codimension one loci of $\Spec\cA$, we deduce that $\VV_{\RR}^{\sh}: \TGR\to \rmod\cA$ is fully faithful. This relies on a case-by-case check for quasi-split groups of semisimple split rank one. Finally, we prove $\cA$ and $\cB$ are canonically isomorphic.

When $G_{\RR}$ is quasi-split but not nice, we prove an analogue of the Struktursatz in \reft{mainnonadj} by reducing the situation a nice group isogenous to $G_{\RR}$.

When $G_{\RR}$ is not quasi-split, the theory of tilting sheaves is still valid, while the Soergel functor should be replaced by an appropriate microlocalization functor. We hope to return to this problem in future work with the goal of proving Koszul duality for real groups in general. 

\ssec{}{Application to Koszul duality for quasi-split groups}

\reft{main0} then allows us to prove a variant of \reft{BV} geometrically. We give a brief statement below and refer to \refs{KD} for more details.

\th{intro KD} (\reft{KD derived}) Let $G_{\RR}$ be semisimple and quasi-split. Let $\CM^{a}_{G_{\RR}}$ be a block of $\CM_{G_{\RR}}$ and $\dcD^{\dua}_{\dK}\subset D^{b}_{\dK}(\dX)$ be the dual (principal) block in the sense of Vogan. There is a triangulated functor
\begin{equation*}
\wt\ka: \dcD^{\dua,gr}_{\dK}:=K^{b}(\SSC(\dcD^{\dua}_{\dK}))\to  \CM^{a}_{G_{\RR}}
\end{equation*}
(where $\SSC(\dcD^{\dua}_{\dK})$ stands for the full subcategory of semisimple complexes in $\dcD^{\dua}_{\dK}$), satisfying the following properties:
\begin{enumerate}
\item $\wt\ka$ is invariant under the internal shift $\{1\}$ on $\SSC(\dcD^{\dua}_{\dK})$.
\item For $\cF_{1},\cF_{2}\in \dcD^{\dua,gr}_{\dK}$, natural maps induce a graded $\cR$-linear  isomorphism
\begin{eqnarray*}
(\op_{m\in\Z}\Hom(\cF_{1},\cF_{2}\{m\})\ot_{\cR^{\bu}}\cR\isom \Hom(\wt\ka(\cF_{1}), \wt\ka(\cF_{2})).
\end{eqnarray*}

\item The functor $\wt\ka$ intertwines the actions of the Hecke categories $\dcH^{gr}_{\dG}$ and $\CH_{G}$.


\item The functor $\wt\ka$ sends a simple perverse sheaf in $\dcD^{\dua}_{\dK}$ and to an indecomposable \fmo tilting sheaf in $\CM^{a}_{G_{\RR}}$, and induces a bijection between the sets of isomorphisms classes of such objects.  

\item The functor $\wt\ka$ sends the graded lift of the standard objects (resp. the costandard objects) to the \fmo standard objects (resp. \fmo costandard objects).
\end{enumerate}
\eth

The last property is a new feature of the tilting version of Koszul duality compared to the version in \reft{BV}. It roughly says that the functor $\wt\ka$ preserves the relative size of support.

As a consequence of Koszul duality, we deduce that the ranks of stalks of the \fmo tilting sheaves $\cT_{\l,\chi}$ are given by the Lusztig-Vogan polynomials of the dual symmetric pair $(\dG,\dK)$, see \refc{LV poly}.

\ssec{}{Organization of the paper}
In \refs{compmoncat} and \refs{tilt} we develop a general theory of free-monodromic tilting sheaves on real analytic varieties stratified by group orbits such that each orbit is equivariantly homotopy equivalent to a torus. It requires an extension of the construction of completed monodromic categories (see \cite[Appendix A]{BY} and \cite{BR}) and a new $t$-structure where these tilting sheaves live.

\refs{realgp} to \refs{loc} are devoted to the study of free-monodromic tilting sheaves on the flag varieties stratified by real group orbits. The discussions in these sections are valid for all real groups. After reviewing some background on real groups in \refs{realgp}, we apply the theory of tilting sheaves developed in \refs{tilt} to the flag variety stratified by real group orbits in \refs{Matsuki}.  We prove that the Matsuki equivalence of \cite{MUV} is a Ringel duality (see \reft{Mat}). In \refs{Hk} we investigate the behavior of tilting sheaves under the Hecke action. In \refs{loc} we develop the technique of localization for completed monodromic categories, which will be used in the proof of the structure theorem later. 

Starting from \refs{realV} we restrict to quasi-split groups.  In \refs{realV}, we construct the real Soergel functor $\VV_{\RR}$ for quasi-split groups. In \refs{str nice} we prove our Soergel's structure theorem for quasi-split groups satisfying a technical condition called {\em nice}, which includes adjoint quasi-split groups.  In \refs{nonadj} we extend the structure theorem to all quasi-split groups.

Finally, in \refs{KD} we prove \reft{intro KD} by constructing a dg-model for the category $\CM_{G_\BR}$ using the structure theorem we proved. We also derive the numerical result on stalks of \fmo tilting sheaves, and prove that the real Soergel functor is copresentated by a ``big tilting'' object. 

\ssec{}{Acknowledgements} We learned the idea of considering $G_{\RR}$-equivariant tilting sheaves on the flag variety from Roman Bezrukavnikov. We thank Kari Vilonen and David Vogan for teaching us about real groups and answering our questions. We thank Simon Riche for pointing out a mistake.

A.I. was funded by RFBR, project number 19-31-90078. Z.Y. is supported partially by the Simon Foundation and the Packard Foundation.

\sec{compmoncat}{Completed monodromic category}

Completed monodromic categories were defined in \cite[Appendix A]{BY} in order to make sense of free-monodromic local systems and tilting sheaves on enhanced flag varieties $G/U$. In \cite{BR}, this construction has been adapted to the topological context allowing an arbitrary coefficient field. The rough idea in both cases is to take certain pro-objects in $D^{b}(\wt X, \bfk)$ that include local systems with unipotent monodromy with infinite Jordan block along the fibers of $\pi$. In these works, each stratum has the homotopy type of a torus of the same dimension. 

For the purpose of this paper we need to extend the construction of completed monodromic categories to the case where different strata are (equivariantly) homotopy equivalent to tori of different dimensions. The ``size'' of the free-monodromy will then vary on different strata. We will show that the completion construction still gives a well-behaved category of sheaves.


Throughout the paper, let $\bfk$ be an algebraically closed field of characteristic $0$. All sheaves will be assumed to be the sheaves of $\bfk$-vector spaces. All (co)homology groups are taken with $\bfk$ coefficients unless otherwise stated.

\ssec{mon}{The setup}
Let $X$ be a real analytic variety (for example the real points of a scheme of finite type over $\RR$). Let $T^c$ be a {\em compact} torus. Let $\pi:\wt X\to X$ be a principal right $T^c$-bundle. 

Let $H$ be a Lie group with an  analytic action on $\wt X$ from the left commuting with $T^c$. Then there is an induced $H$-action on $X$ such that $\pi$ is $H$-equivariant.



\rem{flags}
In application, we consider $X=G/B$ to be the flag manifold of a complex reductive group $G$. Let $Y=G/U$ (where $U$ is the unipotent radical of $B$), then $Y\to X$ is a $T$-torsor, where $T=B/U$ is the abstract Cartan of $G$.  The canonical decomposition $\CC^{\times}=\RR^{>0}\times S^{1}$ gives a canonical decomposition $T=T^{>0}\times T^c$ where $T^{>0}=\RR^{>0}\ot_{\ZZ}\XX_{*}(T)$ and $T^c=S^{1}\ot_{\ZZ}\XX_{*}(T)$, a compact torus. We then let $\wt X=Y/T^{>0}$. Since $T^{>0}$ is contractible, the pullback functor $D^{b}(\wt X)\to D^{b}(Y)$ is fully faithful, so if we are interested only in sheaves on $Y$ monodromic  under the right $T$-action, we may equivalently consider sheaves on $\wt X$ monodromic under the right $T$-action.
\erem

\ssec{}{Completion} Consider the adjoint functors
\begin{equation*}
\xymatrix{D^{b}_{H}(\wt X)\ar@<.5ex>[r]^{\pi_{!}} & D^{b}_{H}(X) \ar@<.5ex>[l]^{\pi^{!}}}.
\end{equation*}
Let $D^{b}_{H}(\wt X)_{T^c\mon}\subset D^{b}_{H}(\wt X)$ be the full subcategory generated by the image of $\pi^{!}$.  

Let $\cR$ be the completion of the group algebra $\bfk[\pi_{1}(T^c)]$ at the augmentation ideal. Equip $\cR$ with the adic topology coming from the augmentation ideal. Then $\cR$ is a complete regular local ring with residue field $\bfk$ and its cotangent space canonically isomorphic to $H_{1}(T^c,\bfk)=\pi_{1}(T^c)\ot_{\ZZ} \bfk$. The monodromy action along the fibers of $\pi$ gives an $\cR$-linear structure on $D^{b}_{H}(\wt X)_{T^c\mon}$, namely $\cR$ acts on the identity functor of $D^{b}_{H}(\wt X)_{T^c\mon}$.

Following \cite[A.3]{BY}, let $\wh D^{b}_{H}(\wt X)_{T^c\mon}$ be the category of pro-objects in $D^{b}_{H}(\wt X)_{T^c\mon}$ indexed by positive integers. Let $\wt D^{b}_{H}(\wt X)_{T^c\mon}\subset \pro D^{b}_{H}(\wt X)_{T^c\mon}$ be the full subcategory consisting of  pro-objects $(\cF_{n})_{n\ge0}$ that satisfy two conditions:
\begin{enumerate}
\item ($\pi$-constancy) The pro-object $(\pi_{!}\cF_{n})_{n}\in \pro D^{b}_{H}(X)$ lies in the essential image of the natural functor $D^{b}_{H}(X)\to \pro D^{b}_{H}(X)$ consisting of constant pro-objects.
\item (uniform boundedness) $(\cF_{n})_{n}$ is isomorphic to a pro-object $(\cF'_{n})_{n}$ in $\pro D^{b}_{H}(\wt X)_{T^c\mon}$ where each $\cF'_{n}$ has perverse degrees in $[-N,N]$ for some $N>0$ independent of $n$.
\end{enumerate}
It is proved in \cite[Theorem A.3.2]{BY} that $\wh D^{b}_{H}(\wt X)_{T^c\mon}$ is an $\cR$-linear triangulated category. By the $\pi$-constancy of objects in $\wh D^{b}_{H}(\wt X)_{T^c\mon}$,  we have an adjunction induced from the adjunction $(\pi_{!},\pi^{!})$
\begin{equation*}
\xymatrix{\wh D^{b}_{H}(\wt X)_{T^c\mon}\ar@<.5ex>[r]^-{\wh\pi_{!}} & D^{b}_{H}(X) \ar@<.5ex>[l]^-{\wh\pi^{!}}}
\end{equation*}
It will be convenient to introduce adjunctions $(\pi_{\da}, \pi^{\da})$ between the same categories:
\begin{equation}\label{pi da}
\pi_{\da}:=\wh \pi_{!}[\dim T^{c}], \quad \pi^{\da}:=\wh \pi_{!}[-\dim T^{c}]\cong \wh\pi^{*}.
\end{equation}

Moreover, given an $H$-equivariant map $f:X\to Y$ that lifts to an $H$-equivariant map $\wt f: \wt X\to \wt Y$ of $T^c$-torsors over $X$ and $Y$, under the assumption that $H$ has finitely many orbits on $X$ and $Y$, the functors $\wt f^{*}, \wt f_{*}, \wt f_{!},\wt f^{!}$ and their adjunctions induce functors $\wh f^{*}, \wh f_{*}, \wh f_{!},\wh f^{!}$ between the completed categories $\wh D_{H}(\wt X)_{T^c\mon}$ and $\wh D_{H}(\wt Y)_{T^c\mon}$.

\ssec{Apt}{Situation over a point} Consider the special case where $X=\pt$, $\wt X=T^c$, and $H=T^c_{1}$ is a closed subgroup of $T^c$ acting on $\wt X$ by left translation.  We shall give an algebraic description of the completed category $\wh D^{b}_{T^c_{1}}(T^c)_{T^c\mon}$.

Let $(T^c_{1})^{\c}$ be the neutral component of $T^c_{1}$, and let $\ol T^c=T^c/(T^c_{1})^{\c}$, the quotient torus.  Then $\pi_{0}(T^c_{1})=T^c_{1}/(T^c_{1})^{\c}$ is a finite subgroup of $\ol T^c$.

For each character $\chi: \pi_{0}(T^c_{1})\to \bfk^{\times}$, let $\bfk_{\chi}$ be the one-dimensional $\bfk$-vector space with an $T^c_{1}$-equivariant structure via $\chi$. Let $D^{b}_{T^c_{1}}(\pt)_{\chi}$ be the full subcategory of $D^{b}_{T^c_{1}}(\pt)$ consisting of objects $\cF$ such that the action of $T^c_{1}$ on $H^{i}\cF$ (an $\bfk$-vector space) is via $\chi$ (pulled back to $T^c_{1}$).  Then we have a decomposition
\begin{equation*}
D^{b}_{T^c_{1}}(\pt)=\bigoplus_{\chi: \pi_{0}(T^c_{1})\to \bfk^{\times}} D^{b}_{T^c_{1}}(\pt)_{\chi}.
\end{equation*}
Tensoring with the rank $1$ local system corresponding to $\chi$ gives an equivalence $D^{b}_{T^c_{1}}(\pt)_{1}\isom D^{b}_{T^c_{1}}(\pt)_{\chi}$.

Similarly, let $D^{b}_{T^c_{1}}(T^c)_{T^c\mon,\chi}$ be the full triangulated subcategory generated by the image of $D_{T^c_{1}}(\pt)_{\chi}$ under the pullback $\pi^{!}: D_{T^c_{1}}(\pt)\to D^{b}_{T^c_{1}}(T^c)_{T^c\mon}$. Again we have a decomposition
\begin{equation*}
D^{b}_{T^c_{1}}(T^c)_{T^c\mon}\cong \bigoplus_{\chi: \pi_{0}(T^c_{1})\to \bfk^{\times}} D^{b}_{T^c_{1}}(T^c)_{T^c\mon,\chi}.
\end{equation*}
Let $\uk_{\chi}$ be the rank one $T^c_{1}$-equivariant local system on $T^c$ with monodromy action given by $\chi$ (whose underlying local system is trivial).

\lem{monA}
\begin{enumerate}
\item The forgetful functor $D^{b}_{T^c_{1}}(T^c)_{T^c\mon}\to D^{b}_{(T^c)_{1}^{\c}}(T^c)_{T^c\mon}$ induces an equivalence
\begin{equation*}
D^{b}_{T^c_{1}}(T^c)_{T^c\mon,\chi}\isom D^{b}_{(T^c)_{1}^{\c}}(T^c)_{T^c\mon}.
\end{equation*}

\item Let $\s: T^c\to \ol T^c$ be the projection. Then $\s^{*}$ induces an equivalence of categories
\begin{equation*}
\s^{*}: D^{b}_{(T^c_{1})^{\c}}(T^c)_{T^c\mon}\cong D^{b}(\ol T^c)_{\ol T^c\mon}.
\end{equation*}
\item Let $\ol \cR$ be the completion of $\bfk[\pi_{1}(\ol T^{c})]$ at the augmentation ideal. We have an equivalence
\begin{equation*}
D^{b}(\ol\cR\lmod_{\nil}) \cong D^{b}(\ol T^c)_{\ol T^c\mon}.
\end{equation*}
Here $\ol\cR\lmod_{\nil}$ is the category of $\ol\cR$-modules of finite dimension over $\bfk$. It sends $M\in \ol\cR\lmod_{\nil}$ to the local system $\cL_{M}$ on $\ol T^c$ whose stalk at $1\in \ol T^c$  is $M$ and the monodromy representation of $\pi_{1}(\ol T^c)$ is given by the $\ol\cR$-module structure on $M$. 

\item Combining (1)(2)(3) we get an equivalence
\begin{equation}\label{DA}
D^{b}_{T^c_{1}}(T^c)_{T^c\mon}\cong \bigoplus_{\chi: \pi_{0}(T^c_{1})\to \bfk^{\times}} D^{b}(\ol\cR\lmod_{\nil}).
\end{equation}
\end{enumerate}
For a collection of finite-dimensional $\ol\cR$-modules $(M_{\chi})_{\chi}$ indexed by characters $\chi: \pi_{0}(T^c_{1})\to \bfk^{\times}$, the equivalence \eqref{DA} sends $\op M_{\chi}$ on the right side to $\op (\s^{*}\cL_{M_{\chi}}\ot \uk_{\chi})\in D^{b}_{T^c_{1}}(T^c)_{T^c\mon}$.
\elem
\begin{proof}
For (3) see \cite[Corollary A.4.7(1)]{BY} or \cite[Lemma 4.1]{BR}. The rest of the lemma is clear.
\end{proof}

We also have a description of $D^{b}_{T^c_{1}}(\pt)$ following \cite{GKM} as modules over the homology algebra of the torus $(T^{c}_{1})^{\c}$.  Let
\begin{equation*}
\L_{\bu}=H_{*}((T^c_{1})^{\c}, \bfk)
\end{equation*}
as a graded algebra in degrees $0,-1,\cdots, -\dim T^{c}_{1}$.

\lem{GKM}
\begin{enumerate}
\item The forgetful functor $D^{b}_{T^c_{1}}(\pt)\to D^{b}_{(T^c)_{1}^{\c}}(\pt)$ restricts to an equivalence for each $\chi$:
\begin{equation*}
D^{b}_{T^c_{1}}(\pt)_{\chi}\isom D^{b}_{(T^c_{1})^{\c}}(\pt).
\end{equation*}

\item Taking the stalk induces an equivalence 
\begin{equation*}
\RG(\pt, -): D^{b}_{(T^c_{1})^{\c}}(\pt)\cong D^{f}_{+}(\L_{\bu}\lmod).
\end{equation*}
Here $D^{f}_{+}(\L_{\bu}\lmod)$ denotes the full subcategory of the bounded below derived category of dg $\L_{\bu}$-modules with cohomology finitely generated over $\L_{\bu}$.

\item  Combining (1)(2), there is an equivalence
\begin{equation}\label{eq pt}
D^{b}_{T^c_{1}}(\pt)\cong \bigoplus_{\chi: \pi_{0}(T^c_{1})\to \bfk^{\times}} D^{b}(\L_{\bu}\lmod).
\end{equation}
\end{enumerate}
\elem
\prf
(2) follows from \cite[Theorem 11.2]{GKM}. The rest of the statements are clear.
\epr

\lem{comp DA}
\begin{enumerate}
\item The equivalence \eqref{DA} extends to a canonical equivalence
\begin{equation}\label{comp DA}
\wh D^{b}_{T^c_{1}}(T^c)_{T^c\mon}\cong \bigoplus_{\chi: \pi_{0}(T^c_{1})\to \bfk^{\times}}  D^{f}(\ol\cR\lmod)
\end{equation}
where $D^{f}(\ol\cR\lmod)$ is the bounded derived category of finitely generated $\ol\cR$-modules. 

\item Under the equivalences \eqref{comp DA} and \eqref{eq pt}, the adjunction $(\pi_{\da}, \pi^{\da})$ (see \eqref{pi da}) can be identified with the composition of adjunctions
\begin{equation*}
\xymatrix{\bigoplus_{\chi}  D^{f}(\ol\cR\lmod)\ar@<.5ex>[r]^-{\bfk\Lot_{\ol\cR}(-)} & \bigoplus_{\chi} D^{f}(\bfk\lmod)\ar@<.5ex>[r]^-{\L_{\bu}\ot_{\bfk}(-)}\ar@<.5ex>[l]^-{\textup{forg}}
& \bigoplus_{\chi} D^{f}_{+}(\L_{\bu}\lmod) \ar@<.5ex>[l]^-{\textup{forg}}}
\end{equation*}
Here both right adjoints are the forgetful functors for the ring homomorphisms $\ol\cR\to \bfk$ (augmentation) and $\bfk\to \L_{\bu}$.

\end{enumerate}

%
\elem
\begin{proof} 
(1) Let $\s: T^c\to \ol T^c$ be the projection.  Combining the equivalences in \refl{monA}(1)(2) we have an equivalence
\begin{equation*}
\Phi: \bigoplus_{\chi} D^{b}(\ol T^c)_{\ol T^c\mon}\isom D^{b}_{T^c_{1}}(T^c)_{T^c\mon}
\end{equation*}
given by sending $(\cF_{\chi})_{\chi}$ to $\op\s^{*}\cF_{\chi}\ot \uk_{\chi}$. Passing to pro-objects we get an equivalence $\pro(\Phi)$. We claim that $\pro(\Phi)$ restricts to an equivalence of full subcategories
\begin{equation*}
\wh\Phi: \bigoplus_{\chi}\wh D^{b}(\ol T^c)_{\ol T^c\mon}\isom \wh D^{b}_{T^c_{1}}(T^c)_{T^c\mon}.
\end{equation*}
Note here the completions on the two sides are with respect to different torus actions. Let $\cF_{\chi}=(\cF_{\chi,n})_{n\ge0}\in \pro D^{b}(\ol T^c)_{\ol T^c\mon}$. We need to show that each $\cF_{\chi}$ satisfies the two conditions  defining $\wh D^{b}(\ol T^c)_{\ol T^c\mon}$ if and only if the pro-object $\op_{\chi}\uk_{\chi}\ot\s^{*}\cF_{\chi,n}$ satisfies the two conditions  defining $\wh D^{b}_{T^c_{1}}(T^c)_{T^c\mon}$. This easily reduces to check the same statement for $\chi=1$: i.e., $(\cF_{n})_{n}$ satisfies $\ol\pi$-constancy (where $\ol \pi: \ol T^c\to \pt$) and  if and only if $(\s^*\cF_{n})_{n}$ satisfies $\pi$-constancy and uniform boundedness. Since $\s^*$ is $t$-exact up to a shift, the equivalence of uniform boundedness is clear. Now
\begin{equation}\label{pibar}
\pi_{!}\s^*\cF_{n}=\ol\pi_{!}\s_{!}\s^*\cF_{n}\cong (\ol \pi_{!}\cF_{n})\ot H^{*}_{c}((T^c_{1})^{\c},\bfk).
\end{equation}
Therefore, $(\pi_{!}\s^{*}\cF_{n})_{n}$ is isomorphic to a constant object in $\pro D^{b}_{T^c_{1}}(\pt)$ if and only if $(\ol\pi_{!}\cF_{n})_{n}$ is isomorphic to a constant object in $\pro D^{b}(\pt)$.

By \cite[Corollary A.4.7(2)]{BY} or \cite[Corollary 4.6]{BR}, we have $D^{b}(\ol T^c)_{\ol T^c\mon}\cong D^{f}(\ol\cR\lmod)$. Combining with $\wh\Phi$ it gives the equivalence \eqref{comp DA}.

(2) By tensoring with $\uk_{\chi}$ the case for general $\chi$ reduced to the case of trivial character $\chi$. In this case we need to describe the functor
\begin{equation*}
\pi_{\da}=\wh\pi_{!}[\dim T^{c}]: \wh D^{b}_{(T^{c}_{1})^{\c}}(T^{c})\to \wh D^{b}_{(T^{c}_{1})^{\c}}(\pt).
\end{equation*}
By \refl{monA}, for $M\in D^{f}(\ol\cR\lmod)$, the corresponding object in $\wh D^{b}_{(T^{c}_{1})^{\c}}(T^{c})$ is $\s^{*}\cL_{M}$. Then \eqref{pibar} implies
\begin{eqnarray*}
\pi_{\da}\s^{*}\cL_{M}&\cong& \wh\pi_{!}\s^{*}\cL_{M}[\dim T^{c}]\\
&=&(\wh{\ol \pi}_{!}\cL_{M})[\dim \ol T^{c}]\ot H^{*}_{c}((T^c_{1})^{\c},\bfk)[\dim T_{1}^{c}]\\
&\cong & \ol\pi_{\da}\cL_{M}\ot \L_{\bu}.
\end{eqnarray*}
Here $\ol\pi_{\da}\cL_{M}\in D^{b}(\pt)\cong D^{f}(\bfk\lmod)$.  It is well-known that $ \ol\pi_{\da}\cL_{M}\cong \bfk\Lot_{\ol \cR}M$. Therefore
\begin{equation*}
\pi_{\da}\s^{*}\cL_{M}\cong (\bfk\Lot_{\ol \cR}M)\ot \L_{\bu}.
\end{equation*}
\end{proof}

 

\sec{tilt}{Real tilting exercises}
 
In this section we develop the theory of (free-monodromic) tilting sheaves on spaces stratified by group orbits in the style of \cite{BBM}. A key assumption that ensures the existence of tilting sheaves is a certain cohomological bound for the links between strata.

\ssec{set}{The setting} 
We are back to the setup of \refss{mon}. Namely,  let $X$ be a real analytic variety. Let $T^c$ be a {\em compact} torus. Let $\pi:\wt X\to X$ be a principal right $T^c$-bundle. Let $H$ be a Lie group with an  analytic action on $\wt X$ from the left commuting with $T^c$. We assume additionally that the action of $H$ on $X$ has finitely many orbits $\{X_\l\}_{\l\in I}$, which then gives a Whitney stratification of $X$. We put $\wt X_\l=\pi^{-1}(X_\l)$. Let $i_\l\colon X_\l\to X$ and $\wt{i}_\l\colon \wt X_\l\to \wt X$ be the inclusions. We say $\l\le\mu$ if $X_\l\subset\overline{X}_\mu$. Put $d_\l:=\dim X_\l$.

Let $X_{\l}\subset X$ be an $H$-orbit, and $\wt X_{\l}=\pi^{-1}(X_{\l})$. We choose a point $x_{\l}\in X_{\l}$. Let $H_{x_{\l}}$ be its stabilizer in $H$. Then $H$ acts on the fiber $\pi^{-1}(x_{\l})$ commuting with the right $T^c$-action. This defines a homomorphism $\ph_{x_{\l}}: H_{x_{\l}}\to T^c$ such that the action of $h\in H$ on $\pi^{-1}(x_{\l})$ is by right translation by $\ph_{x_{\l}}(h)^{-1}$. Let $T^c_{x_{\l}}\subset T^{c}$ be the image of $\ph_{x_{\l}}$. Changing the choice of $x_{\l}$ changes $\ph_{x_{\l}}$ by $H$-conjugation. Since $T^c$ is abelian, $T^c_{x_{\l}}$ stays the same. Therefore $T^c_{x_{\l}}$ is independent of the choice of $x_{\l}$, and we denote it by $T^c_{\l}$. 

For $\l\in I$, let $\LS_{\l}$ denote the set of isomorphism classes of irreducible $H$-equivariant local systems on $X_{\l}$ with $\bfk$-coefficients. Then we have a canonical bijection
\begin{equation}\label{LS Hom}
\LS_{\l}\cong \Hom(\pi_0(T^c_\l),\bfk^{\times}).
\end{equation}
Denote
$$
\wt I=\{(\l,\chi)| \l\in I, \chi\colon\pi_0(T^c_\l)\to\bfk^\x \mbox{ is a character}\}.
$$
For $(\l,\chi)\in \wt I$, let $\uk_{\l,\chi}\in \LS_{\l}$ be the corresponding rank one local system on $X_{\l}$.


\ssec{assump}{Assumptions}
We assume:
\begin{equation}\label{stabilizer}
\parbox{10cm}{The subgroup $T^c_{\l}\subset T^c$ is closed, and  $\ker(H_{x_\l}\to T^c_\l)$ is contractible. }
\end{equation}
In particular, the identity component $(T^c_{\l})^{\c}$ of $T^c_{\l}$ is a compact torus.  For $\l\in I$, we put
$$\ol T^{c}_{\l}=T^{c}_{\l}/(T^c_{\l})^{\c}, \quad d_{\l}=\dim X_{\l}, \quad n_\l=\dim T^c- \dim T^c_\l=\dim \ol T^{c}_{\l}. $$
We impose the following parity condition: 
\begin{equation}\label{parity}
\parbox{10cm}{The parity of the numbers $d_\l+n_\l$ is the same for all $\l\in I$.}
\end{equation}

For $\l< \mu\in I$ and $x_{\l}\in X_{\l}$, let $L^{\mu}_{x_{\l}}$ be the link of $X_{\l}$ in $X_{\mu}$ at $x_{\l}$. More precisely, let $D_{x_{\l}}$ be a sufficiently small transversal slice to $X_{\l}$ at $x_{\l}$, and take $L_{x_{\l}}^{\mu}=D_{x_{\l}}\cap X_{\mu}$. Then $L^{\mu}_{x_{\l}}$ is a smooth manifold of dimension $d_{\mu}-d_{\l}$, well-defined up to diffeomorphism. 


We assume that for any $(\mu,\chi)\in \wt I$ and any $\l<\mu$, we have
\begin{equation}\label{upper van} 
H^i(L_{x_\l}^\mu, \uk_{\mu,\chi})=0 \mbox{  for } i>\frac{1}{2}(d_\mu+n_\mu-d_\l-n_\l).
\end{equation}
Here we use $\uk_{\mu,\chi}$ to denote the restriction of $\uk_{\mu,\chi}$ to $L^{\mu}_{x_{\l}}$.  

Since $L_{x_\l}^\mu$ is a smooth manifold of dimension $d_{\mu}-d_{\l}$, by Poincar\'e duality \eqref{upper van} is equivalent to the following bound for all $\chi: \pi_{0}(T^{c}_{\mu})\to\bfk^{\x}$
\begin{equation}\label{lower van} 
H^i_c(L_{x_\l}^\mu,  \uk_{\mu,\chi})=0 \mbox{ for } i<\frac{1}{2}(d_\mu-n_\mu-d_\l+n_\l).
\end{equation}

\rem{} If $n_{\l}=n_{\mu}$, a typical situation where the bounds \eqref{upper van} and \eqref{lower van} hold is when $L^{\mu}_{x_{\l}}$ is diffeomorphic to a Stein manifold (e.g. smooth affine complex algebraic variety) of complex dimension $\frac{1}{2}(d_\mu-d_\l)$.  If $n_{\mu}>n_{\l}$, then there is a possible overlap of length $n_{\mu}-n_{\l}$ for the nonvanishing degrees of $H^*_{c}(L_{x_\l}^\mu, \uk_{\mu,\chi})$ and $H^*(L_{x_\l}^\mu, \uk_{\mu,\chi})$, which can happen  if $L^{\mu}_{x_{\l}}$ fibers over a Stein manifold of complex dimension $\frac{1}{2}(d_\mu-n_\mu-d_\l+n_\l)$ with fibers compact manifolds of real dimension  $n_{\mu}-n_{\l}$ (e.g. compact torus fibration). On the other hand, if $n_{\mu}<n_{\l}$, then there is a gap of length at least $n_{\l}-n_{\mu}$ between the lowest nonvanishing degree of $H^*_{c}(L_{x_\l}^\mu, \uk_{\mu,\chi})$ and the highest nonvanishing degree of $H^*(L_{x_\l}^\mu, \uk_{\mu,\chi})$, which can happen if $L^{\mu}_{x_{\l}}$ admits a fiber bundle whose total space is diffeomorphic to a Stein manifold of complex dimension $\frac{1}{2}(d_{\mu}-n_{\mu}-d_{\l}+n_{\l})$ and whose fibers are compact manifolds of real dimension $n_{\l}-n_{\mu}$. In our applications, the cohomological bounds hold essentially for these reasons.
\erem

\ssec{t}{A new $t$-structure}
We define a perversity function $p:I\to \ZZ$ by
\begin{equation}\label{def p}
p_{\l}=\lfloor\frac{1}{2}(d_\l + n_\l)\rfloor.
\end{equation}
As in \cite[Section 2.1]{BBD}, it defines a $t$-structure on $D^{b}_{H}(X)$, whose heart we denote by ${}^{p}P_H(X)$.

For $(\l,\chi)\in \wt I$, let $\D_{\l,\chi}$ and  $\Na_{\l,\chi}\in D_{H}^{b}(X)$ be the $!$- and $*$-extensions of the local system $\uk_{\l,\chi}[p_{\l}]$ on $X_{\l}$.

\lem{stand} For $(\mu,\chi)\in \wt I$ and $ \l\le \mu$,  $i^{!}_{\l}\D_{\mu,\chi}$ lies in degrees $\ge -p_{\l}+n_{\l}-n_{\mu}$, and $i^{*}_{\l}\Na_{\mu,\chi}$ lies in degrees $\le -p_{\l}$. In particular $\Na_{\mu,\chi}$ lies in the heart of the $t$-structure ${}^{p}P_{H}(X)$.
\elem
\prf
We first show the statement about $\Na_{\mu,\chi}$. The stalk of $i^{*}_{\l}\Na_{\mu,\chi}$ at $x_{\l}\in X_{\l}$ is $H^{*}(L^{\mu}_{x_{\l}}, \uk_{\mu,\chi})[p_{\mu}]$.  By the cohomological bound \eqref{upper van}, it is concentrated in degrees $\le-p_{\mu}+p_{\mu}-p_{\l}=-p_{\l}$. Since $\Na_{\mu,\chi}$ has vanishing costalks, it lies in the  heart of the $t$-structure ${}^{p}P_{H}(X)$.

For  $i^{!}_{\l}\D_{\mu,\chi}$, we note that $\D_{\mu,\chi}$ is Verdier dual to $\Na_{\mu,\chi'}[d_{\mu}-2p_{\mu}]$ for some $\chi'$. Therefore $i^{!}_{\l}\D_{\mu,\chi}$ is Verdier dual to $i^{*}_{\l}\Na_{\mu,\chi'}[d_{\mu}-2p_{\mu}]$. Since  $i^{*}_{\l}\Na_{\mu,\chi'}[d_{\mu}-2p_{\mu}]$ lies in degrees $\le -p_{\l}-d_{\mu}+2p_{\mu}$, $i^{!}_{\l}\D_{\mu,\chi}$ lies in degrees $\ge -d_{\l}+p_{\l}+d_{\mu}-2p_{\mu}=-p_{\l}+n_{\l}-n_{\mu}$.

\epr

\ssec{singleorbit}{Free-monodromic local systems on an orbit}
We apply results from \refss{Apt} to the situation of a single orbit $H\bs X_{\l}$. We have adjoint functors
\begin{equation*}
\xymatrix{\wh D^{b}_{H}(\wt X_{\l})_{T^c\mon}\ar@<.5ex>[r]^-{\pi_{\l\da}} & D^{b}_{H}(X_{\l}) \ar@<.5ex>[l]^-{\pi^{\da}_{\l}}}
\end{equation*}

Let $\ol\cR_{\l}$ be the completion of the group algebra $\bfk[\pi_{1}(\ol T^c_{\l})]$ at the augmentation ideal. Let $\L_{\l,\bu}=H_{*}((T_{\l}^{c})^{\c}, \bfk)$ be the homology of the torus $(T_{\l}^{c})^{\c}$, viewed as a graded algebra in degrees $0,-1,\cdots, -\dim T_{\l}^{c}$.

\cor{comp orbit} For $\l\in I$,  we have  canonical equivalences
\begin{eqnarray}\label{comp orbit}
\wt \Phi_{\l}: \wh D^{b}_{H}(\wt X_{\l})_{T^c\mon}\cong \bigoplus_{\chi: \pi_{0}(T^c_{\l})\to \bfk^{\times}} D^{f}(\ol\cR_{\l}\lmod),\\
\Phi_{\l}: D^{b}_{H}(X_{\l})\cong \bigoplus_{\chi: \pi_{0}(T^c_{\l})\to \bfk^{\times}} D^{f}_{+}(\L_{\l,\bu}\lmod).
\end{eqnarray}
Under these equivalences, the adjunction $(\pi_{\l\da}, \pi_{\l}^{\da})$ takes the form
\begin{equation*}
\xymatrix{\bigoplus_{\chi}  D^{f}(\ol\cR_{\l}\lmod)\ar@<.5ex>[r]^-{\bfk\Lot_{\ol\cR_{\l}}(-)} & \bigoplus_{\chi} D^{f}(\bfk\lmod)\ar@<.5ex>[r]^-{\L_{\l,\bu}\ot_{\bfk}(-)}\ar@<.5ex>[l]^-{\textup{forg}}
& \bigoplus_{\chi} D^{f}_{+}(\L_{\l,\bu}\lmod) \ar@<.5ex>[l]^-{\textup{forg}}}
\end{equation*}
\ecor
\begin{proof}
Let $\wt x_{\l}\in \wt X_{\l}$ with image $x_{\l}\in X_{\l}$. Then $\pi^{-1}(x_{\l})\cong T^c$ given by the base point $\wt x_{\l}$. Restricting to $\pi^{-1}(x_{\l})$ gives an equivalence $i^{*}: D^{b}_{H}(X_{\l})\isom D^{b}_{H_{x_{\l}}}(T^c)_{T^c\mon}$, which extends to the completions $\wh i^{*}: \wh D^{b}_{H}(X_{\a})\isom \wh D^{b}_{H_{x_{\l}}}(T^c)_{T^c\mon}$. Since $\ker(\ph_{x_{\l}})$ is contractible, we have $D^{b}_{H_{x_{\l}}}(T^c)_{T^c\mon}\cong D^{b}_{T^c_{\l}}(T^c)_{T^c\mon}$ which also extends to completions. Similarly, $D^{b}_{H}(X_{\l})\cong D^{b}_{T^{c}_{\l}}(T^{c})$. It remains to apply \refl{comp DA}. It is easy to check that the equivalence thus defined is independent of the choice of $\wt x_{\l}$.
\end{proof}

In the situation of \refc{comp orbit}, for each  $(\l,\chi)\in \wt I$, we have a {\em free-monodromic} local system $\cL_{\l,\chi}\in \wh D^{b}_{H}(\wt X_{\l})_{T^c\mon}$  that corresponds to the free rank one $\ol\cR_{\l}$-module $\ol\cR_{\l}$ placed in the $\chi$-summand on the right side of \eqref{comp orbit}.



 \ssec{}{Free-monodromic sheaves}
For the triangulated category of free-monodromic sheaves we put $\CM_{H,X}:=\wh D^{b}_{H}(\wt X)_{T^c\mon}$ for short, whenever it does not cause ambiguity. Define a $t$-structure on $\CM_{H,X}$ using the same perversity function $p$ as in \eqref{def p}, whose  heart we denote by $\CP_{H,X}$.


For $\l\in I$ let $\CM_{\l}:=\wh D_{H}(\wt X_{\l})_{T^c\mon}$. We have adjunctions 
\begin{eqnarray*}
\xymatrix{\CM_{\l} \ar@<.5ex>[r]^{\wt i_{\l!}}  & \CM_{H,X} \ar@<.5ex>[l]^{\wt i_{\l}^{!}} } 
\quad 
\xymatrix{\CM_{\l} \ar@<-.5ex>[r]_{\wt i_{\l*}}  & \CM_{H,X} \ar@<-.5ex>[l]_{\wt i_{\l}^{*}} }
\end{eqnarray*}

For $(\l,\chi)\in \wt I$, we have the free-monodromic local system $\cL_{\l,\chi}\in \CM_{\l}$ as defined in  \refss{singleorbit}. We define standard and costandard objects $\tilDel_{\l,\chi}$ and $\tilna_{\l, \chi}$ of $\CM_{H,X}$ as, respectively, the $!$- and $*$-extensions under $\wt i_\l$ of $\CL_{\l,\chi}[p_{\l}]$.

\lem{push std} For any $(\l,\chi)\in \wt I$, we have
\begin{equation*}
\pi_{\da}\wt\D_{\l,\chi}\cong \L_{\l,\bu} \ot \D_{\l,\chi}, \quad \pi_{\da}\wt\Na_{\l,\chi}\cong \L_{\l,\bu} \ot \Na_{\l,\chi}.
\end{equation*}
\elem{}

\begin{proof}
By \refc{comp orbit} we have
\begin{equation*}
\pi_{\l\da}\cL_{\l,\chi}\cong \L_{\l,\bu}\ot \uk_{\l,\chi}.
\end{equation*}
Shifting by $p_{\l}$ and applying $i_{\l!}$ to the above we get
\begin{equation*}
\pi_{\da}\wt\D_{\l,\chi}=i_{\l!}\pi_{\l\da}\cL_{\l,\chi}[p_{\l}]\cong i_{\l!}(\L_{\l,\bu}\ot \uk_{\l,\chi} [p_{\l}])=\L_{\l,\bu} \ot \D_{\l,\chi}.
\end{equation*}

The argument for the costandard sheaf is the same, using that $\pi_{\da}\wt i_{\l*}\cong i_{\l*}\pi_{\l\da}$ because $\wh\pi_{!}=\wh\pi_{*}$ (since $\pi$ is proper).
\end{proof}

\lem{r} Let $(\mu,\chi)\in \wt I$ and $\l<\mu$.
\begin{enumerate}
\item The restriction $\wt i^*_\l\tilna_{\mu,\chi}$ lies in degrees $\le -p_{\l}$.
\item Under the equivalence $\Phi_{\l}$, the corestriction $\wt i^!_\l\tilDel_{\mu,\chi}$ corresponds to a collection $M_{\chi'}\in D^{f}(\ol \cR_{\l}\lmod)$ (where $\chi': \pi_{0}(T_{\l}^{c})\to \bfk^{\times}$) where each $M_{\chi'}$  can be represented by a complex of free $\ol \cR_{\l}$ in degrees  $\ge -p_{\l}$.
\item The standard and costandard objects $\tilDel_{\mu,\chi}$ and $\tilna_{\mu,\chi}$ lie in the heart of the $t$-structure $\CP_{H,X}$.
\end{enumerate}
\elem
\prf
(1)  By \refl{push std}, $\pi_{\l\da}\wt i^*_\l\tilna_{\mu,\chi}\cong  i^{*}_{\l}\pi_{\da}\tilna_{\mu,\chi}\cong \L_{\mu,\bu}\ot  i^*_\l\Na_{\mu,\chi}$.  By \refl{stand}, $i^*_\l\Na_{\mu,\chi}$ lies in degrees $\le -p_{\l}$, hence $\pi_{\l\da}\wt i^*_\l\tilna_\mu$ lies in degrees $\le -p_{\l}$ (note $\L_{\mu,\bu}$ is in non-positive degrees). From  the description of $\pi_{\l\da}$ given in \refc{comp orbit} we see that $\wt i^*_\l\tilna_\mu$ is in degrees $\le -p_{\l}$.


(2) Note the statement is stronger than saying that $M_{\chi'}$ lies in cohomological degrees $\ge -p_{\l}$, but saying that it admits a free resolution (as $\ol\cR_{\l}$-modules) in degrees $\ge-p_{\l}$. 

By \refl{push std}, we have $\pi_{\l\da}\wt i_{\l}^{!}\wt\D_{\mu,\chi}\cong i_{\l}^{!}\pi_{\da}\wt\D_{\mu,\chi}\cong \L_{\mu,\bu}\ot i_{\l}^{!}\D_{\mu,\chi}$. By \refl{stand}, $i_{\l}^{!}\D_{\mu,\chi}$ lies in degrees $\ge-p_{\l}+n_{\l}-n_{\mu}$. Therefore $\pi_{\l\da}\wt i_{\l}^{!}\wt\D_{\mu,\chi}\cong \L_{\mu,\bu}\ot i_{\l}^{!}\D_{\mu,\chi}$ lies in degrees $\ge-p_{\l}+n_{\l}-n_{\mu}-\dim T_{\mu}^{c}=-p_{\l}+n_{\l}-\dim T^{c}=-p_{\l}-\dim  T_{\l}^{c}$.  By \refc{comp orbit}, $\pi_{\l\da}\wt i_{\l}^{!}\wt\D_{\mu,\chi}$ corresponds to $\op_{\chi'}\L_{\l,\bu}\ot (\bfk\Lot_{\ol\cR_{\l}}M_{\chi'})$, hence $\bfk\Lot_{\ol\cR_{\l}}M_{\chi'}$ lies in degrees $-p_{\l}-\dim  T_{\l}^{c}+\dim T_{\l}^{c}=-p_{\l}$ (the lowest degree of $\L_{\l,\bu}$ is $-\dim T_{\l}^{c}$). This implies that $M_{\chi'}$ admits a free resolution (as $\ol\cR_{\l}$-modules) in degrees $\ge-p_{\l}$. 

(3) The statement follows from (1)(2) and the observation that $\wt i^*_\l\tilDel_{\mu,\chi}=0$ and $\wt i^!_\l\tilna_{\mu,\chi}=0$.

\epr

\ssec{tilt}{Tilting sheaves} 

We are in the situation of \refss{set}-\refss{assump}.

\defe{tilt}
An object $\CT$ of $\CM_{H,X}$ is called a {\em free-monodromic tilting sheaf}, if for each $\l\in I$, both complexes $\wt i_\l^*\CT$ and $\wt i_\l^!\CT$ are free-monodromic local systems in degree $-p_{\l}$. 
\edefe
From the definition, we see that an object $\CT$ of $\CM_{H,X}$ is a free-monodromic tilting sheaf if and only if $\CT\in \CP_{H,X}$ and $\CT$ has a $\tilDel$-flag and $\tilna$-flag, i.e. it is both a successive extension of $\tilDel_{\l,\chi}$'s and  a successive extension of $\tilna_{\l,\chi}$'s in $\CP_{H,X}$.

We will denote by $\Tilt(\CM_{H,X})\subset\CM_{H,X}$ the full additive subcategory free-monodromic tilting sheaves. 

\lem{mapTT} Let $\cT,\cT'\in \Tilt(\CM_{H,X})$.
\begin{enumerate}
\item The complex $\RHom_{\CM_{H,X}}(\cT,\cT')$ is concentrated in degree $0$, and $\Hom_{\CM_{H,X}}(\cT,\cT')$ has an increasing filtration indexed by the poset $I$ with associated graded $\Hom_{\CM_{H,X_{\l}}}(\wt i_{\l}^{*}\cT, \wt i_{\l}^{!}\cT')$, which is finite free over $\ol\cR_{\l}$.
\item In particular, if $\wt X_{\l}$ is open in the intersections of the supports of both $\cT$ and $\cT'$, then there is a natural surjection $\Hom_{\CM_{H,X}}(\cT,\cT')\surj \Hom_{\CM_{H,X_{\l}}}(\wt i_{\l}^{*}\cT,\wt i_{\l}^{*}\cT')$.
\end{enumerate} 
\elem
\prf (1) Write $\cT$ as a $\wt\D$-flag and $\cT'$ as a $\wt\nabla$-flag. The spectral sequence calculating $\RHom_{\CM_{H,X}}(\cT,\cT')$ using these filtrations gives a filtration on $\RHom_{\CM_{H,X}}(\cT,\CT')$ with associated graded $\RHom_{\CM_{H,X_{\l}}}(\wt i_{\l}^{*}\cT,\wt i_{\l}^{!}\cT')$, indexed by the poset $\l\in I_{\cT,\cT'}\subset I$, where $I_{\cT,\cT'}$ consists of $\l\in I$ such that $\wt X_{\l}$ is in the support of both $\cT$ and $\cT'$. By \refc{comp orbit}, $\RHom_{\CM_{H,X_{\l}}}(\wt i_{\l}^{*}\cT,\wt i_{\l}^{!}\cT')\cong \op_{\chi:\pi_{0}(T^{c}_{\l})\to \bfk^{\x}}\RHom_{\ol\cR_{\l}}(M_{\chi},M'_{\chi})$ where $(M_{\chi})$ and $(M'_{\chi})$ are free $\ol \cR_{\l}$-modules corresponding to $\wt i_{\l}^{*}\cT$ and $\wt i_{\l}^{!}\cT'$. In particular, $\RHom_{\CM_{H,X_{\l}}}(\wt i_{\l}^{*}\cT,\wt i_{\l}^{!}\cT')$ is concentrated in degree $0$ and is finite free over $\ol\cR_{\l}$. 

(2) Under the assumption  $\l$ is a maximal element in $I_{\cT,\cT'}$. The filtration in (1) can be arranged to have the last quotient $\Hom_{\CM_{H,X_{\l}}}(\wt i_{\l}^{*}\cT,\wt i_{\l}^{!}\cT')$.
\epr

\prop{tilt}
\begin{enumerate}
\item For each $(\l,\chi)\in \wt I$ there exists an indecomposable free-monodromic tilting sheaf $\CT_{\l,\chi}$ whose restriction to $\wt X_{\l}$ is $\CL_{\l,\chi}[p_{\l}]$ and whose support is the closure of $\wt X_\l$. Moreover, such $\CT_{\l,\chi}$ is unique up to isomorphism.
\item The map $(\l,\chi)\mapsto \cT_{\l,\chi}$ is a bijection between $\wt I$ and the set of isomorphism classes of indecomposable objects in $\mathrm{Tilt}(\CM_{H,X})$.
\item Every object $\CT\in\Tilt(\CM_{H,X})$ is isomorphic to a finite direct sum of $\cT_{\l,\chi}$, and the multiplicity of each $\cT_{\l,\chi}$ is an invariant of $\cT$.
\end{enumerate}
\eprop
\prf
(1) We first show the existence of $\cT_{\l,\chi}$. Proceeding by the descending induction on strata we may assume that $Z=X_{\mu}$ is a minimal stratum of $X$ and on the preimage $\wt U=\pi^{-1}(U)$ of its complement $U=X-Z$ there is an indecomposable free-monodromic tilting sheaf $\CT_U$ satisfying the required conditions. Let $\wt j\colon \wt U\to \wt X$ and $\wt i\colon \wt Z=\pi^{-1}(Z) \to \wt X$ be the inclusions. 

Let $\cC=\wt i^{*}\wt j_{*}\CT_{U}$ and $(M_{\chi'})_{\chi'}=\Phi_{\mu}(\cC)\in D^{f}(\ol\cR_{\mu}\lmod)$. Since $\CT_U$ has a $\tilna$-flag,  by \refl{r}, $H^{i}M_{\chi'}=0$ for $i> -p_{\mu}$. Since $\CT_U$ has a $\wt\D$-flag and $\cC[-1]\cong \wt i^{!}\wt j_{!}\CT_{U}$, then by \refl{r}, each $M_{\chi'}[-1]$ can be represented by a bounded complex of free $\ol\cR_{\mu}$-modules in degrees $\ge -p_{\mu}$, therefore $M_{\chi'}$ can be represented by a perfect complex of $\ol\cR_{\mu}$-modules in degrees $\ge -p_{\mu}-1$. Combining these, we see that $M_{\chi'}$ is quasi-isomorphic to a two-step complex of free $\ol\cR_{\mu}$-modules in degrees $-p_{\mu}-1$ and $-p_{\mu}$. Correspondingly, $\cC$ is quasi-isomorphic to a two-step complex $[\CA\xr{\ph} \CB]$ of \fmo local systems on $Z$ in degrees $-p_{\mu}-1$ and $-p_{\mu}$ respectively. We have a map $\cC\to \cA[p_{\mu}+1]$.


As in \cite{BBM} we now put $\CT\in \CP_{H,X}$ to be the extension
\begin{equation}\label{T*}
0\to \wt i_*\CA[p_{\mu}]\to\CT\to \wt j_*\CT_U\to 0
\end{equation}
defined by the map $\wt j_{*}\CT_{U}\to \wt i_{*}\cC\to \wt i_{*}\CA[p_{\mu}+1]$. From this we see that
\begin{equation}\label{TA}
\wt i^!\CT\cong \CA[p_{\mu}].
\end{equation}
Applying $\wt i^{*}$ to the exact sequence \eqref{T*} we see that $\cA[p_{\mu}]\to \wt i^{*}\CT\to \cC$ is a distinguished triangle hence
\begin{equation}\label{TB}
\wt i^{*}\CT\cong \CB[p_{\mu}].
\end{equation}
Therefore we also have an exact sequence
\begin{equation*}
0\to \wt j_!\CT_U\to \CT\to \wt i_*\CB\to 0
\end{equation*}
in $\CP_{H,X}$. From \eqref{TA} and \eqref{TB} and the fact that $\CT_{U}$ is \fmo tilting on $U$ we conclude that $\CT$ is a \fmo tilting sheaf on $X$.

To ensure $\cT$ is indecomposable, we may represent $\cC$ by a two-step complex $[\cA\xr{\ph}\cB]$ of \fmo local systems such that $\rk_{\ol\cR_{\mu}}\cA+\rk_{\ol\cR_{\mu}}\cB$ is minimal. This is equivalent to requiring that $\ol\ph: \cA/\fm\to \cB/\fm$ be zero, where $\fm\subset \ol\cR_{\mu}$ is the maximal ideal. Now if $\cT$ was decomposable as $\cT_{1}\op \cT_{2}$ where $\cT_{1}|_{\wt U}\ne0$ and $\cT_{2}\ne0$, then since $\cT_{U}$ is indecomposable, we must have $\cT_{2}|_{\wt U}=0$, hence $\cT_{2}$ must be supported on $\wt Z$. This would imply that the complex $[\cT_{2}\xr{\id}\cT_{2}]$ is a direct summand of $[\cA\xr{\ph}\cB]$, which contradicts the minimality of $\rk_{\ol\cR_{\mu}}\cA+\rk_{\ol\cR_{\mu}}\cB$.

Proceeding as above by induction on strata, we have constructed an indecomposable $\cT_{\l,\chi}\in \Tilt(\CM_{H,X})$ satisfying the desired properties. 

We show that:
\begin{equation}\label{T isom ext}
\parbox{10cm}{If $\a: \cT_{\l,\chi}\to \cT_{\l,\chi}$ is an isomorphism when restricted to $\wt X_{\l}$, then it is an isomorphism.}
\end{equation}
We show this by induction on strata. Using notation as above, we need to show that if $\a: \cT\to \cT$ is such that $\a_{U}=\a|_{\wt U}: \cT_{U}\to \cT_{U}$ is an isomorphism, then so is $\a$. Indeed, $\a$ induces an endomorphism of complexes $(a,b): [\cA\xr{\ph}\cB]\to  [\cA\xr{\ph}\cB]$, where $a\in \End_{\ol\cR_{\mu}}(\cA)$ and $b\in \End_{\ol\cR_{\mu}}(\cB)$. Since $\a_{U}$ is an isomorphism, it induce an automorphism on $\cC$. Therefore $(a,b)$ is a quasi-isomorphism. After reduction mod $\fm$, $\ph$ becomes zero, and $(a,b)$ mod $\fm$ induces a quasi-isomorphism of $[\cA/\fm\xr{0}\cB/\fm]$.Therefore $a$ and $b$ are both isomorphisms mod $\fm$, and hence $a$ and $b$ are isomorphisms by Nakayama's lemma. This shows in particular that $\a$ restricts to an automorphism of $\wt i^{*}\cT\cong \cB[p_{\mu}]$. Since $\a_{U}$ is an isomorphism as well, $\a$ is then an isomorphism.

The uniqueness of $\cT_{\l,\chi}$ up to isomorphism will be shown together with part (2). 

To prove (2) and (3), we first prove: 
\begin{equation}\label{summand}
\parbox{10cm}{Let $\cT\in\Tilt(\CM_{H,X})$. Let $\wt X_{\l}$ be open in the support of $\cT$, and suppose $\cL_{\l,\chi}[p_{\l}]$ is a direct summand of $\wt i^{*}_{\l}\cT$.  Then  $\cT_{\l,\chi}$ is isomorphic to a direct summand of $\cT$.}
\end{equation}
By assumption, we have maps $\a: \cL_{\l,\chi}[p_{\l}]\to \wt i^{*}_{\l}\cT$ and $\b:\wt i^{*}_{\l}\cT\to \cL_{\l,\chi}[p_{\l}]$ such that $\b\c\a=\id$. By \refl{mapTT}, $\a$ and $\b$ extend to $\wt\a: \cT_{\l,\chi}\to \cT$ and  $\wt\b: \cT\to\cT_{\l,\chi}$. Since $(\wt\b\c\wt\a)|_{\wt X_{\l}}$ is an isomorphism, by \eqref{T isom ext}, $\wt\b\c\wt\a$ is also an isomorphism. Therefore, $\cT_{\l,\chi}$ is isomorphic to a direct summand of $\cT$. 

(2) Clearly $\cT_{\l,\chi}\cong \cT_{\l',\chi'}$ if $(\l,\chi)\ne(\l',\chi')$. Now let $\cT\in\Tilt(\CM_{H,X})$ be indecomposable and let $\wt X_{\l}, \cL_{\l,\chi}$ be  as in \eqref{summand}. By \eqref{summand}, we must have $\cT\cong \cT_{\l,\chi}$. This shows that $\{\cT_{\l,\chi}\}_{(\l,\chi)\in I}$ exhaust all indecomposable objects in $\Tilt(\CM_{H,X})$.

(3) Extend the partial order on $I$ to a total order, and we prove this statement by induction on the strata. For $\cT\in \Tilt(\CM_{H,X})$, let $\l_{\cT}$ be the maximal element such that $\wt i_{\l}^{*}\cT\ne0$. When $\l_{\cT}$ is the minimal element in $I$, $\cT$ is supported on a closed stratum and the statement is clear. If $\l=\l_{\cT}$ is not minimal, let $\cL_{\l,\chi}[p_{\l}]$ is a direct summand of $\wt i^{*}_{\l}\cT$. By \eqref{summand}, $\cT\cong \cT_{\l,\chi}\op\cT_{1}$ for some $\cT_{1}\in \Tilt(\CM_{H,X})$. Then $\wt i^{*}_{\l}\cT_{1}$ has smaller rank than $\cT_{1}$. Repeating this process, some $\cT_{n}$ will have zero restriction to $\wt X_{\l}$. Therefore $\l_{\cT_{n}}<\l_{\cT}$ and we apply inductive hypothesis to $\cT_{n}$.

\epr

We next prove the functoriality of \fmo tilting sheaves under proper pushforward. The result is not used in this paper.


\prop{prop} 
Let $\wt X\to X$ and $\wt Y\to Y$ be $H$-equivariant $T^{c}$-torsors satisfying the conditions of \refss{assump} \footnote{In fact we don't need to assume the cohomological bounds for links in $X$ and $Y$.}. Let $\wt f\colon \wt X\to \wt Y$ be an $H\times T^{c}$-equivariant proper map.  
Then for any free-monodromic tilting sheaf  $\CT\in\CM_{H,X}$,  $\wt f_*\CT\in \CM_{H,Y}$ is also a free-monodromic tilting sheaf.
\eprop
\prf
Since $\wt f$ is proper, $\wt f_{*}$ commutes with restriction and corestriction to $H\times T^{c}$-orbits, it suffices to assume that both $X$ and $Y$ has a single stratum, so that $X=H/H_x, Y=H/H_y$ and $y=f(x)$, $H_{x}\subset H_{y}$. Let $T^c_x\subset T^c_y\subset T^{c}$ be the images of $H_{x}$ and $H_{y}$.

Let $n_{X}=\dim T^{c}/T^{c}_{x}$, $n_{Y}=\dim T^{c}/T^{c}_{y}$, $d_{X}=\dim X$, $d_{Y}=\dim Y$, $p_{X}=\lfloor\frac{d_{X}+n_{X}}{2}\rfloor$ and $p_{Y}=\lfloor\frac{d_{Y}+n_{Y}}{2}\rfloor$. 

We claim that 
\begin{equation}\label{pXY}
p_{X}-n_{X}=p_{Y}-n_{Y}.
\end{equation}
Indeed, let $U_{x}=\ker(H_{x}\to T^{c}_{x})$ and $U_{y}=\ker(H_{x}\to T^{c}_{x})$, which are contractible Lie groups by assumption. Since $f:X\to Y$ is proper, $H_{y}/H_{x}$ is compact. On the other hand, $H_{x}/H_{y}$ is a fibration over $T^{c}_{y}/T^{c}_{x}$ with contractible fiber $U_{y}/U_{x}$. This implies $U_{x}=U_{y}$, and $H_{y}/H_{x}\cong T^{c}_{y}/T^{c}_{x}$. Hence $d_{X}-d_{Y}=\dim H_{y}-\dim H_{x}=\dim T^{c}_{y}-\dim T^{c}_{x}=n_{X}-n_{Y}$. Therefore $d_{X}-n_{X}=d_{Y}-n_{Y}$, which implies \eqref{pXY}.

Restricting to the fibers over $x$ and $y$ respectively we have a commutative diagram
\begin{equation*}
\xymatrix{ \CM_{H,X}\ar[d]^{\wt f_{*}} \ar[r]^-{ \wt i_{x}^{*}}_-{\sim} & \wh D^{b}_{T_{x}^{c}}(T^{c})_{T^{c}\mon}\ar[d]^{\ph_{*}}\\
\CM_{H,Y}\ar[r]^-{ \wt i_{y}^{*}}_-{\sim} &\wh D^{b}_{T_{y}^{c}}(T^{c})_{T^{c}\mon}
}
\end{equation*}
where $\ph_{*}$ is the induction functor for the inclusion $T^{c}_{x}\subset T^{c}_{y}$. In this situation, we may assume $\wt i_{x}^{*}\cT\in D^{b}_{T_{x}^{c}}(T^{c})_{T^{c}\mon}$ is the shifted \fmo  local system $\cL_{\chi}[p_{X}]$ for some $\chi: \pi_{0}(T^{c}_{x})\to \bfk^{\x}$. The fiber of the quotient map $\ol\ph: T_{x}^{c}\bs T^{c}\to T_{y}^{c}\bs T^{c}$ is isomorphic to $T^{c}_{x}\bs T^{c}_{y}$, a compact Lie group whose neutral component is a torus of dimension $n_{X}-n_{Y}$. Therefore, $\ph_{*}\cL_{\chi}$ is a direct sum of \fmo local systems in degree $n_{X}-n_{Y}$. This implies that $\wt f_{*}\cT$ is a direct sum of \fmo local systems in degree $-p_{X}+n_{X}-n_{Y}$. By \eqref{pXY},  $-p_{X}+n_{X}-n_{Y}=-p_{Y}$, therefore $\wt f_{*}\cT$ is a \fmo tilting sheaf on $\wt Y$.
%
\epr

\sec{realgp}{Preliminaries on real groups}

In this section we collect some facts about real algebraic groups and their orbits on the flag variety. 

\ssec{}{Abstract Cartan and abstract Weyl group} Let $G$ be a connected reductive complex Lie group.  Let $X$ be the flag variety of $G$. 
Let $\bW$ be the {\em abstract Weyl group} of $G$. As a set, it is defined as the set of $G$-orbits on $X\times X$. For $w\in \bW$ let $X^{2}_{w}\subset X\times X$ be the corresponding $G$-orbit. Simple reflections in $\bW$ are those $s\in \bW$ such that $\dim X^{2}_{w}=\dim X+1$. When $(B,B')\in X^{2}_{w}$, we write $pos(B,B')=w$.

Consider the space $Y$ of pairs $T\subset B$ where $T$ is a maximal torus of $G$ and $B$ is a Borel subgroup containing it. Let $\scT$ be the space of maximal tori in $G$. If we choose a maximal torus $T\subset G$, we may identify $Y$ with $G/T$ and $\scT$ with $G/N_{G}(T)$. We get the following diagram where the maps are forgetting $T$ or $B$:
\begin{equation*}
\xymatrix{ X &  Y\ar[l]_-{\b}\ar[r]^-{\g} & \scT.
}
\end{equation*}
Both maps $\b,\g$ are $G$-equivariant. For each $(T\subset B)\in Y$ and $w\in \bW$, there is a unique $(T\subset B^{w})\in Y$ such that $pos(B,B^{w})=w$. This defines a group structure on $\bW$ so that $\g$ becomes a $G$-equivariant $\bW$-torsor.

For different choices of Borel subgroups $B$ and $B'$ of $G$, their reductive quotients are canonically identified, which we call the {\em abstract Cartan} $\bT$ of $G$. There is a canonical right action of $\bW$ characterized as follows: for any $(T\subset B)\in Y$, and $w\in \bW$, the following diagram is commutative
\begin{equation}\label{TT}
\xymatrix{ T\ar@{=}[d]\ar@{^{(}->}[r] & B\ar[r]^{can} & \bT\ar[d]^{(-)w}\\
T \ar@{^{(}->}[r] & B^{w}\ar[r]^{can}  & \bT}
\end{equation}
where the maps ``can'' are the canonical quotients.

For each $(T\subset B)\in Y$ we have the based root system $\Phi(G,B,T)$ where positive roots are those appearing in $B$. For different choices of $(T\subset B)\in Y$ these based root systems are canonically identified with one another. We denote the resulting canonical based root system by $\Phi$. It can be viewed as a based root system for the abstract Cartan $\bT$ with Weyl group $\bW$. Let $\underline\Phi$ be the underlying root system of $\Phi$ (i.e., without the basis).

\ssec{}{Real form}
Let $\sigma\colon G\to G$ be an anti-holomorphic involution on $G$ compatible with the group structure. We put $G_{\RR}=G^\sig$ for the corresponding real form, viewed as an algebraic group over $\RR$. We use  $G(\RR)$ to denote the real points of $G_{\RR}$, as a Lie group. Put $\grg=\mathrm{Lie}(G)$ and $\grg_\BR=\mathrm{Lie}(G_\BR)=\grg^\sig$. 

\lem{fxdpt}
\begin{enumerate}
\item Each Borel subgroup $B\subset G$ contains a $\sigma$-stable maximal torus $T$. 
\item Any two $\sigma$-stable maximal tori in $B$ are conjugate under $G(\RR)\cap B$.

\end{enumerate}

\elem
\prf
(1) The subgroup $H:=B\cap\sig(B)$ of $G$ is stable under $\s$, hence is the complexification of a real group $H_{\RR}\subset G_{\RR}$.  Note that $H$ is solvable and contains a maximal torus of $G$. By \cite[Proposition 7.10]{BS}, $H$ contains a maximal torus $T$ defined over $\RR$, then $T\subset H\subset B$ is $\s$-stable and is a maximal torus of $G$. 

(2) If $T,T'$ are two $\s$-stable maximal tori in $B$, then they are both in $H$, hence they are conjugate by some $u\in H^{u}$ (the unipotent radical of $H$). This implies $u^{-1}\s(u)\in N_{H}(T)\cap H^{u}=\{1\}$ hence $u\in H^{u}(\RR)\subset G(\RR)\cap B$.


\epr


\ssec{}{Real orbits on the flag variety} 

Let  $I$ be the set of $G(\RR)$-orbits on $X$ by left translation. For $\l\in I$, we denote the corresponding orbit by $O^{\RR}_{\l}$, so that
\begin{equation*}
X=\bigcup_{\l\in I}O^{\RR}_{\l}.
\end{equation*}

Let $\scT^{\s}\subset \scT$  be the set of $\s$-stable maximal tori in $G$.  Let $Y_{\s}\subset Y=\g^{-1}(\scT^{\s})$ whose points are pairs $(T\subset B)\in Y$ where $T$ is $\s$-stable. Both $\scT^{\s}$ and $Y_{\s}$ carry left actions of $G(\RR)$ by conjugation. We have the following diagram where the maps are forgetting $T$ or $B$:
\begin{equation*}
\xymatrix{ X &  Y_{\s}\ar[l]_-{\b_{\s}}\ar[r]^-{\g_{\s}} & \scT^{\s}
}
\end{equation*}
Both maps $\b_{\s},\g_{\s}$ are $G(\RR)$-equivariant.  The map $\g_{\s}$ is a $G(\RR)$-equivariant $\bW$-torsor.

\lem{real orb} 
\begin{enumerate}
\item The map $\b_{\s}: Y_{\s}\to X$ is surjective and it induces a bijection on $G(\RR)$-orbits $\underline{\b}_{\s}: G(\RR)\bs Y_{\s}\bij G(\RR)\bs X$.
\item The map $\g_{\s}: Y_{\s}\to \scT^{\s}$ is a $\bW$-torsor. It induces a surjective map $\underline{\g_{\s}}: G(\RR)\bs Y_{\s}\bij G(\RR)\bs \scT^{\s}$ whose fibers are $\bW$-orbits.
\item The right $\bW$ action on $G(\RR)\bs Y_{\s}$ defines a right $\bW$-action on $I=G(\RR)\bs X$ via the bijection $\underline{\b_{\s}}$, and the composition $\underline{\g_{\s}}\c\underline{\b_{\s}}^{-1}: I\to G(\RR)\bs \scT^{\s}$ is the quotient map by $\bW$. We denote the right $\bW$-action on $I$ by $\l\mapsto \l\cdot w$ ($\l\in I$, $w\in W$).
\end{enumerate}
\elem
\prf (1) follows from \refl{fxdpt}. (2) and (3) are clear.
\epr


\lem{TRT} Let $\l\in I, B\in O^{\RR}_{\l}$ and $T\subset B$ be a $\s$-stable maximal torus. Consider the isomorphism of tori
\begin{equation*}
\io_{B}: T\subset B\surj \bT.
\end{equation*}
We have:
\begin{enumerate}
\item The real structure $\s|_{T}$ induces via $\io_{B}$ a real structure $\s_{T\subset B}$ (anti-holomorphic involution) on $\bT$. Then $\s_{T\subset B}$ depends only on the orbit $\l$. We denote it by $\s_{\l}$. 
\item The real points $\bT^{\s_{\l}}$ under $\s_{\l}$ is the image of the canonical projection $G(\RR)\cap B\subset B\surj \bT$ (which is then independent of $B\in O^{\RR}_{\l}$). 

\end{enumerate}
\elem
\prf 
(1) By \refl{real orb}(1), any two such $(T\subset B)$ (with $B\in O^{\RR}_{\l}$) are $G(\RR)$-conjugate. For $g\in G(\RR)$, we have a commutative diagram
\begin{equation*}
\xymatrix{ T\ar[d]^{\Ad(g)} \ar[rr]^-{\io_{B}} &&  \bT\ar@{=}[d]^{\id}\\
\Ad(g)T\ar[rr]^-{\io_{\Ad(g)B}} &&  \bT}
\end{equation*}   
From this we conclude that $\s_{T\subset B}$ is the same as $\s_{\Ad(g)T\subset \Ad(g)B}$.

(2) We use the notation $H$ from the proof of part (1) of \refl{fxdpt}. We have $G(\RR)\cap B=H(\RR)$. Moreover, $T(\RR)$ is a maximal torus in the solvable real group $H(\RR)$. Since the kernel of the projection $H(\RR)\to \bT$ is unipotent, its image is the reductive quotient of $H(\RR)$. Therefore $T(\RR)$ maps isomorphically to the image of $H(\RR)\to \bT$. By definition, $T(\RR)$ also maps isomorphically via $\io_{B}$ to $\bT^{\s_{\l}}$. The statement follows.
\epr

\defe{att}
Let $T$ be a $\s$-stable maximal torus of $G$ with real points $T(\RR)$. We say that an orbit $O^{\RR}_{\l}$ is {\em attached to} $T$ if $\underline{\g_{\s}}\c\underline{\b_{\s}}$ maps $\l$ to the $G(\RR)$-orbit of $T$. In other words,  $O^{\RR}_{\l}$ is attached to $T$ if there exists a $T$-fixed point in $O^{\RR}_{\l}$.
\edefe

\ssec{}{Roots}
Fix a $\sig$-stable maximal torus $T\subset G$. Let $\Phi(G,T)$ be the set of roots of $G$ with respect to $T$.  

Note that $\sigma$ acts on the set of roots $\Phi(G,T)$: if $\a\in \Phi(G,T)$ viewed as a homomorphism $T\to \GG_{m}$ over $\CC$, then $\s\a: T\to \GG_{m}$ is defined as $t\mapsto \ol{\a(\s t)}$.

Let $\j{-,-}: \frg\times\frg\to \CC$ be the Killing form. Note that $\j{\s(x),\s(y)}=\overline{\j{x,y}}$ for any $x,y\in \frg$. In particular, $\j{x,\s(x)}\in\RR$.

\defe{roots}
A root $\alp\in \Phi(G,T)$ is called 
\begin{enumerate}
\item {\itshape complex} if $\sig\alpha\ne\pm\alpha$;
\item {\itshape real} if $\sig\alpha=\alpha$;
\item {\itshape compact imaginary} if $\sig\alpha=-\alpha$ and for nonzero $x\in \grg_\alp$,  $\j{x, \sig(x)}<0$. 
\item {\itshape noncompact imaginary} if $\sig\alpha=-\alpha$ and for nonzero $x\in \grg_\alp$, $\j{x, \sig(x)}>0$. 
\end{enumerate}
\edefe

\rem{}
This definition is compatible with the definition for a root system invariant under the corresponding Cartan involution (see, for example, \cite[Section 2]{V2}.
\erem

\ssec{}{Based root system attached to a real orbit}
Recall we have the abstract based root system $\Phi$ on $\bT$ with Weyl group $\bW$.  To each point $(T\subset B)\in Y_{\s}$ we have a canonical isomorphism of based root systems $\Phi(G,B,T)\cong \Phi$, under the isomorphism $T\subset B\surj \bT$. Since $T$ is $\s$-stable, $\s$ acts on the root system $\Phi(G,T)$ without necessarily preserving the positive roots. In particular, we get an involution $\s$ on the underlying root system $\underline\Phi$ of $\Phi$. The assignment $Y_{\s}\ni (T\subset B)\mapsto \Inv(\underline\Phi)$ (the set of involutions on $\underline\Phi$) is $G(\RR)$-invariant, hence it induces a map   $G(\RR)\bs Y_{\s}\to \Inv(\underline\Phi)$. Using the bijection $\underline\b_{\s}$ in \refl{real orb}, we get a map $G(\RR)\bs X=I\to \Inv(\underline\Phi)$. For $\l\in I$, we denote by $\Phi_{\l}$ the based root system $\Phi$ equipped with the involution constructed above on $\underline\Phi$.

For $\a\in \Phi_{\l}$, we can talk about whether it is real, complex or imaginary according to \refd{roots}.


For a simple root $\a\in \Phi$, let $X_{\a}$ be the partial flag variety parametrizing parabolic subgroups conjugate to $P_{\a}$ (generated by a Borel $B$ and root subgroup of $-\a$). Let  $\pi_\alpha\colon X\to X_\alpha$ be the projection which is a $\BP^1$-fibration.



There is a partial order on $I$: $\mu\le\l$ if and only if $O^{\RR}_{\mu}\subset \ol{O^{\RR}_{\l}}$.
The following statement is analogous to the results of \cite[Lemma 5.1]{V} and \cite[Sections 2.2, 2.3]{RS}.

\lem{afbr} Let $\l\in I$. Let $\a\in \Phi_{\l}$ be a simple root and $s_{\a}\in \bW$ be the corresponding simple reflection.
\begin{enumerate}
\item If $\a$ is a complex root, then $\l\cdot s_{\a}\ne \l$,  and $\pi_{\a}^{-1}\pi_{\a}(O_{\l}^{\RR})=O_{\l}^{\RR}\cup O_{\l\cdot s_{\a}}^{\RR}$. 
\begin{itemize}
\item If $\s\a>0$, then $\l<\l\cdot s_{\a}$ and $\pi_{\a}|O_{\l}^{\RR}$ is an isomorphism onto its image. 
\item If $\s\a<0$, then $\l\cdot s_{\a}<\l$ and $\pi_{\a}|O_{\l}^{\RR}$ is an $\AA^{1}$-fibration over its image.
\end{itemize}

\item If $\a$ is a real root, then $\l\cdot s_{\a}=\l$. Moreover, one of the following happens:
\begin{itemize}
\item Type I: there are two orbits $\mu^{+}, \mu^{-}>\l$ such that $\mu^{-}=\mu^{+}\cdot s_{\a}$, and $\pi_{\a}^{-1}\pi_{\a}(O_{\l}^{\RR})=O_{\l}^{\RR}\cup O_{\mu^{+}}^{\RR} \cup O_{\mu^{-}}^{\RR}$. Moreover, $\pi_{\a}|O_{\l}^{\RR}$ is an $S^{1}$-fibration, and $\pi_{\a}|O_{\mu^{+}}^{\RR}$ and $\pi_{\a}|O_{\mu^{-}}^{\RR}$ are $D^{2}$-fibrations ($D^{2}$ is an open real $2$-dimensional disk).
\item Type II:  there is $\mu>\l$ such that  $\mu\cdot s_{\a}=\mu$ and $\pi_{\a}^{-1}\pi_{\a}(O_{\l}^{\RR})=O_{\l}^{\RR}\cup O_{\mu}^{\RR}$, and $\pi_{\a}|O_{\l}^{\RR}$ is an $S^{1}$-fibration over its image.
\end{itemize}
 
\item If $\a$ is compact imaginary, then $\l\cdot s_{\a}=\l$ and $\pi_{\a}^{-1}\pi_{\a}(O_{\l}^{\RR})=O_{\l}^{\RR}$.
\item If $\a$ is noncompact imaginary, then there is a unique $\mu<\l$ with $\mu\cdot s_{\a}=\mu$ such that one of the following happens:
\begin{itemize}
\item Type I: $\l\cdot s_{\a}\ne \l$, $\mu<\l\cdot s_{\a}$, and $\pi_{\a}^{-1}\pi_{\a}(O_{\l}^{\RR})=O_{\l}^{\RR}\cup O_{\l\cdot s_{\a}}^{\RR} \cup O_{\mu}^{\RR}$. Moreover, $\pi_{\a}|O_{\mu}^{\RR}$ is an $S^{1}$-fibration, and $\pi_{\a}|O_{\l}^{\RR}$ and $\pi_{\a}|O_{\l\cdot s_{\a}}^{\RR}$ are $D^{2}$-fibrations ($D^{2}$ is an open real $2$-dimensional disk).
\item Type II:   $\l\cdot s_{\a}=\l$, $\pi_{\a}^{-1}\pi_{\a}(O_{\l}^{\RR})=O_{\l}^{\RR}\cup O_{\mu}^{\RR}$, and $\pi_{\a}|O_{\mu}^{\RR}$ is an $S^{1}$-fibration over its image.
\end{itemize}
 
\end{enumerate}
\elem

%
%
%
%
%
%

\prf See {\itshape loc.cit.}
\epr

\ssec{}{Real Weyl groups}
Let $T$ be a $\s$-stable maximal torus. Let $T_{\RR}=T^{\s}$ be the corresponding real form of $T$. We denote the Weyl group of $T$ by $W$, which carries an action of $\s$.  Let $W(\RR)=W^{\s}\subset W$ be the fixed point subgroup of $\sig$. Indeed, $W$ carries the structure of a finite \'etale group scheme over $\RR$, and $W(\RR)$ thus defined is its group of $\RR$-points. On the other hand we have the Weyl group $W(G(\RR), T(\RR))=N_{G(\RR)}(T(\RR))/T(\RR)$.  Clearly we have $W(G(\RR), T(\RR))\subset W(\RR)$.

\lem{realWeyl}
\begin{enumerate}
\item We have $W(\RR)=\mathrm{Stab}_W(T(\RR))$.
\item Suppose $T$ is a $\s$-stable maximal torus that is maximally split, then the inclusion $W(G(\RR), T(\RR))\subset W(\RR)$ is an equality. 
\end{enumerate}
\elem
\prf
(1) Let $\grt_{\RR}$ (resp. $\grt$) be the Lie algebra of $T(\RR)$ (resp. $T$). Then both $W(\RR)=W^{\s}$ and $\mathrm{Stab}_W(T(\RR))$ consist of $w\in W$ that preserve the decomposition $\grt=\grt_{\RR}\op i\grt_{\RR}$. 


(2) follows from \cite[Propositions 3.12 and 4.16]{V2} as there are no noncompact imaginary roots for maximally split torus.
\epr

If $T\in\scT^{\s}$ and $B$ is a Borel subgroup containing $T$, we get a canonical identification $T\cong \bT$ and $W\cong \bW$ compatible with the actions.

\lem{StabWlam} In the above situation, suppose $B\in O^{\RR}_{\l}$. Then under the isomorphism $W\cong \bW$ (induced by $B$), $W(G(\RR),T(\RR))$ is identified with the stabilizer $\bW_{\l}$ of $\l$ under the right action of $\bW$ on $I$. 
\elem
\prf Let $\io: W\isom \bW$ be the canonical isomorphism. If $\dot{w}\in N_{G(\RR)}(T(\RR))$ has image $w\in W$, then $B$ and $\Ad(\dot w)B$ are both in $O^{\RR}_{\l}$ and both contain $T$. By definition $\io(w)=pos(B,\Ad(\dot w)B)$.  By the definition of the $\bW$-action, we see that $\l\cdot \io(w)=\l$. This proves $\io(W(G(\RR), T(\RR)))\subset\bW_{\l}$. 

Conversely, suppose $v\in \bW$ is such that $\l\cdot v=\l$, then $B^{v}$ (the unique Borel containing $T$ such that $pos(B,B^{v})=v$) lies in $O^{\RR}_{\l}$. Therefore there exists $g\in G(\RR)$ such that $B^{v}=\Ad(g)B$. Now $T$ and $\Ad(g)T$ are both $\s$-stable maximal tori in $B^{v}$, by \refl{fxdpt}, there exists $h\in G(\RR)\cap B^{v}$ such that $\Ad(hg)T=T$, i.e., $hg\in G(\RR)\cap N_{G}(T)=N_{G(\RR)}(T(\RR))$. Since $\Ad(hg)B=\Ad(h)(B^{v})=B^{v}$, we see that the image $w$ of $hg$ in $W$ satisfies $\io(w)=v$. Therefore $v\in \io(W(G(\RR), T(\RR)))$. This finishes the proof. 
\epr

\ssec{cross}{The cross action of $\bW$ on $\wt I$}
Recall from \refl{TRT}(1) the real form $\s_{\l}$ on $\bT$ for $\l\in I$.  Let $\bT^{c}_{\l}\subset \bT^{c}$ be the image of the real points $\bT^{\s_{\l}}$ under the projection $\bT\to\bT^{c}$. By \refl{TRT}(2), $\bT^{c}_{\l}$ is the image of $G(\RR)\cap B\to \bT\surj \bT^{c}$ for any $B\in O^{\RR}_{\l}$, therefore this notation is consistent with that of \refss{set}. 

If $T$ is a $\s$-stable maximal torus and $O^{\RR}_{\l}$ is attached to $T$, then by \refl{TRT},  via a choice of $B\in (O^{\RR}_{\l})^{T}$, $\bT_{\l}^{c}$ can be identified with the compact part of $T(\RR)$.  In particular, we have an isomorphism
\eq{pi0}
\io_{B}: \pi_{0}(T(\RR))\isom \pi_{0}(\bT^{c}_{\l}).
\eeq

Recall the right action of $\bW$ on $\bT$, which induces a right action on $\bT^{c}$.  From the commutative diagram \eqref{TT} we see that 
\begin{equation*}
\bT^{c}_{\l}\cdot w=\bT^{c}_{\l\cdot w}, \quad \forall\l\in I, w\in \bW
\end{equation*}
as subgroups of $\bT^{c}$. In particular, the action of $\bW$ on $\bT$ restricts to an action of $\bW_{\l}$ on $\bT_{\l}^{c}$.

Recall the notations $\LS_{\l}$, the bijection \eqref{LS Hom} and $\wt I=\{(\l,\chi)|\chi:\pi_{0}(\bT^{c}_{\l})\to \bfk^{\x}\}$ from \refss{set}. 
In \cite[Definition 4.1]{V2}, Vogan defines a {\em cross action} of $\bW$ on $\wt I$ that lifts the action of $\bW$ on $I$ from \refl{real orb}.  We will turn the cross action $w\times(-)$ into a right action and denote it by
\begin{equation*}
(\l,\chi)\cdot w:=w\times (\l,\chi), \quad \forall w\in \bW, (\l,\chi)\in \wt I.
\end{equation*}
By \cite[Definition 6.3]{V}, for a simple reflection $s\in \bW$, its action on $(\l,\chi)\in \wt I$ is as follows:
\begin{enumerate}
\item If $\a_{s}$ is a complex root for $O^{\RR}_{\l}$, then there is a canonical isomorphism $\pi_{0}(\bT^{c}_{\l})\cong \pi_{0}(\bT^{c}_{\l\cdot s})$ (both are identified with the $G(\RR)$-equivariant fundamental group of the image of $O_{\l}^{\RR}$ in the partial flag variety $X_{s}$. 
Under this isomorphism, we have $(\l,\chi)\cdot s=(\l\cdot s,\chi)$. 

\item If $\a_{s}$ is type I noncompact imaginary, then there is a canonical isomorphism $\pi_{0}(\bT^{c}_{\l})\cong \pi_{0}(\bT^{c}_{\l\cdot s})$ for the same reason as above. Under this isomorphism, we have $(\l,\chi)\cdot s=(\l\cdot s,\chi)$. 

\item If $\a_{s}$ is type II real, and the local system $\bfk_{\chi}$ on  $O^{\RR}_{\l}$ extends to $\pi_{s}^{-1}\pi_{s}(O^{\RR}_{\l})$. Let $\mu>\l$ be as in \refl{afbr}. Then $\bT^{c}_{\mu}\cap \bT^{c}_{\l}\subset \bT^{c}_{\l}$ has index $2$, which induces a sign character
\begin{equation}\label{signs}
\sgn_{s}: \pi_{0}(\bT^{c}_{\l})\surj \bT^{c}_{\l}/\bT^{c}_{\l}\cap \bT^{c}_{\mu}\cong\{\pm1\}\subset \bfk^{\times}.
\end{equation}
Then $(\l,\chi)\cdot s=(\l, \chi\ot\sgn_{s})$. 

\item In all other cases, $(\l,\chi)\cdot s=(\l,\chi)$.
\end{enumerate}


\sec{Matsuki}{Matsuki correspondence as Ringel duality} 

In this section, we apply the theory of tilting sheaves developed in \refs{tilt} to real group orbits on the enhanced flag variety.

\ssec{Matsuki setup}{Setup}

Let $G$ be a connected semisimple complex Lie group together with the anti-holomorphic involution $\sig$. We put $G_\BR=G^\sig\subset G$ to  be the corresponding real form.  Let $K_\BR\subset G_\BR$ be a maximal compact subgroup of $G_\BR$ and $K\subset G$ be the complexification of $K_\BR$.

Choose a maximal torus $T_{0}$ and a Borel subgroup $B_{0}\subset G$, let $U_{0}\subset B_{0}$ be the unipotent radical. Via $B_{0}$, $T_{0}$ is identified with the abstract Cartan $\bT$. Let $\bT=\bT^{>0}\times \bT^{c}$ be the decomposition of the $\CC$-points of the abstract Cartan $\bT$ into the neutral component  $\bT^{>0}$ and the maximal compact subgroup $\bT^{c}$.

Consider the flag variety $X=G/B_{0}$ of $G$. Consider the $\bT^{c}$-torsor $\pi\colon \wt X=(G/U)/\bT^{>0}\to X$.  When $G$ is adjoint, we can define $\wt X$ as the space of $(B, \{x_{\a}\})$ where $B$ is a Borel subgroup and for each simple root $\a\in \Phi$, $x_{\a}$ is basis of the $\a$-weight space of $U/[U,U]$ under the $\bT$-action. For general $G$,  we need to choose a base point $B_{0}\in X$ in order to define $\wt X$. We will consider the left action of $H=G(\BR)$ on $\wt X$.

\ssec{}{Matsuki correspondence} Let us recall some of the results and constructions of \cite{MUV}.

The Matsuki correspondence is a canonical order-reversing bijection between the $G(\RR)$-orbits and $K$-orbits on $X$. This bijection is realized by a $K(\RR)$-invariant flow $\Phi_t\colon X\to X \, (t\in\BR)$, such that 


\begin{enumerate}
\item The fixed point set of $\Phi_t$ is a finite union of $K(\RR)$-orbits $\{C_\l\}_{\l\in I}$ indexed a finite set $I$.

\item For any $\l\in I$ set $O_{\l}^{\BR}$ (resp. $O^{K}_{\l}$) to be the $G(\RR)$-orbit (resp. $K$-orbit) of $X$ containing $C_{\l}$.  Then we have $O_\l^\BR=\{x\in X| \lim_{t\to +\infty} \Phi_t(x)\in C_\l\}$ and $O_\l^K=\{x\in X| \lim_{t\to -\infty} \Phi_t(x)\in C_\l\}$. The bijection $O_{\l}^{\RR}\bij O^{K}_{\l}$  gives an order reversing bijection between the  $G(\RR)$-orbits and $K$-orbits on $X$ which is called the {\itshape Matsuki correspondence for orbits}.

\item The orbits $\{O^\BR_\l\}$ and $\{O^K_\l\}$ intersect pairwise transversally. The natural projections $O^\BR_\l\to C_\l$ and $O^K_\l\to C_\l$ given by the limits of the flow $\Phi_{t}$ are fibrations with contractible fibers.

\end{enumerate}

For $\l\in I$, let $\wt O^{\RR}_{\l}$ be the preimage of $O^{\RR}_{\l}$ in $\wt X$. For $\l\in I$, recall the subgroup $\bT^{c}_{\l}\subset \bT^{c}$ defined in \refss{cross}, which by \refl{TRT}(2) coincides with the subgroup $T^{c}_{\l}$ defined in \refss{set} for $T^{c}=\bT^{c}$.

\lem{Kstab}
Let $x\in C_{\l}$ and $K(\RR)_{x}$ be the stabilizer of $x$ under $C_{\l}$. Then the projection $K(\RR)_{x}\to \bT^{c}$ is injective and its image is $\bT^{c}_{\l}$. Moreover, the $G(\RR)$-action on $X$ satisfies the condition \eqref{stabilizer}. 
\elem
\prf

Since $K(\RR)_{x}$ is compact and solvable (as an algebraic group over $\RR$), its neutral component is a compact torus.  The projection $K(\RR)_{x}\to \bT^{c}$ is injective with closed image. 

By \refl{TRT}(2), $\bT_{\l}^{c}$ is the image of the projection $\g_{x}: G(\RR)_{x}\to \bT\to \bT^{c}$. The square of the projection $\bT\to \bT^{c}$ is real algebraic, hence $\g^{2}_{x}: G(\RR)_{x}\to \bT^{c}$ is real algebraic,  and its image is therefore a real algebraic subgroup of $\bT^{c}$, hence closed with finitely many components.  This implies that the image $\bT^{c}_{\l}$ is a closed subgroup of $\bT^{c}$ with finitely many components. The kernel $\ker(\g_{x})$ is an extension of a closed subgroup of $\bT^{>0}$ and the unipotent real algebraic group $\ker(\wt\g_{x})$, hence $\ker(\g_{x})$ is contractible, and $G(\RR)_{x}\to \bT^{c}_{\l}$ is a homotopy equivalence.

Since $x$ lies in the critical $K(\RR)$-orbit $C_{\l}$, which is homotopy equivalent to $O_{\l}^\BR$, the inclusion $K(\RR)_{x}\incl G(\RR)_{x}$ is a homotopy equivalence. Therefore $K(\RR)_{x}\incl \bT^{c}_{\l}$ is also a homotopy equivalence. Now $\bT^{c}_{\l}/K(\RR)_{x}$ is both a compact manifold and contractible, hence it must be a point. We conclude that $K(\RR)_{x}$ maps isomorphically to $\bT^c_\l\subset \bT^{c}$.
\epr

\ssec{}{Tilting sheaves on real orbits}

We will now observe that 

\prop{} The $\bT^{c}$-torsor $\pi\colon\wt X\to X$ with the action of $H=G(\RR)$ satisfies the conditions of \refss{assump}. 
\eprop

\prf
The condition \eqref{stabilizer} is already checked in \refl{Kstab}.

We check the parity condition \eqref{parity}. From the transversality of $O^{\RR}_{\l}$ and $O^{K}_{\l}$ we get
\begin{equation*}
d_{\l}+2\dim_{\CC}O^{K}_{\l}=2\dim_{\CC} X+\dim C_{\l}.
\end{equation*}
By \refl{Kstab}, 
\begin{equation}\label{dim C lam}
\dim C_{\l}=\dim K(\RR)-\dim \bT^{c}_{\l}=\dim K(\RR)-\dim \bT^{c}+n_{\l}.
\end{equation}
These imply
\begin{equation}\label{d-n}
d_{\l}-n_{\l}=2\codim_{\CC}O^{K}_{\l}+\dim K(\RR)-\dim \bT^{c}.
\end{equation}
Since $\dim K(\RR)-\dim \bT^{c}$ is independent of $\l$, \eqref{parity} holds.

We check the cohomological bound of links \eqref{upper van}. 
Suppose $\l<\mu$ and consider now the intersection $O^\BR_\mu\cap O^K_\l$. The limit maps $\lim_{t\to\pm\infty}\Phi_t(p)$ provide a diagram
$$
\xymatrix{
& Y:=O^\BR_\mu\cap O^K_\l \ar[dl]_{h_{\l}} \ar[dr]^{h_{\mu}} & \\
C_\l & & C_\mu }
$$
with the maps being the fibrations. For $x_{\l}\in C_{\l}$, the fiber of $O^{K}_{\l}\to C_{\l}$ over $x_{\l}$ is a transversal slice to $O^{\RR}_{\l}$, therefore the fiber $h_{\l}^{-1}(x_{\l})$ is diffeomorphic to the link $L^{\mu}_{x_{\l}}$, which we shall denote by ${}^{\RR}L^{\mu}_{x_{\l}}$ to emphasize it is the link for $G(\RR)$-orbits. Similarly, for $x_{\mu}\in C_{\mu}$, $h_{\mu}^{-1}(x_{\mu})$ is diffeomorphic to the link ${}^{K}L^{\l}_{x_{\mu}}$ for the $K$-orbit $O^{K}_{\mu}$ in $O^{K}_{\l}$. Let
\begin{equation*}
\cF_{\chi}=Rh_{\l*}h_{\mu}^{*}\uk_{\mu,\chi}, \quad \cF^{i}_{\chi}:=R^{i}h_{\l*}h_{\mu}^{*}\uk_{\mu,\chi}.
\end{equation*}
Since $h_{\l}$ is $K(\RR)$-equivariant, $\cF^{i}_{\chi}$ is a $K(\RR)$-equivariant local system on $C_{\l}$. We need to show that
\begin{equation*}
\cF^{i}_{\chi}=0, \quad i>\D:=\frac{1}{2}(d_{\mu}+n_{\mu}-d_{\l}-n_{\l}).
\end{equation*}
As a $K(\RR)$-equivariant local system on $C_{\l}$, $\cF^{i}_{\chi}$  is determined by its stalk at $x_{\l}$ and the monodromy action of $\pi_{0}(K(\RR)_{x_{\l}})=\pi_{0}(\bT^{c}_{\l})$ (by \refl{Kstab}) on $\cF^{i}_{\chi}|_{ x_{\l}}$.  Then $\cF^{i}_{\chi}=0$ if and only if $H^{\dim C_{\l}}(C_{\l}, \cF^{i}_{\chi}\ot \uk_{\l,\theta})=0$ for any character $\theta: \pi_{0}(\bT^{c}_{\l})\to \bfk^{\x}$.

Now we introduce $K(\RR)$-equivariant complexes and local systems on $C_{\mu}$ for any character $\theta: \pi_{0}(\bT^{c}_{\l})\to \bfk^{\x}$
\begin{equation*}
\cG_{\theta}=Rh_{\mu*}h_{\l}^{*}\uk_{\l,\theta}, \quad \cG^{i}_{\theta}=R^{i}h_{\mu*}h_{\l}^{*}\uk_{\l,\theta}.
\end{equation*}
Note that
\begin{equation}\label{HY}
H^{*}(C_{\l}, \cF_{\chi}\ot \uk_{\l,\theta})\cong H^{*}(Y, {}_{\theta}\uk_{\chi})\cong H^{*}(C_{\mu}, \cG_{\theta}\ot \uk_{\mu,\chi}).
\end{equation}
Here ${}_{\theta}\uk_{\chi}$ is the local system $h_{\l}^{*}\uk_{\l,\theta}\ot h_{\mu}^{*}\uk_{\mu,\chi}$ on $Y$. 

Let $N$ be the largest number such that $\cF_{\chi}^{N}\ne0$. We need to show $N\le \D$.  By the first isomorphism in \eqref{HY} and the Leray spectral sequence, $H^{\dim C_{\l}+N}(Y, {}_{\theta}\uk_{\chi}))\cong H^{\dim C_{\l}}(C_{\l}, \cF^{N}_{\chi}\ot \uk_{\l,\theta})$. Therefore, it suffices to show that
\begin{equation}\label{Hi van Y}
H^{\dim C_{\l}+N}(Y, {}_{\theta}\uk_{\chi})),  \quad \mbox{ for } i>\dim C_{\l}+\D, \forall (\theta,\chi).
\end{equation}
Reversing the argument using the second equality in \eqref{HY}, we see that \eqref{Hi van Y} holds if and only if 
\begin{equation}\label{van Gtheta}
\cG^{i}_{\theta}=0, \quad i>\dim C_{\l}+\D-\dim C_{\mu}, \forall \theta.
\end{equation}
Since $\cG^{i}_{\theta}$ is a local system whose stalks calculate link cohomology for $K$-orbits, \eqref{van Gtheta} holds if and only if
\begin{equation}\label{Klink van}
H^{i}({}^{K}L^{\l}_{x_{\mu}},\uk_{\l,\theta})=0, \quad i>\dim C_{\l}+\D-\dim C_{\mu}, \forall \theta.
\end{equation}

By \eqref{dim C lam} and \eqref{d-n}, we have
\begin{equation*}
\dim C_{\l}-\frac{1}{2}(d_{\l}+n_{\l})=\frac{1}{2}(\dim K(\RR)-\dim \bT^{c})-\codim_{\CC}O^{K}_{\l}.
\end{equation*}
Therefore,
\begin{eqnarray}\label{dim Delta}
&&\dim C_{\l}+\D-\dim C_{\mu} \\
\notag&=& \left(\dim C_{\l}-\frac{1}{2}(d_{\l}+n_{\l})\right)-\left(\dim C_{\mu}-\frac{1}{2}(d_{\mu}+n_{\mu})\right)\\
\notag&=&\codim_{\CC}O^{K}_{\mu}-\codim_{\CC}O^{K}_{\l}=\dim_{\CC} O^{K}_{\l}-\dim_{\CC} O^{K}_{\mu}.
\end{eqnarray}

Let $i^{K}_{\l}: O^{K}_{\l}\incl X$ be the inclusion of the $K$-orbits. Then $H^{*}({}^{K}L^{\l}_{x_{\mu}},\uk_{\l,\theta})$ is the stalk of $i^{K}_{\l*}\uk_{\l,\theta}$ at $x_{\mu}$. By \cite[Proposition 4.1]{HMSW} (a result due to Beilinson and Bernstein), $i^{K}_{\l}$ is an affine map. Therefore, $i^{K}_{\l*}\uk_{\l,\theta}[\dim_{\CC}O^{K}_{\l}]$ is perverse, and the stalk of $i^{K}_{\l*}\uk_{\l,\theta}$ vanishes in degrees $>\dim_{\CC}O^{K}_{\l}-\dim_{\CC}O^{K}_{\mu}=\dim C_{\l}+\D-\dim C_{\mu}$ (by \eqref{dim Delta}). This proves \eqref{Klink van} and confirms \eqref{upper van}.
\epr

In this present setting, we further abbreviate the notation by putting $\CM_{G_\BR}:=\CM_{G_\BR, X}$. We denote the $t$-structure on $\CM_{G_{\BR}}$ given by the perversity function $p$ defined in \eqref{def p} by $(\CM^{\le 0}_{G_{\BR}}, \CM^{\ge0}_{G_{\BR}})$, and denote its heart  by $\CP_{G_\BR}$.

The general result in \refp{tilt} then implies:

\cor{tilt GR} Let $\TGR$ be the additive subcategory of $\CP_{G_\BR}$ consisting of free-monodromic tilting sheaves. Then:
\begin{enumerate}

\item For each $(\l,\chi)\in \wt I$,  there is up to isomorphism a unique indecomposable free-monodromic tilting sheaves $\cT_{\l,\chi}$ supported on the closure of $\wt{O}^{\RR}_{\l}$ with $\wt i^{*}_{\l}\cT_{\l,\chi}\cong \cL_{\l,\chi}[p_{\l}]$. 

\item The map $(\l,\chi)\mapsto \cT_{\l,\chi}$ is a bijection between $\wt I$ and the set of isomorphism classes of indecomposable objects in $\TGR$.

\item Every object $\cT\in\TGR$ is isomorphic to a finite direct sum of $\cT_{\l,\chi}$, and the multiplicity of each $\cT_{\l,\chi}$ is an invariant of $\cT$.

\end{enumerate}
\ecor

The Matsuki equivalence functors of \cite[Theorem 6.6 (2)]{MUV} are compatible with passing to the completion. By the general argument of 2.3 in \cite{BBM} the following result is now also a formal consequence of \cite[Theorems 5.7 and 6.6]{MUV}.

\th{Mat} The Matsuki functors in \cite{MUV} induces an equivalence
\begin{equation*}
\wh D^{b}_{K}(G/U)_{\bT\mon}\cong \wh D^{b}_{G(\RR)}(\wt X)_{\bT^{c}\mon}=\CM_{G_{\RR}}
\end{equation*}
It is a Ringel duality in the sense that it sends standard objects to costandard objects, and sends the projective cover of $\IC(O^{K}_{\l}, \bfk_{\l,\chi})$ to the indecomposable free-monodromic tilting sheaf $\cT_{\l,\chi}$, for all $(\l,\chi)\in \wt I$.
\eth





\ssec{rank1}{Quasi-split case with split rank $1$} We describe real group orbits and \fmo tilting sheaves for semisimple quasi-split real groups whose maximally split torus has rank $1$. The following is the list of such groups:
\begin{equation*}
R_{\CC/\RR}\SL_{2}, \quad  R_{\CC/\RR}\PGL_{2}, \quad \SL_{2,\RR}, \quad \PGL_{2,\RR}, \quad \SU(2,1) \mbox{ and }\PU(2,1).
\end{equation*}
In the first two cases, $\CM_{G_{\RR}}$ is identified with the completed version of the monodromic Hecke category (see \refss{Hecke}), for which \fmo tilting sheaves are described in \cite[Appendix C]{BY}. Below we consider the rest of the cases.


1) Let $G=\SL_2$ and $G_\BR=\SL_{2,\BR}$. Respectively, we have $K=\SO_2$ and $K_\BR=\SO_{2,\RR}$. The flag variety is $X=\BP^1(\BC)\simeq S^2$ and $\wt X=(\BC^2-0)/\RR^{>0}\simeq S^{3}$. 

There are three $G(\RR)$-orbits on $\BP^1$: upper and lower half planes $\HH_+, \HH_-$ and the real flag variety $\BP^1(\BR)=S^1$. We will label them by $I=\{+, -, 0\}$. The stabilizer of a point in $\HH_+$ or $\HH_-$ is conjugated to $\SO_2(\BR)$ and $n_+=n_-=0$. The constant local systems are the only $G(\RR)$-equivariant free-monodromic local systems on $\pi^{-1}(\HH_+)$ and $\pi^{-1}(\HH_-)$. The stabilizer of a point in $\BP^1(\BR)$ is conjugated to the group of real points of the Borel subgroup $B(\BR):=\{\begin{pmatrix}
           a & b \\
           0 & a^{-1}  
        \end{pmatrix}|a\in \BR^\x, b\in \BR \}.$ Since $\pi_0(B(\BR))=\BZ/2\BZ$ and $\pi_1(B(\BR))=0$, we have $n_0=1$ and there are two indecomposable $G(\RR)$-equivariant free-monodromic local systems $\cL_{0,\triv}$ and $\cL_{0,\sgn}$ on the closed orbit, corresponding to the trivial and sign characters of $\pi_0(B(\BR))=\BZ/2\BZ$. 
        
        We see that $p_{+}=p_{-}=p_{0}=1$ and the abelian category $\CP_{G_\BR}$ is (up to a shift) the category of constructible free-monodromic sheaves in the orbit stratification. 
        
        The category $\CM_{G_\BR}$ is the direct sum of two blocks: $\CM_{G_\BR}^{\circ}\op\CM_{G_\BR}^{\sgn}$. The block $\CM_{G_\BR}^{\sgn}$ is generated by the standard object $\wt\D_{0,\sgn}$ (extension by zero of $\cL_{0,\sgn}$), and the other standard objects $\wt \D_{+}, \wt\D_{-}$ and $\wt \D_{0,\triv}$ generate the other block $\CM_{G_\BR}^{\circ}$.   Correspondingly $\cP_{G_\BR}=\CP_{G_\BR}^{\circ}\op \CP^{\sgn}_{G_\BR}$.



Below we focus on the block $\CP^{\circ}_{G_\BR}$. We represent an object in $\cP^{\circ}_{G_\BR}$  by the following diagram
$$
\xymatrix{
V_+ & V_0 \ar[r]^{s_-}\ar[l]_{s_+}\ar@(ul,ur)^m  & V_-},
$$
where $V_+, V_{0}$ and $V_{-}$ are the vector spaces of stalks on $\pi^{-1}(\HH_{+}), \pi^{-1}(\PP^{1}(\RR))$ and $\pi^{-1}(\HH_{+})$, $m\colon V_0\to V_0$ is the pro-unipotent monodromy operator along the fibers of $\pi$, and $s_\pm\colon \Coker(m-1)\to V_\pm$ are the cospecialization maps. In these terms we have
$$
\tilDel_+=\left(\xymatrix{
\bfk & 0 \ar[r]\ar[l]\ar@(ul,ur)  & 0}\right), \quad
\tilDel_-=\left(\xymatrix{
0 & 0 \ar[r]\ar[l]\ar@(ul,ur)  & \bfk}\right),
$$
$$
\tilna_+=\left(\xymatrix{
\bfk & \bfk \ar[r]\ar[l]_{\id}\ar@(ul,ur)^\id  & 0}\right), \quad
\tilDel_+=\left(\xymatrix{
0 & \bfk \ar[r]^{\id}\ar[l]\ar@(ul,ur)^\id  & \bfk}\right),
$$
$$
\tilDel_0=\tilna_0=\CT_0=\left(\xymatrix{
0 & \bfk[[x]] \ar[r]\ar[l]\ar@(ul,ur)^{(1+x)\cdot}  & 0}\right).
$$
We see that the following are the tilting extensions of the local systems on the open orbits
$$
\CT_+=\left(\xymatrix{
\bfk & \bfk[[x]] \ar[r]\ar[l]_{\id}\ar@(ul,ur)^{(1+x)\cdot}  & 0}\right),  \quad 
\CT_-=\left(\xymatrix{
0 & \bfk[[x]]  \ar[r]^{\id}\ar[l]\ar@(ul,ur)^{(1+x)\cdot}  & \bfk}\right).
$$


2) Let $G=\PGL_2$ and $G_\BR=\PGL_{2,\BR}$. Respectively, we have $K=\mathrm{O}_2(\BC)/\{\pm I_{2}\}$ and $K_\BR=\mathrm{O}_{2,\BR}/\{\pm I_{2}\}$. The flag variety is $X=\BP^1(\BC)\simeq S^2$ and $\wt X=(\BC^2-0)/\BR^{\x}\simeq\BP^3(\BR)$. 

There are two $G(\RR)$-orbits on $\BP^1$: the real flag variety $\BP^1(\BR)$ and its complement $\HH_{\pm}=\HH_{+}\coprod \HH_{-}$. 
We label them by $I=\{0, h\}$. The stabilizer of a point in $\BP^1(\BR)$ is conjugated to the group of real points of the Borel subgroup $B(\BR):=\{\begin{pmatrix}
           a & b \\
           0 & c  
        \end{pmatrix}|a, c\in \BR^\x, b\in \BR \}/\{\begin{pmatrix}
           d & 0 \\
           0 & d  
        \end{pmatrix}|d\in \BR^\x\}.$ Since $\pi_0(B(\BR))=\BZ/2\BZ$ and $\pi_1(B(\BR))=0$, we have $n_0=1$, there are two indecomposable $G(\RR)$-equivariant free-monodromic local systems $\cL_{0,\triv}$ and $\cL_{0,\sgn}$ on the closed orbit corresponding to the trivial and sign characters of $\pi_0(B(\BR))$. The stabilizer of a point in $\HH$ is conjugated to $\SO_2(\BR)/\{\pm1\}$ and $n_h=0$.
        
We see that $p_{h}=p_{0}=1$ and the abelian category $\CP_{G_\BR}$ is (up to a shift) the category of constructible free-monodromic sheaves in the orbit stratification. We represent an object of this category by the following diagram
$$
\xymatrix{
V_{\triv} \ar@(ul,ur)^{m_{\triv}}\ar[r]^{s_{\triv}} & V_{h} & V_{\sgn}\ar[l]_{s_{\sgn}}\ar@(ul,ur)^{m_{\sgn}}},
$$
where $V_{h}$ is the stalk on $\pi^{-1}(\HH_{\pm})$, $V_{\triv}$ and $V_{\sgn}$ are eigenspaces of the stalk at a point of $\pi^{-1}(\PP^{1}(\BR))$ corresponding, respectively, to the trivial and sign characters of $\pi_0(B(\BR))=\BZ/2\BZ$. The maps $m_{\triv}: V_{\triv}\to V_{\triv}$ and $m_{\sgn}: V_{\sgn}\to V_{\sgn}$ are the pro-unipotent monodromy automorphisms along the fibers of $\pi$, and $s_{\triv}\colon \Coker(m_{\triv}-1)\to V_{h}$, $s_{\sgn}\colon \Coker(m_{\sgn}-1)\to V_{h} $ are the cospecialization maps. In these terms we have
$$
\tilDel_{0,\triv}=\tilna_{0,\triv}=\CT_{0,\triv}=\left(\xymatrix{
\bfk[[x]] \ar@(ul,ur)^{(1+x)\cdot}\ar[r] & 0 & 0\ar[l]\ar@(ul,ur)}\right),
$$   
$$
\tilDel_{0,\sgn}=\tilna_{0,\sgn}=\CT_{0,\sgn}=\left(\xymatrix{
0 \ar@(ul,ur)\ar[r] & 0 & \bfk[[x]]\ar[l]\ar@(ul,ur)^{(1+x)\cdot}}\right),
$$
$$
\tilDel_h=\left(\xymatrix{
0 \ar@(ul,ur)\ar[r] & \bfk & 0\ar[l]\ar@(ul,ur)}\right),
\tilna_h=\left(\xymatrix{
\bfk \ar@(ul,ur)^{\id}\ar[r]^{\id} & \bfk & \bfk\ar[l]_{\id}\ar@(ul,ur)^{\id}}\right).
$$ 
We see that the following is the tilting extensions of the local system on the open orbit
$$
\CT_h=\left(\xymatrix{
\bfk[[x]] \ar@(ul,ur)^{(1+x)\cdot}\ar[r]^{\id} & \bfk & \bfk[[x]]\ar[l]_\id\ar@(ul,ur)^{(1+x)\cdot}}\right).
$$

3) 
Let $G=\SL_{3}$ and $G_\BR=\SU(2,1)$ or $G=\PGL_3$ and $G_\BR=\PU(2,1)$. These cases have an identical geometry and, so we will discuss only the first one. 

The six orbits are defined by the signature of the hermitian form on each vector space in the flag. Namely, we put $I=\{0|+0,0|+-,+|+0,-|+-, +|+-, +|++\}$, where to the left of $|$ we have a signature of the hermitian form restricted to $V_1$ and to the right restricted on $V_2$ for a given flag $V_1\subset V_2$. The poset of the orbits is 
\begin{equation}\label{SU21 orbits}
\xymatrix{
& & 0|+0 \ar[dl] \ar[dr] & & \\
& 0|+- \ar[dl] \ar[dr] & & +|+0\ar[dl] \ar[dr] & \\
-|+- & & +|+- & & +|++.
}
\end{equation}
We have $$n_{0|+0}=n_{0|+-}=n_{+|+0}=1, n_{-|+-}=n_{+|+-}=n_{+|++}=0$$ and
$$d_{0|+0}=3, d_{0|+-}=d_{+|+0}=5, d_{-|+-}=d_{+|+-}=d_{+|++}=6.$$ The stabilizer of a point in each orbit is connected. 

The local geometry of the orbits in a transversal slice to a point in the closed orbit can be modeled as follows. Consider the space $\RR^{3}$ together with two smooth surfaces $S_{0|+-}$ and $S_{+|+0}$ that are tangent to each other at their only intersection point $S_{0|+0}$ in $\RR^{3}$. Then the complement of $S_{0|+-}\cup S_{+|+0}$ has 3 connected components $S_{-|+-}, S_{+|+-}$ and $S_{+|++}$, such that the closure relation of these strata is given by \eqref{SU21 orbits}.

From the local model, we see that the closures $\overline{O}^\BR_{0|+-}$ and $\overline{O}^\BR_{+|+0}$ are smooth and are tangent to each other along the closed orbit $O^\BR_{0|+0}$. We conclude that $\nabla_{-|+-}$ and $\nabla_{+|++}$ are shifted (in degree $-3$) constant sheaves on the closures $\overline{O}^\BR_{-|+-}$ and $\overline{O}^\BR_{+|++}$. 
For $\nabla_{+|+-}$, if we restrict to the complement of the closed orbit $\overline{O}^\BR_{+|+-}\bs O^{\RR}_{0|+0}$, it is the constant sheaf in degree $-3$. Along the closed orbit $O^{\RR}_{0|+0}$ its stalks are one dimensional at degree $-3$ and one dimensional at degree $-2$. 

We have $\tilDel_{0|+0}=\tilna_{0|+0}=\CT_{0|+0}$. Since the stratifications of $S_{0|+-}$ and $S_{+|+0}$ are by a smooth point $S_{0|+0}$ and its complement we conclude that the free-monodromic tilting objects $\CT_{0|+-}$ and $\CT_{+|+0}$ behave fiberwise over the closed orbit as the tilting sheaves in the complex group case (see \cite[Appendix C]{BY}).

For $\l\in I$ put $\calC_\l:=\ker(\CT_\l\to \wt i_{0|+0,*}\wt i_{0|+0}^*\CT_\l)\in \cP_{G_{\RR}}$. It follows that we have distinguished triangles:
\eq{su(2,1)-1}\tilDel_{-|+-}\to\calC_{-|+-}\to\tilDel_{0|+-}\to,\eeq
\eq{su(2,1)-2}\tilDel_{+|++}\to\calC_{+|++}\to\tilDel_{+|+0}\to,\eeq
\eq{su(2,1)-3}\tilDel_{+|+-}\to\calC_{+|+-}\to\tilDel_{0|+-}\oplus\tilDel_{+|+0}\to.\eeq
These calculations are local near a point of the orbits $\wt O^{\RR}_{0|+-}$ and $\wt O^{\RR}_{+|+0}$, and are identical to the case of $\SL_{2,\RR}$.  

\sec{Hk}{Hecke action}

In this section we study the action of the Hecke category on $\CM_{G_{\RR}}$, and the behavior of tilting sheaves under the action.

\ssec{Hecke}{Hecke category}
If we view the complex group $G$ as a real group, i.e., $R_{\CC/\RR}G$, its complexification $(R_{\CC/\RR}G)_{\CC}$ is $G\x G'$, where $G'$ is $G$ equipped with the opposite complex structure. Similarly, the flag variety of $(R_{\CC/\RR}G)_{\CC}$ is $X\times X'$ (where $X'$ is $X$ equipped with the opposite complex structure).  We can identify $\CM_{R_{\CC/\RR}G}$ with the  (free-monodromic) {\em Hecke category}
\begin{equation*}
\CH_{G}=\wh D^{b}_{G}(\wt X\times \wt X)_{\bT^{c}\times \bT^{c}\mon}.
\end{equation*}
It is equivalent to the triangulated category $\wh D^{b}_{U}(\wt X)_{\bT^c\mon}$, which was studied in \cite{BY} and \cite{BR}.



Note that in this case we have $\wt I\isom I=W$, where $W$ is the Weyl group of $G$, which we consider equipped with the Bruhat order. For $w\in W$ let $X^{2}_{w}\subset X^{2}$ be the corresponding $G$-orbit,  and $\wt X^{2}_{w}$ be its preimage in $\wt X^{2}$. In particular, for $w=e$ (identity in $W$),  $\wt X^{2}_{e}$ is the preimage of the diagonal $\D(X)\subset X^{2}$ in $\wt X^{2}$.  

For any $w\in W$, we have that $\bT^{c}_{w}\subset \bT^{c}\times \bT^{c}$ is the graph of the $w$-action on $\bT^c$. In particular, $n_{w}=r=\dim \bT^{c}$. Also $\dim_{\RR}X^{2}_{w}=2(\dim_{\CC}X+\ell(w))$. The perversity function is 
$p_{w}=\dim_{\CC}X+\ell(w)+\lfloor r/2\rfloor$.
We will use a slightly different perversity function $p':W\to \Z$
\begin{equation*}
p'_{w}=\ell(w)+r
\end{equation*}
to define  a $t$-structure on $\CH_{G}$. Denote its heart by $\CH_{G}^{\hs}$.

We have the free-monodromic standard objects and costandard objects $\wt \D_{w}, \wt\Na_{w}\in \CH^{\hs}_{G}$ being the $!$- and $*$- extensions of the \fmo $G$-equivariant local system on $\wt X^{2}_{w}$ placed in degree $-p'_{w}$. In particular, we denote $\tilDel_e=\tilna_e$ by $\wt\d$. This is a \fmo local system on the closed stratum $\wt X^{2}_{e}$ placed in degree $-r$. If we identify $\wt X^{2}_{e}$ with $\wt X\times \bT^{c}$ such that $(\wt x, t)\in \wt X\times \bT^{c}$ corresponds to $(\wt x, \wt xt)$, then $\wt\d$ is the extension by zero of $\bfk\bt\cL[r]$ on  $\wt X\times \bT^{c}$.

We define a monoidal structure on $\CH_{G}$ by convolution. Let $\pr_{ij}: \wt X^{3}\to \wt X^{2}$ be the projection to the $(i,j)$-factors. The convolution product on $\CH_{G}$ is defined using
\begin{equation*}
\cK_{1}\star\cK_{2}:=\pr_{13*}(\pr_{12}^{*}\cK_{1}\ot \pr_{23}^{*}\cK_{2})
\end{equation*}
for $\cK_{1},\cK_{2}\in \CH_{G}$. This equips $\CH_{G}$ with a monoidal structure with monoidal unit $\wt\d$. If we identify $\CH_{G}$ with $\wh D^{b}_{U}(\wt X)_{\bT^{c}\mon}$, this monoidal structure is the same as the one defined in  \cite[Section 4.3]{BY} and \cite[Section 7]{BR}.


\ssec{}{Hecke action}
We define a right action of $\CH_{G}$ on $\CM_{G_\BR}$ as follows.

For a $G(\RR)$-equivariant, $\bT^c$-monodromic sheaf $\CF_1$ on $\wt X$ and an $U$-equivariant $\bT^c$-monodromic sheaf $\CF_2$ on $\wt X$ we have a $G(\RR)$-equivariant, $\bT^c$-monodromic sheaf $\CF_1\boxtimes \CF_2$ on the fibered product $G\times^{\bT^{>0}U} \wt X$. We push it forward with respect to the action map $G\times^{\bT^{>0}U} \wt X\to \wt X$ to define a $G(\RR)$-equivariant, $\bT^c$-monodromic sheaf $\CF_1\star\CF_2$ on $\wt X$. As in Lemma 4.3.1 in \cite{BY} or Lemma 7.4 in \cite{BR} the operation is compatible with passing to the limit and we have an action of the monoidal category $\CH_G$ on $\CM_{G_\BR}$:
$$\star\colon\CM_{G_\BR}\x\CH_G\to\CM_{G_\BR}.$$

Let $\cF\in \CM_{G_\BR}$ and $\cK\in \CH_{G}$.  Consider the two projections
\begin{equation*}
\pr_{1}, \pr_{2}: \wt X\times \wt X\to \wt X.
\end{equation*}
Define
\begin{equation*}
\cF\star \cK=\pr_{2*}(\pr_{1}^{*}\cF\ot \cK).
\end{equation*}
More precisely, we start with the above definition for the uncompleted monodromic categories and pass to the limit as in  \cite[Lemma 4.3.1]{BY} or  \cite[Lemma 7.4]{BR}. This defines a right action of the Hecke category $\CH_{G}$ on $\CM_{G_\BR}$:
$$\star\colon\CM_{G_\BR}\x\CH_G\to\CM_{G_\BR}.$$

We now compute the action of some (co)standard objects in $\CH_G$ on some (co)standard objects of $\CM_{G_\BR}$. The calculation is parallel to \cite[Lemma 3.5]{LV} (which is for $K$-orbits on $X$). 


\lem{stdconv} Let  $(\l,\chi)\in \wt I$  and $s\in \bW$ be a simple reflection.  Let $\a_{s}$ be the corresponding simple root in the based root system $\Phi_{\l}$. We will use the notations from \refl{afbr} and \refss{cross}. 

Suppose $O_{\l}^{\RR}$ is closed inside $\pi_{s}^{-1}\pi_{s}(O^{\RR}_{\l})$. Then we have the following cases:


\begin{enumerate}
\item If $\alp_{s}$ is a complex root and $\sig\alp>0$. Let $\mu=\l\cdot s>\l$ be such that $O^{\RR}_{\mu}=\pi_{s}^{-1}\pi_s(O_\lambda^\BR)-O_{\l}^{\BR}$. Then there is a canonical isomorphism $\pi_0(\bT^c_\l)\cong \pi_0(\bT^c_{\mu})$, through which we may view $\chi$ as a character of $\pi_0(\bT^c_{\mu})$. 
We then have 
$$\tilDel_{\l,\chi}\star\tilDel_{s}\cong\tilDel_{\mu,\chi}, \quad \tilna_{\l,\chi}\star\tilna_{s}\cong\tilna_{\mu,\chi},$$
$$\tilDel_{\mu,\chi}\star\tilna_{s}\cong\tilDel_{\l,\chi}, 
\quad \tilna_{\mu,\chi}\star\tilDel_{s}\cong\tilna_{\l,\chi}.$$

\item If $\alp_{s}$ is a compact imaginary root. Then 
$$\tilDel_{\l,\chi}\star\tilDel_{s}\cong \tilDel_{\l,\chi}[-1], \quad \tilDel_{\l,\chi}\star\tilna_{s}\cong\tilDel_{\l,\chi}[1],$$
$$\tilna_{\l,\chi}\star\tilDel_{s}\cong \tilna_{\l,\chi}[-1], \quad \tilna_{\l,\chi}\star\tilna_{s}\cong\tilna_{\l,\chi}[1].$$

\item If $\alp_{s}$ is a real root. 
Then exactly one of the following options holds:
\begin{itemize}
\item The local system $\uk_{\l,\chi}$ on $O^{\RR}_{\l}$ does not extend to $\pi_{s}^{-1}\pi_{s}(O^{\RR}_{\l})$. \footnote{This happens if and only if the composition $\{\pm1\}=\pi_{0}(\RR^{\times})\xr{\a^{\vee}_{s}}\pi_{0}(T(\RR))\isom\pi_{0}(\bT^{c}_{\l})\xr{\chi} \bfk^{\times}$ is nontrivial, where the $\s$-stable maximal torus $T$ is such that $O^{\RR}_{\l}$ is attached to $T$.} Then $$\tilDel_{\l,\chi}\star\tilDel_{s}\cong\tilDel_{\l,\chi}\star\tilna_{s} \cong\tilDel_{\l,\chi},$$ 
$$\tilna_{\l,\chi}\star\tilna_{s}\cong\tilna_{\l,\chi}\star\tilDel_{s} \cong\tilna_{\l,\chi}.$$
\item The local system $\uk_{\l,\chi}$ on $O^{\RR}_{\l}$ extends  (uniquely) to a $G(\RR)$-equivariant local system $\wt{\uk_{\l,\chi}}$ on $\pi_{s}^{-1}\pi_{s}(O^{\RR}_{\l})$, and $\a_{s}$ is type II real. 

Let $\mu>\l$ be such that $O_\mu^\BR=\pi_{s}^{-1}\pi_s(O_\lambda^\BR)-O_\l^\BR$. Let $\uk_{\mu,\psi}$ be the restrictions of $\wt{\uk_{\l,\chi}}$ to $O^{\RR}_{\mu}$. There are distinguished  triangles
\begin{eqnarray}
\label{real II DlamDs}\tilDel_{\mu,\psi}\to\tilDel_{\l,\chi}\star\tilDel_{s}\to\tilDel_{(\l,\chi)\cdot s}\to,\\
\label{real II NlamNs} \tilna_{(\l,\chi)\cdot s}\to\tilna_{\l,\chi}\star\tilna_{s}\to\tilna_{\mu,\psi}\to,\\
\label{real II DmuNs}\tilDel_{\mu,\psi}[-1]\to \tilDel_{\mu,\psi}\star\tilna_{s}\to \tilDel_{\l,\chi}\op\tilDel_{(\l,\chi)\cdot s}\to,\\
\label{real II NmuDs} \tilna_{\l,\chi}\op\tilna_{(\l,\chi)\cdot s}\to\tilna_{\mu,\psi}\star\tilDel_{s}\to  \tilna_{\mu,\psi}[1]\to.
\end{eqnarray}


\item The local system $\uk_{\l,\chi}$ on $O^{\RR}_{\l}$ extends (uniquely) to a $G(\RR)$-equivariant local system $\wt{\uk_{\l,\chi}}$ on $\pi_{s}^{-1}\pi_{s}(O^{\RR}_{\l})$, and $\a_{s}$ is type I real. 

Let $\mu^{+},\mu^{-}>\l$ be such that $O_{\mu^{+}}^\BR\coprod O_{\mu^{-}}^\BR=\pi_{s}^{-1}\pi_s(O_\l^\BR)-O_{\l}^{\RR}$. Let $\uk_{\mu^{+},\psi^{+}}$ and $\uk_{\mu^{-},\psi^{-}}$ be the restrictions of $\wt{\uk_{\l,\chi}}$ to $O^{\RR}_{\mu^{+}}$ and $O^{\RR}_{\mu^{-}}$. There are distinguished triangles
\begin{eqnarray}
\label{real I DlamDs} \tilDel_{\mu^{+},\psi^{+}}\oplus\tilDel_{\mu^{-}\psi^{-}}\to\tilDel_{\l,\chi}\star\tilDel_{s}\to\tilDel_{\l,\chi}\to,\\
\label{real I DmuNs} \wt\D_{\mu^{\mp},\psi^{\mp}}[-1]\to \wt\D_{\mu^{\pm},\psi^{\pm}}\star\wt\Na_{s}\to \wt\D_{\l,\chi}\to,\\
\label{real I NlamNs}
\tilna_{\l,\chi}\to\tilna_{\l,\chi}\star\tilna_{s}\to\tilna_{\mu^{+},\psi^{+}}\oplus\tilna_{\mu^{-},\psi^{-}}\to,\\
\label{real I NmuDs}
\wt\Na_{\l,\chi}\to \wt\Na_{\mu^{\pm},\psi^{\pm}}\star\wt\D_{s}\to\wt\Na_{\mu^{\mp},\psi^{\mp}}[1]\to 
\end{eqnarray}


\end{itemize}
\end{enumerate}
\elem

\prf In the proof we will compute the non-monodromic versions $\D_{\l,\chi}\bstar\D_{s}$. Here $\D_{w}\in D_{G}(X\times X)$ is the $!$-extension of $\uk[\ell(w)]$ on the $G$-orbit $X^{2}_{w}$, and for $\cF\in D_{G(\RR)}(X)$, $\cK\in D_{G}(X\times X) $, $\cF\bstar\cK:=\pr_{2*}(\pr_{1}^{*}\cF\ot \cK)$.   Note the degree shift for $\D_{w}$ differs from the one for $\wt\D_{w}$ by $\r=\dim T^{c}$.  Using \refl{push std} we have 
\begin{equation}\label{conv push}
\pi_{\da}(\wt\D_{\l,\chi}\star \wt\D_{s})\cong \pi_{\da}(\wt\D_{\l,\chi})\bstar\D_{s} \cong \L_{\l,\bu}\ot (\D_{\l,\chi}\bstar\D_{s}), \quad \forall (\l,\chi)\in \wt I.
\end{equation}
In order to recover stalks of $\wt\D_{\l,\chi}\star \wt\D_{s}$ from that of $\D_{\l,\chi}\bstar\D_{s}$,  we will use the following simple observation which follows from  \refc{comp orbit}. Suppose $(\mu,\psi)\in \wt I$, $\cF\in \CM_{G_{\RR}}$:
\begin{equation}\label{recover stalk}
\parbox{10cm}{If $i_{\mu}^{*}\pi_{\dagger}\cF\cong \L_{\mu,\bu}\ot \uk_{\mu,\psi}\ot V$ as a free $\L_{\mu,\bu}$-module (where $V$ is a complex of $\bfk$-vector spaces), then $\wt i_{\mu}^{*}\cF\cong \cL_{\mu,\psi}\ot V$. }
\end{equation}


Choose a point $x\in O_{\l}^{\RR}$ corresponding to a Borel $B_{x}$, and let $\PP^{1}_{s}=\pi_{s}^{-1}\pi_{s}(x)$. Let $P_{x}$ be the parabolic containing $B_{x}$ with negative root $-\a_{s}$.  Note that $G(\RR)\cap P_{x}$ acts on $\PP^{1}_{s}$. The stalk of $\D_{\l,\chi}\bstar\D_{s}$ at $y\in \PP^{1}_{s}$ is
\begin{equation}\label{sty}
i_{y}^{*}(\D_{\l,\chi}\bstar\D_{s})\cong \begin{cases}H_{c}^{*}( O^{\RR}_{\l}\cap \PP^{1}_{s}-\{y\}, \uk_{\l,\chi})[p_{\l}+1] & y\in O^{\RR}_{\l}; \\
H^{*}( O^{\RR}_{\l}\cap \PP^{1}_{s}, \uk_{\l,\chi}) [p_{\l}+1] & y\notin O^{\RR}_{\l}.
\end{cases}
\end{equation}
Moreover, in the case $y\notin O^{\RR}_{\l}$ (corresponding to a Borel $B_{y}$), the inclusion $B_{y}\incl P_{x}$ induces an isomorphism $\pi_{0}(G(\RR)\cap B_{y})\isom \pi_{0}(G(\RR)\cap P_{x})$, and the isomorphism above is equivariant under the actions of $\pi_{0}(G(\RR)\cap B_{y})$ on the left and $\pi_{0}(G(\RR)\cap P_{x})$ on the right via this isomorphism.

Moreover, we have the triangle $\d\to \D_{s}\to \IC_{s}\to $ in $D_{G}(X\times X)$ gives a distinguished triangle
\begin{equation}\label{DDC}
\D_{\l,\chi}\to \D_{\l,\chi}\bstar\D_{s}\to \pi_{s}^{*}\pi_{s*}\D_{\l,\chi}[1]\to 
\end{equation}



(1) In this case, $O_{\l}^{\RR}\cap\PP^{1}_{s}=\{x\}$. By \eqref{sty}, $\D_{\l,\chi}\bstar\D_{s}$ has nonzero ($1$-dimensional) stalk only at $y\in \PP^{1}_{s}-\{x\}$, and $\pi_{0}(G(\RR)\cap B_{y})$ acts via $\chi$ via the natural isomorphisms $\pi_{0}(G(\RR)\cap B_{y})\cong \pi_{0}(G(\RR)\cap P_{x})\cong \pi_{0}(G(\RR)\cap B_{x})= \pi_{0}(\bT^{c}_{\l})$. We conclude that 
\begin{equation}\label{DchiDs}
\D_{\l,\chi}\star\D_{s}\cong \D_{\mu, \chi}
\end{equation}
By \eqref{recover stalk}, this implies $\wt\D_{\l,\chi}\star\wt\D_{s}\cong \wt\D_{\mu, \chi}$. The formula $\tilna_{\l,\chi}\star\tilna_{s}\cong \tilna_{\mu, \chi}$ follows from the same argument from $\Na_{\l,\chi}\star\Na_{s}\cong \Na_{\mu, \chi}$, which follows from \eqref{DchiDs} by Verdier duality. The third and the forth isomorphisms now follow from the fact that 
\begin{equation}\label{invert s Hk}
\wt\D_{s}\star\tilna_{s}\cong\tilna_{s}\star\wt\D_{s}\cong \wt\d\in \CH_{G}
\end{equation}
(see  \cite[Lemma 7.7 (1)]{BR}).


(2) In this case $\PP^{1}_{s}\subset O^{\RR}_{\l}$. From \eqref{sty} we see that $\D_{\l,\chi}\bstar\D_{s}$ has stalk $H^{*}_{c}(\PP^{1}_{s}\bs\{y\}, \bfk)[p_{\l}+1]\cong \bfk[p_{\l}-1]$ at every point $y\in \PP^{1}_{s}$. Moreover, the first map in \eqref{DDC} induces an isomorphism on degree $1-p_{\l}$ stalks, and the monodromy on $\pi_{s}^{*}\pi_{s*}\D_{\l,\chi}$ is given by the same character $\chi$, we conclude that the same is true for the monodromy of $\D_{\l,\chi}\bstar\D_{s}$, i.e., $\D_{\l,\chi}\bstar\D_{s}\cong \D_{\l,\chi}[-1]$. By \eqref{recover stalk}, this implies $\wt\D_{\l,\chi}\star\wt\D_{s}\cong \wt\D_{\mu, \chi}[-1]$. The other formulae are proved similarly.

%

(3) Fix a $\s$-stable maximal torus $T\subset B_{x}$.  Let $L$ be the Levi subgroup of $G$ generated by $T$ and root spaces $\pm\a_{s}$. Then $L$ is $\s$-stable with real form $L_{\RR}=L^{\s}$. We have an $L$-equivariant isomorphism $\wt X_{L}\cong \pi^{-1}(\PP^{1}_{s})$. Therefore we may assume $G=L$. Note that the derived group of $L_{\RR}$ is either $\SL_{2,\RR}$ or $\PGL_{2,\RR}$. The flag variety $X_{L}\cong \PP^{1}$ is defined over $\RR$ with real points the real projective line $\PP^{1}(\RR)=O^{\RR}_{\l}$.



In the first option, $\uk_{\l,\chi}|_{\PP^1(\RR)}$ has nontrivial monodromy, hence $\pi_{s*}\D_{\l,\chi}=0$. By  \eqref{DDC}  we have an isomorphism $\D_{\l,\chi}\isom \D_{\l,\chi}\bstar\D_{s}$. By \eqref{conv push}, we conclude that $\tilDel_{\l,\chi}\star\tilDel_{s}\cong\tilDel_{\l,\chi}$. The other isomorphisms follow from this one by Verdier duality and \eqref{invert s Hk}.



In the second option, we have $O^{\RR}_{\mu}=\PP^{1}_{s}- \PP^1(\RR)$. We first compute  $\wt i_{\l}^{*}(\wt\D_{\l,\chi}\star\wt\D_{s})$. Using \eqref{sty}, the stalk of $\Del_{\l,\chi}\bstar\Del_{s}$ at a point $y\in \PP^1(\RR)$ is $H^*_c(\PP^1(\RR)-\{y\},\uk_{\l,\chi})[p_{\l}+1]\cong \bfk[p_{\l}]$. Moreover, the action of $\pi_0(G(\RR)\cap B_{y})$ on $H^*_c(\PP^1(\RR)-\{y\},\uk_{\l,\chi})$ is $\chi$ twisted by the sign character $\sgn_{s}$ (see \eqref{signs}) through which $\pi_0(G(\RR)\cap B_{y})$ acts on $H^{1}_{c}(\PP^1(\RR)-\{y\})\cong H^{1}(\PP^1(\RR))$.  We conclude that 
\begin{equation}\label{pi DlamDs}
i^{*}_{\l}(\D_{\l,\chi}\bstar\D_{s})\cong \uk_{\l,\chi\ot\sgn_{s}}[p_{\l}]
\end{equation}
By \eqref{recover stalk} we conclude that
\begin{equation}\label{stalk lam}
\wt i_{\l}^{*}(\wt\Del_{\l,\chi}\star\wt\Del_{s})\cong \cL_{\l,\chi\ot\sgn_{s}}[p_{\l}]=\cL_{(\l,\chi)\cdot s}[p_{\l}].
\end{equation}

Next we show that
\begin{equation}\label{stalk mu}
\wt i_{\mu}^{*}(\wt\D_{\l,\chi}\star\wt\D_{s})\cong \cL_{\mu,\psi}[p_{\mu}].
\end{equation}
For this, computing $i^{*}_{\mu}(\Del_{\l,\chi}\bstar\Del_{s})$ as a local system is not enough, since we need to keep track of the $\L_{\mu,\bu}$-action in order to recover $\wt i_{\mu}^{*}(\wt\D_{\l,\chi}\star\wt\D_{s})$. Nevertheless we first use \eqref{sty} to see the stalk of $\Del_{\l,\chi}\bstar\Del_{s}$ at a point  $y\in \PP^{1}_{s}- \PP^1(\RR)$ is equal to $H^*(\PP^1(\RR), \uk_{\l,\chi})[p_{\l}+1]$. Using the long exact sequence attached to the triangle $j_{!}\uk_{\mu,\psi}\to \wt\uk_{\l,\chi}\to i_{!}\uk_{\l,\chi}\to $ on $\pi^{-1}_{s}\pi_{s}(O^{\RR}_{\l})$, we see that the action of $\pi_{0}(G(\RR)\cap B_{y})\cong \pi_{0}(\bT^{c}_{\mu})$ on the stalk $i_{y}^{*}(\Del_{\l,\chi}\bstar\Del_{s})$ is via $\psi$. To show \eqref{stalk mu}, it remains to show that 
\begin{equation}\label{mu fm}
\mbox{$\wt i_{\mu}^{*}(\wt\D_{\l,\chi}\star\wt\D_{s})$ is free-monodromic of rank 1 in degree $-p_{\mu}=-p_{\l}$. }
\end{equation}
Let $A=(ZG)^{\c}$, then we have an isogeny $A\times \SL_{2}\to G$ defined over $\RR$ with kernel either trivial or of order two.  It is then sufficient to prove \eqref{mu fm} for $A\times \SL_{2}$ (by pullback from $\wt X$), and eventually for $G_{\RR}=\SL_{2,\RR}$ (just for the statement \eqref{mu fm}). In the rest of this paragraph we assume $G=\SL_{2}$.  In this case, the statement \eqref{mu fm} becomes
\begin{equation}\label{mu fm SL2}
\mbox{For $\wt y\in \pi^{-1}(\PP^{1}-\PP^{1}(\RR))$, $i_{\wt y}^{*}(\wt\D_{\l,\triv}\star\wt\D_{s})\cong\bfk[1]$.}
\end{equation}
We identify $G/U\cong \AA^{2}\bs\{0\}$ hence $G/UT^{>0}\cong S^{3}$ (unit sphere in $\CC^{2}$). The preimage $\wt O_{\l}^{\RR}\subset S^{3}$ consists of $(z_{1},z_{2})\in \CC^{2}$ with $|z_{1}|^{2}+|z_{2}|^{2}=1$ and $z_{1}/z_{2}\in \RR\cup \{\infty\}$. It is easy to see that $\wt O_{\l}^{\RR}$ is a $2$-dimensional torus. Let $\wt y=(y_{1},y_{2})\in S^{3}$. Then 
\begin{equation}\label{DchiDs SL2}
i_{\wt y}^{*}(\wt\D_{\l,\chi}\star\wt\D_{s})\cong H^{*}(\wt O^{\RR}_{\l}, \cL_{\l,\chi}[1]\ot \ep^{*}\cL[2]).
\end{equation}
Here $\cL$ is the  rank one \fmo local system on $S^{1}$, and $\ep: \wt O_{\l}^{\RR}\to S^{1}$ sends $(z_{1},z_{2})$ to $(z_{1}y_{2}-z_{2}y_{1})/|z_{1}y_{2}-z_{2}y_{1}|$, and   $\ep^{*}\cL[2]$ is the contribution of $\wt\D_{s}$ to the the fiber of the convolution.  Consider the $G(\RR)$-equivariant embedding $\tet: S^{1}\to \wt O^{\RR}_{\l}$ sending $u+iv\mapsto (u,v)$. Then $\tet^{*}\cL_{\l,\chi}$ is trivial since $G(\RR)$ acts transitively on $S^{1}$.   However, calculation shows that $\ep\c\tet: S^{1}\to S^{1}$ is a homeomorphism. These facts combined imply that  $\cL_{\l,\chi}\ot \ep^{*}\cL$ is a rank one \fmo local system on the $2$-dimensional torus $\wt O^{\RR}_{\l}$. Using \eqref{DchiDs SL2} we see that \eqref{mu fm SL2} holds. Now \eqref{stalk mu} is proved.




Combining \eqref{stalk lam} and \eqref{stalk mu} we get the distinguished triangle \eqref{real II DlamDs}.

The same argument for showing \eqref{stalk lam} shows 
\begin{equation}\label{stalk cost lam}
\wt i^{*}_{\l}(\wt\D_{\l,\chi}\star\wt\Na_{s})\cong \cL_{\l,\chi}[p_{\l}+1].
\end{equation}
Here we are using that for $y\in \PP^{1}(\RR)$, $i_{y}^{*}(\D_{\l,\chi}\star\Na_{s})\cong H^{*}(\PP^{1}(\RR)-\{y\}, \uk_{\l,\chi})[p_{\l}+1]\cong \bfk[p_{\l}+1]$. On the other hand, we have
\begin{equation*}
\wt i_{\mu}^{*}(\tilDel_{\l,\chi}\star\tilDel_{s})\cong \wt i_{\mu}^{*}(\tilDel_{\l,\chi}\star\tilna_{s}),
\end{equation*}
which is isomorphic to $\cL_{\mu,\psi}[p_{\mu}]$ by \eqref{stalk mu}. This together with \eqref{stalk cost lam} imply a distinguished triangle 
\begin{equation*}
\wt\D_{\mu,\psi}\to\tilDel_{\l,\chi}\star\tilna_{s}\to\tilDel_{\l,\chi}[1]\to
\end{equation*}
Replacing $(\l,\chi)$ by $(\l,\chi)\cdot s$, and shift by $[-1]$, we get
\begin{equation}\label{Nh0sgn}
\wt\D_{\mu,\psi}[-1]\to\tilDel_{(\l,\chi)\cdot s}\star\tilna_{s}[-1]\to\tilDel_{(\l,\chi)\cdot s}\to
\end{equation}
Now we convolve \eqref{real II DlamDs} with $\tilna_{s}$, rotating it and using \eqref{invert s Hk} 
we obtain a distinguished triangle 
$$\tilDel_{(\l,\chi)\cdot s}\star\tilna_{s}[-1]\to \tilDel_{\mu,\psi}\star\tilna_{s}\to\tilDel_{(\l,\chi)\cdot s}\to.$$ 
Combined with \eqref{Nh0sgn}, observing that there is nontrivial nonzero extension between $\wt\D_{\l,\chi}$ and $\wt \D_{(\l,\chi)\cdot s}$,  we get the asserted distinguished triangle \eqref{real II DmuNs}.

To prove \eqref{real II NlamNs}, we take the Verdier dual of \eqref{pi DlamDs} to get $i_{\l}^{!}(\Na_{\l,\chi}\bstar\Na_{s})\cong \uk_{(\l,\chi)\cdot s}[p_{\l}]$. Using \eqref{recover stalk} we get $\wt i_{\l}^{!}(\wt\Na_{\l,\chi}\star\wt\Na_{s})\cong \cL_{(\l,\chi)\cdot s}[p_{\l}]$. The calculation of $\wt i_{\mu}^{!}(\wt\Na_{\l,\chi}\star\wt\Na_{s})$ boils down to the same thing as $\wt i_{\mu}^{*}(\wt\D_{\l,\chi}\star\wt\D_{s})$, and we have $\wt i_{\mu}^{!}(\wt\Na_{\l,\chi}\star\wt\Na_{s})\cong \cL_{\mu,\psi}[p_{\mu}]$. Combining these facts we get the distinguished triangle \eqref{real II NlamNs}. Then the same deduction from \eqref{real II DlamDs} to \eqref{real II DmuNs} allows us to deduce \eqref{real II NmuDs} from \eqref{real II NlamNs}.

Finally consider the third option. In this case  $O^{\RR}_{\mu^{+}}$ and $O^{\RR}_{\mu^{+}}$ are the two hemispheres of $\PP^{1}-\PP^{1}(\RR)$, which we denote by $\HH^{+}$ and $\HH^{-}$.   The calculations made for \eqref{real II DlamDs} shows \eqref{real I DlamDs}; the argument for \eqref{real II DmuNs} gives a distinguished triangle
\begin{equation*}
\wt\D_{\mu^{+},\psi^{+}}[-1]\op\wt\D_{\mu^{-},\psi^{-}}[-1]\to \wt\D_{\mu^{+},\psi^{+}}\star\wt\Na_{s}\op \wt\D_{\mu^{-},\psi^{-}}\star\wt\Na_{s}\to \wt\D_{\l,\chi}^{\op 2}\to .
\end{equation*}
To deduce \eqref{real I DmuNs}, it remains to show that $\wt i^{*}_{\mu^{+}}(\wt\D_{\mu^{+},\psi^{+}}\star\wt\Na_{s})=0$, or $i_{\mu^{+}}^{*}(\D_{\mu^{+},\psi^{+}}\bstar\Na_{s})=0$. For $y\in \HH^{+}$, $i^{*}_{y}(\D_{\mu^{+},\psi^{+}}\bstar\Na_{s})\cong H^{*}_{c}(\HH^{+}, j_{*}\uk)[2]$ where $j: \HH^{+}-\{y\}\incl \HH^{+}$ is the inclusion. The vanishing of $H^{*}_{c}(\HH^{+}, j_{*}\uk)$ is clear. 

The proofs of \eqref{real I NlamNs} and \eqref{real I NmuDs} are similar to those of \eqref{real II NlamNs} and \eqref{real II NmuDs}. We omit details.
\epr

Let $\mathrm{Tilt}(\CH_G)\subset\CH_G$ be the additive subcategory of \fmo tilting objects. The category $\Tilt(\CH_G)$ is closed under convolution. See \cite[Proposition 4.3.4]{BY} and \cite[Lemma 7.8, Remark 7.9]{BR}. Hence $\Tilt(\CH_G)$ (as an additive category) inherits a monoidal structure from $\CH_{G}$.



\prop{tiltconv} For $\CT_1$ in $\mathrm{Tilt}(\CM_{G_\BR})$ and $\CT_2$ in $\mathrm{Tilt}(\CH_G)$ the convolution product $\CT_1\star\CT_2$ is in $\mathrm{Tilt}(\CM_{G_\BR})$.  In other words, $\Tilt(\CM_{G_\BR})$ is a module category for $\Tilt(\CH_G)$ under convolution.
\eprop

\prf 
Every object $\CT\in \Tilt(\CH_G)$ is a direct summand of  successive convolutions of $\CT_{s}$ for simple reflections $s\in \bW$ (see \cite[Corollary 5.2.3]{BY} and \cite[Remark 7.9]{BR}). Hence, it is sufficient to assume that $\CT_2=\CT_{s}$. We start by checking that for any $(\l,\chi)\in \wt I$, the convolution product $\tilDel_{\l,\chi}\star\CT_{s}$ is a successive extension of the objects of the form $\tilDel_{\mu,\psi}[-n]$ with $(\mu,\psi)\in \wt I$ and $n\ge0$. 

If $\l$ and $s$ are in position of one of the options of \refl{stdconv} we use the distinguished triangle $\tilde{\del}\to\CT_{s}\to\tilDel_{s_\alp}\to$ and the calculation of $\tilDel_{\l,\chi}\star\wt\D_{s}$ in \refl{stdconv} 
to obtain needed filtration. Otherwise, we consider the distinguished triangle $\tilna_{s}\to\CT_{s}\to\tilde{\del}\to$ and use the calculation of $\tilDel_{\l,\chi}\star \tilna_{s}$ in \refl{stdconv} to conclude.

It follows that $\tilDel_{\l,\chi}\star\CT_{s}$ is a successive extension of the objects of the form $\tilDel_{\mu,\psi}[-n]$, $n\ge0$. By considering a standard filtration on $\CT_1$, we see that  the same is true for $\CT_1\star\CT_{s}$. In particular, 
\begin{equation}\label{stalk T1Ts}
\parbox{10cm}{For each $\mu\in I$, the restriction $\wt i^*_\mu(\CT_1\star\CT_{s})$ is a bounded complex of free-monodromic local systems concentrated in degrees $\ge-p_\mu$. }
\end{equation}

Similarly, using the calculations in \refl{stdconv}, the convolution product $\tilna_{\l,\chi}\star\CT_{s}$ is a successive extension of the objects of the form $\tilna_{\mu,\psi}[\ge0]$ and $(\mu,\psi)\in \wt I$. By considering the costandard filtration on $\CT_1$, we see that  the same is true for $\CT_1\star\CT_{s}$.  In particular, 
\begin{equation}\label{costalk T1Ts}
\parbox{10cm}{For each $\mu\in I$, $\wt i_{\mu}^{!}(\CT_1\star\CT_{s})$ is in degrees $\le -p_{\mu}$.}
\end{equation}

By \refl{r}(1),  $\wt i^{*}_{\mu}\tilna_{\mu',\psi}$ lies in degrees $\le -p_{\mu}$. Since $\CT_1\star\CT_{s}$ is a successive extension of $\tilna_{\mu',\psi}[\ge0]$,  $\wt i^{*}_{\mu}(\CT_1\star\CT_{s})$ also lies in degrees $\le -p_{\mu}$. Combined with \eqref{stalk T1Ts}, we conclude that $\wt i^{*}_{\mu}(\CT_1\star\CT_{s})$ is concentrated in degree $-p_{\mu}$ and is free-monodromic. 

By \refl{r}(2), $\wt i^{!}_{\mu}\wt\D_{\mu',\psi}$ is a complex of free-monodromic local systems in degrees $\ge -p_{\mu}$. Since  $\CT_1\star\CT_{s}$ is a successive extension of $\wt\D_{\mu',\psi}[\le0]$, $\wt i^{!}_{\mu}(\CT_1\star\CT_{s})$ is also a complex of free-monodromic local systems in degrees $\ge -p_{\mu}$.  Combined with \eqref{costalk T1Ts}, we conclude that $\wt i^{!}_{\mu}(\CT_1\star\CT_{s})$ is concentrated in degree $-p_{\mu}$ and is free-monodromic. 

Combining the last two paragraphs, we conclude that $\CT_1\star\CT_{s}$ is a \fmo tilting sheaf.


\epr


\lem{tiltconv2}
Recall the assumptions and notations of \refl{stdconv}. Then
\begin{enumerate}
\item If $\a_{s}$ is complex and $\s\a_{s}>0$, then $\CT_{\l,\chi}\star\CT_{s}$ contains $\CT_{\l\cdot s, \chi}$ as a direct summand with multiplicity one. 
\item If $\a_{s}$ is real and $\uk_{\l,\chi}$ extends to a $G(\RR)$-equivariant local system on $\pi_{s}^{-1}\pi_{s}(O^{\RR}_{\l})$. Then
\begin{itemize}
\item If $\a_{s}$ is type II real,  then $\CT_{\l,\chi}\star\CT_{s}$ contains $\CT_{\mu, \psi}$ as a direct summand with multiplicity one. 
\item If $\a_{s}$ is type I real, then  $\CT_{\l,\chi}\star\CT_{s}$ contains $\CT_{\mu^{+}, \psi^{+}}\oplus\CT_{\mu^{-}, \psi^{-}}$ as a direct summand with multiplicity one. 
\end{itemize}
\end{enumerate}
\elem
\prf Let $O^{\RR}_{\mu}=\pi_{s}^{-1}\pi_{s}(O^{\RR}_{\l})-O^{\RR}_{\l}$ (which is a union of two orbits in type I real case), and let $\wt O^{\RR}_{\mu}$ be its preimage under  $\pi$. Let $\wt i_{\mu}: \wt O^{\RR}_{\mu}\incl \wt X$ be the inclusion. Then $\wt O^{\RR}_{\mu}$ is the open stratum in the support of $\CT_{\l,\chi}\star\CT_{s}$. By support considerations, we have
\begin{equation*}
\wt i^{*}_{\mu}(\CT_{\l,\chi}\star\CT_{s})\cong \wt i^{*}_{\mu}(\D_{\l,\chi}\star \wt\D_{s}).
\end{equation*}
By \refl{stdconv}, the above is $\cL_{\mu,\chi}[p_{\mu}]$ in case (1), $\cL_{\mu,\psi}[p_{\mu}]$ in case (2) type II, and $\cL_{\mu^{+},\psi^{+}}[p_{\mu^{+}}]\op\cL_{\mu^{-},\psi^{-}}[p_{\mu^{-}}]$ in case (2) type I. The conclusion follows.  
\epr

The Hecke action on $\Tilt(\CM_{G_\BR})$ enjoys the following self-adjunction property.

\prop{adj}
Let $s\in W$ be a simple reflection and $\CT_{s}$ the corresponding tilting object of $\CH_G$. Then the convolution endo-functor $-\star\CT_{s}$ on $\mathrm{Tilt}(\CM_{G_\BR})$ is self-adjoint. Namely for $\CT_1, \CT_2$ in $\mathrm{Tilt}(\CM_{G_\BR})$ we have a canonical isomorphism functorial in $\cT_{1}$ and $\cT_{2}$
$$\RHom_{\CM_{G_\BR}}(\CT_1\star\CT_{s},\CT_2)\cong\RHom_{\CM_{G_\BR}}(\CT_1,\CT_2\star\CT_{s}).$$
\eprop

\prf
To construct a unit and counit of the adjunction it suffices to construct maps $u\colon \tilde{\del}\to \CT_{s}\star\CT_{s}$ and $c\colon\CT_{s}\star\CT_{s}\to \tilde{\del}$ in $\CH_G$, such that  both compositions
\begin{equation}\label{cu id}
\xymatrix{\CT_{s}\ar[r]^-{u\star \id} & \CT_{s}\star\CT_{s} \star\CT_{s} \ar[r]^-{\id\star c} & \CT_{s}\\
\CT_{s}\ar[r]^-{\id\star u} & \CT_{s}\star\CT_{s} \star\CT_{s} \ar[r]^-{c\star \id} & \CT_{s}}
\end{equation}
are equality to the identity of $\CT_{s}$.

It is known that the Soergel's functor $\VV$ (recalled in \refss{V Hk}) provides a fully-faithful embedding of $\mathrm{Tilt}(\CH_{G})$ into the category of Soergel bimodules (i.e. \cite[Proposition 11.2]{BR} in the present setting). It is thus sufficient to construct $u$ and $c$ in the category of Soergel bimodules, where we have $\VV(\CT_s)=\CR\ten_{\CR^s}\CR$. Note that $\CR^{s=-1}$, as an $\CR^{s}$-bimodule,  is isomorphic to the regular bimodule. Let us fix an element $x_{s}\in\CR^{s=-1}$ such that multiplication by $x_{s}$ gives an isomorphism $\CR^{s}\isom \CR^{s=-1}$ of $\CR^{s}$-bimodules. We get a splitting $\CR=\CR^s\oplus x_{s}\CR^s$ as $\CR^{s}$-bimodules. 

We then have a splitting $$\VV(\CT_s\star\CT_s)=\CR\ten_{\CR^s}\CR\ten_{\CR^s}\CR=\CR\ten_{\CR^s}(\CR^s\oplus x_{s}\CR^s)\ten_{\CR^s}\CR=$$ $$=\CR\ten_{\CR^s}\CR^s\ten_{\CR^s}\CR\oplus\CR\ten_{\CR^s}x_{s}\CR^s\ten_{\CR^s}\CR.$$ We put $c\colon\CR\ten_{\CR^s}\CR\ten_{\CR^s}\CR\to\CR$ for a composition of the projection onto the second summand $\CR\ten_{\CR^s}x_{s}\CR^s\ten_{\CR^s}\CR\simeq\CR\ten_{\CR^s}\CR$ and the multiplication map $\CR\ten_{\CR^s}\CR\to \CR$. We put $u\colon\CR\to\CR\ten_{\CR^s}\CR\ten_{\CR^s}\CR$ for the composition of  the map $\CR\to\CR\ten_{\CR^s}\CR$ given by $1\mapsto x_{s}\ten1+1\ten x_{s}$ and the inclusion of the first summand $\CR\ten_{\CR^s}\CR^s\ten_{\CR^s}\CR\simeq\CR\ten_{\CR^s}\CR.$ The verification of the adjunction identities is now straightforward.




\epr

\ssec{}{Generation}

The following generation property for the $\Tilt(\CH_{G})$-module $\TGR$ will play an important role later.

Recall there is a unique closed $G(\RR)$-orbit $O^{\RR}_{\l_{0}}$ in $X$ and its preimage $\wt O^{\RR}_{\l_{0}}:=\pi^{-1}(O^{\RR}_{\l})$ in $\wt X$.

\prop{gen}
The category $\TGR$ is generated under taking direct sums and summands and applying the convolution action of $\mathrm{Tilt}(\CH_G)$ by the free-monodromic local systems supported on $\wt O^{\RR}_{\l_{0}}$. 
\eprop

\prf



Let $\Tilt'\subset \Tilt(\CM_{G_{\BR}})$ be the full subcategory generated under taking direct sums and summands and the convolution action of $\mathrm{Tilt}(\CH_{G})$ by the free-monodromic local systems supported on $\wt O^{\RR}_{\l_{0}}$. Suppose for purpose of contradiction that $\Tilt'$ is not the whole $\Tilt(\CM_{G_{\BR}})$, let $(\l,\chi)\in \wt I$ be such that $\CT_{\l,\chi}\notin \Tilt'$ and that $\dim O^{\RR}_{\l}$ is minimal among such $(\l,\chi)$.


Let $B$ be a Borel corresponding to a point in $O^{\RR}_{\l}$ and let $\grb$ be its Lie algebra and $\grt\subset\grb$ its $\sig$-invariant subtorus. Let $\alp$ be a simple root of $\grt$ in $\grg$ positive with respect to $\grb$. 
\refl{tiltconv2} implies that if $\alp$ is noncompact imaginary or $\alp$ is complex and $\sig\alp$ is negative, there is $(\mu,\psi)\in \wt I$ with $\mu<\l$ such that $\CT_{\l,\chi}$ is a direct summand in $\CT_{\mu,\psi}\star\CT_{s_\alp}$. This would contradict the minimality assumption.

From the above we conclude that any simple root $\a$ of $(\grb,\grt)$ either satisfies $\s\a>0$, or $\a$ is compact imaginary. By \refl{cl orb} below, in this situation $O^{\RR}_{\l}$ is closed, therefore $\cT_{\l,\chi}$ is already in $\Tilt'$. We arrived at a contradiction.
\epr

\lem{cl orb} Let $\l\in I$. The orbit $O^{\RR}_{\l}$ is closed if and only if any simple root $\a\in \Phi_{\l}$ either satisfies $\s\a>0$, or is compact imaginary. 
\elem
\prf
The ``only if'' part follows from \refl{afbr} as the intersections with the $\a$-lines has to be closed subvarieties of $\BP^1$.

Let us proof the ``if'' part. Under the Matsuki correspondence, the closed $G(\RR)$-orbit corresponds to the unique open $K$-orbit in $X$.  Consider now the corresponding $K$-orbit $O^{K}_{\l}$ under the Matsuki correspondence. Let $\grk$ be the Lie algebra of $K$ and let $\tet$ be the corresponding Cartan involution. We choose $B\in C_{\l}=O^{K}_{\l}\cap O^{\RR}_{\l}$. We would like to prove that $\grb+\grk=\grg$, which then implies that $O^{K}_{\l}$ is open and $O^{\RR}_{\l}$ is closed.

Translating the constraints on the roots in terms of $\tet$, we see that for each simple root $\a$ of $(\grb,\grt)$, either $\tet\a<0$, or $\a$ is compact imaginary (i.e., $\tet\a=\a$ and $\grg_{\pm\a}\subset \grk$). 

We want to show that for each root $\a$ we have $\grg_{\alp}\subset\grk+\grb$. We proceed by downward induction on the height of $\a$. If $ht(\a)>0$ then $\a>0$ and $\frg_{\a}\subset \frb$. If $ht(\a)<0$ and suppose $\frg_{\b}\subset\grk+\frb$ for all $ht(\b)>ht(\a)$. Write $\a$ as a sum of simple roots $\a_{i}$, and each $\a_{i}$ either satisfies $\tet\a_{i}<0$ or $\tet\a_{i}=\a_{i}$. Therefore $ht(\tet\a)\le ht(\a)$. 

If $ht(\tet\a)<ht(\a)$ then by inductive hypothesis we have $\frg_{\tet\a}\subset \grk+\grb$. On the other hand, $(\frg_{\a}\op\frg_{\tet\a})^{\tet}\subset \grk$. Therefore $\frg_{\a}\subset \frg_{\tet\a}+(\frg_{\a}\op\frg_{\tet\a})^{\tet}\subset \grb+\grk$.

If $ht(\tet\a)=ht(\a)$, then this can happen only when $\a$ is compact imaginary. In this case, $\frg_{\a}\subset\grk$.  This completes the induction. 
\epr

We define a measure of complexity of an indecomposable \fmo tilting sheaf $\cT_{\l,\chi}$.

\defe{length}
Let $(\l,\chi)\in \wt I$. Let $\ell(\l)$ be the minimal non-negative integer $n$ such that there exists $\psi\in \LS_{\l_{0}}$ and $w\in \bW$ such that $\ell(w)=n$ and $\cT_{\l,\chi}$ is a direct summand of $\cT_{\l_{0},\psi}\star \cT_{w}$.
\edefe

For example, $\ell(\l_{0},\chi)=0$ for any $\chi\in\LS_{\l_{0}}$.

\lem{length codim} For any $(\l,\chi)\in \wt I$, we have
\begin{equation*}
\ell(\l,\chi)=\codim_{\CC}O^{K}_{\l}. \quad \mbox{(codimension in $X$)}
\end{equation*}
In particular, $\ell(\l,\chi)$ depends only on $\l\in I$, and we denote it by $\ell(\l)$.
\elem
\prf
For $\l\in I$, let $q_{\l}=(d_{\l}-n_{\l})/2$. By \eqref{d-n}, we see $q_{\l}-q_{\l_{0}}=\codim_{\CC}O^{K}_{\l}$ (where $\l_{0}$ indexes  the closed $G_{\RR}$-orbit). Therefore we only need to show that $\ell(\l,\chi)=q_{\l}-q_{\l_{0}}$.

We first show $\ell(\l,\chi)\ge q_{\l}-q_{\l_{0}}$. For $\mu<\l$, we have $q_{\mu}<q_{\l}$ because $q_{\mu}-q_{\l}=\dim_{\CC} O^{K}_{\mu}-\dim_{\CC} O^{K}_{\l}$. For an indecomposable tilting sheaf $\cT_{\mu,\psi}$ and a simple reflection $s\in \bW$, the support of $\cT_{\mu,\psi}\star \cT_{s}$ is $\pi_{s}^{-1}\pi_{s}(\ol{O^{\RR}_{\mu}})$. Suppose $O^{\RR}_{\mu'}\subset \pi_{s}^{-1}\pi_{s}(\ol{O^{\RR}_{\mu}})$, then either $\mu'\le \mu$, in which case $q_{\mu'}\le q_{\mu}$, or $O^{\RR}_{\mu'}$ is open in $\pi_{s}^{-1}\pi_{s}(O^{\RR}_{\mu'})$, which contains a smaller orbit $\mu''\le \mu$. By examining each case in \refl{afbr}, we see that $q_{\mu'}-q_{\mu''}=1$. Therefore $q_{\mu'}\le q_{\mu}+1$. In other words, if $\cT_{\mu',\psi'}$ is a direct summand of $\cT_{\mu,\psi}\star \cT_{s}$, then $q_{\mu'}\le q_{\mu}+1$. Applying this repeatedly, we see that if $\cT_{\l,\chi}$ appears as a direct summand of $\cT_{\l_{0}, \psi}\star \cT_{w}$, then $q_{\l}-q_{\l_{0}}\le \ell(w)$.  This implies $\ell(\l,\chi)\ge q_{\l}-q_{\l_{0}}$.

We now show $\ell(\l,\chi)\le q_{\l}-q_{\l_{0}}$. We prove this by induction on $q_{\l}$. If $q_{\l}=q_{\l_{0}}$, then $\l$ is the closed orbit,  hence $\ell(\l,\chi)=0$ so the equality holds. Now suppose $\ell(\l',\chi')\le q_{\l'}-q_{\l_{0}}$ holds for all $q_{\l'}<q_{\l}$. Since $\l$ is not the closed orbit, by \refl{cl orb}, there is a simple root $\a_{s}\in \Phi_{\l}$ that is either complex and $\s\a_{s}<0$, or non-compact imaginary. In either case there is a smaller orbit $\l'<\l$ inside $\pi_{s}^{-1}\pi_{s}(O^{\RR}_{\l})$ with $q_{\l'}=q_{\l}-1$, by checking the cases listed in \refl{afbr}. By  \refl{tiltconv2}, there exists $\chi'\in \LS_{\l'}$ such that $\cT_{\l,\chi}$ is a direct summand of $\cT_{\l',\chi'}\star\cT_{s}$. This implies $\ell(\l, \chi)\le \ell(\l',\chi')+1$. By inductive hypothesis, we conclude that $\ell(\l, \chi)\le q_{\l'}-q_{\l_{0}}+1=q_{\l}-q_{\l_{0}}$. This proves the inequality $\ell(\l,\chi)\le q_{\l}-q_{\l_{0}}$. 

Combining both inequalities we conclude that $\ell(\l,\chi)= q_{\l}-q_{\l_{0}}=\codim_{\CC}O^{K}_{\l}$.
\epr

\rem{} Using the multiplicity one statement in \refl{tiltconv2}, the above argument shows the following stronger statement: for any $(\l,\chi)\in \wt I$, and any $w\in \bW$ of length equal to $\ell(\l)$ such that $\cT_{\l,\chi}$ is a direct summand of $\cT_{\l_{0},\psi}\star\cT_{w}$ for some $\psi\in\LS_{\l_{0}}$, $\cT_{\l,\chi}$  appears in $\cT_{\l_{0},\psi}\star\cT_{w}$ with multiplicity one.
\erem
\sec{loc}{Localization} 

This section will play a technical role in the proof of the main results in \refs{str nice}. The localization procedure here can be viewed as the counterpart of equivariant localization under Koszul duality.

\ssec{}{Preliminaries}

Let $V=\Spec \CR$.  Let $\bT^{\vee}=\Hom(\XX_{*}(\bT), \GG_{m,\bfk})=\Hom(\pi_{1}(\bT^{c}), \GG_{m,\bfk})$ be the dual torus of $\bT$ over $\bfk$. Then the formal version $\Spf \CR$ of $V$ can be canonically identified with the formal completion of $\bT^{\vee}$ at the identity element. We have a natural $\bW$-action on $V$.

For any $\l\in I$ let $\ol\CR_{\l}$ be the completion of the group ring $\bfk[\pi_{1}(\ol \bT^{c}_{\l})]$ at the augmentation ideal. Then   $V_{\l}:=\Spec \ol\CR_{\l}\subset V$ is a closed subscheme. By \refl{TRT}(1), the abstract Cartan $\bT$ carries a real structure $\s_{\l}$. Let $\theta_{\l}$ be the Cartan involution on $\bT$ corresponding to the real structure $\s_{\l}$ (so $\theta_{\l}$ is the composition of $\s_{\l}$ with the compact real form $\s_{c}$ of $\bT$). The Cartan involution $\theta_{\l}$ acts on $\bT^{\vee}$ hence on $V$. Write $-\theta_{\l}$ to be $\theta_{\l}$ composed with inversion. Then $V_{\l}$ is the fixed points  subscheme $V^{-\theta_{\l}}$ of $V$.  Again the formal version $\Spf \ol\CR_{\l}$ of $V_{\l}$ can be identified with the formal completion of $(\bT^{\vee})^{-\tet_{\l}}$ at the identity element. The group $\bW^{\tet_{\l}}=\bW^{\s_{\l}}$ acts on $V_{\l}$. By \refl{StabWlam},  we have $\bW_{\l}\subset \bW^{\s_{\l}}$, which then acts on $V_{\l}$.

Recall $\l_{0}\in I$ indexes the closed $G(\RR)$-orbit.
Let $\CS=\ol \CR_{\l_{0}}$.  We have $V_{\l_{0}}=\Spec\CS\subset V$. Let $T_{0}$ be a maximally split $\s$-stable Cartan. By \refl{StabWlam}, $\bW_{\l_{0}}$ can be identified with the real Weyl group $W(G(\RR), T(\RR))$ for a maximally split $\s$-stable torus $T$, which by \refl{realWeyl} is also $W^{\s}=W(\RR)$ for the induced real structure on $W=W(G,T)$. We have a natural map $\CR^\bW\to\CS^{\bW_{\l_{0}}}=\CS^{W(\RR)}$. Define
\begin{equation}
(\CS^{\bW_{\l_{0}}})':=\textup{Im}(\CR^\bW\to\CS^{\bW_{\l_{0}}}).
\end{equation}
Then $(\CS^{\bW_{\l_{0}}})'$ is a subring of $\CS^{\bW_{\l_{0}}})$ that has the same fraction field with $\CS^{\bW_{\l_{0}}})$. It is shown in \cite{H} that $(\CS^{\bW_{\l_{0}}})'=\CS^{\bW_{\l_{0}}}$ outside $4$ exceptional cases in type $E$. We note that the only quasi-split simple group for which $(\CS^{\bW_{\l_{0}}})'\ne\CS^{\bW_{\l_{0}}}$ is the non-split quasi-split form of $E_{6}$.

We consider the projection map
\begin{equation*}
\pr_{V}: V_{\l_{0}}\times_{V\sslash \bW}V=\Spec (\cS\ot_{\cR^{\bW}}\cR)\to V=\Spec \cR.
\end{equation*}
Let $V_{r}$ be the scheme-theoretic image of $\pr_{V}$; this is a $\bW$-stable closed subscheme of $V$. It is easy to see that the reduced structure of $V_{r}$ is the union
\begin{equation*}
V^{\red}_{r}=\bigcup_{w\in \bW}w V_{\l_{0}}\subset V.
\end{equation*}
Equivalently, $V^{\red}_{r}$ is the union of $V_{\l}$ for those $\l$ attached to a fixed maximally split $\s$-stable torus $T_{0}$.  


The category $\TGR$ is linear over $\CR$, i.e., $\CR$ acts on the identity functor $\id_{\TGR}$.
 
\lem{R supp} The action of $\CR$ on $\id_{\TGR}$ factors through the quotient $\cO(V_{r})$ (regular functions on $V_{r}$). The action of $\CR^{\bW}$ on $\id_{\TGR}$ factors through $(\CS^{\bW_{\l_{0}}})'$.
\elem
\prf By definition, $\cO(V_{r})$ is the image of the natural map $\cR\to \cS\ot_{\cR^{\bW}}\cR$ sending $x\mapsto 1\ot x$. Taking $\bW$-invariants, $\cO(V_{r})^{\bW}$ is the image of the natural map $\cR^{\bW}\to \cS$ which lands in $\cS^{\bW_{\l_{0}}}$. Hence $\cO(V_{r})^{\bW}=(\CS^{\bW_{\l_{0}}})'$. From this we see that the second assertion follows from the first.

We prove the first assertion. By \refp{gen}, it suffices to check that the action of $\CR$ on $\CT\star\CT'$ (via right monodromy on $\CT'$) factors through $\cO(V_{r})$, for $\CT\in\TGR$ supported on $\wt O^{\RR}_{\l_{0}}$, and $\CT'\in \THG$. The action of $\CR\ot\CR$ on $\CT'$ by the left and right monodromy factors through $\CR\ot_{\CR^{\bW}}\CR$ by \refp{VHk}, and the first copy of $\cR$-action is the same as the right monodromy on $\CT$. For $\CT\in\TGR$ supported on the closed orbit, $\CR$ acts on $\CT$ through the quotient $\CS$, hence the $\CR\ot_{\CR^{\bW}}\CR$-action on $\CT\star\CT'$ factors through $\CS\ot_{\CR^{\bW}}\CR$. Therefore the second copy of $\CR$-action factors through the image of $\CR\to \CS\ot_{\CR^{\bW}}\CR$ ($r\mapsto 1\ot r$), which is $\cO(V_{r})$ by definition.
\epr

\ssec{locJ}{Localization}


We fix a Borel $B_{0}\in O^{\RR}_{\l_{0}}$ and a $\s$-stable maximal torus $T_{0}\subset B_{0}$. Let $P_{0}\supset B_{0}$ be the minimal $\s$-stable parabolic subgroup containing $B_{0}$. Let $A_{0}\subset T_{0}$ be the subtorus corresponding to the split part of $T_{0,\RR}$. Via the isomorphism $\io_{B_{0}}: T_{0}\subset B_{0}\surj \bT$, $A_{0}$ gets identified with the subtorus $\bA\subset \bT$, which is the neutral component of $\bT^{-\theta_{\l_{0}}}$. We denote by $\bA^{\vee}$ the $\bfk$-torus dual to $\bA$. Then $\Spf \CS$ (the formal version of $V_{\l_{0}}$) is canonically identified with the completion of $\bA^{\vee}$ at the identity.

Consider the restricted root system $\Phi(G,A_{0})$, with basis given by $B_{0}$. Via the isomorphism $\io_{B_{0}}: A_{0}\isom \bA$, we view $\a\in \Phi(G,A_{0})$ as characters on $\bA$, and the corresponding coroot $\a^{\vee}$ as characters on $\bA^{\vee} $. Let $J$ be a subset of simple roots in $\Phi(G,A_{0})$. Let $\bA_{J}\subset \bA$  and $A_{J}\subset A_{0}$ be the neutral component of $\cap_{j\in J}\ker(\a_{j})$. Similarly, let $\bA^{\vee}_{J}\subset \bA^{\vee}$ be the neutral component of $\cap_{j\in J}\ker(\a^{\vee}_{j})$. Let $V_{J}\subset V_{\l_0}$ be a closed subscheme provided by the inclusion $\bA^{\vee}_{J}\subset \bA^{\vee}$. It is identified with the completion of the torus $\bA^{\vee}_J$ at the identity again considered as a scheme and not a formal scheme.

Let $\CK_{J}$ be the localization of $V$ at the generic point of $V_{J}$. The category $\TGR$ is linear over $\CR$. We define the localization of $\TGR$ at the generic point of $V_{J}$, denoted  $\TGR\ot_{\CR}\CK_{J}$, as the idempotent completion of the $\CK_{J}$-linear additive category whose objects are the same as those in $\TGR$, and the morphisms are defined as
$$\Hom_{\TGR\ot_{\CR}\CK_{J}}(\CT_1,\CT_2):=\Hom_{\TGR}(\CT_1,\CT_2)\ten_{\CR} \CK_{J}.$$

\ssec{Levi}{Passing to a Levi}
We continue with the above notations. Let $L:=L_{J}=C_{G}(A_{J})\subset G$ and $P:=P_{J}\supset P_{0}$ be the unique parabolic subgroup of $G$ containing $P_{0}$ with $L$ as a Levi subgroup. Then $L$ and $P$ are also $\s$-stable hence defined over $\RR$.

Let $X_{L}$ be the flag variety of $L_{J}$ and let $\wt X_{L}$ be the $\bT^{c}$-torsor over $X_{L}$. We have the monodromic category $\CM_{L_{\RR}}$ and the subcategory of \fmo tilting sheaves $\Tilt(\CM_{L_{\RR}})$. Applying \refl{R supp} to $L_{\RR}$ we see that the action of $\CR$ on $\id_{\Tilt(\CM_{L_{\RR}})}$ also factors through $\cO(V_{r})$ (in fact we can replace $V_{r}$ by a further subscheme, which is the image of $V_{\l_{0}}\times_{V\sslash W_{L}}V\to V$). Therefore it makes sense to define the localization $\Tilt(\CM_{L_{\RR}})\ot_{\CR}\CK_{J}$.

Let $p_L\colon P\to L$ be the projection, whose kernel is the unipotent radical $U_{P}$ of $P$. We have a closed embedding
\begin{equation*}
i_{P}: X_{L}\incl X
\end{equation*}
sending $B'\in X_{L}$ to $p_{L}^{-1}(B')\in X$. The image of $i_{P}$ is the set of Borel subgroups of $G$ contained in $P$. Similarly we have a closed embedding of enhanced flag varieties
\begin{equation*}
\wt i_{P}: \wt X_{L}\incl \wt X
\end{equation*}
covering $i_{P}$.  Consider the diagram of real algebraic stacks
\begin{equation}\label{corr L G}
\xymatrix{L_{\RR}\bs \wt X_{L} &\ar[l]_-{\wt \pi_{P}} P_{\RR}\bs \wt X_{L} \ar[r]^{\wt \io_{P}} & G_{\RR} \bs \wt X}
\end{equation}
where $\wt\pi_{P}$ is induced by the projection $p_{L,\RR}:P_{\RR}\to L_{\RR}$ (so $\wt \pi$ is a $U_{P,\RR}$-gerbe) and $\wt\io_{P}$ is induced by $\wt i_{P}$. Similarly, we have a diagram
\begin{equation*}
\xymatrix{L_{\RR}\bs X_{L} &\ar[l]_-{\pi_{P}} P_{\RR}\bs \wt X_{L} \ar[r]^{\io_{P}} & G_{\RR} \bs X}
\end{equation*}

\lem{closedembeddingofstacks}
The maps of real algebraic stacks $\io_{P}: P_\BR\backslash X_L\to G_\BR\backslash X$ and $\wt \io_{P}: P_\BR\backslash \wt X_L\to G_\BR\backslash \wt X$ are closed embeddings. 
\elem
\prf
The embedding $\wt i_P$ induce isomorphisms $G\x^P\wt X_L\iso \wt X$. We have a closed embedding $G_\BR\x^{P_\BR}\wt X_L\hookrightarrow G\x^P\wt X_L$. Composing the two maps we get a closed embedding $G_\BR\x^{P_\BR}\wt X_L\incl \wt X$. Quotienting by $G_\BR$ we get the desired closed embedding $\wt\io_{P}$. Since $\io_{P}$ is $\bT^{c}$-equivariant, quotienting by $\bT^{c}$ shows that $\io_{P}$ is also a closed embedding.
\epr

Recall $d_{\l_{0}}$ (resp. $d_{\mu_{0}}$) is the real dimensions of the closed $G(\RR)$-orbit in $X$ (resp. closed $L(\RR)$-orbit on $X_{L}$). Note that $n_{\l_{0}}=n_{\mu_{0}}$. Let
\begin{equation*}
\d^{G}_{L}=\lfloor\frac{d_{\l_{0}}+n_{\l_{0}}}{2}\rfloor-\lfloor\frac{d_{\mu_{0}}+n_{\mu_{0}}}{2}\rfloor.
\end{equation*}
Note that $d_{\l_{0}}-d_{\mu_{0}}=\dim_{\CC}(G/P)$. When $\dim_{\CC}(G/P)$ is even, $\d^{G}_{L}=\frac{1}{2}\dim_{\CC}(G/P)$; when $\dim_{\CC}(G/P)$ is odd, $\d^{G}_{L}$ is one of the integers closest to $\frac{1}{2}\dim_{\CC}(G/P)$. 

\prop{blockandlevi} 
Consider the functor
\begin{equation*}
\g:=\wt\io_{P*}\wt\pi^{*}_{P}[\d^{G}_{L}]: \CM_{L_{\RR}}\to \CM_{G_{\RR}}.
\end{equation*}
\begin{enumerate}
\item $\g$ is an $\cH_{L}$-equivariant full embedding.
\item $\g$ restricts to a $\Tilt(\cH_{L})$-equivariant full embedding of tilting sheaves
\begin{equation*}
\g_{\Tilt}: \Tilt(\CM_{L_{\RR}})\to \Tilt(\CM_{G_{\RR}}).
\end{equation*}

\item $\g_{\Tilt}$ induces an equivalence after localization
\begin{equation*}
\Tilt(\CM_{L_{\RR}})\ot_{\CR}\CK_{J}\isom \Tilt(\CM_{G_{\RR}})\ot_{\CR}\CK_{J}.
\end{equation*}
\end{enumerate}
\eprop


The proof will be given after some preparations.

\lem{intersectingxl} Recall $J$ is a subset of simple roots of $\Phi(G,A_{0})$ that cut out $A_{J}, \bA_{J}, \bA^{\vee}_{J}$ and $V_{J}$.  For a fixed $\l\in I$ we have $V_J\subset V_\l$ if and only if the intersection $O_\l^\BR\cap i_P(X_L)$ is nonempty, i.e. if and only if there exists a Borel subgroup $B'\subset P\subset G$ contained in $O_\l^\BR$ as a point of $X$.
\elem

\prf
Note that the following are equivalent:
\begin{equation}\label{VJ incl}
V_{J}\subset V_{\l} \iff \bA_{J}\subset \bT^{-\theta_{\l}}.
\end{equation}


Suppose $B\subset P$ is a Borel subgroup and $B_{L}=B\cap L$.  Let $T\subset B_{L}$ be a $\s$-stable torus. Note that $A_{J}$ is the split center of $L$, hence $A_{J}\subset T$. Consider the image of $A_{J}$ under $\io_{B}: T\subset B\surj \bT$. On the one hand, $\io_{B}$ and $\io_{B_{0}}$ restricts to the same map $A_{J}\to \bT$ (because they differ by $\bW_{L}$), hence $\io_{B}(A_{J})=\bA_{J}$. On the other hand,  $A_{J}$ is contained in complexification of the split part of $T_{\RR}$,  hence $\io_{B}(A_{J})\subset \bT^{-\theta_{\l}}$. Therefore $\bA_{J}\subset \bT^{-\theta_{\l}}$, and $V_{J}\subset V_{\l} $ by \eqref{VJ incl}.

Conversely, suppose $V_{J}\subset V_{\l}$, hence $\bA_{J}\subset \bT^{-\theta_{\l}}$ by \eqref{VJ incl}. Let $B\in O^{\RR}_{\l}$ and $T\subset B$ be a $\s$-stable maximal torus. Let $A_{1}\subset T$ be the complexification of the split part of $T_{\RR}$. Changing $(T,B)$ by $G(\RR)$-conjugacy, we may assume $A_{1}\subset A_{0}$. Under $\io_{B}: T\subset B\surj \bT$, we have $\io_{B}(A_{1})=\bT^{-\theta_{\l},\c}\supset \bA_{J}$. Let $A'_{J}=\io^{-1}(\bA_{J})\subset A_{1}$.

Now we have two subtori $A_{J},A'_{J}$ of $A_{0}$, which is in turn in $T_{0}$. Via $\io_{B_{0}}$ and $\io_{B}$ respectively, they map isomorphically to $\bA_{J}$. Let $a'_{J}\in A'_{J}$ be a generic element, and let $a_{J}\in A_{J}$ such that $\io_{B}(a'_{J})=\io_{B_{0}}(a_{J})$. Both $a'_{J}$ and $a_{J}$ are in $T_{0}$ and they are in the same conjugacy class of $G$, there exists $w\in W(G,T_{0})$ such that $w(a'_{J})=a_{J}$. Since $a'_{J}$ is generic in $A'_{J}$,  $w$ restricts to the isomorphism $A'_{J}\xr{\io_{B}} \bA_{J}\xr{\io_{B_{0}}^{-1}} A_{J}$.  Since $a_{J}, a'_{J}\in A_{0}$ are the same $W(G,T_{0})$-orbit, they are also in the same $W(G,A_{0})$-orbit.   Let  $u\in W(G,A_{0})$ be such that $u(a'_{J})=a_{J}$, then $u|_{A'_{J}}=w|_{A'_{J}}$. Since $W(G,A_{0})=W(G(\RR), T_{0,\RR})$, we can lift $u$ to $\dot u\in G(\RR)$ normalizing $T_{0}$.  We have a commutative diagram
\begin{equation*}
\xymatrix{ A_{J}\ar@/^2pc/@{=}[rr]\ar[dr]_{\io_{\Ad(\dot u)B}}  & A'_{J} \ar[l]_{u}\ar[r]^{w}\ar[d]^{\io_{B}} & A_{J}\ar[dl]^{\io_{B_{0}}} \\
& \bT}
\end{equation*} 
By \refl{A fixed Borel} below, $B_{0}$ and $\Ad(\dot u)B$ are in the same  $L=C_{G}(A_{J})$-orbit of $X$. Now $B_{0}\in i_{P}(X_{L})$, which is the $L$-orbit through $B_{0}$, we have $\Ad(\dot u)B\in i_{P}(X_{L})\cap O^{\RR}_{\l}$.  
\epr

\lem{A fixed Borel}
Let $A\subset G$ be a torus. Consider the map $\ka: X^{A}\to \Hom(A,\bT)$ sending $B\in X^{A}$ (i.e., a Borel subgroup $B$ containing $A$) to the map $\io_{B}: A\subset B\surj \bT$. Then each non-empty fiber of $\ka$ is stable under the Levi subgroup $C_{G}(A)$ of $G$, and is $C_{G}(A)$-equivariantly isomorphic to the flag variety of $C_{G}(A)$.
\elem
\prf
The map $\ka$ is equivariant under $N_{G}(A)$, therefore each fiber has an action by $C_{G}(A)$.

Let $B_{1},B_{2}\in X^{A}$ with the same image under $\ka$. Let $T_{i}\subset B_{i}$ be a maximal torus containing $A$, $i=1,2$. Let $\io_{i}: T_{i}\subset B_{i}\surj \bT$ be the isomorphisms induced by $B_{i}$. The fact that $\ka(B_{1})=\ka(B_{2})$ implies $\io_{1}|_{A}=\io_{2}|_{A}$. Let $g\in G$ be such that $\Ad(B_{1})=B_{2}$ and $\Ad(T_{1})=T_{2}$.  Then $\io_{2}\c\Ad(g)=\io_{1}\in \Hom(T_{1}, \bT)$. Restricting to $A$ we get that $\io_{2}(\Ad(g)a)=\io_{1}(a)$, which is $\io_{2}(a)$ for all $a\in A$. Therefore $\Ad(g)a=a$ hence $g\in C_{G}(A)$. This shows that each non-empty fiber of $\ka$  is a homogeneous space for $C_{G}(A)$.

The stabilizers of $C_{G}(A)$ on $X^{A}$ are clearly Borel subgroups of $C_{G}(A)$. Hence each non-empty fiber of $\ka$  is isomorphically to the flag variety of $C_{G}(A)$.
\epr


\prf[Proof of \refp{blockandlevi}] 
(1) Since $\wt\pi_{P}$ is a gerbe for a unipotent group,  $\wt\pi_{P}^{*}:\CM_{L_{\RR}}\to \CM_{P_{\RR}}$ is an equivalence. Since $\wt\io_{P}$ is a closed embedding by \refl{closedembeddingofstacks}, $\wt\io_{P*}:\CM_{P_{\RR}}\to \CM_{G_{\RR}}$ is a full embedding. Therefore $\g$ is a full embedding. It is $\cH_{L}$-equivariant because both $\wt\pi_{P}^{*}$ and $\wt\io_{P*}$ are.

(2) 
Let $I_{L}$ and $\wt I_{L}$ be the counterparts of $I$ and $\wt I$ for $L_{\RR}$. For $\mu\in I_{L}$, let $O^{\RR}_{L,\mu}\subset X_{L}$ be the corresponding orbit. The diagram \eqref{corr L G} induces a map
$\ka: I_{L}\incl I$ sending $\mu$ to the unique $\l\in I$ such that $i_{P}(O^{\RR}_{L,\mu})\subset O^{\RR}_{\l}$.  Since $\io_{P}$ is a closed embedding by \refl{closedembeddingofstacks}, $\ka$ is injective. 

For any $\mu\in I_{L}$, we have a canonical isomorphism $\bT^{c}_{\mu}\cong \bT^{c}_{\ka(\mu)}$, therefore the injection $\ka$ lift canonically to an injection $\wt\ka: \wt I_{L}\incl \wt I$ such that the following diagram is Cartesian
\begin{equation*}
\xymatrix{ \wt I_{L} \ar[d]\ar[r]^-{\wt \ka}& \wt I\ar[d]\\
 I_{L}\ar[r]^-{\ka} & I}
\end{equation*}

For any $(\mu,\psi)\in \wt I_{L}$, with $\wt\ka(\mu,\psi)=(\l,\chi)\in \wt I$. Restricting the diagram \eqref{corr L G} to $P(\RR)\bs\wt O^{\RR}_{L,\mu}$ we get maps
\begin{equation*}
\xymatrix{L(\RR)\bs\wt O^{\RR}_{L,\mu} & \ar[l]_-{\wt\pi_{P,\mu}} P(\RR)\bs\wt O^{\RR}_{L,\mu}\ar[r]^-{\wt\io_{P,\mu}} & G(\RR)\bs \wt O^{\RR}_{\l}
}
\end{equation*}
where $\wt\pi_{P,\mu}$ is a $\bT^{c}$-equivariant $U_{P,\RR}$-gerbe and $\wt\io_{P,\mu}$ is a $\bT^{c}$-equivariant isomorphism by \refl{closedembeddingofstacks}. This implies that $\g$ sends $\wt\D_{\mu,\psi}$ to a shift of $\wt\D_{\l, \chi}$. The perversity functions for $X_{L}$ and $X$ are related by
\begin{equation*}
p_{\l}-p_{\mu}=\lfloor\frac{d_{\l}+n_{\l}}{2}\rfloor-\lfloor\frac{d_{\mu}+n_{\mu}}{2}\rfloor=\lfloor\frac{d_{\l_{0}}+n_{\l_{0}}}{2}\rfloor-\lfloor\frac{d_{\mu_{0}}+n_{\mu_{0}}}{2}\rfloor=\d^{G}_{L}
\end{equation*}
because $d_{\l}-d_{\mu}=d_{\l_{0}}-d_{\mu_{0}}=\dim_{\RR}G_{\RR}/P_{\RR}$ (which follows from the fact that $\wt\io_{P,\mu}$ is an isomorphism), $n_{\l}=n_{\mu}$, $d_{\l}+n_{\l}\equiv d_{\l_{0}}+n_{\l_{0}}\mod 2$ and $d_{\mu}+n_{\mu}\equiv d_{\mu_{0}}+n_{\mu_{0}}\mod 2$. Therefore, $\g$ sends $\wt\D_{\mu,\psi}$ to $\wt\D_{\l, \chi}$. 

For the same reason, $\g$ sends $\tilna_{\mu,\psi}$ to $\tilna_{\l,\chi}$. Therefore $\g$ sends $\Tilt(\CM_{L_{\RR}})$ to $\Tilt(\CM_{G_{\RR}})$.

(3) By (2) we see that $\g_{\tilt}\ot \cK_{J}$ is fully faithful. Let us show it is essentially surjective.

Let $X'\subset X$ be the union of $G(\RR)$-orbits that intersect $i_{P}(X_{L})$, and let $\wt X'\subset \wt X$ be its preimage in $\wt X$. By \refl{closedembeddingofstacks}, $X'\cong G_{\RR}\times^{P_{\RR}}X_{L}$ is closed in $X$. The essential image of $\g_{\Tilt}$, hence also the essential image of $\g_{\Tilt}\ot\cK_{J}$, consists of tilting sheaves supported on $\wt X'$. 

On the other hand, by \refl{intersectingxl}, after the localization $-\ten_\CR\CK_J$, objects in $\Tilt(\CM_{G_{\RR}})\ot_\CR\CK_J$ are exactly the images of tilting sheaves supported on $\wt X'$, which is the image of $\g_{\Tilt}\ot \cK_{J}$ by the above paragraph. This proves the essential surjectivity of $\g_{\Tilt}\ot \cK_{J}$. 

\epr


\ssec{locHk}{Localized Hecke action} 
We keep working in the above setting. Consider the second projection $$\pr_{J,V}\colon V_J\x_{V\sslash \bW}V\to V$$ and let $V_{J,r}$ be its scheme-theoretic image. Let $\CR_J$ be the localization of $V$ at the generic points of the irreducible components of $V_{J,r}$. It is isomorphic to the product of the copies of $\CK_J$ indexed by the $\bW$-orbit of $V_J$ as a subspace of $V$. 

Recall that the Hecke category is a $\CR\ten_{\CR^{\bW}}\CR$-linear category. We define the localization of the Hecke category $\CR_J\ten_\CR\Tilt(\CH_G)\ten_\CR\CR_J$ to be the idempotent completion of the $\CR_J\ten_{\CR^{\bW}}\CR_J$-linear category  whose objects are the same as the objects of $\Tilt(\CH_G)$ and the morphisms are given by 
$$\Hom_{\CR_J\ten_\CR\Tilt(\CH_G)\ten_\CR\CR_J}(\CT_1,\CT_2):=\CR_J\ten_\CR\Hom_{\Tilt(\CH_G)}(\CT_1,\CT_2)\ten_\CR\CR_J.$$ 
Note that $\CR_J\ten_\CR\Tilt(\CH_G)\ten_\CR\CR_J$ splits into the direct product of categories numbered by a pair of elements in the $\bW$-orbit of $V_J$. 
The $\star$-action of $\Tilt(\CH_G)$ on $\TGR$ passes to a functor 
\eq{locact}
\star\colon (\TGR\ten_\CR\CR_{J})\times(\CR_{J}\ten_\CR\Tilt(\CH_G)\ten_\CR\CR_{J})\to\TGR\ten_\CR\CR_{J}.\eeq
The action is compatible with the direct product decompositions.

\ssec{loc0}{Localization at codimension $0$}
We continue with the notations of \refss{locJ}.
Let $\CK$ be the field of fractions of $\CS$. 
We also put $\CQ:=\CK^{\bW_{\l_{0}}}$ for the field of fractions of $\CS^{\bW_{\l_{0}}}$, which is also the field of fractions of $(\CS^{\bW_{\l_{0}}})'=\textup{Im}(\CR^\bW\to\CS^{\bW_{\l_{0}}})$. Consider the natural composition $\CR^\bW\to\CS^{\bW_{\l_{0}}}\to\CQ$.  Then $\CS\ten_{\CR^\bW}\CQ\cong\CK$.



Recall $T_{0}$ is a maximally split real torus and $A_{0}\subset T_{0}$ is the split subtorus. Let $I_{0}\subset I$ be the set of orbits attached to $T_{0}$ in the sense of \refd{att}, and $\wt I_{0}$ be the preimage of $I_{0}$ in $\wt I$. Note that for $\l\in I_{0}$ we have $n_\l=\dim A_{0}$; for $\l\notin I_{0}$ we have $n_\l<\dim A_{0}$. For $\l\in I_{0}$, $V_{\l}=\Spec \ol\CR_{\l}$ is a $w$-translate of $V_{\l_{0}}=\Spec \CS\subset \Spec\CR$ for some $w\in \bW$.

For $\l\in I_{0}$,  let $\CK_{\l}=\Frac(\ol\CR_{\l})$. We have $\ol\CR_{\l}\ot_{\CR^{\bW}}\CQ=\CK_{\l}$. We have an isomorphism  $\CK\cong\CK_{\l}$ unique up to precomposing with the action of $W(\RR)$. Then $\cR\ot_{\cR^{\bW}}\cQ\cong \prod_{\l\in I_{0}}\cK_{\l}$.

Consider the localization $\CM_{G_{\RR},\cQ}$ of the category $\CM_{G_{\RR}}$, obtained as the idempotent completion of the additive category whose objects are the same as the objects of $\CM_{G_\BR}$ and the morphisms are given by $$\Hom_{\CM_{G_\BR,\CQ}}(\CF_1,\CF_2):=\Hom_{\CM_{G_\BR}}(\CF_1,\CF_2)\ten_{\CR^\bW} \CQ.$$ 


\prop{loc}
\begin{enumerate}
\item If $(\l,\chi)\in \wt I -  \wt I_{0}$, then $\wt\D_{\l,\chi}$ and $\wt \Na_{\l,\chi}$ are zero in $\CM_{G_\BR,\CQ}$.
\item For $(\l,\chi)\in \wt I_{0}$, we have $\End_{\CM_{G_\BR,\CQ}}(\wt\D_{\l,\chi},\wt\D_{\l,\chi})\cong \CK_{\l}$.
\item  The functor
\begin{equation*}
\bigoplus_{(\l,\chi)\in \wt I_{0}}D^{b}(\CK_{\l}\lmod)\to \CM_{G_\BR,\CQ}
\end{equation*}
sending $(M_{\l,\chi})_{(\l,\chi)\in \wt I_{0}}$ to $\bigoplus_{(\l,\chi)\in \wt I_{0}} M_{\l,\chi}\ot _{\CK_{\l}}\wt\D_{\l,\chi}$ is an equivalence. 
\end{enumerate}
\eprop
\prf
(1) If $\l\notin I_{0}$, the action of $\CR$ on $\wt\D_{\l,\chi}$ and $\wt \Na_{\l,\chi}$ factor through $\CR_{\l}$ but $\Spec\ol\cR_{\l}$ has dimension $n_{\l}$ which is smaller than the dimension of $\Spec\CS^{W(\RR)}$. Therefore the actions of $\CR^{\bW}$ on $\wt\D_{\l,\chi}$ and $\wt \Na_{\l,\chi}$ also factor through a quotient with smaller dimension than $\Spec\CS^{W(\RR)}$,  hence  localizing to the generic point of $\Spec\CS^{W(\RR)}$ kills  $\wt\D_{\l,\chi}$ and $\wt \Na_{\l,\chi}$.


(2) If $(\l,\chi)$ and $(\l,\psi)\in \wt I_{0}$, then by \refl{comp DA}, 
$$\RHom_{\CM_{G_\BR}}(\tilDel_{\l,\chi}, \tilDel_{\l,\psi})=\begin{cases}\ol\CR_{\l}, & \chi=\psi,\\
0, &\chi\ne\psi. \end{cases}$$
Tensoring with $\CQ$ we get
$$\RHom_{\CM_{G_\BR,\CQ}}(\tilDel_{\l,\chi}, \tilDel_{\l,\psi})=\begin{cases}\CK_{\l}, & \chi=\psi,\\
0, &\chi\ne\psi. \end{cases}$$

(3) Since the $\{\wt\D_{\l,\chi}\}_{(\l,\chi)\in\wt I}$ generate $\CM_{G_\BR,\CQ}$, in view of part (1) and (2), it remains to show that $\RHom_{\CM_{G_\BR,\CQ}}(\tilDel_{\l,\chi}, \tilDel_{\mu,\psi})=0$ for distinct orbits $\l,\mu\in I_{0}$. Note that the support of $\RHom_{\CM_{G_\BR}}(\tilDel_{\l,\chi}, \tilDel_{\mu,\psi})$ as an $\CR$-module is contained in $V_{\l}\cap V_{\mu}$. By \refl{realWeyl} and \cite[Proposition 12.9 and 12.14]{AdC},  that $(\bT^c_\l)^{\c}\ne (\bT^c_\mu)^{\c}$ as subtori of $\bT^{c}$, hence the intersection $V_{\l}\cap V_{\mu}$ has smaller dimension than $V_{\l_{0}}$. The same holds for the support of $\RHom_{\CM_{G_\BR}}(\tilDel_{\l,\chi}, \tilDel_{\mu,\psi})$ as an $\CR^\bW$-module. Therefore tensoring with $\CQ$ kills $\RHom_{\CM_{G_\BR}}(\tilDel_{\l,\chi}, \tilDel_{\mu,\psi})$.   


\epr



By \refp{loc}, indecomposable objects in the localized category $\Tilt(\CM_{G_{\RR}})_{\CQ}$ (idempotent completion of the image of $\Tilt(\CM_{G_{\RR}})$ in $\CM_{G_{\RR}, \CQ}$) are of the form $\wt\D_{\l,\chi}$ where $(\l,\chi)\in \wt I_{0}$ (these are objects in the localized tilting category because they are direct summands of $\cT_{\l,\chi}$ after localization).  
On the other hand,  we can define the localized Hecke category $\cH_{G,\cQ}$ and tilting sheaves $\Tilt(\cH_{G})_{\cQ}$ inside it. For $w\in \bW$, $\wt \D_{w}$ is a direct summand of $\CT_{w}$ in $\Tilt(\cH_{G})_{\cQ}$, hence $\wt\D_{w}\in \Tilt(\cH_{G})_{\cQ}$. The assignment $w\mapsto \wt\D_{w}$ gives a monoidal functor from $\bW$ (viewed as a category with objects $\bW$ and only identity morphisms) to $\Tilt(\cH_{G})_{\cQ}$. In particular,  the  $\Tilt(\cH_{G})_{\cQ}$-action on $\Tilt(\CM_{G_\BR})_{\CQ}$ induces a  right $\bW$-action on the set of isomorphisms classes of indecomposable objects in $\Tilt(\CM_{G_\BR})_{\CQ}$, and hence induces a right action of $\bW$ on $\wt I_{0}$.

\lem{same as cross} Assume that $G_{\RR}$ is  quasi-split. The action of $\bW$ on $\wt I_{0}$ defined above is the same as the restriction of the cross action. In other words, in $\CM_{G_\BR,\CQ}$ we have an isomorphism $\wt\D_{\l,\chi}\star\wt\D_{w}\cong \wt\D_{(\l,\chi)\cdot w}$ for $(\l,\chi)\in \wt I_{0}$ and $w\in \bW$.
\elem
\prf When $G_{\RR}$ is quasi-split and $\l\in I_{0}$, all simple roots in $\Phi_{\l}$ are in case (1) or (3) in \refl{stdconv}.  For $(\l,\chi)\in \wt I_{0}$, by \refl{stdconv} we have $\wt\D_{\l,\chi}\star\wt\D_{s}\cong \wt\D_{(\l,\chi)\cdot s}$ in $\CM_{G_\BR,\CQ}$ (the contribution of orbits $\mu$ or $\mu^{\pm}$ become zero after localization). Writing any $w\in \bW$ as a product of simple reflections,  we see that $\wt\D_{\l,\chi}\star\wt\D_{w}$ and $\wt\D_{(\l,\chi)\cdot w}$ become isomorphic in $\CM_{G_\BR,\CQ}$. This implies the lemma.
\epr

\sec{realV}{Real Soergel functor} 

In this section we extend Soergel's functor $\VV$ from complex groups to quasi-split real groups, and study the behavior of tilting sheaves under the real Soergel functor.

From this section on, we assume that $G_\BR$ is quasi-split. 

\ssec{}{Regular covector} Let $\cN^{*}\subset \grg^{*}$ be the nilcone. Also we have a decomposition $\grg^{*}=\grg_{\RR}^{*}\op i\grg^{*}_{\RR}$. Let 
\begin{equation*}
i\cN^{*}(\RR):=\cN^{*}\cap i\grg^{*}_{\RR}.
\end{equation*}
Let $\cO_{\reg}\subset \cN^{*}$ be the regular nilpotent orbit.

\lem{reg real orbit} 
\begin{enumerate}
\item We have $i\cN^{*}(\RR)\cap \cO_{\reg}$ if and only if $G(\RR)$ is quasi-split.
\item Suppose $\xi\in i\cN^{*}(\RR)\cap \cO_{\reg}$, then the Springer fiber $\cB_{\xi}$ is a single point, and is contained in the closed $G(\RR)$-orbit of $X$.
\end{enumerate}

\elem
\prf Suppose $\xi\in i\cN^{*}(\RR)\cap \cO_{\reg}$, then the Springer fiber $\cB_{\xi}\subset X$ is a single point, hence a real point of the flag variety $X$ since $\xi$ is pure imaginary. In particular, the point $\cB_{\xi}$ gives a Borel subgroup of $G$ defined over $\RR$. Since the closed $G(\RR)-$-orbit in $X$ parametrizes Borel subgroups that are defined over $\RR$, $\cB_{\xi}$ is contained in the closed $G(\RR)$-orbit of $X$.

Conversely, assume $G_{\RR}$ is quasi-split and $B_{\RR}\subset G_{\RR}$ is a Borel subgroup defined over $\RR$. Let $\grn_{\RR}$ be the nilpotent radical of $\Lie B_{\RR}$. Then a generic element $\xi\in i\grn_{\RR}$ (i.e., its projection to each simple root space is nonzero), viewed as an element in $i\grg^{*}_{\RR}$ using the Killing form, is regular.
\epr

\ssec{GVC}{Generic vanishing cycles}
In the rest of the section we assume $G_{\RR}$ is quasi-split.  In this case, the closed $G(\RR)$ orbit $O_{\l_{0}}^{\RR}\subset X$ is the set of real points of $X$.  

Consider the moment map of $\wt X=G/U\bT^{>0}$ for the left action of $G$:
\begin{equation*}
\mu: T^{*}\wt X\to \grg^{*}.
\end{equation*}
For $\bT^{c}$-monodromic sheaves $\cF\in D^{b}(\wt X)_{\bT^{c}\mon}$, its singular support $SS(\cF)$ is contained in the image of the pullback $T^{*}X\times_{X}\wt X\to T^{*}\wt X$, which is equal to $\mu^{-1}(\cN^{*})$. On the other hand, if $\cF\in D^{b}_{G(\RR)}(\wt X)$, then $SS(\cF)\subset \mu^{-1}(i\grg^{*}_{\RR})$. Let 
\begin{equation*}
\L_{\RR}=\mu^{-1}(i\cN^{*}(\RR))\subset T^{*}\wt X.
\end{equation*}
Then the above discussion shows that for $\cF\in D^{b}_{G(\RR)}(\wt X)_{\bT^{c}\mon}$, $SS(\cF)\subset \L_{\RR}$. Also,  since $G(\RR)\times \bT^{c}$ has finitely many orbits on $\wt X$, $\L_{\RR}$ is the union of conormals of the orbits $\{\wt O^{\RR}_{\l}\}$:
\begin{equation*}
\L_{\RR}=\bigcup_{\l\in I} T^{*}_{\wt O^{\RR}_{\l}}\wt X.
\end{equation*}

Let $\xi\in i\cN^{*}(\RR)\cap \cO_{\reg}$ and let
\begin{equation*}
\Xi:=\mu^{-1}(\xi)\subset T^{*}\wt X.
\end{equation*}
Then by \refl{reg real orbit}, the  Springer fiber $\cB_{\xi}$ is a single point $x$ contained in the closed $G(\RR)$-orbit $O^{\RR}_{\l_0}\subset X$ that is a real form of $X$.  Hence projection to $\wt X$ is an isomorphism $\Xi\isom \pi^{-1}(x)\subset \wt X$ (a single $\bT^{c}$-orbit).  Moreover, $\Xi\subset T^{*}_{\wt O^{\RR}_{\l_0}}\wt X$ but is disjoint from the conormals of other $\wt O^{\RR}_{\l}$.  

For a manifold $M$ and a submanifold $N\subset M$, and $\cF\in D^{b}(M)$, we use $\mu_{N}\cF\in D^{b}(T^{*}_{N}M)$ to denote the microlocalization of $\cF$ along $N$.   See \cite[Definition 4.3.1]{KS}.

For $\cF\in D^{b}_{G(\RR)}(\wt X)_{\bT^{c}\mon}$, denote the microlocalization $\mu_{\wt O^{\RR}_{\l_0}}(\cF)\in D^{b}(T^{*}_{\wt O^{\RR}_{\l_0}}\wt X)_{\bT^{c}\mon}$ along $\wt O^{\RR}_{\l_0}$ simply by $\mu_{\l_{0}(\cF)}$;  it is  locally constant on  the generic part of $T^{*}_{\wt O^{\RR}_{\l_0}}$. Restricting to $\Xi$ gives a functor
\begin{equation*}
 D^{b}_{G(\RR)}(\wt X)_{\bT^{c}\mon}\xr{\mu_{\l_0}} D^{b}(T^{*}_{\wt O^{\RR}_{\l_0}}\wt X)_{\bT^{c}\mon} \to  D^{b}(\Xi)_{\bT^{c}\mon} \cong D^{b}(\pi^{-1}(x))_{\bT^{c}\mon} 
\end{equation*}
Passing to completions and shifting so that a free-monodromic local system on the closed orbit in the heart of the $t$-structure is sent to an object concentrated in degree $0$, i.e. just its stalk shifted to degree $0$, we get a functor
\begin{equation*}
\VV_{\RR, \xi}: \CM_{G_\BR}=\wh D^{b}_{G(\RR)}(\wt X)_{\bT^{c}\mon}\to \wh D^{b}(\Xi)_{\bT^{c}\mon}
\end{equation*}

If we choose a point $\wt x \in \pi^{-1}(x)$ and let $\wt\xi=(\wt x,\xi)\in \Xi$, the base point $\wt\xi$ trivializes  $\Xi$ as a $\bT^{c}$-torsor, and gives an equivalence $\wh D^{b}(\Xi)_{\bT^{c}\mon}\cong D^{b}(\rmod\cR)$. Then $\VV_{\RR,\xi}$ induces a functor
\begin{equation}\label{V real}
\VV_{\RR, \wt \xi}: \CM_{G_\BR}\to D^{b}(\rmod\cR)
\end{equation}
that depends on the choice of $\wt \xi\in T^{*}\wt X$ over the regular $\xi$. When $\wt \xi$ is fixed, we also write the functor as $\VV_{\RR}$.


\prop{exact} Let $(\l,\chi)\in \wt I$.
\begin{enumerate}
\item $\VV_{\RR}(\Na_{\l,\chi})$ is concentrated in degree $0$.
\item $\VV_{\RR}(\D_{\l,\chi})$ is concentrated in degree $n_{\l_{0}}-n_{\l}$. 
\end{enumerate}
\eprop

\prf
(1) 
We argue by induction on $d_{\l}=\dim_{\RR} O^{\RR}_{\l}$. If $\l=\l_{0}$ corresponds to the closed orbit, then  $\VV_\BR(\nabla_{\l,\chi})$ is the stalk of $\CL_{\l,\chi}$ along $\wt O^{\RR}_{\l_{0}}$ (up to a shift), and it is normalized to be in degree $0$. Otherwise, choose a point $x\in O^{\RR}_{\l}$ (corresponding to a Borel $x$) such that $B$ contains a $\s$-stable maximal torus $T$, and we can talk about the based root system $\Phi(G,T)$ with positive roots defined by $B$.  Since $O^{\RR}_{\l}$ is not closed, there is a simple root $\alp\in \Phi(G,T)$ and $\mu<\l$, such that $\pi_\alp(O_\lambda^\BR)=\pi_\alp(O_\mu^\BR)$ for the projection $\pi_\alp\colon X\to X_{\a}=G/P_\alp$. Since $O_{\l}^{\RR}$ is not closed, we have the following cases according to \refl{cl orb}:
\begin{itemize}
\item $\a$ is complex and $\s\a<0$. In this case, we have $p_{\l}-p_{\mu}=1$, and a distinguished triangle $\pi_\alp^*\pi_{\a*}\nabla_{\lambda, \chi}\to\nabla_{\lambda, \chi}\to \nabla_{\mu,\psi}\to$ for some $\psi: \pi_{0}(\bT^{c}_{\mu})\to \bfk^{\times}$.
\item $\a$ is noncompact imaginary and $\pi_{\a}^{-1}\pi_{\a}(O^{\RR}_{\l})=O^{\RR}_{\l}\cup O^{\RR}_{\mu}$. We have $p_{\l}=p_{\mu}$, and a distinguished triangle $\pi_\alp^*\pi_{\a*}\nabla_{\lambda, \chi}\to\nabla_{\lambda, \chi}\to \nabla_{\mu,\psi}\to$ for some $\psi:\pi_{0}(\bT^{c}_{\mu})\to \bfk^{\times}$.
\item $\a$ is noncompact imaginary and $\pi_{\a}^{-1}\pi_{\a}(O^{\RR}_{\l})=O^{\RR}_{\l}\cup O^{\RR}_{\l'}\cup O^{\RR}_{\mu}$, with $\mu<\l$ and $\mu<\l'$. Then  $p_{\l}=p_{\l'}=p_{\mu}$, and we have a distinguished triangle $\pi_\alp^*\pi_{\a*}\nabla_{\lambda, \chi}\to\Na_{\lambda, \chi}\op\Na_{\l',\chi'}\to \Na_{\mu,\psi}\to$ for some $\chi':\pi_{0}(\bT^{c}_{\l'})\to \bfk^{\times}$ and $\psi:\pi_{0}(\bT^{c}_{\mu})\to \bfk^{\times}$.
\end{itemize}
We have $\VV_\BR(\pi^{*}_\alp(\CG))=0$ for any $\CG\in D^{b}_{G(\RR)}(X_{\a})$, as the covector $\xi$ does not lie in the image of the pullback of cotangent bundles  $d\pi_{\a}: (T^{*}X_{\a})\times_{X_{\a}}X\to T^{*}X$. From the distinguished triangles above we see that $\VV_{\RR}(\Na_{\l,\chi})$ is a direct summand of $\VV_{\RR}(\Na_{\mu,\psi})$. Since $d_{\mu}<d_{\l}$, by inductive hypothesis $\VV_{\RR}(\Na_{\mu,\psi})$ is concentrated in degree $0$, therefore the same is true for $\VV_{\RR}(\Na_{\l,\chi})$.


(2) The argument is similar to the costandard case. In the induction step, we have the following cases
\begin{itemize}
\item $\a$ is complex and $\s\a<0$. In this case, we have $p_{\l}-p_{\mu}=1$, and a distinguished triangle $\D_{\mu,\psi}\to\D_{\lambda, \chi}\to \pi_{\a}^{!}\pi_{\alp!}\D_{\lambda, \chi}\to$. We conclude that $\VV_{\RR}(\D_{\lambda, \chi})\cong \VV_{\RR}(\D_{\mu,\psi})$. Note that $n_{\mu}=n_{\l}$ in this case.
\item $\a$ is noncompact imaginary and $\pi_{\a}^{-1}\pi_{\a}(O^{\RR}_{\l})=O^{\RR}_{\l}\cup O^{\RR}_{\mu}$. We have $p_{\l}=p_{\mu}$, and a distinguished triangle $\D_{\lambda, \chi}\to\pi_\alp^!\pi_{\a!}\D_{\l,\chi}\to \D_{\mu,\psi}\to$. We conclude that $\VV_{\RR}(\D_{\lambda, \chi})\cong \VV_{\RR}(\D_{\mu,\psi})[-1]$. Note that $n_{\mu}-n_{\l}=1$ in this case.
\item $\a$ is noncompact imaginary and $\pi_{\a}^{-1}\pi_{\a}(O^{\RR}_{\l})=O^{\RR}_{\l}\cup O^{\RR}_{\l'}\cup O^{\RR}_{\mu}$, with $\mu<\l$ and $\mu<\l'$. Then  $p_{\l}=p_{\l'}=p_{\mu}$, and we have a distinguished triangle $\D_{\lambda, \chi}\op\D_{\l',\chi'}\to \pi_\alp^!\pi_{\a!}\D_{\l,\chi}\to\D_{\mu,\psi}\to$. We conclude that $\VV_{\RR}(\D_{\lambda, \chi})$ is a summand of $\VV_{\RR}(\D_{\mu,\psi})[-1]$. Again $n_{\mu}-n_{\l}=1$ in this case.
\end{itemize}




\epr

\cor{VR exact tilting} For any \fmo tilting sheaf $\CT\in \TGR$, $\VV_{\RR}(\CT)$ is concentrated in degree $0$.
\ecor

\ssec{V Hk}{Soergel functor for the Hecke category}
Consider the Hecke category $\CH_{G}=\wh D^{b}_{G}(\wt X\times \wt X)_{\bT^{c}\times \bT^{c}\mon}$. The construction of the Soergel functor can be applied to the complex group $G$ viewed as a real group $R_{\CC/\RR}G$ and giving a Soergel functor for $\CH_{G}$. We spell this out.

Consider the doubled moment map $\mu^{(2)}: T^{*}\wt X\times T^{*}\wt X\to \frg^{*}\op \frg^{*}$. Let $\D^{-}(\cN^{*})\subset \D^{-}(\frg^{*})\subset \frg^{*}\op \frg^{*}$ be the anti-diagonally embedded nilcone. Let
\begin{equation*}
\L:=\mu^{(2),-1}(\D^{-}(\cN^{*})).
\end{equation*}
It is well-known that 
\begin{equation*}
\L=\bigcup_{w\in W}T^{*}_{\wt X^{2}_{w}}(\wt X^{2}).
\end{equation*}

Let $\xi\in \cN^{*}\cap \cO_{\reg}$ and consider  $(\xi,-\xi)\in \D^{-}(\cN^{*})$. Let $\Xi=\mu^{-1}(\xi), \Xi^{-}=\mu^{-1}(-\xi)$ and consider 
\begin{equation*}
\mu^{(2),-1}(-\xi,\xi)=\Xi^{-}\times \Xi\subset T^{*}(\wt X^{2}).
\end{equation*}
Since the Springer fiber $\cB_{\xi}$ is a single point $x\in X$, $\mu^{(2),-1}(-\xi,\xi)$ projects isomorphically onto $\pi^{-1}(x)^{2}\subset \wt X^{2}$, which is a $\bT^{c}\times \bT^{c}$-torsor.  In particular, $\Xi^{-}\times \Xi$ is contained in the conormal bundle of $\wt X^{2}_{e}\subset \wt X^{2}$ and not in the closure of the conormals of $\wt X^{2}_{w}$ for $e\ne w\in W$ (i.e., it is contained in the generic part of the $T^{*}_{\wt X^{2}_{e}}(\wt X^{2})$).

For $\cK\in \CH_{G}$, its microlocalization along $\wt X^{2}_{e}$ is locally constant on the generic part of $T^{*}_{\wt X^{2}_{e}}(\wt X^{2})$ and $\bT^{c}\times \bT^{c}$-unipotently monodromic. Restricting to $\Xi^{-}\times \Xi$ gives a functor 
\begin{equation*}
\VV_{(-\xi,\xi)}: \CH_G\to \wh D^{b}(\Xi^{-}\times \Xi)_{\bT^{c}\times \bT^{c}\mon}
\end{equation*}

If we choose $\wt x\in \pi^{-1}(x)$ hence $\wt\xi=(\wt x,\xi)\in \Xi$ and $-\wt\xi=(\wt x,-\xi)\in \Xi^{-}$,  we can then identify $\wh D^{b}(\Xi^{-}\times \Xi)_{\bT^{c}\times \bT^{c}\mon}$ with $\wh D^{b}(\bT^{c}\times \bT^{c})_{\bT^{c}\times \bT^{c}\mon}\cong D^{b}(\rmod\CR\ot\CR)$. Then $\VV_{(-\xi,\xi)}$ induces a functor
\begin{equation}\label{VHkpre} 
\VV_{(-\wt \xi,\wt\xi)}: \CH_G\to 
\wh D^{b}(\bT^{c}\times \bT^{c})_{\bT^{c}\times \bT^{c}\mon} \cong D^{b}(\rmod\CR\ot\CR).
\end{equation}

It will be more convenient to turn right $\cR\ot \cR$-modules into $\cR$-bimodules.  Let $\io:\cR\to \cR$ be the involution given by the inversion on $\pi_{1}(\bT^{c})$. We consider the equivalence
\begin{equation}\label{rmod to bimod}
\t: \rmod\cR\ot\cR\isom \cR\bimod\cR
\end{equation}
given by sending $M\in \rmod\cR\ot\cR$ to the same underlying vector space $M$ equipped with the left and right actions of $\cR$ defined by
\begin{equation*}
r_{1}\cdot m\cdot r_{2}=m\cdot (\iota(r_{1})\ot r_{2}), \quad r_{1},r_{2}\in \cR, m\in M.
\end{equation*}
Composing \eqref{VHkpre} with the equivalence \eqref{rmod to bimod}, we get a functor
\begin{equation}\label{V Hk}
{}_{-\wt \xi}\VV_{\wt \xi}: \CH_G\to  D^{b}(\CR\bimod\CR).
\end{equation}
When $\wt\xi$ is understood from the context we simply denote ${}_{-\wt \xi}\VV_{\wt \xi}$ by $\VV$.

Recall in \cite[\S4.5]{BY} (in the $\ell$-adic setting), a similar functor $\VV$ was defined using the action of $\cH_{G}$ on a certain Whittaker category.  A topological analogue has been constructed in \cite[\S11.1]{BR}.


We summarize the main properties of $\VV$ in the following proposition. Most of the assertions are proved in the literature with an a priori different definition of $\VV$. 

\prop{VHk}
\begin{enumerate}
\item Let $\wt \cP_{e}\in \CH^{\hs}_{G}$ \footnote{$\CH^{\hs}_{G}$ denotes the heart of the $t$-structure defined by $p$ for the complex group, which is essentially the middle perversity. } be a projective cover of the constant sheaf on the preimage of $\D(X)\subset X\times X$ in $\wt X\times \wt X$. Then $\VV\cong \RHom(\wt\cP_{e},-)$.
\item\label{Hom Tw0} $\wt\cP_{e}$ is isomorphic to the \fmo indecomposable tilting sheaf $\cT_{w_{0}}$ with full support. Therefore the functor $\VV\cong\RHom(\cT_{w_{0}},-)$.
\item\label{Endo P0} We have an isomorphism of $\cR$-bimodules $\End_{\CH_{G}}(\CT_{w_{0}})\cong \cR\ot_{\cR^{\bW}}\cR$. In particular, if we let $(\CR\bimod\CR)_{\CR^{\bW}}$ be the category of $\CR$-bimodules where the actions of $\CR^{\bW}$ through the left and right copies of $\CR$ coincide, then $\VV$ takes values in $D^{b}((\CR\bimod\CR)_{\CR^{\bW}})$.
\item\label{VHk ff} $\VV$ is fully faithful on \fmo tilting sheaves.
\item For $\CT'\in \Tilt(\CH_{G})$, $\VV(\CT')$ is a Soergel bimodule (see, for example, \cite[Theorem 11.8(1)]{BR} for the definition). Let $\SBim$ be the category of Soergel $\cR$-bimodules.  Then $\VV|_{\Tilt}$ (shorthand for the restriction of $\VV$ to $\Tilt(\CH_{G})$) upgrades to a monoidal functor
\begin{equation}\label{Vsh}
\VV^{\sh}: \Tilt(\CH_{G})\to \SBim.
\end{equation}
\end{enumerate}
\eprop
\prf
(1)We only need to check that $\VV(\IC_{w})=0$ for $w\ne e$ and $\VV(\IC_{e})\cong \bfk$. If $w\ne e$, then $\IC_{w}$ is pulled back from  either $X\times G/P$ or $G/P\times X$ (where $P$ is a parabolic subgroup that is not a Borel). In this case, the singular support of $\IC_{w}$ along the diagonal lies in the pullback of the conormal to the diagonal of  $X\times G/P$ or $G/P\times X$, hence the generic vanishing cycle along the diagonal vanishes, i.e., $\VV(\IC_{w})=0$. The isomorphism $\VV(\IC_{e})\cong \bfk$ is clear since $\IC_{e}$ is the constant sheaf on the diagonal.

(2) See \cite[Lemma 6.9(2)]{BR}.

(3) See \cite[Theorem 9.1]{BR}.

(4) See \cite[Proposition 11.2]{BR}.

(5) See \cite[Theorem 11.8(1)]{BR}.
\epr

\prop{H monoidal} The functor $\VV$ carries a canonical monoidal structure with respect to the convolution product on $\cH_{G}$ and the tensor product $(-)\ot^{\LL}_{\cR}(-)$ on $D^{b}(\CR\bimod\CR)$.
\eprop
\prf First we construct a functorial isomorphism
\begin{equation*}
c_{\cK,\cK'}: \VV(\cK\star \cK')\cong \VV(\cK)\ot_{\cR}^{\LL}\VV(\cK')
\end{equation*}
for $\cK,\cK'\in \CH_{G}$.  Recall $\cK\star \cK'=\pr_{13*}(\pr_{12}^{*}\cK\ot \pr_{23}^{*}\cK')$. Let $\cG=\pr_{12}^{*}\cK\ot \pr_{23}^{*}\cK'$. Let $\wt\D_{13}=\pr_{13}^{-1}(\wt X_{e}^{2})\subset \wt X^{3}$. Let $\pi_{13}: T^{*}_{\wt\D_{13}}(\wt X^{3})\to T^{*}_{\wt X^{2}_{e}}(\wt X^{2})$ be the natural projection. By \refp{KS}(1), we have
\begin{equation*}
\pi_{13*}(\mu_{\wt\D_{13}}\cG)\cong \mu_{\wt X^{2}_{e}}\pr_{13*}\cG.
\end{equation*}
Recall that $\L\subset T^{*}(\wt X^{2})$ is the union of conormals of $\wt X^{2}_{w}$. Now observe that $SS(\cG)\subset (\L\times 0_{\wt X})+(0_{\wt X}\times \L)$ (here we apply \cite[Proposition 5.4.14(i)]{KS}, using that $(\L\times 0_{\wt X})\cap(0_{\wt X}\times \L)$ is contained in the zero section). This implies that $\pi_{13}^{-1}(\Xi^{-}\times\Xi)\cap SS(\cG)\subset \Xi^{-}\times 0_{\pi^{-1}(x)} \times \Xi\subset T^{*}_{\wt\D}(\wt X^{3})$, where $\wt \D\subset \wt X^{3}$ is the preimage of the small diagonal $\D(X)\incl X^{3}$.  Therefore
\begin{equation}\label{push G}
\pi_{13,\Xi*}(\mu_{\wt\D}(\cG)|_{\Xi^{-}\times 0_{\pi^{-1}(x)} \times \Xi})\cong (\mu_{\wt X^{2}_{e}}\pr_{13*}\cG)|_{\Xi^{-}\times \Xi}=\VV_{(-\xi,\xi)}(\cK\star\cK').
\end{equation}
Here $\pi_{13,\Xi}: \Xi^{-}\times 0_{\pi^{-1}(x)} \times \Xi\to \Xi^{-}\times \Xi$ is the projection.

Let $\d_{23}: \wt X^{3}\incl \wt X^{4}$ be the diagonal embedding of the middle two factors. Let $\ol\d_{23}$ be the composition $\wt X^{3}\xr{\d_{23}} \wt X^{4}\to \wt X^{2}\times(G\bs \wt X^{2})$. Then $\cG=\ol\d^{*}_{23}(\cK\bt\cK')$, where we view $\cK'\in D^{b}_{G}(\wt X^{2})$. Now we are ready to apply \refp{KS}(2) (together with \refr{KS stack}) to the smooth map $\ol\d_{23}$, the submanifold $\wt X^{2}_{e}\times(G\bs \wt X^{2}_{e})\subset \wt X^{2}\times (G\bs \wt X^{2})$, and its preimage $\wt \D\subset \wt X^{3}$ under $\ol\d_{23}$.  
There is a canonical embedding induced from $\ol\d_{23}$ (which is an isomorphism in conormal fibers)
\begin{equation*}
\ol\d_{23}^{\na}: T^{*}_{\wt\D}(\wt X^{3})\to T^{*}_{\wt X^{2}_{e}\times(G\bs\wt X^{2}_{e})}(\wt X^{2}\times(G\bs \wt X^{2})).
\end{equation*}
By \refp{KS}(2), we have
\begin{equation}\label{23pull}
\ol\d_{23}^{\na*}(\mu_{\wt X^{2}_{e}}\cK\bt\mu_{G\bs \wt X^{2}_{e}}\cK')=\d_{23}^{\na*}\mu_{\wt X^{2}_{e}\times (G\bs \wt X^{2}_{e})}(\cK\bt\cK')\isom \mu_{\wt\D}(\cG).
\end{equation}
Since $\ol\d^{\na}_{23}$ factors as the composition
\begin{equation*}
\ol\d_{23}^{\na}: T^{*}_{\wt\D}(\wt X^{3})\xr{\d^{\na}_{23}}T^{*}_{\wt X^{2}_{e}\times\wt X^{2}_{e})}(\wt X^{2}\times \wt X^{2}) \to T^{*}_{\wt X^{2}_{e}\times(G\bs\wt X^{2}_{e})}(\wt X^{2}\times(G\bs \wt X^{2}))
\end{equation*}
where the second map is the quotient by $G$, we can rewrite \eqref{23pull} as an isomorphism (viewing $\cK'\in D^{b}(\wt X^{2})$)
\begin{equation*}
\d_{23}^{\na*}(\mu_{\wt X^{2}_{e}}\cK\bt\mu_{\wt X^{2}_{e}}\cK')\cong \mu_{\wt\D}(\cG).
\end{equation*}
Restricting both sides to $\Xi^{-}\times 0_{\pi^{-1}(x)}\times \Xi$ we get
\begin{eqnarray*}
\d_{23,\Xi}^{\na*}(\VV_{(-\xi,\xi)}(\cK)\bt\VV_{(-\xi,\xi)}(\cK'))
&=&\d_{23,\Xi}^{\na*}(\mu_{\wt X^{2}_{e}}\cK|_{\Xi^{-}\times \Xi}\bt\mu_{\wt X^{2}_{e}}\cK'|_{_{\Xi^{-}\times \Xi}})\\
&\isom &\mu_{\wt\D}(\cG)|_{\Xi^{-}\times 0_{\pi^{-1}(x)}\times\Xi}.
\end{eqnarray*}
where $\d^{\na}_{23,\Xi}: \Xi^{-}\times 0_{\pi^{-1}(x)}\times \Xi\incl \Xi^{-}\times \Xi\times \Xi^{-}\times \Xi$ is the restriction of $\d^{\na}_{23}$.  Combined with \eqref{push G} we get a canonical isomorphism
\begin{equation}\label{Vxixi}
\VV_{(-\xi,\xi)}(\cK\star\cK')\cong \pi_{13,\Xi*}\d_{23,\Xi}^{\na*}(\VV_{(-\xi,\xi)}(\cK)\bt\VV_{(-\xi,\xi)}(\cK')).
\end{equation}

Identifying $\Xi^{-}$, $\Xi$ and $\pi^{-1}(x)$ with $\bT^{c}$ using the base points $\pm\wt\xi$ and $\wt x$,  \eqref{Vxixi} becomes
\begin{equation}\label{Vwtxi}
\VV_{(-\wt\xi,\wt\xi)}(\cK\star\cK')\cong \dot\pr_{13*}\dot\d_{23}^{*}(\VV_{(-\wt\xi,\wt\xi)}(\cK)\bt\VV_{(-\wt\xi,\wt\xi)}(\cK')).
\end{equation}
where $\dot\d_{23}: (\bT^{c})^{3}\incl (\bT^{c})^{4}$ is the diagonal embedding of the middle factors and $\dot\pr_{13}: (\bT^{c})^{3}\to (\bT^{c})^{2}$ the projection. If we write $M=\VV_{(-\wt\xi,\wt\xi)}(\cK)$ and $M'=\VV_{(-\wt\xi,\wt\xi)}(\cK')$ viewed as objects in $D^{b}(\rmod\cR\ot\cR)$, then the right side of \eqref{Vwtxi} becomes $(M\ot M')\ot^{\LL}_{\cR_{23}}\bfk$, where the notation $\ot^{\LL}_{\cR_{23}}\bfk$ means $\ot^{\LL}_{\cR}\bfk$ taken using the $\cR$-action on $M\ot M'$ given by $(m,m')\cdot r=(m,m')\cdot (1\ot \d(r)\ot1)$, where $\d: \cR\to \cR\ot \cR$ is the comultiplication induced by the diagonal map $\pi_{1}(\bT^{c})\to \pi_{1}(\bT^{c})\times \pi_{1}(\bT^{c})$. Finally note canonical isomorphism of right $\cR\ot \cR$-modules (recall $\t$ from \eqref{rmod to bimod})
\begin{equation*}
(M\ot M')\ot^{\LL}_{\cR_{23}}\bfk\cong M\ot_{\cR}^{\LL} \t M'.
\end{equation*}
which induces a canonical isomorphism of right $\cR$-bimodules
\begin{equation*}
\t ((M\ot M')\ot^{\LL}_{\cR_{23}}\bfk)\cong \t M\ot_{\cR}^{\LL} \t M'.
\end{equation*}
Plugging into \eqref{Vwtxi} we get the desired isomorphism $c_{\cK,\cK'}$. We omit the verification that $c_{\cK,\cK'}$ satisfies the axioms of a monoidal structure on $\CH_{G}$.
\epr

\ssec{}{Module structure on $\VV_\BR$}
We now establish the relation between $\VV, \VV_\BR$ and the $\star$-action.


Choose $\xi\in i\cN^{*}(\RR)\cap \cO_{\reg}$, and let $x\in O^{\RR}_{\l_0}$ be the unique point in the Springer fiber $\cB_{\xi}$. Choose a lifting $\wt x\in \pi^{-1}(x)$ and let $\wt\xi=(\wt x,\xi)\in T^{*}\wt X$.  We use $\wt \xi$ to define the functor $\VV_{\RR}:=\VV_{\RR,\wt \xi}$ as in \eqref{V real}. On the other hand, we use $(-\wt\xi,\wt\xi)\in T^{*}(\wt X^{2})$ to defined the functor $\VV:={}_{-\wt\xi}\VV_{\wt\xi}$ as in \eqref{V Hk}.  


%
%


\prop{monoidal} 
Under the above notations,  the functor $\VV_\BR$ intertwines the right $\CH_G$-action on $\CM_{G_\BR}$ by $\star$ and the right action of $D^{b}(\cR\bimod\cR)$ on $D^{b}(\rmod\cR)$ given by $(M, N)\mapsto M\ot^{\LL}_{\cR}N$ (for $M\in D^{b}(\rmod\cR)$ and $N\in D^{b}(\cR\bimod\cR)$). I.e., there is a functorial isomorphism 
\begin{equation}\label{aFK}
\a_{\cF,\cK}:\VV_{\RR}(\cF)\ot^{\LL}_{\cR}\VV(\cK)\to \VV_{\RR}(\cF\star\cK).
\end{equation}
for $\cF\in \CM_{G_\BR}$ and $\cK\in \CH_{G}$, such that the following diagram is commutative for any $\cF\in \CM_{G_\BR}$ and $\cK_{1},\cK_{2}\in \CH_{G}$
\begin{equation}\label{comm FK1K2}
\xymatrix{\VV_{\RR}(\cF)\ot^{\LL}_{\cR}\VV(\cK_{1})\ot^{\LL}_{\cR}\VV(\cK_{2})\ar[rr]^-{a_{\cF,\cK_{1}}\ot\id} && \VV_{\RR}(\cF\star\cK_{1})\ot^{\LL}_{\cR}\VV(\cK_{2})\ar[d]^{a_{\cF\star\cK_{1},\cK_{2}}}\\
\VV_{\RR}(\cF)\ot^{\LL}_{\cR}\VV(\cK_{1}\star\cK_{2})\ar[u]^{\id\ot c_{\cK_{1},\cK_{2}}}\ar[rr]^-{a_{\cF,\cK_{1}\star\cK_{2}}} && \VV_{\RR}(\cF\star\cK_{1}\star\cK_{2})
}
\end{equation}
\eprop
\prf
We first construct a natural isomorphism \eqref{aFK} for $\cF\in \CM_{G_\BR}$ and $\cK\in \CH_{G}$.


Let $\cG=\pr_{1}^{*}\cF\ot\cK\in \wh D^{b}(\wt X\times \wt X)_{\bT^{c}\times \bT^{c}\mon}$. We know that $SS(\cF)\subset \L_{\RR}$ (hence $SS(\pr_{1}^{*}\cF)\subset \L_{\RR}\times 0_{\wt X}\subset T^{*}(\wt X^{2})$), and $SS(\cK)\subset \L$. One checks that $(\L_{\RR}\times 0_{\wt X})\cap \L$ is contained in the zero section of $T^{*}(\wt X^{2})$. By \cite[Proposition 5.4.14(i)]{KS}, $SS(\cG)=SS(\pr_{1}^{*}\cF\ot \cK)$ is contained in the pointwise addition $\L':=(\L_{\RR}\times 0_{\wt X})+\L\subset T^{*}(\wt X^{2})$.  The cone $\L'$ consists of $((\wt x_{1},v-w), (\wt x_{2}, w))\in T^{*}(\wt X)^{2}$ (where $\wt x_{1}, \wt x_{2}\in \wt X$ with image $x_{1},x_{2}\in X$,  $v,w\in \frg^{*}$) such that $w\in {}^{x_{1}}\frb^{\bot}\cap {}^{x_{2}}\frb^{\bot}$ and $v\in i\cN^{*}(\RR)\cap {}^{x_{1}}\frb^{\bot}$. 

By definition, $\VV_{\RR}(\cF\star\cK)$ is the restriction of the microlocalization $\mu_{\l_0}(\pr_{2*}\cG)$ to $\Xi=\mu^{-1}(\xi)\subset T^{*}_{\wt O^{\RR}_{\l_0}}\wt X$. By \refp{KS}(1),  we consider 
\begin{equation*}
\xymatrix{T^{*}_{\wt X\times \wt O^{\RR}_{\l_0}}(\wt X^{2}) & 0_{\wt X}\times T^{*}_{\wt O^{\RR}_{\l_0}}(\wt X)\ar[l]_-{\sim}\ar[r]^-{\pi_{2}} & T^{*}_{\wt O^{\RR}_{\l_0}}(\wt X)
}
\end{equation*}
Let $\pi_{2,\Xi}: 0_{\wt X}\times \Xi\to \Xi$ be the second projection, then we have
\begin{equation}\label{VRFK1}
\VV_{\RR,\xi}(\cF\star\cK)\cong (\pi_{2*}(\mu_{\wt X\times \wt O^{\RR}_{\l_0}}\cG))|_{\Xi}\cong \pi_{2,\Xi*}((\mu_{\wt X\times \wt O^{\RR}_{\l_0}}\cG)|_{0_{\wt X}\times \Xi}).
\end{equation}
Now $SS(\cG)\subset \L'$. We claim that
\begin{equation}\label{Xi'}
\Xi':=\L'\cap (0_{\wt X}\times \Xi)=\{(\wt x_{1}, 0), (\wt x_{2}, \xi))|\wt x_{1}, \wt x_{2}\in \pi^{-1}(x)\}.
\end{equation}
Indeed, a point in $\L'$ takes the form $((\wt x_{1}, v-w), (\wt x_{2}, w))$. Intersecting with $0_{\wt X}\times \Xi$ means imposing $w=\xi$ and $v=w$, which also forces $x_{1}=x_{2}=x$ since $w=\xi\in {}^{x_{1}}\frb^{\bot}\cap {}^{x_{2}}\frb^{\bot}$. The projection to $\wt X^{2}$ gives an isomorphism $\Xi'\cong \pi^{-1}(x)^{2}$.

Now $(\mu_{\wt X\times \wt O^{\RR}_{\l_0}}\cG)|_{0_{\wt X}\times \Xi}$ is supported in $\Xi'$. Let $\wt \D^{\RR}_{\l_{0}}\subset \wt X^{2}$ be the preimage of the diagonal $\D(O^{\RR}_{\l_0})\subset X^{2}$ under the projection $\pi^{2}: \wt X^{2}\to X^{2}$. By \eqref{Xi'} we have $\Xi'\subset T^{*}_{\wt \D^{\RR}_{\l_{0}}}(\wt X^{2})$. Therefore
\begin{equation*}
(\mu_{\wt X\times \wt O^{\RR}_{\l_0}}\cG)|_{\Xi'}\cong (\mu_{\wt\D^{\RR}_{\l_{0}}}\cG)|_{\Xi'}
\end{equation*}
and \eqref{VRFK1} gives an isomorphism
\begin{equation}\label{VRFK2}
\VV_{\RR,\xi}(\cF\star\cK)\cong \pr_{2,\Xi*}((\mu_{\wt\D^{\RR}_{\l_{0}}}\cG)|_{\Xi'}
)
\end{equation}
where $\pr_{2,\Xi}: \Xi'\to \Xi$ is the projection.

To compute $(\mu_{\wt\D^{\RR}_{\l_{0}}}\cG)|_{\Xi'}$, we consider the diagonal embedding  $\d_{12}: \wt X^{2}\incl \wt X^{3}$ given by $(\wt x_{1}, \wt x_{2})\mapsto (\wt x_{1}, \wt x_{1}, \wt x_{2})$. Let $\ol\d_{12}: \wt X^{2}\to \wt X^{3}\to \wt X\times (G\bs \wt X^{2})$, where the second map is the quotient map on the second and third factors. Then $\cG=\ol\d_{12}^{*}(\cF\bt\cK)$, where we view $\cK\in D^{b}_{G}(\wt X^{2})$. We are ready to apply \refp{KS}(2) (together with \refr{KS stack}) to the smooth map $\ol\d_{12}$, the submanifold $\wt O^{\RR}_{\l_0}\times (G\bs \wt X^{2}_{e})\subset \wt X\times (G\bs \wt X^{2})$, and its preimage under $\ol\d_{12}$ which is $\wt\D^{\RR}_{\l_{0}}\subset \wt X^{2}$.
We have a canonical map induced by $\ol\d_{12}$ which is an isomorphism on conormal fibers:
\begin{equation*}
T^{*}_{\wt\D^{\RR}_{\l_{0}}}(\wt X^{2})\xr{\ol\d_{12}^{\na}}T^{*}_{\wt O^{\RR}_{\l_0}\times (G\bs \wt X^{2}_{e})}(\wt X\times (G\bs \wt X^{2})).
\end{equation*}
By \refp{KS}(2), we have a canonical isomorphism
\begin{equation}\label{D0}
\ol\d^{\na*}_{12}(\mu_{\l_0}\cF\bt\mu_{G\bs \wt X^{2}_{e}}\cK)\cong\d_{12}^{\na*}\mu_{\wt O^{\RR}_{\l_0}\times (G\bs\wt X^{2}_{e})}(\cF\bt\cK)\isom \mu_{\wt \D^{\RR}_{\l_{0}}}\cG.
\end{equation}
Since $\ol\d_{12}^{\na}$ factors as the composition
\begin{equation*}
T^{*}_{\wt\D^{\RR}_{\l_{0}}}(\wt X^{2})\xr{\d_{12}^{\na}}T^{*}_{\wt O^{\RR}_{\l_0}\times  \wt X^{2}_{e}}(\wt X^{3})\to T^{*}_{\wt O^{\RR}_{\l_0}\times (G\bs \wt X^{2}_{e})}(\wt X\times (G\bs \wt X^{2}))
\end{equation*}
where the second map is the quotient map by $G$, we can rewrite \eqref{D0} as an isomorphism
\begin{equation}\label{D1}
\d^{\na*}_{12}(\mu_{\l_0}\cF\bt\mu_{\wt X^{2}_{e}}\cK)\cong\mu_{\wt \D^{\RR}_{\l_{0}}}\cG.
\end{equation}
Here, when writing $\mu_{\wt X^{2}_{e}}\cK$, we are viewing $\cK$ as an object in $D^{b}(\wt X^{2})$. Recall $\Xi^{-}=\mu^{-1}(-\xi)$. Then $\d_{12}^{\na}(\Xi')=(\Xi\times_{\wt X}\Xi^{-})\times \Xi$, and let $\d^{\na}_{12,\Xi}: \Xi'\incl \Xi\times \Xi^{-}\times\Xi$ be the induced embedding.  Restricting \eqref{D1} to $\Xi'$ we get
\begin{equation*}
\d^{\na*}_{12,\Xi}((\mu_{\l_0}\cF)|_{\Xi}\bt(\mu_{\wt X^{2}_{e}}\cK)|_{\Xi^{-}\times \Xi})\cong (\mu_{\wt \D^{\RR}_{\l_{0}}}\cG)|_{\Xi'}.
\end{equation*}
Combined with \eqref{VRFK2} we get a canonical isomorphism in $\wh D^{b}(\Xi)_{\bT^{c}\mon}$
\begin{eqnarray}\label{VRFK3}
\VV_{\RR,\xi}(\cF\star\cK)\cong\pr_{2, \Xi*}\d^{\na*}_{12,\Xi}((\mu_{\l_0}\cF)|_{\Xi}\bt(\mu_{\wt X^{2}_{e}}\cK)|_{\Xi^{-}\times \Xi})\\
\notag
=\pr_{2, \Xi*}\d^{\na*}_{12,\Xi}(\VV_{\RR,\xi}(\cF)\bt\VV_{(-\xi,\xi)}(\cK)).\end{eqnarray}
If we identify $\Xi$ and $\Xi^{-}$ with $\bT^{c}$ using the liftings $\wt\xi=(\wt x,\xi)\in \Xi$ and $-\wt\xi=(\wt x,-\xi)\in \Xi^{-}$, then $\d_{12}^{\na}$ becomes the diagonal embedding $\dot\d_{12}: \bT^{c}\times \bT^{c}\incl \bT^{c}\times \bT^{c}\times \bT^{c}$ into the first two factors, and $\pr_{2,\Xi}$ becomes the second projection $\dot\pr_{2}: \bT^{c}\times \bT^{c}\to \bT^{c}$. Then \eqref{VRFK3} becomes a canonical isomorphism in $\wh D^{b}(\bT^{c})_{\bT^{c}\mon}\cong D^{b}(\rmod\cR)$
\begin{equation}\label{VRwtxi}
\VV_{\RR,\wt\xi}(\cF\star\cK)\cong \dot\pr_{2*}\dot\d_{12}^{*}(\VV_{\RR,\wt \xi}(\cF)\bt\VV_{(-\wt\xi,\wt\xi)}(\cK))\cong \dot\pr_{2*}(\dot\pr_{1}^{*}\VV_{\RR,\wt\xi}(\cF)\ot \VV_{(-\wt\xi,\wt\xi)}(\cK)).
\end{equation}

Let $M=\VV_{\RR,\wt\xi}(\cF)\in D^{b}(\rmod\cR)$, $N\in \VV_{(-\wt\xi,\wt\xi)}(\cK)\in D^{b}(\rmod\cR\ot\cR)$, hence $\t N={}_{-\wt\xi}\VV_{\wt\xi}(\cK)\in D^{b}(\cR\bimod\cR)$ (see \eqref{rmod to bimod} for $\t$). Then $\dot\pr_{2*}(\dot\pr_{1}^{*}M\ot N')\cong (M\ot N')\ot^{\LL}_{\cR_{12}}\bfk$ where $\ot^{\LL}_{\cR_{12}}\bfk$ means the functor $\ot^{\LL}_{\cR}\bfk$ taken using the following right $\cR$-module structure on $M\ot N'$: $(m\ot n')\cdot r=(m\ot n)\cdot \d(r)$, where $\d: \cR\to \cR\ot \cR$ is the comultiplication induced by the diagonal map $\pi_{1}(\bT^{c})\to \pi_{1}(\bT^{c})\times \pi_{1}(\bT^{c})$. Here we are using the fact that the pushforward $\bT^{c}\to \pt$ of monodromic sheaves corresponds to the functor $(-)\ot^{\LL}_{\cR}\bfk: D^{b}(\rmod\CR)\to D^{b}(\rmod\bfk)$. It is easy to see there is a canonical isomorphism of right $\cR$-modules (using the second $\cR$-action on $N$ and the right $\cR$-action on $\t N$)
\begin{equation*}
(M\ot N)\ot^{\LL}_{\cR}\bfk\cong M\ot^{\LL}_{\CR} \t N.
\end{equation*}
Therefore the right side of \eqref{VRwtxi} is canonically isomorphic to $M\ot^{\LL}_{\CR} \t N=\VV_{\RR,\wt\xi}(\cF)\ot^{\LL}_{\CR} {}_{-\wt\xi}\VV_{\wt\xi}(\cK)$ as a right  $\cR$-module. This completes the construction of $\a_{\cF,\cK}$.

The argument for checking the commutativity of \eqref{comm FK1K2} is similar. We omit it.

\epr

We record here the functoriality properties of microlocalization from \cite{KS} that we used in the above proof. Let $f: Y\to X$ be a map of manifolds, $M\subset X$ a closed submanifold and $N=f^{-1}(M)$.  Then there is a natural correspondence between the conormals
\begin{equation*}
\xymatrix{ T^{*}_{N}Y &  N\times_{M}(T^{*}_{M}X) \ar[r]^-{f^{\na}_{N}}\ar[l]_-{df_{N}}& T^{*}_{M}X }
\end{equation*}
We have the microlocalization functors $\mu_{M}: D^{b}(X)\to D^{b}(T^{*}_{M}X)$ and $\mu_{N}:  D^{b}(Y)\to D^{b}(T^{*}_{N}Y)$.

\prop{KS} 
\begin{enumerate}
\item (\cite[Proposition 4.3.4]{KS}) Assume that  $f$ is transversal to $M$ (i.e.  $df_{N}$ above is an isomorphism). Let $\cG\in D^{b}(Y)$ and suppose that $f|_{\supp(\cG)}$ is proper, then there is a canonical isomorphism
\begin{equation*}
\mu_{M}(f_{*}\cG)\isom f^{\na}_{N*}df^{*}_{N}\mu_{N}(\cG).
\end{equation*}
 
\item (\cite[Proposition 4.3.5]{KS}) Assume $f$ is smooth (which again implies that $df_{N}$ above is an isomorphism). Let $\cF\in D^{b}(X)$. Then there is a canonical isomorphism
\begin{equation*}
df_{N*}f^{\na,*}_{N}\mu_{M}(\cF)\isom \mu_{N}(f^{*}\cF).
\end{equation*}

\end{enumerate}

\eprop

\rem{KS stack} In the arguments for \refp{H monoidal} and \refp{monoidal}, we applied \refp{KS}(2) in the stacky situation $X=H\bs \wt X$ for a smooth manifold $\wt X$ with a Lie group $H$ acting, and a smooth map $f:Y\to X$. Let $\wt M\subset \wt X$ be an $H$-stable submanifold and $M=H\bs \wt M$, $N=f^{-1}(M)$. Then $\mu_{M}: D^{b}(X)=D^{b}_{H}(\wt X)\to D^{b}(T^{*}_{M}X)=D^{b}_{H}(T^{*}_{\wt M}\wt X)$ is defined, and \refp{KS}(2) holds as stated. To deduce this stacky version, we simply apply \refp{KS}(2) to the map $\wt f: \wt Y\to \wt X$ obtained from $f$ by base change along $\pi_{X}: \wt X\to X$, the submanifold $\wt M$, $\wt N=\wt f^{-1}(\wt M)$ and the sheaf $\pi^{*}_{X}\cF\in D^{b}(\wt X)$.
\erem

\sec{str nice}{Structure theorem for nice quasi-split groups}
In this section we assume $G_{\RR}$ is quasi-split. In this section we will prove our main result (\reft{main1}, \reft{main2} and \refc{maincor}). It is a generalization of Soergel's Struktursatz to quasi-split real groups satisfying a technical condition called {\em nice} (which in particular includes quasi-split groups of adjoint type).


\ssec{}{The algebra $\CA$} We choose $\xi\in i\CN^{*}(\RR)\cap \cO_{\reg}$ and $\wt\xi\in T^{*}\wt X$ over $\xi$, and define the functor $\VV_{\BR}=\VV_{\BR,\wt\xi}$. 

We define $\CA:=\End(\VV_\BR|_{\Tilt})^{opp}$ for the opposite ring of the endomorphism ring of the functor $\VV_{\RR}|_{\Tilt}$ (shorthand for the restriction of $\VV_{\BR}$ to $\mathrm{Tilt}({\CM_{G_\BR}})$).  For an element $a\in \CA$ and a tilting sheaf $\CT$ we put $a_\CT\in\End(\VV_{\RR}(\CT))$ for the (right) action of $a$ on $\VV_\BR(\CT)$. Then $\VV_{\BR}$ upgrades to an exact functor
\begin{equation}\label{VRsh}
\VV_{\RR}^{\sh}: \Tilt(\CM_{G_\BR})\to \rmod\CA.
\end{equation}


To state the extension of Soergel's structure theorem for real groups, we need to single out a class of quasi-split groups.

\defe{Nice quasi-split groups}
Let $G_{\RR}$ be a quasi-split connected reductive group over $\RR$. We say that $G_{\RR}$ is {\em nice}, if for every real parabolic subgroup $P_{\RR}\subset G_{\RR}$ whose Levi factor $L_{\RR}$ satisfies $L^{\ad}_{\RR}\cong \PGL_{2,\RR}$, the map on real points $L(\RR)\to L^{ad}(\RR)$ is surjective. 
\edefe

The following lemma ensures that in particular adjoint quasi-split groups are nice.

\lem{leviRpts} Assume $G_{\RR}$ is quasi-split, and $H^{1}(\CC/\RR, Z(G))=1$. Then for any real parabolic $P\subset G$ and its Levi factor $L$, the map $L(\RR)\to L^{ad}(\RR)$ is surjective. In particular, $G_{\RR}$ is nice.
\elem
\prf Let $B_{0}$ is a Borel subgroup defined over $\RR$, $T_{0}\subset B_{0}$ is a maximally split real Cartan, and $L$ is a standard Levi subgroup with respect to $B_{0}$, i.e., it is the Levi subgroup of a real parabolic subgroup containing $B_{0}$. To show $L(\RR)\to L^{ad}(\RR)$ is surjective, it suffices to prove that the Galois cohomology $H^{1}(\CC/\RR, Z(L))$ vanishes. 

Let $S$ be the set of simple roots in $\Phi(G,T_{0})$ with respect to the real Borel $B_{0}$. Then $\sig$ acts on $S$ by involution. By our assumption, $Z(L)\subset T_{0}$ is the vanishing locus of a $\sig$-stable subset $J'\subset S$ of simple roots, hence $Z(L)$ fits into an exact sequence
\begin{equation*}
1\to Z(G)\to Z(L)\to (\Gm)^{S-J'}\to 1.
\end{equation*}
Since $\sig$ permutes $S$, the real structure on $(\Gm)^{S-J'}$ is a product of  $R_{\CC/\RR}\Gm$ (one for each nontrivial orbit of $\sig$ on $S-J'$), and $\Gm$ (one for each fixed point of $\sigma$ on $S-J'$). Since $H^{1}(\CC/\RR, R_{\CC/\RR}\Gm)=1$ and $H^{1}(\CC/\RR, \Gm)=1$, and $H^{1}(\CC/\RR, Z(G))$ by assumption, we conclude that $H^{1}(\CC/\RR, Z(L))=1$.
\epr

\ex{} Suppose $G_{\RR}$ is of the form $\SU(p,q)/\G$ where $|p-q|\ge1$, and $\G\subset \mu_{p+q}$ is a subgroup. If $p+q$ is odd, then all such $G_{\RR}$ are nice. If $p+q$ is even, then $G_{\RR}$ is nice if and only if $\G$ contains $-1$.
\eex

\th{main1} Suppose $G_{\RR}$ is quasi-split and nice.  Then the functor $\VV_{\RR}^{\sh}$ is  fully faithful.
\eth

The proof will be completed in \refss{pf main1}.

\ssec{}{The algebra $\cA_{0}$}
Recall $\CR=\widehat{\bfk[\pi_1(\bT^c)]}$ is the group algebra of the fundamental group of $\bT^{c}$ completed at the augmentation ideal. 
The monodromy action along the fibers of $\pi$ induces a ring map
\begin{equation}\label{RtoA}
\ph_{\cR}: \CR\to Z(\CA)
\end{equation}
to the center of $\CA$. 

Let $\CA_0\subset\CA$ be the subalgebra of endomorphisms of $\VV_\BR|_{\Tilt}$ commuting with the Hecke action. More precisely, $a\in \cA_{0}$ if for any $\CT\in \Tilt(\CM_{G_\BR})$ and $\CT'\in \Tilt(\CH_{G})$, the endomorphism $a_{\CT\star\CT'}$ of $\VV_{\BR}(\CT\star\CT')\cong \VV_{\BR}(\CT)\ot_{\CR}\VV(\CT')$ is equal to $a_{\CT}\ot\id_{\VV(\CT')}$. 

\lem{RWA0}
\begin{enumerate}
\item The image of $\CR^{\bW}$ under the map $\ph_{\CR}$ lies in $\CA_{0}$.
\item\label{comm Tw0} Recall the big tilting object $\CT_{w_{0}}\in \THG$. Let $a\in \CA$. Then $a\in \CA_{0}$ if and only if for any $\CT\in \Tilt(\CM_{G_\BR})$, the endomorphism $a_{\CT\star\CT_{w_{0}}}$ of $\VV_{\BR}(\CT\star\CT_{w_{0}})\cong \VV_{\BR}(\CT)\ot_{\CR}\VV(\CT_{w_{0}})$ is equal to $a_{\CT}\ot\id_{\VV(\CT_{w_{0}})}$. 
\end{enumerate}
\elem
\prf
(1) Let  $b\in \CR$ and $a=\ph_{\CR}(b)\in Z(\CA)$. Let $\CT\in \Tilt(\CM_{G_\BR})$ and $\CT'\in \Tilt(\CH_{G})$. Then $a_{\CT\star\CT'}$ is induced from the action of $b$ on $\CT\star\CT'$ by right $\bT^{c}$-monodromy on $\CT'$. We know that the left and right monodromy actions on $\CT'\in \Tilt(\CH_{G})$ factor through $\CR\ot_{\CR^{\bW}}\CR$ (as it is so for $\VV(\CT')$ (\refp{VHk}\eqref{Endo P0}) and $\VV|_{\Tilt}$ is fully faithful (\refp{VHk}\eqref{VHk ff})). Therefore,  if $b\in \CR^{\bW}$, then the right monodromy action of $b$ on $\CT'$ is the same as left monodromy, and the induced action on $\CT\star\CT'$ is the same as the action of $b$ on $\CT$ by (right) $\bT^{c}$-monodromy. Therefore, both $a_{\CT\star\CT'}$ and $a_{\CT}\ot\id_{\VV(\CT')}$ on $\VV_{\RR}(\CT\star\CT')\cong \VV_{\RR}(\CT)\ot_{\CR}\VV(\CT')$ are given by the action of $b\in \CR^{\bW}$ on the first factor $\VV_{\RR}(\CT)$. This shows $a\in \CA_{0}$.

(2) Suppose $a\in \CA$ satisfies $a_{\CT\star\CT_{w_{0}}}=a_{\CT}\ot\id_{\VV(\CT_{w_{0}})}$ for all $\CT\in\Tilt(\CM_{G_\BR})$. Let $\CT'\in\Tilt(\CH_{G})$, we want to show that $a_{\CT\star\CT'}=a_{\CT}\ot\id_{\VV(\CT')}$ as endomorphisms of $\VV_{\BR}(\CT\star\CT')\cong \VV_{\BR}(\CT)\ot_{\CR}\VV(\CT')$. We have a canonical map
\begin{equation}\label{ev Hom from Tw0}
\e: \Hom(\CT_{w_{0}},\CT')\ot\VV(\CT_{w_{0}})\to \VV(\CT')
\end{equation}
sending $f\ot v$ to the image of $v\in \VV(\CT_{w_{0}})$ under the map $\VV(f): \VV(\CT_{w_{0}})\to \VV(\CT')$. By \refp{VHk}\eqref{Hom Tw0}, $\Hom(\CT_{w_{0}},\CT')\cong \VV(\CT')$ and $\VV(\CT_{w_{0}})\cong\End(\CT_{w_{0}})$. Take $\id_{\CT_{w_{0}}}\in \End(\CT_{w_{0}})\cong \VV(\CT_{w_{0}})$, we have $\e(f\ot \id_{\CT_{w_{0}}})=f$ for any $f\in \Hom(\CT_{w_{0}},\CT')\cong \VV(\CT')$. This shows that $\e$ is surjective. The similarly defined map 
$$\e_{\CT}: \Hom(\CT_{w_{0}},\CT')\ot\VV_{\RR}(\CT\star\CT_{w_{0}})\to \VV_{\RR}(\CT\star\CT')$$ 
can be identified with 
$$\e\ot\id_{\VV_{\RR}(\CT)}: \Hom(\CT_{w_{0}},\CT')\ot\VV_{\RR}(\CT)\ot_{\CR}\VV(\CT_{w_{0}})\to \VV_{\RR}(\CT)\ot_{\CR}\VV(\CT'),$$
hence is also surjective. By the definition of $\CA$, we have a commutative diagram
\begin{equation}\label{HV1}
\xymatrix{\Hom(\CT_{w_{0}},\CT')\ot\VV_{\RR}(\CT\star\CT_{w_{0}})\ar[r]^-{\e_{\CT}}\ar[d]^{\id\ot a_{\CT\star\CT_{w_{0}}}} & \VV_{\RR}(\CT\star\CT')\ar[d]^{a_{\CT\star\CT'}}\\
\Hom(\CT_{w_{0}},\CT')\ot\VV_{\RR}(\CT\star\CT_{w_{0}})\ar[r]^-{\e_{\CT}} & \VV_{\RR}(\CT\star\CT')}
\end{equation}
Rewriting $\VV_{\RR}(\CT\star\CT_{w_{0}})$ as $\VV_{\RR}(\CT)\ot_{\CR}\VV(\CT_{w_{0}})$ and $\VV_{\RR}(\CT\star\CT')$ as $\VV_{\RR}(\CT)\ot_{\CR}\VV(\CT')$, we similarly we have a commutative diagram
\begin{equation}\label{HV2}
\xymatrix{\Hom(\CT_{w_{0}},\CT')\ot\VV_{\RR}(\CT)\ot_{\CR}\VV(\CT_{w_{0}})\ar[r]^-{\e\ot \id_{\VV_{\RR}(\CT)}}\ar[d]^{\id\ot a_{\CT}\ot\id} & \VV_{\RR}(\CT)\ot_{\CR}\VV(\CT')\ar[d]^{a_{\CT}\ot\id}\\
\Hom(\CT_{w_{0}},\CT')\ot\VV_{\RR}(\CT)\ot_{\CR}\VV(\CT_{w_{0}})\ar[r]^-{\e\ot \id_{\VV_{\RR}(\CT)}} & \VV_{\RR}(\CT)\ot_{\CR}\VV(\CT')}
\end{equation}
Now compare \eqref{HV1} and \eqref{HV2}, the left vertical maps are equal in the two diagrams by assumption. Since the horizontal maps are equal and surjective, the right vertical maps are also equal in the two diagrams, i.e., $a_{\CT\star\CT'}=a_{\CT}\ot\id_{\VV(\CT')}$.
\epr

By \refl{RWA0}, we have a ring homomorphism
\begin{equation*}
\ph: \CA_{0}\ot_{\CR^{\bW}}\CR\to \CA
\end{equation*}
where the image of $\CR$ is central. 




\prop{A descent} The ring map $\ph$ is an isomorphism.
\eprop
\prf

We construct a map in the other direction as follows. Let $w_{0}\in W$ be the longest element, and $\CT_{w_0}\in\Tilt(\CH_{G})$ be the indecomposable free-monodromic tilting object with full support. It is known (Theorem 9.1 in \cite{BR}, Proposition 4.7.3 1) in \cite{BY}) that $\VV(\CT_{w_0})=\CR\ten_{\CR^{\bW}}\CR$. Consider the functor $\UU_{\BR}: \Tilt(\CM_{G_{\RR}})\to \CR\bimod\CR$ given by $\UU_{\BR}(\cT):=\VV_\BR(\cT\star\CT_{w_0})$, with the two $\CR$-actions given by the left and right monodromy actions on $\CT_{w_0}$. Using \refp{VHk}\eqref{Endo P0}
\begin{equation}\label{UR}
\UU_{\BR}(\CT)\cong \VV_\BR(\cT)\ot_{\CR}\VV(\CT_{w_{0}})\cong \VV_{\BR}(\CT)\ot_{\CR^{\bW}}\CR
\end{equation}
with the first copy of  $\CR$ acting on $\VV_{\BR}(\CT)$ and the second copy acting on the second tensor factor $\CR$ by multiplication. 
From \eqref{UR} we see that $\End(\UU_{\BR})^{opp}\cong\CA\ten_{\CR^{\bW}}\CR$. 

We have an action of $\CA$ on $\UU_{\BR}$: for $a\in \CA$ and $\CT\in \TGR$,  $a$ acts on $\UU_{\BR}(\CT)$ by $a_{\CT\star\CT_{w_{0}}}$. This gives a ring map
\begin{equation*}
\psi': \CA\to \End(\UU_{\BR})^{opp}\cong \CA\ten_{\CR^{\bW}}\CR.
\end{equation*}

\begin{claim} The image of $\psi'$ is contained in $\CA_{0}\ten_{\CR^{\bW}}\CR$.
\end{claim}
\begin{proof}[Proof of Claim] Using the characterization of $\CA_{0}$ given in \refl{RWA0}\eqref{comm Tw0}, $\CA_{0}\ten_{\CR^{\bW}}\CR$ consists exactly of those $b\in \End(\UU_{\BR})$ such that $b_{\CT\star \CT_{w_{0}}}=b_{\CT}\ot \id_{\VV(\CT_{w_{0}})}$. Therefore, to show $\psi'(a)\in \CA_{0}\ot_{\CR^{\bW}}\CR$ for $a\in \CA$, we need to check for any $\CT\in\TGR$,
\begin{equation}\label{aa}
a_{\CT\star \CT'_{w_{0}}\star\CT''_{w_{0}}}=a_{\CT\star\CT''_{w_{0}}}\ot\id_{\VV(\CT'_{w_{0}})}\in\End(\VV_{\RR}(\CT\star\CT'_{w_{0}}\star\CT''_{w_{0}})).
\end{equation}
Here, to distinguish two copies of $\CT_{w_{0}}$ we denote them by $\CT'_{w_{0}}$ and $\CT''_{w_{0}}$. Consider the map $\b: \cT''_{w_{0}}\to \CT_{w_{0}}'\star\CT''_{w_{0}}$ such that $\VV(\b)$ is given by
\begin{eqnarray*}
\cR\ot_{\cR^{\bW}} \cR &\to & \cR\ot_{\cR^{\bW}} \cR \ot_{\cR^{\bW}} \cR\\
a\ot b &\mapsto & a\ot 1\ot b.
\end{eqnarray*}
Then $\b$ induces $\id_{\CT}\star\b:  \CT\star\cT''_{w_{0}}\to \CT\star\CT_{w_{0}}'\star\CT''_{w_{0}}$, hence a commutative diagram 
\begin{equation*}
\xymatrix{\VV_{\RR}(\CT\star\cT''_{w_{0}})\ar[d]^{a_{\CT\star\cT''_{w_{0}}}}\ar[rr]^-{\VV(\id_{\CT}\star\b)} && \VV_{\RR}(\CT\star\CT_{w_{0}}'\star\CT''_{w_{0}})\ar[d]^{a_{\CT\star\CT'_{w_{0}}\star\cT''_{w_{0}}}}\\
\VV_{\RR}(\CT\star\cT''_{w_{0}})\ar[rr]^-{\VV(\id_{\CT}\star\b)} && \VV_{\RR}(\CT\star\CT_{w_{0}}'\star\CT''_{w_{0}})}
\end{equation*}
Using the description of $\VV(\b)$, we see that the above diagram can be written as
\begin{equation*}
\xymatrix{\VV_{\RR}(\CT)\ot_{\CR^{\bW}}\CR\ar[d]^{a_{\CT\star\cT''_{w_{0}}}}\ar[rr]^-{\id\ot 1\ot \id} && \VV_{\RR}(\CT)\ot_{\CR^{\bW}}\CR\ot_{\CR^{\bW}}\CR\ar[d]^{a_{\CT\star\CT'_{w_{0}}\star\cT''_{w_{0}}}}\\
\VV_{\RR}(\CT)\ot_{\CR^{\bW}}\CR\ar[rr]^-{\id\ot 1\ot \id} && \VV_{\RR}(\CT)\ot_{\CR^{\bW}}\CR\ot_{\CR^{\bW}}\CR
}
\end{equation*}
This shows that \eqref{aa} holds on the subspace $\VV_{\RR}(\CT)\ot1\ot\CR\subset\VV_{\RR}(\CT)\ot_{\CR^{\bW}}\CR\ot_{\CR^{\bW}}\CR=\VV_{\RR}(\CT\star\CT_{w_{0}}'\star\CT''_{w_{0}})$. Since both $a_{\CT\star \CT'_{w_{0}}\star\CT''_{w_{0}}}$ and $a_{\CT\star\CT''_{w_{0}}}\ot\id_{\VV(\CT'_{w_{0}})}$ are linear with respect to the three $\CR$-actions on $\VV_{\RR}(\CT\star\CT_{w_{0}}'\star\CT''_{w_{0}})$, we conclude that \eqref{aa} holds.
\end{proof}

By the Claim, we have a map 
\begin{equation*}
\psi: \cA\to \CA_{0}\ot_{\CR^{\bW}} \CR.
\end{equation*}
Note that this map is $\CR$-linear. We check that $\ph$ and $\psi$ are inverse to each other. If $a\in \CA_{0}$, then $\psi\ph(a)$ acts on $\UU_{\RR}(\cT)=\VV_{\RR}(\CT\star\CT_{w_{0}})$ by $a_{\CT\star\CT_{w_{0}}}$. Since $a\in \CA_{0}$, $a_{\CT\star\CT_{w_{0}}}=a_{\CT}\ot\id_{\VV(\CT_{w_{0}})}$, which implies $\psi\ph(a)=a\ot 1\in \CA_{0}\ot_{\CR^{\bW}} \CR$. By $\CR$-linearity, this implies $\psi\ph=\id$.

On the other hand, to check $\ph\psi=\id_{\CA}$, it suffices to show that the composition
\begin{equation*}
\CA\xr{\psi'}\CA\ot_{\CR^{\bW}}\CR\xr{m}\CA
\end{equation*}
is the identity, where $m$ is the multiplication map. We have a canonical map $\e: \CT_{w_0}\to\tilde{\delta}$ that  induces the multiplication map $\VV(\CT_{w_0})=\CR\ten_{\CR^{\bW}}\CR\to\CR=\VV(\tilde{\delta})$.  It induces a natural transformation of functors $\g: \UU_{\RR}\to \VV_{\RR}$ ($\CR$-linear with respect to the second $\CR$-action on $\UU_{\RR}$).  Under the isomorphism $\UU_{\RR}\cong \VV_{\RR}\ot_{\CR^{\bW}}\CR$, $\g$ corresponds to the multiplication map $ \VV_{\RR}\ot_{\CR^{\bW}}\CR\to \VV_{\RR}$. Therefore for any $b\in \End(\UU_{\RR})^{opp}\cong \cA\ot_{\CR^{\bW}}\cR$, we have a commutative diagram for any $\CT\in \TGR$
\begin{equation}\label{UV}
\xymatrix{\UU_{\RR}(\CT) \ar[d]^{b_{\CT}}\ar[r]^-{\g} & \VV_{\RR}(\CT)\ar[d]^{m(b)_{\CT}} \\
\UU_{\RR}(\CT)\ar[r]^-{\g} & \VV_{\RR}(\CT)}
\end{equation}
On the other hand, by the definition of $\CA$ we have a commutative diagram 
\begin{equation}\label{UV2}
\xymatrix{\VV_{\RR}(\CT\star\CT_{w_{0}}) \ar[d]^{a_{\CT\star\CT_{w_{0}}}}\ar[rr]^-{\VV(\id_{\CT}\star\e)} && \VV_{\RR}(\CT)\ar[d]^{a_{\CT}} \\
\UU_{\RR}(\CT)\ar[rr]^-{\VV(\id_{\CT}\star\e)} && \VV_{\RR}(\CT)}
\end{equation}
Now taking $b=\psi'(a)$ in \eqref{UV}, it becomes the same diagram as \eqref{UV2}, from which we conclude that $m(\psi'(a))=m(b)=a$. This implies $m\psi'=\id_{\CA}$ hence $\ph\psi=\id_{\CA}$ and finishes the proof.
\epr

Now we define an action of Soergel bimodules $\SBim$ on $\rmod\CA$ as follows: $M\in \rmod\CA$ and $N\in \SBim$, the action of $N$ on $M$ is the tensor product $M\ot_{\CR}N$. Now $M\ot_{\CR}N$ is naturally a $\CA_{0}\ot_{\CR^{\bW}}\CR$-module using the (right) $\CA_{0}$-action on $M$ and the right $\CR$-action on $N$. Since $\CA_{0}\ot_{\CR^{\bW}}\CR\isom \CA$ by \refp{A descent}, we may view $M\ot_{\CR}N$ as a right $\CA$-module.

The following is an immediate consequence of \refp{monoidal}, and the action defined above.

\cor{act} The functors $\VV^{\sh}_\BR$ in \eqref{VRsh}  and $\VV^{\sh}$ in \eqref{Vsh} intertwine the convolution action of $\mathrm{Tilt}({\CH_G})$ on $\mathrm{Tilt}({\CM_{G_\BR}})$ and the action of $\SBim$ on $\rmod\CA$ defined above.
\ecor

The following is parallel to \refp{adj}.
\lem{SB adj} Let $\CB_{s}=\CR\ot_{\CR^{s}}\CR\in \SBim$ for each simple reflection $s\in \bW$. Then the action of $\CB_{s}$ on $\rmod\CA$ is self-adjoint: there is an isomorphism functorial in $M_{1},M_{2}\in \rmod\CA$
\begin{equation*}
\Hom_{\CA}(M_{1}\ot_{\CR}\CB_{s}, M_{2})\cong \Hom_{\CA}(M_{1}, M_{2}\ot_{\CR}\CB_{s})
\end{equation*}
\elem
\begin{proof}
As in the proof of \refp{adj}, it suffices to give unit and counit maps $u:\CR\to \CB_{s}\ot_{\CR}\CB_{s}$ and $\CB_{s}\ot_{\CR}\CB_{s}\to \CR$ in $\SBim$ satisfying identities analogous to \eqref{cu id}. The maps $u$ and $c$ are given in the proof of \refp{adj} because $\VV(\CT_{s})\cong \CB_{s}$.
\end{proof}

\ssec{}{Real Soergel functor on localized categories}
We are in the setup of \refss{locJ}. In particular, we fix $B_{0}, T_{0}\supset A_{0}$, and a subset $J$ of simple roots in $\Phi(G,A_{0})$.   The real Soergel functor $\VV_\BR$ is $\CR$-linear and so we can define 
$$\VV_\BR\ten_\CR\CK_J\colon \mathrm{Tilt}({\CM_{G_\BR}})\ten_\CR\CK_J \to \rmod\CK_J.$$
Its endomorphism algebra is opposite to $\CA\ten_\CR\CK_J$ and we also get the functor
$$\VV^{\sh}_\BR\ten_\CR\CK_J\colon \mathrm{Tilt}({\CM_{G_\BR}})\ten_\CR\CK_J \to \rmod\CA\ten_\CR\CK_J.$$

They enjoy the following properties.

\lem{blockandlevi2} 
The equivalence of \refp{blockandlevi} intertwines the localized real Soergel functors on both sides.
\elem

\prf

We put $\grp\supset\grp_\BR$ and $\grl\supset\grl_\BR$ for the relevant Lie algebras. Let $x\in X_L$ be a point in the closed orbit and let a regular nilpotent $\xi\in\CN\cap i\grg_\BR^*$ be the generic conormal to the closed orbit at $x$ inside $X$. Since $x$ lies in $X_L$, the corresponding Borel is contained in $P$ and $\xi\in\CN\cap\grp$ we conclude that $\xi$ is orthogonal to the nilpotent radical $\grp_P$ of $\grp$ and, thus, $p_L(\xi)$ lands in $i\grl^*_\BR$. Moreover, $p_L(\xi)$ is regular inside $i\grl_\BR^*$. Indeed, for a nonregular element $e_0\in\CN_L$ the elements in $e_0+\grn_P$ are also not regular. By construction we can also match the fibers $\pi^{-1}(x)$ and $\pi_L^{-1}(x)$ together with the compact part of the stabilizers of $x$ inside $G$ and $L$. We conclude that for a sheaf $\CF$ we have an isomorphism
\begin{equation*}
\mu_{(x,\xi)}(\mathrm{Av}_{G(\RR)}\circ \wt i_{P,*}(\CF))\cong\mu_{(x,p_L(\xi))}(\CF)
\end{equation*}
of free-monodromic local systems on $\bT^c$.

\epr

Recall the ring $\cR_{J}$ from \refss{locHk}.
\prop{blocktoall}
Assume that the localized Soergel functor $\VV^\sh_\BR\ten_\CR\CK_J$ is fully faithful on $\TGR\ten_\CR\CK_J$. Then $\VV^\sh_\BR\ten_\CR\CR_{J}$ is fully faithful on $\TGR\ten_\CR \CR_{J}$.
\eprop

\prf
By \refp{gen} and the fact that $V_J\subset V_{\l_0}$, the subcategory $\TGR\ten_\CR\CK_J$ generates $\TGR\ten_\CR \CR_{J}$ under the localized Hecke action \refe{locact}. 
Therefore, by \refp{adj} it is sufficient to check that 
the map
$$
\Hom_{\CM_{G_\BR}}(\CT_1,\CT_2)\ot_{\CR}\CR_{J}\to\Hom_{\CA\ot_{\CR}\CR_{J}}(\VV_\BR(\CT_1)\ot_{\CR}\CR_{J},\VV_\BR(\CT_2)\ot_{\CR}\CR_{J})
$$
is an isomorphism for $\CT_2\in\TGR\ten_\CR\CK_J$ and $\CT_1\in\TGR\ten_\CR w(\CK_J)$, where $w(\CK_J)\ne \CK_J$ is the localization of $V$ at the generic point of $w(V_J)\ne V_J$. As an $\CR$-module $\Hom_{\CM_{G_\BR}}(\CT_1,\CT_2)$ is supported on $V_J\cap w(V_J)$ and, therefore, vanishes after applying $-\ot_{\CR}\CR_{J}$. Since $\CM_{G_\BR}\ten_\CR\CR_{J}$ decomposes into the direct sum, so does $\CA\ot_{\CR}\CR_{J}$. Since $\CT_1$ and $\CT_2$ belong to the different summands of $\CM_{G_\BR}\ten_\CR\CR_{J}$, the algebra $\CA\ot_{\CR}\CR_{J}$ acts on $\VV_\BR(\CT_1)$ and $\VV_\BR(\CT_2)$ through different direct summands. It follows that $\Hom_{\CA\ot_{\CR}\CR_{J}}(\VV_\BR(\CT_1)\ot_{\CR}\CR_{J},\VV_\BR(\CT_2)\ot_{\CR}\CR_{J})=0$ and we are done.
\epr



In the codimension $0$ case recall that $\cR_{\cQ}=\cR\ot_{\cR^{\bW}}\cQ\cong \prod_{\l\in I_{0}}\cK_{\l}$ and consider the localized  functors
\begin{equation*}
\VV_{\RR,\cQ}: \Tilt(\CM_{G_\BR})_{\CQ}\to \rmod\cR_{\cQ}
\end{equation*}
and
\begin{equation*}
\VV^\sh_{\RR,\cQ}: \Tilt(\CM_{G_\BR})_{\CQ}\to \rmod\CA\ot_{\cR^{\bW}}\cQ.
\end{equation*}
We observe that 

\lem{VRQ} For any $(\l,\chi)\in \wt I_{0}$, $\VV_{\RR,\cQ}(\wt\D_{\l,\chi})\cong \cK_{\l}$ as an $\cR_{\cQ}$-module.
\elem
\prf The statement is clear for $\l=\l_{0}$. For general $\l\in I_{0}$ we have $(\l,\chi)=(\l_{0},\chi')\cdot w$ for some $w\in \bW$. By \refl{same as cross}, $\VV_{\RR,\cQ}(\wt\D_{\l,\chi})\cong \VV_{\RR,\cQ}(\wt\D_{\l_{0},\chi'}\star \wt \D_{w})$ is the translation under $w$ of $\VV_{\RR,\cQ}(\wt\D_{\l_{0},\chi'})\cong \cK$, which is $\cK_{\l}$ as an  $\cR_{\cQ}$-module.
\epr

We can now check the version of \reft{main1} for the localized categories.

\lem{loc}
The functor $\VV^\sh_{\RR,\cQ}$ is fully-faithful.
\elem 
\prf
We need to check that for $\CF_{1},\CF_{2}\in \CM_{G_\BR}$, the map
\begin{equation}\label{F1F2}
\RHom_{\CM_{G_\BR}}(\CF_1,\CF_2)\ot_{\CR^{\bW}}\CQ\to\RHom_{\CA\ot_{\CR^{\bW}}\CQ}(\VV_\BR(\CF_1)\ot_{\CR^{\bW}}\CQ,\VV_\BR(\CF_2)\ot_{\CR^{\bW}}\CQ)
\end{equation}
is an isomorphism.  Since the image of $\TGR\to \CM_{G_\BR, \CQ}$ generates $\CM_{G_\BR,\CQ}$ by taking direct summands (for $\wt \D_{\l,\chi}$ is a direct summand of $\CT_{\l,\chi}$ after localization), it suffices to check the case $\CF_{1},\CF_{2}\in \TGR$. By \refp{gen}, \refp{adj} and \refc{act} we can reduce to the case where $\CF_{2}$ is supported on  $\wt O^{\RR}_{\l_{0}}$, hence we may assume $\CF_{2}=\wt\D_{\l_{0},\psi}$.


Now instead of assuming $\CF_{1}$ is a \fmo tilting sheaf, by \refp{loc} it suffices to treat the case $\CF_{1}=\wt\D_{\l,\chi}$ for some $\l\in I_{0}$. If $\l\ne\l_{0}$, then the left side of \eqref{F1F2} is zero by \refp{loc}(2), and the right side vanishes for the same reason: the action of $\CR$ on $\VV_\BR(\CF_1)$ and $\VV_\BR(\CF_2)$ has support contained in $\Spec \ol\CR_{\l}\cap \Spec \ol\CR_{\l_{0}}$, which has dimension less than $\dim \CS^{\bW_{\l_{0}}}$. If $\l=\l_{0}$ then $\VV_{\RR}(\CF_{i})$ is just the stalk of $\CF_{i}$ at a point in $\wt O^{\RR}_{\l_{0}}$. If $\chi=\psi$, then both sides of \eqref{F1F2} are isomorphically equal to $\CK$. If $\chi\ne\psi$, then the left side is zero by \refp{loc}. By \refl{VRQ} we have $\VV_{\RR,\cQ}(\wt\D_{\l_0,\chi})\cong \VV_{\RR,\cQ}(\wt\D_{\l_0,\psi})\cong\cK_{\l_0}$. Since by \refp{loc} the objects $\wt\D_{\l_0,\chi}$ and $\wt\D_{\l_0,\psi}$ are two different simple objects of the semisimple category $\Tilt(\CM_{G_\BR})_{\CQ}$ we have an element $a\in\CA\ot_{\CR^{\bW}}\CQ$ acting on $\VV_{\RR,\cQ}(\wt\D_{\l_0,\chi})$ and $\VV_{\RR,\cQ}(\wt\D_{\l_0,\psi})$ by different elements of $\CK$. This implies that the right hand side also vanishes. The statement now follows.

%


\epr

The following lemma will allow us to transfer the results from the localized category to the original category. 

\lem{tors-free}
For a \fmo tilting object $\CT\in \Tilt(\CM_{G_{\BR}})$, the $\CR^{\bW}$-action on $\VV_\BR(\CT)$ factors through the quotient $(\CS^{\bW_{\l_{0}}})'$ and $\VV_\BR(\CT)$ is torsion free as an $(\CS^{\bW_{\l_{0}}})'$-module.
\elem

\prf
By \refl{R supp}, the action of $\CR^{\bW}$ on $\CT$ factors through $(\CS^{\bW_{\l_{0}}})'$, therefore so does its action on $\VV_{\RR}(\CT)$.
  
If $\CL$ is a \fmo local system supported on the closed orbit $\wt O^{\RR}_{\l_0}$ and extended by zero to $\wt X$, then $\VV_\BR(\CL)$ is the same as the stalk of $\CL$ along $\wt O^{\RR}_{\l_0}$, hence a free $\CS$-module. The $(\CS^{\bW_{\l_{0}}})'$-action on $\VV_\BR(\CL)$ comes from the inclusion $(\CS^{\bW_{\l_{0}}})'\subset \CS$, hence $\VV_\BR(\CL)$ is torsion free over $(\CS^{\bW_{\l_{0}}})'$. 

Now for any $\CT'\in \Tilt(\CH_{G})$, $\VV_{\RR}(\CL\star\CT')\cong \VV_{\RR}(\CL)\ot_{\CR}\VV(\CT')$. Since $\VV(\CT')$ is a Soergel bimodule, it is free as a left $\CR$-module. Therefore, as an $(\CS^{\bW_{\l_{0}}})'$-module, $\VV_{\RR}(\CL)\ot_{\CR}\VV(\CT')$ is a direct sum of $\VV_{\RR}(\cL)$, which is again torsion free. Finally, by \refp{gen} each \fmo tilting object is the direct summand of one of $\CL\star\CT'$,  the statement holds for all \fmo tilting objects.
\epr

\ssec{pf main1}{Proofs of \reft{main1}} In this subsection we assume $G$ is nice.

By \refp{gen}, \refp{adj} and \refl{SB adj}, to show the full faithfulness of $\VV_{\RR}^{\sh}$, it is sufficient to check that $\VV_\BR$ induces and isomorphism $$\Hom_{\CM_{G_\BR}}(\CT_1,\CT_2)\cong\Hom_{\CA}(\VV_\BR(\CT_1),\VV_\BR(\CT_2))$$ for $\CT_2$ supported on the preimage $\wt O^{\RR}_{\l_{0}}$ of the closed orbit. Let $\wt i: \wt O^{\RR}_{\l_{0}}\incl \wt X$ be the inclusion. We may assume $\CT_{2}=\tilna_{\l_{0},\chi}\cong \wt\D_{\l_{0},\chi}$ for some character $\chi$ of $\pi_{0}(\bT^{c}_{\l_{0}})$.

By adjunction we have $$\Hom_{\CM_{G_\BR}}(\CT_1,\tilna_{\l_{0},\chi})=\Hom_{\CM_{G_\BR}}(\wt i_*\wt i^*\CT_1,\tilna_{\l_{0},\chi})=\Hom_{\CM_{G_\BR}}(\wt i_*(\wt i^*\CT_1)_{\chi},\tilna_{\l_{0},\chi}).$$
Here $(\wt i^*\CT_1)_{\chi}$ is the direct summand of $\wt i^{*}\CT_{1}$ where $\pi_{0}(\bT^{c}_{\l_{0}})$ acts by $\chi$. 

We claim that the natural map
$$\Hom_{\CM_{G_\BR}}(\wt i_*(\wt i^*\CT_1)_{\chi}, \tilna_{\l_{0},\chi})\to\Hom_{\CA}(\VV_\BR(\wt i_*(\wt i^*\CT_1)_{\chi}),\VV_\BR(\tilna_{\l_{0},\chi}))$$ 
is an isomorphism. Indeed, it suffices to replace $\wt i_*(\wt i^*\CT_1)_{\chi}$ by $\tilna_{\l_{0},\chi}$ and show $\End(\tilna_{\l_{0},\chi})\cong \Hom_{\CA}(\VV_\BR(\tilna_{\l_{0},\chi}))$. Now the left side is $\CS$.  As $\VV_\BR$ is the stalk functor when restricted to local systems on $\wt O^{\RR}_{\l_{0}}$, the right side is $\End_{\CA}(\CS)$ where $\CS$ is viewed as an $\CA$-module via the $\CR$-algebra homomorphism 
\begin{equation*}
\ev_{\chi}: \CA\to \ol\CR_{\l_{0}}=\CS.
\end{equation*}
given by the action of $\CA$ on $\VV_{\BR}(\tilna_{\l_{0},\chi})\cong \CS$. Since $\CR\surj \CS$, $\ev_{\chi}$ is surjective, hence $\End_{\CA}(\CS)=\CS$, which coincides with the left side.


We denote $\VV_{\BR}(\tilna_{\l_{0},\chi})$ by $\CS_{\chi}$ to emphasize that it is isomorphic to $\CS$ as an $\CR$-module, and $\CA$ acts on it via $\ev_{\chi}$. Therefore, it remains to check that the map 
\begin{equation}\label{Vii}
\Hom_{\CA}(\VV_\BR(\wt i_*(\wt i^*\CT_1)_{\chi}),\CS_{\chi})\to\Hom_{\CA}(\VV_\BR(\CT_1),\CS_{\chi})
\end{equation}
induced by the unit map $u: \CT_{1}\to \wt i_*\wt i^*\CT_1\to \wt i_*(\wt i^*\CT_1)_{\chi}$ is an isomorphism. The left side of \eqref{Vii} is 
\begin{equation*}
\Hom_{\CS}(\VV_\BR(\wt i_*(\wt i^*\CT_1)_{\chi}),\CS)
\end{equation*}
since the action of $\CA$ on $\VV_\BR(\wt i_*(\wt i^*\CT_1)_{\chi})$ factors through $\CS$ via $\ev_{\chi}$. The right side of \eqref{Vii} can be identified with
\begin{equation*}
\Hom_{\CS}(\VV_\BR(\CT_{1})\ot_{\CA,\ev_{\chi}}\CS,\CS).
\end{equation*}
Therefore \eqref{Vii} comes from the map of $\CS$-modules
\begin{equation*}
\r: \VV_\BR(\CT_{1})\ot_{\CA,\ev_{\chi}}\CS\to \VV_\BR(\wt i_*(\wt i^*\CT_1)_{\chi})
\end{equation*}
by taking the $\CS$-linear dual.  Since $\CS$ is regular, to show \eqref{Vii} is an isomorphism, it suffices to show
\begin{eqnarray}
\label{ker r}&&\mbox{$\ker(\r)$ is a torsion $\CS$-module;}\\
\label{coker r}&&\mbox{$\coker(\r)$ has support of codimension $\ge2$ in $\Spec\CS=V_{\l_{0}}$.} 
\end{eqnarray}

To check \eqref{ker r}, we localize $\CM_{G_{\RR}}$ at the generic point of $V_{\l_{0}}=\Spec\CS\subset\Spec \CR=V$ as in \refss{loc0} and use \refl{loc}. We conclude that both $\ker(\r)$ and $\coker(\r)$ are torsion $\CS$-modules.

To check \eqref{coker r}, we first observe that the cokernel of $\r$ is supported on $\cup_{\l\in I_{1}}  (V_{\l}\cap V_{\l_{0}})$ as an $\CS$-module, where
\begin{equation*}
I_{1}=\{\l\in I|n_{\l}=n_{\l_{0}}-1\}.
\end{equation*}
Indeed, let $\CK\in \CM_{G_{\BR}}$ fit into the distinguished triangle $\CK\to \CT_{1}\to \wt i_*(\wt i^*\CT_1)_{\chi}\to \CK[1]$. Then $\CK$ is a successive extension of $\wt\D_{\l,\psi}$ where $(\l,\psi)\ne (\l_{0},\chi)$. Taking $\VV_{\BR}$, using that $\VV_{\BR}(\CT_{1})$ is concentrated in degree $0$,  we get an exact sequence
\begin{equation*}
0\to H^{0}\VV_{\BR}(\CK)\to \VV_{\BR}(\CT_{1})\to \VV_{\BR}(\wt i_*(\wt i^*\CT_1)_{\chi}) \to H^{1}\VV_{\BR}(\CK)\to 0.
\end{equation*}
We have $\coker(\r)=H^{1}\VV_{\BR}(\CK)$. As an $\CR$-module, the support of $H^{1}\VV_{\BR}(\CK)$  is contained in the union of supports of $H^{1}\VV_{\BR}(\wt\D_{\l,\psi})$, which can be nonzero only when $\l\in I_{1}$ by \refp{exact}. Therefore, $\supp_{\CR}(H^{1}\VV_{\BR}(\CK))\subset \cup_{\l\in I_{1}}  V_{\l}$. Since $H^{1}\VV_{\BR}(\CK)$ is also a quotient of $\VV_{\BR}(\wt i_*(\wt i^*\CT_1)_{\chi})$ which is supported on $V_{\l_{0}}$ as an $\CR$-module, we conclude that $\coker(\r)=H^{1}\VV_{\BR}(\CK)$ is supported on $\cup_{\l\in I_{1}}  (V_{\l}\cap V_{\l_{0}})$ as an $\CS$-module.

Therefore, to show \eqref{coker r}, it suffices to show that for any $\l\in I_{1}$ such that $V_{\l}\subset V_{\l_{0}}$, letting $\y_{\l}$ be the generic point of $V_{\l}\subset V_{\l_{0}}$ and $\wh\CS_{\y_{\l}}$ be the completed local ring of $\CS$ at $\y_{\l}$, the map
\begin{equation}\label{local comp r}
\wh\r_{\y_{\l}}=\r\ot_{\CS}\id_{\wh\CS_{\y_{\l}}}: \VV_\BR(\CT_{1})\ot_{\CA,\ev_{\chi}}\wh\CS_{\y_{\l}}\to \VV_\BR(\wt i_*(\wt i^*\CT_1)_{\chi})\ot_{\CS}\wh\CS_{\y_{\l}}
\end{equation}
is surjective. 

\refp{blockandlevi}, \refl{blockandlevi2} and \refp{blocktoall} allow to reduce the localized statement to the case of Levi subgroup $L=L_J$ with $J$ a subset of simple roots cutting out $V_\l\subset V_{\l_0}$. Such Levi is quasi-split and $L^{ad}$ has split rank $1$. Our general strategy for dealing with $L$ is to relate it to the case of the adjoint group $L^{ad}$, making use of the niceness assumption of $G_{\RR}$, and then deal with quasi-split adjoint groups of split rank $1$ case by case.

Let $T_{L}=L/L^{der}$ so that $L$ fits into an exact sequence
\begin{equation*}
1\to \G_{L}\to L\to L^{ad}\x T_{L}\to1.
\end{equation*}
Here $\G_{L}=Z(L^{der})$. We are going to use the discussions in \refss{tw Hk}, \refss{cent} and \refss{nonadj}. Let $\omega: \G_{L}(\RR)\to \bfk^{\x}$ and extend it to $\wt\omega: \G_{L}\to \bfk^{\x}$, which corresponds to a rank one local system $\cL\bt\cL'$ on $(\bT^{ad})^{c}\times T_{L}^{c}$. By \refe{GtoGad} we get
\begin{equation*}
\CM^{\omega}_{L_{\RR}}\cong (\CM_{L^{ad}_{\RR}\times T_{L,\RR}, \cL\bt\cL'})^{\grS_{L}}
\end{equation*}
where $\grS_{L}$ is dual to the cokernel $(L^{ad}(\RR)\times T_{L}(\RR))\mathrm{Im}(L(\RR))$. The irreducible $T_{L}(\RR)$-equivariant local systems on $T_{L}^{c}$ that appear in $\CM_{T_{L,\RR},\cL'}$ are either empty, in which case there is nothing to prove, or are indexed by a $\pi_{0}(T_{L}(\RR))^{*}$-torsor $\Xi_{\cL'}$.  We have 
\begin{equation}\label{TLdecomp}
\CM_{L^{ad}_{\RR}\times T_{L,\RR}, \cL\bt\cL'}\simeq \bigoplus_{\chi\in\Xi_{\cL'}}\CM_{L^{ad}_{\RR},\cL}\wh\ot \ol \cR'
\end{equation}
where $\ol\cR'$ is the completed group ring of $\pi_{1}(T_{L}^{c}/\mathrm{Im}(T_{L}(\RR)))$.

By our assumption of niceness, $L(\RR)\to L^{ad}(\RR)$ is surjective (although the niceness assumption only involves the case $L^{ad}_{\RR}\cong \PGL_{2,\RR}$, the surjectivity is automatic in all other cases as $L^{ad}(\RR)$ will be connected). Therefore we have a surjection $\pi_{0}(T_{L}(\RR))\surj \grS^{*}$, and dually an injection $\grS\incl \pi_{0}(T_{L}(\RR))^{*}$.  The $\grS$-action on $\CM_{L^{ad}_{\RR}\times T_{L,\RR}, \cL\bt\cL'}$ corresponds to the permutation of summands on the right side of \eqref{TLdecomp} via the injective map $\grS\incl \pi_{0}(T_{L}(\RR))^{*}$ and the $\pi_{0}(T_{L}(\RR))^{*}$-action on $\Xi_{\cL'}$. Therefore, taking $\grS$-invariants we get
\begin{equation}\label{MLRdecomp}
\CM^{\omega}_{L_{\RR}}\cong\bigoplus_{\chi\in\Xi_{\cL'}/\grS}\CM_{L^{ad}_{\RR},\cL}\wh\ot \ol \cR'.
\end{equation}
This allows us to reduce the proof of the surjectivity of $\r$ for $L$ to the same statement in each summand of \eqref{MLRdecomp}, which is clearly equivalent to the surjectivity of $\r$ in the case of $L^{ad}_{\RR}$.  

Thus we reduce to showing the surjectivity of $\r$ for indecomposable tilting sheaves $\cT_{1}\in \CM_{G_{\RR}, \cL}$, $G_{\RR}$ is adjoint and quasi-split of split rank one (so $G_{\RR}\cong \PU(2,1), R_{\CC/\RR}\PGL_{2}$ or $\PGL_{2,\RR}$),  and $\cL$ is a rank one local system on $\bT^{c}$ extendable to $G$.

\lem{surjrk1}
Under the above assumptions on $G_{\RR}$, the map $\rho$ is surjective.
\elem

\prf

We should go over the cases of \refss{rank1}. In each case, we need to show that for any indecomposable tilting $\cT_{1}$, the restriction map to the closed orbit induces a surjection
\begin{equation*}
\wt\r: \VV_{\BR}(\cT_{1})\to \VV_{\BR}(\wt i_{*}(\wt i^*\cT_{1})_{\chi})
\end{equation*}
for any character $\chi$ of $\pi_{0}(\bT^{c}_{\l_{0}})$.

If $G_\BR=R_{\CC/\RR}\PGL_2$, we have $G=\PGL_{2}\times \PGL_{2}$. The $\cL$ that makes $\CM_{G_{\RR},\cL}\ne0$ is either trivial or nontrivial on both factors of $G$.  In both cases, $\CM_{G_{\RR},\cL}$ is  equivalent to the Hecke category $\CH_{\PGL_{2}}$, and the surjectivity of $\rho$ is well-known.
 
If $G_\BR=\PGL_{2,\BR}$ and $\cL$ is nontrivial, all objects in $\CM_{G_{\RR}, \cL}$ are supported on the closed orbit, in which case there is nothing to prove. Below we assume $\cL$ is trivial and use notations from \refss{rank1}. The functor $\VV_\BR$ is equal to $\ker(s_{\mathrm{triv}}+ s_{\mathrm{sgn}})$. We should consider the case $\CT_{1}=\CT_h$. We have $\VV_\BR(\CT_{0,\chi})=\ker(\bfk[[x]]\oplus\bfk[[x]]\to\bfk)$. On the other hand, $\wt i_{0,*}(\wt i^*_0\CT_h)\simeq\CT_{0,\triv}\oplus\CT_{0,\sgn}$. For $\chi=\triv,\sgn$ we have $\VV_\BR(\CT_{0,\chi})=\bfk[[x]]$. The map $\wt\r: \ker(\bfk[[x]]\oplus\bfk[[x]]\to\bfk)\to \bfk[[x]]$ is induced by projection on one of the summands. We conclude that it is indeed surjective.

Consider the case $G_\BR=\PU(2,1)$. If $\cL$ is nontrivial, by looking at the stabilizers of orbits we see that $\CM_{G_{\RR}, \cL}=0$. Below we consider the case $\cL$ is trivial. When $\cT_{1}$ is not supported on the open orbit, the calculation is the same as in the case of $\PGL_{2}(\CC)$, which we omit. Therefore we assume that the support of $\cT_{1}$ is open, i.e., $\cT_{1}=\cT_{\l}$ for $\l\in \{-|+-, +|++, +|+-\}$.  Using notation $\cC_{\l}=\ker(\cT_{\l}\to \wt i_{*}\wt i^{*}\cT_{\l})$, we reduce to show that $H^{>0}\VV_{\RR}(\cC_{\l})=0$. Short exact sequences \refe{su(2,1)-1},  \refe{su(2,1)-2},  \refe{su(2,1)-3} after applying $\VV_\BR$ yield long exact sequences of cohomology. They imply in particular $H^{>1}\VV_{\RR}(\cC_{\l})=0$ and a surjection $H^{1}\VV_{\RR}(\tilDel_{\l})\surj H^{1}\VV_{\RR}(\cC_{\l})$. For $\l=-|+-$,  using the local model of a transversal slice to the closed orbit,  we see that $SS(\tilDel_{\l})|_{\cO^{\RR}_{\l_{0}}}$ is fiberwise one-dimensional (in the conormal direction of the real hypersurface $\ol\cO^{\RR}_{0|+-}$). Therefore $\VV_{\RR}(\tilDel_{\l})=0$. This implies $H^{1}\VV_{\RR}(\cC_{\l})=0$. Similarly argument works for $\l=+|++$. 


In case of $\l=+|+-$, using \refe{su(2,1)-3} and \refp{exact} we get the exact sequence:
$$H^0\VV_\BR(\tilDel_{+|+0}\oplus\tilDel_{0|+-})\to H^1\VV_\BR(\tilDel_{+|+-})\to H^1\VV_\BR(\cC_{+|+-})\to0.$$ It is therefore sufficient to verify the surjectivity of the first map. Since $\tilDel_{+|+-}$ is the pullback of $\D_{+|+-}$ on $X$, we have $\VV_\BR(\tilDel_{+|+-})\cong \VV_\BR(\D_{+|+-})$. Hence it suffices to show the surjectivity of the map $$H^0\VV_\BR(\Del_{+|+0}\oplus\Del_{0|+-})\to H^1\VV_\BR(\Del_{+|+-}),$$ where the standard sheaves are now on $X$. The latter map coincides with the boundary homomorphism obtained by applying $\VV_\BR$ to the standard filtration of $\nabla_{+|+-}$, which we know to be surjective as $H^1\VV_\BR(\nabla_{+|+-})=0$ by \refp{exact}.

\epr

This finishes the proof of \eqref{coker r}, and completes the argument for \reft{main1}.

\ssec{}{The algebras $\CB$ and $\CB_{0}$}
In this subsection we assume $G_{\RR}$ is quasi-split but not necessarily nice.

Let $\LS_{\l_0}=\{\chi: \pi_{0}(\bT^{c}_{\l_{0}})\to \bfk^{\times}\}$,  the set of rank one $G(\RR)$-equivariant  local systems on the closed orbit $O^{\RR}_{\l_{0}}$. We have an action of $\bW_{\l_0}$ on $\LS_{\l_0}$ by the cross action. The group $\bW_{\l_0}$ acts on $\LS_{\l_0}$ via the cross action. It also acts on $\CS$ compatibly with the action of $\bW$ on $\cR$. Define an action of $\bW_{\l_0}$ on $\op_{\chi\in\LS_{\l_0}}\cS$ as follows:  for $f\in \cS$ put in summand $\chi$, denoted $f_{\chi}$, $w(f_{\chi}):=(wf)_{\chi\cdot w^{-1}}$ for $w\in \bW_{\l_0}$.  Let
\begin{equation*}
\CB_{0}=(\bigoplus_{\chi\in\LS_{\l_0}}\cS)^{\bW_{\l_{0}}}.
\end{equation*}


By definition, $\CB_{0}$ is an $\CS^{\bW_{\l_{0}}}$-algebra.  View it as an $\cR^{\bW}$-algebra via the natural map $\cR^{\bW}\to \CS^{\bW_{\l_{0}}}$. Define another algebra
\begin{equation*}
\CB:=\CB_{0}\ot_{\cR^{\bW}}\CR.
\end{equation*}

\rem{BV}
The algebra $\CB$ is related to the block variety of \cite{BV} in the following way. The orbits of $\LS_{\l_0}$ under the cross action of $\bW_{\l_0}$ are called blocks. This coincides with the notion of blocks for $(\frg,K)$-modules with a fixed regular integral infinitesimal character, see \cite[Claim 2.2]{BV}. Then $\CB$ defined above is the direct product of completions of the algebras of functions on the block varieties $\mathfrak{B}_{mon}$ from \cite[Remark 2.3]{BV} for all blocks.

\erem

The action of $\CA$ on $\VV_{\RR}(\cL_{\l_{0}, \chi})$ for $\chi\in \LS_{\l_0}$ gives a homomorphism
\begin{equation*}
\act_{\l_{0}}: \CA\to \bigoplus_{\chi\in\LS_{\l_0}}\End_{\cR}(\VV_{\RR}(\cL_{\l_{0}, \chi}))^{opp}=\bigoplus_{\chi\in\LS_{\l_0}} \cS.
\end{equation*}
Here we use that $\End_{\cR}(\VV_{\RR}(\cL_{\l_{0}, \chi}))^{opp}\cong \End_{\cR}(\cS)^{opp}=\cS$.

Recall from \refp{A descent} that we have $\CA=\CA_0\ten_{\CR^{\bW}}\CR$, where the subalgebra $\CA_0\subset\CA$ is exactly the endomorphisms commuting with the Hecke action. 


\th{main2} Let $G_{\RR}$ be any quasi-split real group. The map $\act_{\l_{0}}$ restricts to an $\cS^{\bW_{\l_{0}}}$-algebra isomorphism $\CA_0\iso\CB_{0}$. In particular, we have an isomorphism of $\cR$-algebras $\cA\isom \cB$.
\eth

\prf 
By \refp{gen} the map $\act_{\l_{0}}$ restricted to $\cA_{0}$ is injective: if $a\in \CA_{0}$ acts by zero on $\VV_{\RR}(\cL_{\l_{0},\chi})$ for all $\chi\in \LS_{\l_0}$, it will act by zero on all $\VV_{\RR}(\cL_{\l_{0},\chi}\star\cT_{w})$ for all $w\in \bW$, which contain $\VV_{\RR}(\cT_{\l,\psi})$ as a direct summand for all $(\l,\psi)\in \wt I$, hence $a=0$. In particular, $\CA_{0}$ is torsion-free as an $\cS^{\bW_{\l_{0}}}$-module.

Let us prove that $\act_{\l_{0}}$ sends $\CA_{0}$ to the $\bW_{\l_{0}}$-invariants of $\op_{\chi\in \LS_{\l_0}}\cS$. Since $\CA_{0}$ is $\cS^{\bW_{\l_{0}}}$-torsion free, it suffices to show the statement after tensoring with $\cQ$. In fact we will show that  $\act_{\l_{0}}$ induces an isomorphism
\begin{equation}\label{A0B0Q}
\act_{\l_{0},\cQ}: \cA_{0,\cQ}:=\cA_{0}\ot_{\cR^{\bW}} \cQ\to (\op_{\chi\in\LS_{\l_0}}\cQ)^{\bW_{\l_{0}}}=\cB_{0,\cQ}.
\end{equation}

Consider the localized category $\Tilt(\CM_{G_{\RR}})_{\cQ}$ under the action of the base-changed Hecke category $\Tilt(\CH_{G})_{\cQ}$. Then $\cA_{\cQ}:=\cA\ot_{\cS^{\bW_{\l_{0}}}}\cQ$ is the endomorphism ring of the functor $\VV_{\RR,\cQ}: \Tilt(\CM_{G_{\RR}})_{\cQ}\to \rmod\cR_{\cQ}$, and $\cA_{0,\cQ}$ is the subalgebra of $\cA_{\cQ}$ commuting with the Hecke action.  By \refl{VRQ} we can compute $\CA_{\cQ}$ explicitly as 
\begin{equation}\label{AQ}
\CA_{\cQ}\isom \bigoplus_{(\l,\chi)\in \wt I_{0}}\cK_{\l},
\end{equation}
where the $(\l,\chi)$-factor is the action of $\cA_{\cQ}$ on $\VV_{\RR,\cQ}(\wt\D_{\l,\chi})\cong \cK_{\l}$. By \refl{same as cross}, the part of $\cA$ that commutes with the Hecke action corresponds to the $\bW$-invariants of the right side of \eqref{AQ} under the cross action. Since $\bW$ acts transitively on $I_{0}$, we may rewrite the $\bW$-invariants of the right side as $(\op_{\chi\in \LS_{\l_{0}}}\cK)^{\bW_{\l_{0}}}=\cB_{0,\cQ}$. This proves \eqref{A0B0Q}. In particular, $\act_{\l_{0}}$ restricts to a ring homomorphism
\begin{equation*}
\act'_{\l_{0}}: \cA_{0}\to \cB_{0}
\end{equation*}
that is injective and becomes an isomorphism after tensoring with $\cQ$. 



%

It remains to show that $\act'_{\l_{0}}$ is surjective. Given a collection $b=(b_{\chi})_{\chi\in\LS_{\l_{0}}}\in \cB_{0}$, we would like to construct $a\in \cA_{0}$ that acts on $\VV_{\RR}(\CT_{\l_{0},\chi})=\VV_{\RR}(\wt\D_{\l_{0},\chi})$ by $b_{\chi}$. For $\chi\in\LS_{\l_{0}}$ and $w\in \bW$, define an endomorphism $a_{\chi,w}$ of $\VV_{\RR}(\CT_{\l_{0},\chi}\star \CT_{w})\cong \VV_{\RR}(\CT_{\l_{0},\chi})\ot_{\cR}\VV(\cT_{w})$ by $b_{\chi}\ot\id_{\VV(\cT_{w})}$. For any morphism $\ph: \CT_{\l_{0},\chi}\star \CT_{w}\to \CT_{\l_{0},\chi'}\star \CT_{w'}$ in $\Tilt(\CM_{G_{\RR}})$, we claim that the following diagram is commutative
\begin{equation}\label{achiw}
\xymatrix{ \VV_{\RR}(\CT_{\l_{0},\chi}\star \CT_{w}) \ar[d]^{\VV_{\RR}(\ph)}\ar[r]^{a_{\chi,w}}& \VV_{\RR}(\CT_{\l_{0},\chi}\star \CT_{w})\ar[d]^{\VV_{\RR}(\ph)}\\
\VV_{\RR}(\CT_{\l_{0},\chi'}\star \CT_{w'}) \ar[r]^{a_{\chi',w'}} & \VV_{\RR}(\CT_{\l_{0},\chi'}\star \CT_{w'})
}
\end{equation}
Indeed, after tensoring with $\CQ$ the above diagram is commutative since $\cA_{0,\cQ}\isom \cB_{0,\cQ}$. Since $\VV_{\RR}(\CT_{\l_{0},\chi'}\star \CT_{w'})$ is torsion-free as an $(\cS^{\bW_{\l_{0}}})'$-module by \refl{tors-free}, the diagram is commutative. Let $\Tilt'\subset \Tilt(\CM_{G_{\RR}})$ be the full subcategory whose objects are finite direct sums of $\CT_{\l_{0},\chi}\star \CT_{w}$ for various $\chi\in\LS_{\l_{0}}$ and $w\in \bW$. By the commutativity of \eqref{achiw}, the collection $\{a_{\chi,w}\}_{\chi\in\LS_{\l_{0}}, w\in \bW}$ gives an endomorphism of $\VV_{\RR}|_{\Tilt'}$. By \refp{gen}, all objects in $\Tilt(\CM_{G_{\RR}})$ are direct summands of objects in $\Tilt'$, therefore the $\{a_{\chi,w}\}$ defines an endomorphism of $\VV_{\RR}$, i.e., an element $a\in \CA$. By construction, $a$ acts on $\VV_{\RR}(\CT_{\l_{0},\chi})$ by $b_{\chi}$, and $a$ commutes with the Hecke action. Therefore $a\in\CA_{0}$ satisfies $\act_{\l_{0}}(a)=b$.

%


\epr

Combining \reft{main1} and \reft{main2}, we get:
\cor{maincor} Suppose $G_{\RR}$ is quasi-split and nice. The enhanced real Soergel functor gives a fully faithful embedding
\begin{equation*}
\VV_{\RR}^{\sh}: \Tilt(\CM_{G_{\RR}})\to \rmod\cB.
\end{equation*}

\ecor

\sec{nonadj}{Structure theorem for general quasi-split groups}

In this section we state and prove the Soergel structure theorem for general quasi-split groups without assuming they are nice. To treat the general case we first need to consider variants of the previous setup: allowing non-unipotent monodromy along $\bT^{c}$ and fixing a central character.  

We put ourselves in the situation of \refss{Matsuki setup}.

\ssec{tw Hk}{Twisted version of $\CM_{G_{\RR}}$} We recall the notion of sheaves on $\wt X$ equivariant under $\bT$ with respect to a local system, following \cite[\S2.1, 2.5]{LY}. 

Let $\CL$ be a rank one local system on $\bT^{c}$ with finite monodromy (in $\bfk$-coefficients). 
Such local systems all come from the following construction: take a finite isogeny $\nu\colon \wt \bT^{c} \to \bT^{c}$ and a character $\kappa\colon\ker(\nu)\to\bfk^{\x}$, then take $\CL$ to be the rank one direct summand of $\nu_*\underline{\bfk}$ where $\ker(\nu)$ acts via $\kappa$. 

Consider the equivariant category $D^b_{G(\RR)}(\wt X/\wt \bT^{c})$ with $\wt \bT^{c}$ acting on $\wt X$ through $\nu$. Since the finite abelian group $\ker(\nu)$ acts trivially on $X$, it acts on the identity functor of $D^b_{G(\RR)}(\wt X/\wt \bT^{c})$. This allows us to decompose $D^b_{G(\RR)}(\wt X/\wt \bT^{c})$ into a direct sum of full triangulated subcategories according to
characters of $\ker(\nu)$. Let $D^b_{G(\RR)}(X)_{\CL}$ be the direct summand of $D^b_{G(\RR)}(\wt X/\wt \bT^{c})$ corresponding to $\kappa$. One can check that $D^b_{G(\RR)}(X)_{\CL}$ depends only on the local system $\cL$ on $\bT^{c}$ and not on the choices of the isogeny $\nu:\wt \bT^{c}\to \bT^{c}$ and $\kappa$. 

Let $\LS_{\l, \cL}$ be the set of isomorphism classes of irreducible $G(\RR)$-equivariant local systems on $O^{\RR}_{\l}$ that appear in $D^{b}_{G(\RR)}(X)_{\CL}$. 
We describe $\LS_{\l, \cL}$ more explicitly. Choose a finite isogeny $\nu\colon \wt \bT^{c} \to \bT^{c}$ such that $\CL$ corresponds to a character $\kappa\colon\ker(\nu)\to\bfk^{\x}$. Let $\wt\bT^{c}_{\l}=\nu^{-1}(\bT^{c}_{\l})$. Then there is a natural bijection
\begin{equation}\label{LS orbit tw}
\LS_{\l,\cL}\cong\{\wt\chi: \pi_{0}(\wt\bT^{c}_{\l})\to \bfk^{\times}; \wt\chi|_{\ker(\nu)}=\kappa\}.
\end{equation}
In particular, $\LS_{\l, \cL}\ne\varnothing$ if and only if the restriction of $\kappa$ to the kernel of  the map $\ker(\nu)\subset \wt\bT^{c}_{\l}\surj \pi_{0}(\wt\bT^{c}_{\l})$ is trivial; and when this happens, $\LS_{\l, \cL}$ is a torsor under $\pi_{0}(\bT^{c}_{\l})^{*}$.

We define the monodromic category $D^b_{G(\RR)}(\wt X)_{(\bT^{c},\CL)-\mathrm{mon}}$ as the full subcategory of $D^b_{G(\RR)}(\wt X)$ generated by the image of $D^b_{G(\RR)}(X)_{\CL}$ via pullback.  Finally, we can define the completed version
$$\CM_{G_{\RR}, \cL}:=\wh D^b_{G(\RR)}(\wt X)_{(\bT^{c},\CL)\mon}$$ 
as in \refs{compmoncat}.

We define the $\CL$-twisted version of the algebra $\cB$ as follows. Recall $\l_{0}\in I$ is the index of the closed $G(\RR)$-orbit. Let  
$$\cB_{0,\cL}=(\bigoplus_{\wt\chi\in \LS_{\l_{0}, \cL}}\cS)^{\bW_{\l_{0}}},$$
and let
$$\cB_{\cL}=\cB_{0,\cL}\ot_{\cR^{\bW}}\cR.$$

Applying the twisting construction to the setting of \refss{Hecke} we define $\CL$-twisted monoidal Hecke category $$\CH_{G,\CL}:=\wh D^b_G(\wt X\x \wt X)_{(\bfT^{c}\x \bfT^{c},\CL\boxtimes\CL^{-1})-\mathrm{mon}}.$$ This is equivalent to the one defined in \cite[\S2.7]{LY}.

In the following we are only interested in the case where $\CL$ extends to a rank one local system $\cL_{G}$ on $G$. This happens if and only if one can choose a finite isogeny $\nu': \wt G\to G$ inducing  the isogeny of abstract Cartan groups $\nu:\wt \bT^{c}\to \bT^{c}$ such that $\CL$ is attached to some character $\kappa$ of $\ker(\nu)$. In this case, tensoring with $\cL_{G}$ induces a monoidal equivalence
\begin{equation}\label{tw Hk equiv}
\CH_{G}\simeq\CH_{G,\CL}.
\end{equation} 

The following theorem generalizes \reft{main1} and \refc{maincor} with the same proof (where we use \eqref{tw Hk equiv}).
\th{twisted main} Let $G_{\RR}$ be quasi-split and nice.  Let $\cL$ be a rank one local system on $\bT$ that extends to $G$. Then the real Soergel functor naturally lifts to a fully faithful embedding
\begin{equation*}
\VV^{\sh}_{\RR}: \Tilt(\CM_{G_{\RR}, \cL})\to\rmod \cB_{\cL}. 
\end{equation*}
\eth

\ssec{cent}{Central characters}
Let $\G\subset ZG(\RR)$ be a finite subgroup. Let $\omega: \G\to \bfk^{\times}$ be a character. We define $D^{b}_{G(\RR)(\G, \omega)}(X)$ to be the full subcategory of $D^{b}_{G(\RR)\G}(X)$ consisting of objects $\cF$ on which $\G$ acts via the character $\omega$ (note that $\G$ acts trivially on $X$, hence acts on every $\cF\in D^{b}_{G(\RR)\G}(X)$). Similarly define $D^{b}_{G(\RR)(\G, \omega)}(O^{\RR}_{\l})$ for each $G(\RR)$-orbit $O^{\RR}_{\l}$.

For $\l\in I$,  the group $\bT^{c}_{\l}$ contains the compact part of $Z(G)$, hence contains $\G$. Therefore, the set of isomorphism classes of irreducible local systems in $D^{b}_{G(\RR)(\G, \omega)}(O^{\RR}_{\l})$ is in natural bijection with the set of characters
\begin{equation}\label{LS orbit cen}
\LS_{\l}^{\ome}:=\{\chi: \pi_{0}(\bT^{c}_{\l})\to \bfk^{\times};\chi|_{\G}=\omega\}.
\end{equation}
In particular, $\LS_{\l}^{\ome}\ne\varnothing$ if and only if $\omega$ is trivial on the kernel of the map $\G\subset \bT^{c}_{\l}\surj \pi_{0}(\bT^{c}_{\l})$; and when this happens, $\LS_{\l}^{\ome}$ is a torsor under the dual group $\coker(\G\to \pi_{0}(\bT^{c}_{\l}))^{*}$.

We define $D^{b}_{G(\RR)(\G, \omega)}(\wt X)_{\bT^{c}\mon}$ to be the full subcategory of $D^{b}_{G(\RR)\G}(\wt X)$ generated by the pullbacks of $D^{b}_{G(\RR)(\G, \omega)}(X)$. We can define the completed version as in \refs{compmoncat}:
$$\CM_{G_{\RR}}^{\omega}:=\wh D^{b}_{G(\RR)(\G, \omega)}(\wt X)_{\bT^{c}\mon}.$$ 

Note that $\LS_{\l}^{\ome}$ is stable under the cross action of $\bW_{\l}$. Define $\CB^{\ome}_{0}=(\op_{\chi\in \LS_{\l_{0}}^{\ome}}\CS)^{\bW_{\l_{0}}}$, and $\CB^{\ome}=\CB_{0}^{\ome}\ot_{\cR^{\bW}}\cR$. Let $\cA^{\ome}$ be opposite to the endomorphism ring of $\VV_{\RR}|_{\Tilt(\CM_{G_{\RR}}^{\omega})}$. Then \reft{main2} implies an isomorphism $\cA^{\ome}\cong \cB^{\ome}$. In particular, $\VV_{\RR}|_{\Tilt(\CM_{G_{\RR}}^{\omega})}$ is enhanced to a functor
\begin{equation}\label{VR ome}
\VV^{\sh,\ome}_{\RR}: \Tilt(\CM_{G_{\RR}}^{\omega})\to \rmod\CB^{\ome}.
\end{equation}
However, this functor is not fully faithful in general. To remedy, we will modify this functor by replacing the right side with an equivariant and twisted version of $\cB$ for the adjoint group of $G_{\RR}$.

\ssec{deeq}{De-equivariantization}
We recall the procedure of de-equivariantization, as explained in  \cite[21 and 22]{G}. Let $\G$ be a finite abelian group such that $|\G|$ is prime to $\ch(\bfk)$. Assume $\bfk$ contains enough roots of unity such that $\G^{*}=\Hom(\G,\bfk^{\times})$ has the same cardinality as $\G$.  Assume $\G$ acts on a $\bfk$-linear idempotent complete category $\cC$. Let $\cD=\cC^{\G}$ be the category of $\G$-equivariant objects in $\cC$: an object $d\in \cD$ is tuple $(c, \{\a_{\g}\}_{\g\in \G})$ where $X\in \cC$ and $\a_{\g}:\g(c)\isom c$ are isomorphisms indexed by $\g\in\G$ satisfying $\a_{1}=\id_{c}$ and $a_{\g_{1}\g_{2}}=\a_{\g_{1}}\c\g_{1}(\a_{\g_{2}})$. Then there is an action of the dual group $\G^{*}$ on $\cD$ as follows: for $\chi\in \G^{*}$ and $d=(c,\{\a_{\g}\})\in \cD$, define $\chi(d)=(c, \{\a'_{\g}\})$ where $\a'_{\g}$ is $\a_{\g}$ multiplied by $\j{\chi,\g}\in \bfk^{\times}$. Then we can recover $\cC$ as the category of $\G^{*}$-equivariant objects in $\cD$. We give the functors as follows.  There is a  functor $\av_{\G}: \cC\to \cD$ sending $c$ to the object $\av_{\G}(c)=\op_{\g\in \G} \g(c)$ with its obvious $\G$-equivariant structure. Then $\av_{\G}(c)\in\cD$ in fact carries a canonical $\G^{*}$-equivariant structure and hence lifts to an object $\av_{\G}(c)^{\sh}\in\cD^{\G^{*}}$. The assignment $c\mapsto \av_{\G}(c)^{\sh}$ gives a functor $\cC\to \cD^{\G^{*}}$. On the other hand, let $(d,\{\a_{\chi}\}_{\chi\in\G^{*}})\in \cD^{\G^{*}}$, with $d=(c, \{\a_{\g}\}_{\g\in\G})\in \cD$. The data of $\a_{\chi}$ means automorphisms $\b_{\chi}\in \Aut(c)$ that form a $\G^{*}$-action on $c$ compatible with $\{\a_{\g}\}_{\g\in \G}$. In particular, we can extract the direct summand $c_{1}=c^{\G^{*}}\subset c$ corresponding to the trivial character of $\G^{*}$. The assignment $(d,\{\a_{\chi}\}_{\chi\in\G^{*}})\mapsto c$ gives a functor $\cD^{\G^{*}}\to \cC$. These two functors are inverse to each other and give an equivalence $\cC\cong \cD^{\G^{*}}$.

\ssec{isogen}{Monodromic categories for isogenous groups} We first treat a more general situation and then see how we can apply it to the non-adjoint group case. 

Consider a short exact sequence of real algebraic groups 
$$1\xrightarrow{} \Gam\xrightarrow{} G_1\xrightarrow{p} G_2\xrightarrow{}1$$ 
with $\Gam$ being finite. Fix a character $\omega\in \Gam(\BR)^*$ and choose an extension $\wt\omega$ to $\Gam$. Let $\pi_i\colon \wt X_i\to X_i$ be the respective flag varieties and let $\bT_i$ be the respective abstract Cartan groups. Finally, put $\grS$ for the dual group of the cokernel of the map $p(\BR)\colon G_1(\BR)\to G_2(\BR)$. 

We have the canonical equivalence of categories given by the forgetful functor from right to left
\begin{equation}\label{wt om}
\CM^{\omega}_{G_{\RR}}=\wh D^{b}_{G_1(\BR)(\G(\RR),\omega)}(\wt X_1)_{\bT^c_1\mon} \simeq \wh D^{b}_{G_1(\BR)(\Gam,\wt\omega)}(\wt X_1)_{\bT^c_1\mon}.
\end{equation}
Here the centrally twisted categories are defined in \refss{cent}.

Let $\cL=\cL_{\wt\omega}$ be the rank $1$ local system on $\bT^c_2$ defined by the finite isogeny $p|_{\bT_{1}^{c}}:\bT^{c}_{1}\to \bT_{2}^{c}$ and the character $\wt\omega$ on $\G=\ker(p|_{\bT_{1}^{c}})$. Note that $\cL$ becomes trivial after pulling back to $\bT^c_1$. 

\lem{} Pullback along $p$ induces an equivalence
\begin{equation}\label{pG1}
\wh D^{b}_{p(G_1(\BR))}(\wt X_2)_{(\bT^c_2,\cL)\mon}\simeq\wh D^{b}_{G_1(\BR)(\Gam,\wt\omega)}(\wt X_1)_{\bT^c_1\mon}.
\end{equation}
\elem
\prf
It suffices to check the same statement before completion. By definition, $D^{b}_{G_1(\BR)(\Gam,\wt\omega)}(\wt X_1)_{\bT^c_1\mon}\subset D^{b}_{G_{1}(\RR)\G}(\wt X_{1})$ is the full subcategory generated by pullbacks of objects in $D^{b}_{G_1(\BR)(\Gam,\wt\omega)}(X_1)$. On the other hand, pullback along $p$ identifies $D^{b}_{p(G_1(\BR))}(\wt X_2)_{(\bT^c_2,\cL)\mon}$ with the full subcategory of $D^{b}_{G_{1}(\RR)\G}(\wt X_{1})$ generated by pullbacks from $D^{b}_{p(G_{1}(\RR))}(X_{2})_{\cL}$. It therefore suffices to identify simple perverse sheaves in $D^{b}_{p(G_{1}(\RR))}(X_{2})_{\cL}$ and in $D^{b}_{G_1(\BR)(\Gam,\wt\omega)}(X_1)$. 

Note that $X_{1}=X_{2}$. A $G_{1}(\RR)$-orbit $O^{\RR}_{\l}\subset X_{1}$ is the same as a $p(G_{1}(\RR))$-orbit on $X_{2}$. For such an orbit, the irreducible local systems in both $D^{b}_{G_1(\BR)(\Gam,\wt\omega)}(O^{\RR}_{\l})$ and in $D^{b}_{p(G_{1}(\RR))}(O^{\RR}_{\l})_{\cL}$ are parametrized by characters of $\pi_{0}(\bT_{1,\l}^{c})$ whose pullback to $\G$ is $\wt\omega$ (compare \eqref{LS orbit tw} and \eqref{LS orbit cen}). This finishes the proof.
\epr

Note that $G_2(\BR)/p(G_1(\BR))\cong\grS^*$ and, hence, there is an action of $\grS^*$ on $\wh D^{b}_{p(G_1(\BR))}(\wt X_2)_{(\bT^c_2,\cL)\mon}$, such that 
$$
(\wh D^{b}_{p(G_1(\BR))}(\wt X_2)_{(\bT^c_2,\cL)\mon})^{\grS^*}\simeq\wh D^{b}_{G_2(\BR)}(\wt X_2)_{(\bT^c_2,\cL)\mon}=\CM_{G_{2,\RR},\cL}.
$$
Applying the observations of \refss{deeq} to this action, we have an action of $\grS$ on $\CM_{G_{2,\RR}, \cL}$ and a canonical equivalence
\begin{equation}\label{grSeq}
\wh D^{b}_{p(G_1(\BR))}(\wt X_2)_{(\bT^c_2,\cL)\mon}\simeq(\CM_{G_{2,\RR}, \cL})^\grS.
\end{equation}

Combining \eqref{wt om}, \eqref{pG1} and \eqref{grSeq}, we have verified that
\begin{equation}\label{G1G2S}
\CM^{\omega}_{G_{1,\RR}}\simeq(\CM_{G_{2,\RR},\cL})^\grS.
\end{equation}
We remark that this equivalence depends on the choice of the extension $\wt\omega$ of $\omega$ to $\G$ (which determines the local system $\cL$ on $\bT^{c}_{2}$).

 \ssec{nonadj}{Structure theorem for general quasi-split groups}

Now let $G$ be a connected reductive algebraic group over $\CC$ with real form $G_{\RR}$. Let $X, \wt X$ be defined in terms of $G$. Let $G^{ad}$ be the adjoint form of $G$, which carries a real form $G_{\RR}^{ad}$ compatible with $G_{\RR}$. Let $X^{ad}, \wt X^{ad}$ be defined as in \refss{set} in terms of $G^{ad}$. In particular, $\pi^{ad}: \wt X^{ad}\to X^{ad}$ is a $(\bT^{ad})^{c}$-torsor.  

Let $T'=G/G^{der}$ be the abelianization of $G$. We apply the result of \refss{isogen} to the short exact sequence $$1\xrightarrow{} \Gam\xrightarrow{} G\xrightarrow{p}  G^{ad}\x T'\xrightarrow{}1,$$ 
with $\Gam=Z(G^{der)}$.  Since
$$\CM_{G_{\RR}}=\bigoplus_{\omega: \G(\RR)\to \bfk^{\times}}\CM^{\omega}_{G_{\RR}},$$
we will consider one summand $\CM^{\omega}_{G_{\RR}}$ at a time. So we fix a character $\omega: \G(\RR)\to \bfk^{\times}$, and choose an extension of $\omega$ to $\wt\omega:\G\to \bfk^{\x}$. The map $p$ induces a short exact sequence $1\to \G\to \bT\to \bT^{\ad}\times T'\to 1$.  Using this and $\wt\omega$ we get a rank one local system $\cL\bt\cL'$ on $(\bT^{\ad})^{c}\times T'^{c}$ as in \refss{tw Hk}. By \eqref{G1G2S} we get
an equivalence 
\eq{GtoGad}
\CM_{G_{\RR}}^{\omega} \simeq(\CM_{G^{ad}_{\RR}\times T'_{\RR}, \cL\bt\cL'})^\grS,
\eeq
where $\grS:=(G^{ad}(\RR)\x T'(\RR)/\im(G(\RR)))^*$.

The category on the right side is the completed tensor product of $\CM_{G^{ad}_{\RR}, \cL}$ and $\CM_{T'_{\RR}, \cL'}=\wh D^{b}_{T'(\BR)}(T'^{c})_{(T'^{c},\cL')\mon}$. The latter category is zero if $\cL'|_{T'(\RR)}$ is a nontrivial local system. If this happens then $\wh D^{b}_{G(\RR)(\G(\RR),\omega)}(\wt X)_{\bT^c\mon}$ and there is nothing to prove. If $\cL'|_{T'(\RR)}$ is a trivial local system, then irreducible local systems that appear in $\CM_{T'_{\RR}, \cL'}$ are parametrized by a torsor $\Xi_{\cL'}$ of $\pi_{0}(T'(\RR))^{*}$. By \refc{comp orbit}, we get a decomposition
\begin{equation*}
\CM_{T'_{\RR}, \cL'}\simeq\bigoplus_{\chi\in \Xi_{\cL'}}D^{f}(\rmod\ol\cR')
\end{equation*}
where $\ol\cR'$ is the completed group ring of $\pi_{1}(\ol{T}'^{c})$, and $\ol{T}'^{c}:=T'^{c}/\mathrm{Im}(T'(\RR))$.

By definition, $\grS^{*}$ is a quotient of $\pi_{0}(G^{ad}(\RR))\times \pi_{0}(T'(\RR))$. In particular, we have a homomorphism $\pi_{0}(T'(\RR))\to \grS^{*}$, hence dually $\grS\to \pi_{0}(T'(\RR))^{*}$. Therefore $\grS$ acts on $\Xi_{\cL'}$ via the map to $\pi_{0}(T'(\RR))^{*}$. We thus get an $\grS$-equivariant equivalence
\begin{equation}\label{decomp MGadT'}
\CM_{G^{ad}_{\RR}\times T'_{\RR}, \cL\bt\cL'}\simeq\bigoplus_{\chi\in \Xi_{\cL'}}\CM_{G^{\ad}_{\RR}, \cL}\wh\ot_{\bfk}\ol\cR'
\end{equation}
where $\grS$ permutes the summands according to its action on $\Xi_{\cL'}$. 



By \refc{maincor} or rather its generalization \reft{twisted main}, the enhanced real Soergel functor
\begin{equation*}
\VV_{\RR}^{ad, \sh}: \Tilt(\CM_{G^{ad}_{\RR}, \cL})\to \rmod\CB^{ad}_{\cL}
\end{equation*}
is fully-faithful. Here $\cB^{ad}_{\cL}=\cB_{0,\cL}^{ad}\ot_{\cR^{ad, \bW}}\cR^{ad}$ is the algebra $\cB$ for the group $G^{ad}_{\RR}$ and the twisting $\cL$. For a point $x$ in the closed orbit $\cO^{\RR}_{\l_{0}}$ of $X^{ad}$ we have a map $\pi_{0}((\bT^{\ad}_{\l_{0}})^{c})\to \pi_{0}(G^{ad}(\RR))\to \grS^{*}$. Dually, $\grS$ acts on the set $\LS^{ad}_{\l_{0},\cL}$ of $G^{ad}(\RR)$-equivariant irreducible local systems on $\cO^{\RR}_{\l_{0}}$ that appear in $\CM_{G_{\RR}^{ad}, \cL}$. This yields a $\grS$-action on $\CB^{ad}_{0,\cL}$ and on $\CB^{ad}_{\cL}$, such that the functor $\VV_\BR^{ad,\sh}$ respects the $\grS$ action on $\CM_{G^{ad}_{\RR}}$ and on $\rmod\CB^{ad}_{\cL}$. 

We conclude that:


\th{mainnonadj} Consider the functor \footnote{The notation emphasizes the dependence on the extension $\wt\ome$ of $\ome$, which defines $\cL$.}
$$\VV^{\sh,\wt\ome}_\BR\colon\Tilt(\CM^{\omega}_{G_{\RR}})\xrightarrow{}\left(\bigoplus_{\chi\in\Xi_{\cL'}}\rmod\CB^{ad}_{\cL}\wh\ten\ol\CR'\right)^{\grS},$$ obtained by composing the equivalences \refe{GtoGad}, \eqref{decomp MGadT'} and taking $\VV^{ad,\sh}_\BR\ten\id_{\ol\CR'}$ on each summand $\CM_{G^{ad}_{\RR}, \cL}\wh\ot\ol\cR'$. Here $\grS$ acts simultaneously on $\CB^{ad}_{\cL}$ as explained above and permuting the summands via its action on $\Xi_{\cL'}$.

Then $\VV^{\sh,\wt\ome}_\BR$ is fully-faithful. 

In particular, if $G$ is semisimple (so that $\G=Z(G)$ and $T'=1$), then the functor 
\begin{equation}\label{VR ss}
\VV^{\sh,\wt\ome}_\BR\colon\Tilt(\CM^{\omega}_{G_{\RR}})\to(\rmod\CB^{ad}_{\cL})^{\grS}
\end{equation}
is fully faithful. 
\eth

\ssec{}{Relation between $\VV^{\sh,\ome}_{\RR}$ and $\VV^{ad,\sh}_{\RR}$} In this subsection we assume $G_{\RR}$ is quasi-split and semisimple. In this case we want to describe $\VV^\sh_\BR$ more directly in terms of the vanishing cycles functors for the conormals to the closed $G(\RR)$-orbit on $X$. Recall a choice of the generic conormal $\xi$ to the open orbit and its lift $\wt\xi$ made in \refss{GVC} in order to define $\VV^{ad}_\BR=\VV^{ad}_{\BR,\wt\xi}$. For an element $\sig\in\grS^*$ pick its lift to $G^{ad}(\RR)$ and let $\sig(\wt\xi)$ be the image of $\wt\xi$ under the adjoint action of this lift. 
There is a natural $\grS^*$-action on the direct sum of functors $\bigoplus_{\sig\in\grS^*}\VV^{\sh,\ome}_{\BR,\sig(\wt\xi)}$ on $\Tilt(\CM^{\omega}_{G_{\RR}})$ and each of the summands lands in $\rmod\CB_\ome$ (see \eqref{VR ome}).

Recall the descriptions of $\LS^{ad}_{\l_{0},\cL}$ and $\LS^{\ome}_{\l_{0}}$ in \eqref{LS orbit tw} and \eqref{LS orbit cen}. By \refl{TRT}, $\bT$ carries a canonical real form $\bT^{\s_{\l_{0}}}$ such that $\bT_{\l_{0}}^{c}=\textup{Im}(\bT^{\s_{\l_{0}}}\to \bT^{c})$. We simply denote the real points of this real form by $\bT(\RR)$, keeping in mind that it is the real form attached to the closed orbit. Similarly we define $\bT^{ad}(\RR)$ from the real form of $\bT^{ad}$ attached to the closed orbit.

Let $\bT(\RR)'\subset \bT$ be the preimage of $\bT^{ad}(\RR)$ under the projection $\nu: \bT\to \bT^{ad}$. Then we have a short exact sequence 
$$1\to ZG\to \bT(\RR)'\to \bT^{ad}(\RR)\to 1.$$
Then $\LS^{ad}_{\l_{0},\cL}$ is the subset of $\pi_{0}(\bT(\RR)')^{*}$ of characters whose restriction to the image of $ZG$ is $\wt\ome$. On the other hand, $\LS^{\ome}_{\l_{0}}$ is the subset of $\pi_{0}(\bT(\RR))^{*}$ of characters whose restriction to the  image of $ZG(\RR)$ is $\ome$. We have a  short exact sequence 
\begin{equation}\label{TR'}
1\to \bT(\RR)\to \bT(\RR)'\to \grS^{*}\to 1,
\end{equation}
and $ZG\cap \bT(\RR)=ZG(\RR)$. Therefore we have a restriction map
\begin{equation}\label{LSadL to LSome}
\r_{\l_{0}}: \LS^{ad}_{\l_{0},\cL}\to \LS^{\ome}_{\l_{0}}
\end{equation}
that is surjective. When $\LS^{\ome}_{\l_{0}}\ne\varnothing$, $\LS^{ad}_{\l_{0},\cL}$ is a torsor under $\pi_{0}(\bT^{ad}(\RR))^{*}$ while $\LS^{\ome}_{\l_{0}}$ is a torsor under $\pi_{0}(\bT(\RR)/ZG(\RR))^{*}$, and the map $\r_{\l_{0}}$ is compatible with their torsor structures via the surjection $\pi_{0}(\bT^{ad}(\RR))^{*}\to \pi_{0}(\bT(\RR)/ZG(\RR))^{*}$ whose kernel is $\grS$ by \eqref{TR'}. This means the map $\r_{\l_{0}}$ is a $\grS$-torsor. Comparing the definitions of $\cB^{ad}_{\cL}$ and $\cB^{\ome}$, we conclude that
$$(\CB^{ad}_\CL)^\grS\cong\CB^{\ome}.$$
Using this isomorphism we get a functor
\begin{equation}\label{BadBom}
(\rmod\CB^{ad}_\CL)^{\grS}\to \bigoplus_{\sig\in\grS^*}\rmod\CB^{\ome}
\end{equation}
sending a $\grS$-equivariant $\CB^{ad}_{\cL}$-module $M$ to $\op_{\s\in \grS^{*}}M(\s)$, where $M(\s)$ is the $\s$-isotypical summand of $M$ which is a module over $(\CB^{ad}_\CL)^{\grS}\cong \CB^{\ome}$.
 
\prop{}
The following diagram is commutative
\begin{equation*}
\xymatrix{ \Tilt(\CM^{\omega}_{G_{\RR}}) \ar[rrr]^-{\bigoplus_{\sig\in\grS^*}\VV^{\sh,\ome}_{\BR,{\sig(\wt\xi)}}} &&&  \bigoplus_{\sig\in\grS^*}\rmod\CB^{\ome} \\
\Tilt(\CM_{G^{ad}_{\RR},\CL})^\grS\ar[rrr]^-{\VV^{ad,\sh}_\BR}\ar[u]^{\sim}_{\refe{GtoGad}} &&&  (\rmod\CB^{ad}_\CL)^{\grS},\ar[u]^{\eqref{BadBom}}}
\end{equation*}
\eprop

\prf
The diagram is compatible with the $\grS^*$-actions, therefore it is sufficient to verify the commutativity after passing to the $\grS^*$-equivariant objects. It then becomes the diagram
\[
\xymatrix{ \Tilt(\CM^{\omega}_{G_{\RR}})^{\grS^*} \ar[rrr]^-{\VV^{\sh,\ome}_{\BR}} &&&  \rmod\CB^{\ome}\\
\Tilt(\CM_{G^{ad}_{\RR},\CL})\ar[rrr]^-{\VV^{ad,\sh}_\BR}\ar[u]^{\sim} &&&  \rmod\CB^{ad}_\CL,\ar[u]}
\]
which is commutative tautologically, as these are the same vanishing cycles functors.




\epr

\sec{KD}{Koszul duality for quasi-split groups}
In this section we reprove the main result of \cite{BV}, Koszul duality for quasi-split real groups, using a more geometric argument based on our study of tilting sheaves on $G_{\RR}$-orbits.

\ssec{}{Vogan duality} Recall from \cite[Theorem 8.5]{V2} that the blocks of $G(\BR)$-representations are in bijection with $\bW_{\lambda_0}$-orbits on $\LS_{\lambda_0}$ under the cross action. Let $a\in \LS_{\lambda_0}/\bW_{\l_{0}}$. Then the corresponding block (direct summand) $\CM_{G_\BR}^a\subset \CM_{G_{\RR}}$ is generated (as a triangulated category) by summands of $\wt\D_{\l_0,\chi}\star \cK$ for $\chi\in a$ and $\cK\in \CH_{G}$.  Note that all objects in $\CM_{G_\BR}^a$ have the same central character, say $\ome: ZG(\RR)\to \bfk^{\times}$, therefore $\CM_{G_\BR}^a$ is a direct summand of $\CM_{G_\BR}^{\ome}$. 

Now we pass to the Langlands dual group $\dG$. Denote its flag variety by $\dX$. For a symmetric subgroup $\dK\subset \dG$, consider the equivariant derived category $\dcD_{\dK}:=D^{b}_{\dK}(\dX)$. We have the equivariant Hecke category 
\begin{equation*}
\dcH_{\dG}:=D^{b}_{\dG}(\dX\times \dX)
\end{equation*}
that acts on $\dcD_{\dK}$ by convolution on the right. 

Let $O_{\dl_{0}}^{\dK}\subset \dX$ be the open $\dK$-orbit. There is also the notion of blocks for $\dcD_{\dK}$, and the blocks are parametrized by $\bW_{\dl_{0}}$-orbits on the set $\LS_{\dl_{0}}$ of isomorphism classes of $\dK$-equivariant irreducible local systems on $O^{\dK}_{\dl_{0}}$ under the cross action.  For $\dua\in \LS_{\dl_{0}}/\bW_{\dl_{0}}$, denote the corresponding block (direct summand) by $\dcD^{\dua}_{\dK}\subset \dcD_{\dK}$; it is generated by direct summands of $\IC(\cL_{\dl_{0},\chi})\star \cK$ for $\chi\in\dua$ and $\cK\in\dcH_{\dG}$.

To the pair $(G_{\RR},a)$, where $G_{\RR}$ is a real form of $G$ and $a\in \LS_{\l_{0}}/\bW_{\l_{0}}$, Vogan's duality of \cite{V2} associates a pair $(\dK,\dua)$ consisting of a symmetric subgroup $\dK\subset\dG$ and $\dua\in \LS_{\dl_{0}}/\bW_{\dl_{0}}$. We say the corresponding blocks $\CM^{a}_{G_{\RR}}$ and $\dcD^{\dua}_{\dK}$ are in {\em Vogan duality}. 

When $G_\BR$ is quasi-split, the block $\dcD^{\dua}_{\dK}$ is what is called {\itshape principal}, i.e., $\dua$ contains the trivial local system, or equivalently, $\dcD^{\dua}_{\dK}$ contains the constant sheaf on $\dX$. In this case, it is proved in \cite{V2} and \cite[Proposition 3.6]{BV} that $\dcD^{\dua}_{\dK}$ is generated as a triangulated category by $\IC_{\dl}$ for {\em closed} orbits $O^{\dK}_{\dl}\subset \dX$ under the $\dcH_{\dG}$-action and taking direct summands.

\ssec{}{Koszul duality--additive category version} We continue with the notations above. Fix a block $\CM^{a}_{G_{\RR}}$ for $\CM_{G_{\RR}}$ and its Vogan dual (principal) block $\dcD^{\dua}_{\dK}$ for $\dcD_{\dK}$. 

Let $\dT$ be the abstract Cartan for $\dG$. Let 
$$\cR^{\bu}=H^{*}_{\check{\bT}}(\pt,\bfk)\cong \Sym(\XX_{*}(\dT)\ot\bfk[-2]),$$ 
viewed as a formal dg algebra in even non-negative degrees. 
We have the ring homomorphism
\begin{equation*}
\cR^{\bu}=\Sym(\XX_{*}(\bT)\ot \bfk)\to \cR=\wh{\bfk[\XX_{*}(\bT)]}
\end{equation*}
sending $x\in \XX_{*}(\bT)$ to $\log(x)=\sum_{n\ge1}(-1)^{n-1}\frac{(x-1)^{n}}{n}\in \cR$. This identifies $\cR$ with the completion of $\cR^{\bu}$ along the ideal $\cR^{>0}$, i.e., $\cR\cong\prod_{n\ge0}\cR^{n}$. 

The category $\dcD_{\dK}$ is enriched over (right) graded $\cR^{\bu}$-modules in the following way: for $\cF_{1},\cF_{2}\in \dcD_{\dK}$, the Ext group
\begin{equation*}
\Ext^{\bu}(\cF_{1},\cF_{2}):=\bigoplus_{n\in \ZZ}\Ext^{n}_{\dK\bs \dX}(\cF_{1},\cF_{2})
\end{equation*}
is equipped with a graded action of $H^{*}_{\dK}(\dX,\bfk)$, hence a graded module over $\cR^{\bu}$ via the map $\cR^{\bu}\to H^{*}_{\dK}(\dX,\bfk)$ obtained as the pullback along $\dK\bs \dX\to \pt/\dB\to \pt/\check{\bT}$. 

Let $\SSC(\dcD^{\dua}_{\dK})$ be the full subcategory of $\dcD^{\dua}_{\dK}$ consisting of semisimple complexes, i.e., direct sums of shifts of simple perverse sheaves. This is an additive category equipped with the cohomological shift endo-functor $[1]$. We denote this shift functor by $\{1\}$, and call it the {\em internal} degree shift functor on $\SSC(\dcD^{\dua}_{\dK})$. For $\cF_{1},\cF_{2}\in\SSC(\dcD^{\dua}_{\dK})$, we use the notation
\begin{equation*}
\Hom^{\bu}(\cF_{1},\cF_{2}):=\op_{m\in\Z}\Hom(\cF_{1},\cF_{2}\{m\})
\end{equation*}
instead of $\Ext^{\bu}(\cF_{1},\cF_{2})$; it is a graded $\cR^{\bu}$-module.

The following can be deduced from the equivariant-monodromic equivalence \cite[Theorem 5.2.1]{BY}. There is a canonical monoidal functor
\begin{equation}\label{KD Hk}
\Phi_{G}: \SSC(\dcH_{\dG})\to \Tilt(\cH_{G})
\end{equation}
that sends the simple perverse sheaf $\IC_{w}$ to the \fmo tilting sheaf $\cT_{w}$ for any $w\in W$. Indeed, let $\SBim^{gr}$ be the category of graded Soergel $\cR^{\bu}$-bimodules. Then Soergel's classical result gives a monoidal equivalence $\SSC(\dcH_{\dG})\cong \SBim^{gr}$ (as categories enriched in $\cR^{\bu}\bimod$). Now $\Tilt(\cH_{G})\cong \SBim$, then functor $\Phi_{G}$ corresponds to the monoidal functor
\begin{equation*}
(-)\ot_{\cR^{\bu}}\cR\cong \cR\ot_{\cR^{\bu}}(-): \SBim^{gr}\to \SBim.
\end{equation*}




The following is a reformulation of the main result of \cite{BV}. 
It essentially states that $\SSC(\dcD^{\dua}_{\dK})$ can be viewed as a graded version of $\Tilt(\CM^{a}_{G_\BR})$. We will give a geometric proof using our results on $\TGR$ and the results on $\dcD_{\dK}$ proved in \cite{BV}. 

\th{KD}(cf. \cite{BV})
Fix a regular covector $\xi\in i\cN^{*}(\RR)\cap\cO_{\reg}$ and its lift $\wt\xi$ as in \refss{GVC}, and fix an extension $\wt\ome$ of $\ome$ (the central character of the block $a$) to $ZG$. With these choices, there is a canonical functor
\begin{equation*}
\ka: \SSC(\dcD^{\dua}_{\dK})\to \Tilt(\CM^{a}_{G_\BR})\end{equation*}
satisfying the following properties:
\begin{enumerate}
\item\label{KD shift inv} $\ka$ is invariant under the internal shift $\{1\}$ on $\SSC(\dcD^{\dua}_{\dK})$.

\item\label{KD Hom} For $\cF_{1},\cF_{2}\in \SSC(\dcD^{\dua}_{\dK})$, the natural maps $\Hom(\cF_{1},\cF_{2}\{n\})\to \Hom(\ka(\cF_{1}), \ka(\cF_{2}))$ induce an isomorphism of $\cR$-modules
\begin{equation*}
\Hom^{\bu}(\cF_{1},\cF_{2})\ot_{\cR^{\bu}}\cR\isom \Hom(\ka(\cF_{1}), \ka(\cF_{2})).
\end{equation*}
\item\label{KD Hk eq} $\ka$ intertwines the actions of $\SSC(\dcH_{\dG})$ and $\Tilt(\cH_{G})$ on both sides under the equivariant-monodromic functor \eqref{KD Hk}.
\item\label{KD IC to tilting} $\ka$ is essentially surjective, and induces a bijection
\begin{eqnarray*}
|\ka|&:&\{\mbox{simple perverse sheaves in $\dcD^{\dua}_{\dK}$}\}/\\
&\bij &\{\mbox{indecomp. free mono. tilting sheaves in $\CM^{a}_{G_{\RR}}$}\}/\cong.
\end{eqnarray*}
\end{enumerate}
\eth
\prf

Let $\ome:ZG(\RR)\to \bfk^{\times}$ be the restriction of any $\chi\in a$ to the real center $ZG(\RR)$. Let $\wt\ome:ZG(\CC)\to \bfk^{\times}$ be an extension of $\ome$. Recall the $\bW_{\l_{0}}$-equivariant map $\r_{\l_{0}}: \LS^{ad}_{\l_{0},\cL}\to \LS^{\ome}_{\l_{0}}$ that is a $\grS$-torsor when both are non-empty. Let $\wt a=\r^{-1}_{\l_{0}}(a)$, with the action of $\grS$. 

Let $\cB^{ad,a}_{0,\cL}=(\prod_{\chi\in \wt a}\CS)^{\bW_{\l_{0}}}$ and $\cB^{ad,a}_{\cL}=\cB^{ad, a}_{0,\cL}\ot_{\cR^{\bW}}\cR$. This is a direct factor of $\cB^{ad}_{\cL}$ that is stable under the action of $\grS$. The $\SBim$ action on $\rmod\cB^{ad}_{\cL}$ restricts to an action on $\rmod\cB^{ad,a}_{\cL}$ commuting with the $\grS$-action. The fully faithful functor \eqref{VR ss} restricts to full embedding
\begin{equation}\label{VR wta}
\VV^{\sh, a}_\BR\colon\Tilt(\CM^{a}_{G_{\RR}})\to(\rmod\CB^{ad,  a}_{\cL})^{\grS}.
\end{equation}
For $\chi\in a$, denote by $\CS_\chi=\CS$ the $\CB^{ad,a}_{\cL}$-module on which $\CB^{ad,a}_{0,\cL}$ acts via the $\chi$-factor $\CS$ and $\CR$ acts via the projection to $\CS$. Denote by $\BB^{ad,a}_{\cL}$ be the full subcategory of $\rmod\CB^{ad,a}_{\cL}$ generated by $\{\CS_\chi\}_{\chi\in \wt a}$ under the action of Soergel bimodules $\SBim$ and taking direct summands. Then $\grS$ acts on $\BB^{ad,a}_{\cL}$. 
The full embedding \eqref{VR wta} induces an equivalence of $\cR$-linear additive categories
\begin{equation}\label{VR a}
\VV^{a}_{\cL}: \Tilt(\CM^{a}_{G_{\RR}})\simeq (\BB^{ad,a}_{\cL})^{\grS}.
\end{equation}
As indicated in the notation, the functor $\VV^{a}_{\cL}$ depends on the choice of $\wt\ome$ (same datum as $\cL$) extending $\ome$.

Now consider the dual side.  Our category $\SSC(\dcD^{\dua}_{\dK})$ is what's denoted by $\CS^{gr}$ in \cite[3.1]{BV}. In \cite{BV}, the authors introduced the counterpart $(\dcB^{sc, \dua,\bu}, \dgS)$ of $(\cB^{ad, a}, \grS)$ on the equivariant side. Here $\dcB^{sc,\dua,\bu}=H^{*}_{\dK^{\c}}(\dX,\bfk)$ (where $\dK^{\c}$ is the neutral component of $\dK$) as graded $\cR^{\bu}$-algebras, and $\dgS=\pi_{0}(\dK)$. The pair $(\dcB^{sc, \dua,\bu}, \dgS)$ is denoted by $(\grB,\grS)$ in \cite[3.9]{BV}. Note that
\begin{equation*}
\dcB^{sc, \dua,\bu}\cong \dcB^{sc,\dua,\bu}_{0}\ot_{(\cR^{\bu})^{\bW}}\cR^{\bu}, \quad \mbox{ where }\dcB^{sc,\dua,\bu}_{0}=H^{*}_{\dK^{\c}}(\pt,\bfk).
\end{equation*}

Let $(\dK^{\c}\bs \dX)_{cl}$ be the set of closed $\dK^{\c}$-orbits on $\dX$. For $\dl\in (\dK^{\c}\bs \dX)_{cl}$, denote the corresponding $\dK^{\c}$-orbit by $O^{\dK^{\c}}_{\dl}$, then we have
\begin{equation*}
\cS^{\bu}_{\dl}:=H^{*}_{\dK^{\c}}(O^{\dK^{\c}}_{\dl}, \bfk)\in \grrmod\dcB^{sc,\dua,\bu}.
\end{equation*}
Let $\dBB^{sc,\dua,gr}$ be the full subcategory of graded right $\dcB^{sc,\dua,\bu}$-modules generated by $\{\CS^{\bu}_{\dl}\}$ ($\dl\in (\dK^{\c}\bs \dX)_{cl}$) under the action of $\SBim^{gr}$ and taking direct summands. Note that $\dgS$ acts on $\dBB^{sc,\dua,gr}$. Then \cite[Theorem 3.12]{BV} gives an equivalence of $\grrmod\cR^{\bu}$-linear categories
\begin{equation*}
\HH^{\dua}: \SSC(\dcD^{\dua}_{\dK})\simeq (\dBB^{sc,\dua,gr})^{\dgS}.
\end{equation*}

By \cite[Theorem 2.6]{BV}, there is a canonical isomorphism $\dgS\cong \grS$ and an $\cR$-algebra isomorphism
\begin{equation}\label{cBgr}
\dcB^{sc,\dua}\ot_{\cR^{\bu}}\cR\isom \cB^{ad,a}_{\cL}
\end{equation}
equivariant under the respective actions of $\dgS$ and $\grS$. It induces a functor
\begin{equation*}
c':\grrmod\dcB^{sc,\dua}\to \rmod\cB^{ad,a}_{\cL}
\end{equation*}
sending a graded $\dcB^{sc,\dua}$-module $M=\op_{n} M^{n}$ (where $M^{n}=0$ for $n\ll0$) to its completion $\wh M=\prod_{n}M^{n}$, which is a $\cB^{ad,a}_{\cL}$-module via \eqref{cBgr}. 

We claim that $c'$ sends $\dBB^{sc,\dua,gr}$ to $\BB^{ad,a}_{\cL}$. To check this we may assume $G$ is adjoint. Then as mentioned in the proof of \cite[Theorem 2.4]{BV}, there is a canonical bijection
\begin{equation}\label{bij cl a}
(\dK\bs\dX)_{cl}\bij a.
\end{equation}
Moreover, if $\dl\in (\dK\bs\dX)_{cl}$ corresponds to $\chi\in a$, then $\CS^{\bu}_{\dl}\ot_{\cR^{\bu}}\cR\cong \CS_{\chi}$ as $\cB^{a}$-modules. Since $\dBB^{sc,\dua,gr}$ (resp. $\BB^{ad,a}_{\cL}$) is generated by $\CS^{\bu}_{\dl}$ (resp. $\CS_{\chi}$) under the $\SBim^{gr}$ (resp. $\SBim$) action and taking direct summands, we conclude that $c'$ sends $\dBB^{sc,\dua,gr}$ to $\BB^{ad,a}_{\cL}$.

It is easy to see that $c'$ is equivariant under $\dgS\cong\grS$. 
Therefore, $c'$ induces a functor
\begin{equation*}
c: (\dBB^{sc,\dua,gr})^{\dgS}\to (\BB^{ad,a}_{\cL})^{\grS}.
\end{equation*}
We then define $\ka$ to be the composition
\begin{equation*}
\xymatrix{\SSC(\dcD^{\dua}_{\dK})\ar[r]^-{\HH^{\dua}}_-{\sim} & (\dBB^{sc,\dua,gr})^{\dgS} \ar[r]^-{c} & (\BB^{ad,a}_{\cL})^{\grS}\ar[r]^-{(\VV^{a}_{\cL})^{-1}}_-{\sim} & \Tilt(\CM^{a}_{G_{\RR}}).
}
\end{equation*}
 
Now we check the required properties of $\ka$. Properties \eqref{KD shift inv} and \eqref{KD Hom} follows easily from the similar properties of the completion functor $c$. Property \eqref{KD Hk eq} follows from the fact that $c$ intertwines the $\SBim^{gr}$-action on $(\dBB^{sc,\dua, gr})^{\dgS}$ with the $\SBim$-action on $(\BB^{ad,a}_{\cL})^{\grS}$ via the completion functor $\SBim^{gr}\to \SBim$.

For \eqref{KD IC to tilting}, we first show that if $\cF\in \dcD^{\dua}_{\dK}$ is a simple perverse sheaf, then $\ka(\cF)$ is indecomposable. Indeed, $\Hom^{\bu}(\cF,\cF)$ is concentrated in non-negative degrees with degree zero part being $\bfk$. Therefore, by \eqref{KD Hom}, $\End(\ka(\cF))\cong \prod_{n\ge0}\Ext^{n}(\cF,\cF)$, hence is a local ring with residue field $\bfk$. This implies $\ka(\cF)$ is indecomposable.  Therefore the map $|\ka|$ in property  \eqref{KD IC to tilting} is defined.

We show $|\ka|$ is injective. If $\cF_{1},\cF_{2}\in \dcD^{\dua}_{\dK}$ are simple perverse sheaves and $\a: \ka(\cF_{1})\isom \ka(\cF_{2})$ is an isomorphism with inverse $\b$. Using that $\Hom(\ka(\cF_{1}), \ka(\cF_{2}))\cong \prod_{n\ge0}\Ext^{n}(\cF_{1},\cF_{2})$, let $\a_{0}$ and $\b_{0}$ be the degree zero projections of $\a$ and $\b$. Then $\a_{0}$ and $\b_{0}$ define inverse isomorphisms between $\cF_{1}$ and $\cF_{2}$. This proves that $|\ka|$ is injective.

To show $|\ka|$ is surjective, it suffices to consider the case $G$ is adjoint hence $\dG$ is simply-connected; the case of general semisimple $G$ follows by quotienting by $\dgS\cong\grS$ on both sides of $|\ka|$. Now assume $G$ is adjoint. Let $\cT$ be an indecomposable \fmo tilting sheaf in $\CM^{a}_{G_{\RR}}$. By \refp{gen}, $\cT$ is a direct summand of $\CS_{\chi}\star \cK$ for some $\chi\in a$ and $\cK\in \Tilt(\CH_{G})$. Let $\dl\in (\dK\bs\dX)_{cl}$ correspond to $\chi\in a$ under \eqref{bij cl a}. Let $\dcK\in \SSC(\dcH_{\dG})\cong \SBim^{gr}$ be an object that maps to $\cK\in \Tilt(\CH_{G})\cong \SBim$ under $\Phi_{G}$.  Let $\CS^{\bu}_{\dl}\star \dcK=\op_{i=1}^{N}\cF_{i}$ be a decomposition into shifted simple perverse sheaves. We have proved that each $\ka(\cF_{i})$ is an indecomposable \fmo tilting sheaf. Since $\op_{i=1}^{N}\ka(\cF_{i})\cong \ka(\CS^{\bu}_{\dl}\star \dcK)\cong \CS_{\chi}\star \cK$, and the latter contains $\cT$ as an indecomposable summand, hence by \refc{tilt GR}(3), $\cT\cong \ka(\cF_{i})$ for some $i$. This proves that $|\ka|$ is surjective.  We have finished checking property \eqref{KD IC to tilting}.


%
%
%
%
%

\epr

Recall the function $\ell: I\to \Z_{\ge0}$ from \refd{length}. For $n\in\Z_{\ge0}$, let $X_{\le n}$ be the union of $O^{\RR}_{\l}$ such that $\ell(\l)\le n$. Since the similar union $\cup_{\ell(\l)\le n}O^{K}_{\l}\subset X$ is open by \refl{length codim}, $X_{\le n}$ is closed in $X$. Let $\CM^{a}_{G_{\RR}, \le n}\subset \CM^{a}_{G_{\RR}}$ be the full subcategory of complexes supported on $X_{\le n}$.

On the dual side, for the principal block $\dcD_{\dK}^{\dua}$, let $\wt{\dI^{\dua}}$ be the set of isomorphisms classes of simple perverse sheaves in $\dcD_{\dK}^{\dua}$. Let $\dI$ be the set of $\dK$-orbits on $\dX$. Then we have the map $\wt{\dI^{\dua}}\to\dI$ sending $\cF$ to the unique dense orbit in $\Supp(\cF)$. We define a function $\dell: \wt{\dI^{\dua}}\to \Z_{\ge0}$ in a similar way as $\ell$: for $\cF\in \wt{\dI^{\dua}}$, let $\dell(\cF)$ be the minimal non-negative integer $n$ such that a shift of $\cF$ appears as a direct summand of $\IC_{\dl}\star\IC_{w}$ for some $\dl\in (\dK\bs \dX)_{cl}$ and $w\in \bW$ with $\ell(w)=n$. It is shown in the argument of \cite[Proposition 3.6]{BV} that 
\begin{equation}\label{dual length}
\dell(\cF)=\dim_{\CC}\Supp(\cF)-\check{d}_{\min}
\end{equation}
where $\check{d}_{\min}$ is the complex dimension of any closed $\dK$-orbit on $\dX$. In particular, $\dell$ factors through $\dI$, which we still denote by $\dell: \dT\to \Z_{\ge0}$. For $n\in\Z_{\ge0}$, let $\dX_{\le n}$ be the union of $\dK$-orbits $O^{\dK}_{\dl}$ such that $\dell(\dl)\le n$. Then $\dX_{\le n}$ is closed in $\dX$. Let $\dcD^{\dua}_{\dK, \le n}\subset \dcD^{\dua}_{\dK}$ be the full subcategory of complexes supported on $\dX_{\le n}$.

\cor{pres supp}
\begin{enumerate}
\item The bijection $|\ka|$ preserves the length functions on both sides: if $\cF$ has support $\ol O^{\dK}_{\dl}$ for some $\dl\in \dI$, and it corresponds to $\cT_{\l,\chi}$ under $|\ka|$, then $\dell(\dl)=\ell(\l)$.
\item The functor $\ka$ in \reft{KD} sends $\SSC(\dcD^{\dua}_{\dK, \le n})$ surjectively to $\Tilt(\CM^{a}_{G_{\RR},\le n})$ for any $n\in \Z_{\ge0}$. 
\end{enumerate}
 \ecor
\prf 
(1) Comparing the definitions of $\dell$ and $\ell$, we see $\dell(\cF)=\ell(\l)$ because for a closed $\dK$-orbit $O^{\dK}_{\dmu}$, $\ka(\IC_{\dmu})\cong \cT_{\l_{0},\chi}$ for some $\chi\in\LS_{\l_{0}}$, and $\ka$ is equivariant under Hecke actions by \reft{KD}\eqref{KD Hk eq}. 

(2) By \refl{length codim}, objects in $\Tilt(\CM^{a}_{G_{\RR},\le n})$ are direct sums of $\cT_{\l,\chi}$ with $\ell(\l)\le n$. By \eqref{dual length}, objects in $\SSC(\dcD^{\dua}_{\dK,\le n})$ are direct sums of shifts of simple perverse sheaves supported on $O^{\dK}_{\dl}$ with $\dell(\dl)\le n$. Therefore (2) follows from (1).
\epr

\ssec{}{Koszul duality--triangulated category version}
To extend the Koszul duality functor in \reft{KD} to triangulated categories,  we recall how $\CM^{a}_{G_\BR}$ and $\dcD^{\dua}_{\dK}$ are recovered from the additive subcategories $\Tilt(\CM^{a}_{G_\BR})$ and $\SSC(\dcD^{\dua}_{\dK})$ respectively.

First consider the monodromic side. The following constructions and results make sense in the general setting of \refss{set}, i.e. in the case of a Lie group $H$ acting on a real analytic space $X$ satisfying the assumptions in \refss{assump}. 


Recall from \refl{mapTT}(1) that $\Ext^{>0}_{\CM_{H,X}}(\CT_1,\CT_2)=0$. Denote $$\CT_\oplus:=\bigoplus_{(\l,\chi)\in \wt I} \CT_{\l,\chi}$$ for the sum of all indecomposable \fmo tilting sheaves in $\CM_{H,X}$ and let 
$$E=\End(\CT_\oplus)^{opp}.$$ 
Also, put $K^b(\mathrm{Tilt}(\CM_{H,X}))$ for the homotopy category of bounded complexes in $\mathrm{Tilt}(\CM_{H,X})$. 

\prop{dgmod}
\begin{enumerate}
\item Let $\Proj(E\lmod)$ be the category of finitely generated projective $E$-modules. Then the functor
$$\Hom(\CT_{\op}, -): \Tilt(\CM_{H,X})\to \Proj(E\lmod)$$ 
is an equivalence of categories.
\item The natural functor $K^b(\mathrm{Tilt}(\CM_{H,X}))\to\CM_{H,X}$ is an equivalence of triangulated categories. 
\item Combining (1) and (2),  there is a canonical equivalence of triangulated categories
\begin{equation*}
\CM_{H,X}\cong \Perf(E\lmod):=K^{b}(\Proj(E\lmod))
\end{equation*}
under which indecomposable \fmo tilting sheaves correspond to indecomposable projective $E$-modules.
\end{enumerate}

\eprop

\prf
(1) The functor lands in projective $E$-modules because all objects in $\Tilt(\CM_{H,X})$ are direct summands of $\CT_{\op}^{n}$ for some $n$ by \refp{tilt}. The left adjoint of the functor is given by $M\mapsto \CT_{\op}\ten_{E}M$. One checks that these functors are inverse to each other. 

(2) Follows from \refp{tilt} and \refl{mapTT}(1) as in \cite[1.5]{BBM} (see also  \cite[Proposition B.1.7]{BY}).
\epr

On the dual side, we introduce the notation
\begin{equation*}
\dcD^{\dua,gr}_{\dK}:=K^{b}(\SSC(\dcD^{\dua}_{\dK})).
\end{equation*}
This can be viewed as a mixed (in the sense of mixed sheaves or mixed Hodge modules) version of $\dcD^{\dua}_{\dK}$. Similarly we define the graded version of the Hecke category
\begin{equation*}
\dcH^{gr}_{\dG}:=K^{b}(\SSC(\dcH_{\dG}))\cong K^{b}(\SBim^{gr}).
\end{equation*}
that acts on $\dcD^{\dua,gr}_{\dK}$.

The internal degree shift $\{1\}$ on $\SSC(\dcD^{\dua}_{\dK})$ induces an endo-functor on $\dcD^{\dua,gr}_{\dK}$ which we still denote by $\{1\}$. We denote the cohomological shift in $K^{b}(\SSC(\dcD^{\dua}_{\dK}))$ by $[1]$. For $\cF\in \SSC(\dcD^{\dua}_{\dK})$ viewed as an object in $\dcD^{\dua,gr}_{\dK}$,  we write it as $\cF\{0\}$.
  
For $\cF_{1},\cF_{2}\in \dcD^{\dua,gr}_{\dK}$, define a bigraded $\cR^{\bu}$-module 
\begin{equation*}
\HOM(\cF_{1},\cF_{2}):=\op_{m,n\in\Z}\Hom_{\dcD^{\dua,gr}_{\dK}}(\cF_{1},\cF_{2}\{m\}[n]).
\end{equation*}

For simple perverse sheaf $\cF\in \dcD^{\dua}_{\dK}$ with support $O^{\dK}_{\dl}$, let $\D_{\cF}=i_{\dl,!}i_{\dl}^{*}\cF$ and $\Na_{\cF}=i_{\dl,*}i_{\dl}^{*}\cF$. 

The following proposition makes precise the meaning that $\dcD^{\dua,gr}_{\dK}$ is a graded version of $\dcD^{\dua}_{\dK}$.

\prop{gr eq} There is a canonical triangulated functor
\begin{equation*}
\ph: \dcD^{\dua,gr}_{\dK}=K^{b}(\SSC(\dcD^{\dua}_{\dK}))\to \dcD^{\dua}_{\dK}
\end{equation*}
extending the natural embedding $\SSC(\dcD^{\dua}_{\dK})\incl \dcD^{\dua}_{\dK}$ and satisfying the following properties:
\begin{enumerate}
\item $\ph\c\{1\}\cong [1]\c\ph$.
\item For $\cF_{1},\cF_{2}\in \dcD^{\dua,gr}_{\dK}$, natural maps induce a graded $\cR^{\bu}$-linear isomorphism
\begin{eqnarray*}
\HOM(\cF_{1},\cF_{2})\isom \Ext^{\bu}(\ph(\cF_{1}), \ph(\cF_{2}))
\end{eqnarray*}
that maps $\Hom(\cF_{1},\cF_{2}\{m\}[n])$ to $\Ext^{m+n}(\ph(\cF_{1}),\ph(\cF_{2}))$, for $m,n\in\Z$.
\item The functor $\ph$ intertwines the actions of $\dcH^{gr}_{\dG}$ and $\dcH_{\dG}$ via the similarly defined monoidal functor $\ph_{\dG}: \dcH^{gr}_{\dG}\to \dcH_{\dG}$ in the case of complex groups.

\item\label{gr st cst}  For a simple perverse sheaf $\cF\in \dcD^{\dua}_{\dK}$, there exists $\D^{gr}_{\cF}\in \dcD^{\dua,gr}_{\dK}$ and a map $\imath^{gr}_{\cF}: \D^{gr}_{\cF}\to \cF\{0\}$ whose image under $\ph$ is the canonical map $\imath_{\cF}: \D_{\cF}\to\cF$; such a pair $(\D^{gr}_{\cF}, \imath_{\cF}^{gr})$ is unique up to a unique isomorphism. Similarly, there exists $\Na^{gr}_{\cF}\in \dcD^{\dua,gr}_{\dK}$ and a map $\jmath^{gr}_{\cF}: \cF\to\Na^{gr}_{\cF}$ whose image under $\ph$ is the canonical map $\jmath_{\cF}: \cF\to\Na_{\cF}$; such a pair is unique up to a unique isomorphism. 
\end{enumerate}
\eprop

\prf The functor $\ph$ is constructed in \cite[5.1]{BV}; the essential point being that for $\dL=\op_{\wt{\dI^{\dua}}}\cF$, the dg algebra $\RHom(\dL,\dL)$ is formal, which is proved by a standard purity argument. Properties (1)-(3) for $\ph$ are easy to see. 

Now we prove (4). We only give the argument for $\D^{gr}_{\cF}$, and the case of $\Na^{gr}_{\cF}$ can be proved similarly by replacing all $\D$'s by $\Na$'s. 

We first construct $\D^{gr}_{\cF}$. Note that graded liftings $\D^{gr}_{w}\in \dcH_{\dG}^{gr}$ of $\D_{w}\in \dcH_{\dG}$ exist for any $w\in \bW$ (in the Soergel bimodule setup, they are the Rouquier complexes). We construct $\D^{gr}_{\cF}$ by induction on $\dell(\cF)$. For $\dell(\cF)=0$, $\D_{\cF}\isom \cF$, hence we can take $\D^{gr}_{\cF}$ to be $\cF\{0\}$. Now suppose $\cF$ is not supported on a closed orbit, and suppose $\D^{gr}_{\cF'}$ together with the map $\imath_{\cF'}:\D^{gr}_{\cF'}\to \cF'$ has been constructed for all simple perverse sheaves $\cF'\in \dcD^{\dua}_{\dK}$ with $\dell(\cF')<\dell(\cF)$ (i.e., $\dim\Supp(\cF')<\dim\Supp(\cF)$). Let $\cF=\IC(O^{\dK}_{\dl}, \uk_{\psi})$ for some $\dK$-equivariant local system $\uk_{\psi}$ on $O^{\dK}_{\dl}$. Since $O^{\dK}_{\dl}$ is not a closed $\dK$-orbit, there exists a simple root $\check\a_{s}$ in the root system $\Phi_{\dl}$ such that $O^{\dK}_{\dl}$ is open in (but not equal to) $\pi_{s}^{-1}\pi_{s}(O^{\dK}_{\dl})$. We have the following cases: \footnote{The calculations of convolutions $\star\D_{s}$ that appear in these cases are analogs of \refl{stdconv}, which can be found in \cite[Lemma 3.5]{LV}, or can be deduced from \refl{stdconv} by applying the Matsuki equivalence.}
\begin{itemize}
\item If $\check\a_{s}$ is a complex root, then $\pi_{s}^{-1}\pi_{s}(O^{\dK}_{\dl})-O^{\dK}_{\dl}$ is a unique orbit $O^{\dK}_{\dmu}$, and the local system $\uk_{\dl,\psi}$ extends to $\pi_{s}^{-1}\pi_{s}(O^{\dK}_{\dl})$, whose restriction to $O^{\dK}_{\dmu}$ we still denote by $\uk_{\dmu,\psi}$. We have $\D_{\cF}=\D_{\dl,\psi}\cong \D_{\dmu, \psi}\star\D_{s}$. By inductive hypothesis, $\D^{gr}_{\dmu,\psi}$ has already been constructed. We define $\D^{gr}_{\cF}:=\D^{gr}_{\dmu, \psi}\star\D^{gr}_{s}$.

\item If $\check\a_{s}$ is a real root, $\uk_{\dl,\psi}$ extends to a local system $\wt{\uk_{\dl,\psi}}$ on $\pi_{s}^{-1}\pi_{s}(O^{\dK}_{\dl})$, and $\pi_{s}^{-1}\pi_{s}(O^{\dK}_{\dl})-O^{\dK}_{\dl}$ is a single orbit $O^{\dK}_{\dmu}$. Denote the restriction of $\wt{\uk_{\dl,\psi}}$ to $O^{\dK}_{\dmu}$ again by $\uk_{\dmu,\psi}$. Then there is a distinguished triangle $\D_{\dl,\psi}\op\D_{\dl,\psi'}\to \D_{\dmu,\psi}\star\D_{s}\to\D_{\dmu,\psi}[1]$ (where $\psi'\ne \psi$ is another local system on $O^{\dK}_{\dl}$). Note that $\D^{gr}_{\dmu,\psi}$ has already been constructed by inductive hypothesis. Since $\Hom(\D_{\dmu,\psi}\star\D_{s},\D_{\dmu,\psi}[1])$ is one-dimensional over $\bfk$, by \refp{gr eq}(2), there is a unique $m\in \Z$ and a nonzero map $f: \D^{gr}_{\dmu,\psi}\star\D^{gr}_{s}\to\D^{gr}_{\dmu,\psi}\{m\}[1-m]$ (unique up to scalar). Let $C=Cone(f)[-1]\in \dcD^{\dua,gr}_{\dK}$, then $\ph(C)\cong \D_{\dl,\psi}\op\D_{\dl,\psi'}$. By \refp{gr eq}(2), idempotents of $\End(\ph(C))$ lift to idempotents of $\End(C)$, therefore $C$ decomposes into two summands $C_{1}\op C_{2}$ whose images under $\ph$ are $\D_{\dl,\psi}$ and $\D_{\dl,\psi'}$ respectively. We define $\D^{gr}_{\cF}:=C_{1}$.

\item If $\check\a_{s}$ is a real root, $\uk_{\dl,\psi}$ extends to a local system $\wt{\uk_{\dl,\psi}}$ on $\pi_{s}^{-1}\pi_{s}(O^{\dK}_{\dl})$, and $\pi_{s}^{-1}\pi_{s}(O^{\dK}_{\dl})-O^{\dK}_{\dl}=O^{\dK}_{\dmu^{+}}\sqcup O^{\dK}_{\dmu^{-}}$. Let $\uk_{\dmu^{+}, \psi^{+}}$ and $\uk_{\dmu^{-}, \psi^{-}}$ be the restrictions of $\wt{\uk_{\dl,\psi}}$ to $O^{\dK}_{\dmu^{+}}$ and $O^{\dK}_{\dmu^{-}}$ respectively. Then we have a distinguished triangle $\D_{\dl,\psi}\to \D_{\dmu^{+}, \psi^{+}}\star\D_{s}\to \D_{\dmu^{-}, \psi^{-}}[1]$. As in the previous case, there is a unique $m\in\Z$ such that there is a nonzero map $f: \D^{gr}_{\dmu^{+}, \psi^{+}}\star\D^{gr}_{s}\to \D^{gr}_{\dmu^{-}, \psi^{-}}\{m\}[1-m]$. We define $\D^{gr}_{\cF}:=Cone(f)[-1]$.

\item If $\check\a_{s}$ is a real root and $\psi$ does not extend to $\pi_{s}^{-1}\pi_{s}(O^{\dK}_{\dl})$. If this happens, then there must exist another simple root $\check\a_{t}$ that falls into the previous cases. This is proved in the argument of \cite[Proposition 3.6]{BV}. Therefore there is no need to consider this case.
\end{itemize}

In each of the cases,  $\D^{gr}_{\cF}$ is defined to be equipped with a map to $\D^{gr}_{\cF'}\star \D_{s}$ for some simple perverse sheaf $\cF'$ with $\dell(\cF')=n-1$ and simple reflection $s\in \bW$. We then define $\imath_{\cF}$ to be the composition
\begin{equation*}
\D^{gr}_{\cF}\to \D^{gr}_{\cF'}\star \D_{s}\xr{\imath_{\cF'}\star\imath_{\IC_{s}}} \cF'\{0\}\star \IC_{s}\{0\}\to \cF\{0\}
\end{equation*}
where the last map is the projection of $\cF'\star \IC_{s}$ to the unique summand isomorphic to $\cF$. It is easy to see that $\imath_{\cF}$ is sent to the canonical map $\D_{\cF}\to \cF$ under $\ph$.

Finally we prove the uniqueness of the pair $(\D^{gr}_{\cF}, \imath_{\cF})$. Let $(\cE, i:\cE\to \cF\{0\})$ and $(\cE',i':\cE'\to \cF\{0\})$ be two such pairs. Let $\e: \ph(\cE)\isom\ph(\cE')$ be the unique isomorphism compatible with $\ph(i)$ and $\ph(i')$. By \refp{gr eq}(2), for a unique $m\in\Z$, $\e$ lifts to a map $\e^{gr}: \cE\to \cE'\{m\}[-m]$. Now both $i: \cE\to \cF$ and $i'\c\e^{gr}: \cE\to \cE'\{m\}[-m]\to \cF\{m\}[-m]$ map to the canonical map $\D_{\cF}\to \cF$ under $\ph$. By \refp{gr eq}(2) again, we conclude that $m=0$ and $i=i'\c\e^{gr}$. This proves the uniqueness of the pair up to unique isomorphism.
\epr

\th{KD derived} Under the same notations and choices of $\xi,\wt\xi$ and $\wt\ome$ as in \reft{KD}, there is a canonical triangulated functor
\begin{equation*}
\wt\ka:  \dcD^{\dua,gr}_{\dK}\to \CM^{a}_{G_{\RR}} \end{equation*}
extending the functor $\ka$ and satisfying the following properties:
\begin{enumerate}
\item $\wt\ka\c\{1\}\cong \wt\ka$.
\item For $\cF_{1},\cF_{2}\in \dcD^{\dua,gr}_{\dK}$, natural maps induce a graded $\cR$-linear  isomorphism
\begin{eqnarray*}
\HOM(\cF_{1},\cF_{2})\ot_{\cR^{\bu}}\cR\isom \Ext^{\bu}(\wt\ka(\cF_{1}), \wt\ka(\cF_{2}))
\end{eqnarray*}
that sends $\Hom(\cF_{1},\cF_{2}\{m\}[n])$ to  $\Ext^{n}(\wt\ka(\cF_{1}), \wt\ka(\cF_{2}))$, for $m,n\in\Z$.

\item The functor $\wt\ka$ intertwines the actions of $\CH_{G}$ and $\dcH^{gr}_{\dG}$ via the monoidal functor $\dcH^{gr}_{\dG}\xr{\ph_{\dG}}\dcH_{\dG}\xr{\wt\Phi}\cH_{G}$. Here $\wt\Phi$ is defined similarly as $\wt\ka$ in the case of complex groups. 


\item\label{KD st cst} Suppose $|\ka|$ sends a simple perverse sheaf $\cF\in \dcD^{\dua}_{\dK}$ to $\cT_{\l,\chi}\in \Tilt(\CM^{a}_{G_{\RR}})$ as in \reft{KD}\eqref{KD IC to tilting}. Then $\wt\ka$ sends the graded lift of the standard object $\D^{gr}_{\cF}$ (resp. the costandard object $\Na^{gr}_{\cF}$, see \refp{gr eq}\eqref{gr st cst}) to the \fmo standard object $\wt\D_{\l,\chi}$ (resp. \fmo costandard object $\wt\Na_{\l,\chi}$).
\end{enumerate}
\eth

\prf 
By \refp{dgmod}(2), $\CM^{a}_{G_{\RR}}\cong K^{b}(\Tilt(\CM^{a}_{G_{\RR}}))$. We define the functor $\wt\ka$  to be $K^{b}(\ka)$. Properties (1)-(3) follow from the similar properties for $\ka$ proved in \reft{KD}. 

It remains to check \eqref{KD st cst}; we give the argument for $\wt\ka(\D^{gr}_{\cF})\cong \wt\D_{\l,\chi}$, and the other isomorphism $\wt\Na(\D^{gr}_{\cF})\cong \wt\Na_{\l,\chi}$ is proved similarly. Let $n=\dell(\cF)$. Then the object $\D^{gr}_{\cF}$ together with the canonical map $\imath_{\cF}: \D^{gr}_{\cF}\to \cF\{0\}$ constructed in \refp{gr eq}\eqref{gr st cst} can be characterized as the unique such pair $(\cE, i:\cE\to \cF)$ such that $\cE\in \dcD^{\dua,gr}_{\dK}$ such that $\Hom(\cE, \dcD^{\dua,gr}_{\dK,<n})=0$ and $i$ induces an isomorphism in the Verdier quotient $\dcD^{\dua,gr}_{\dK}/\dcD^{\dua,gr}_{\dK,<n}$ (this follows from the construction of $\D^{gr}_{\cF}$). Similarly, $\wt\D_{\l,\chi}$ together with the canonical map $\wt\D_{\l,\chi}\to \cT_{\l,\chi}$ can be characterized as the unique pair $(\cG, i:\cG\to \cT_{\l,\chi})$ such that $\Hom(\cG, \CM^{a}_{G_{\RR},<n})=0$ and $i$ induces an isomorphism in the Verdier quotient $\CM^{a}_{G_{\RR}}/\CM^{a}_{G_{\RR},<n}$. This proves $\wt\ka(\D^{gr}_{\cF})\cong \wt\D_{\l,\chi}$.
\epr

\rem{} The above theorem shows that $\dcD^{\dua,gr}_{\dK}$ is a graded version of $\CM^{a}_{G_{\RR}}$. It would be interesting to define a graded version of $\CM_{G_{\RR}}$ intrinsically using only the geometry of the $G_{\RR}$-action on $\wt X$. If we use the Matsuki equivalence to identify $\CM_{G_{\RR}}$ with $\wh D_{K}(G/U)_{\bT\mon}$. Such a graded version may be obtained by considering mixed Hodge modules on $G/U$. 

In particular, the stalks and costalks of an indecomposable tilting object in $\TGR$ should carry gradings reflecting a mysterious weight structure, canonical up to an overall shift. We refer to \cite{Y} where the weight structure on stalks of tilting sheaves in the complex group case is determined.
\erem

\ssec{}{Numerical consequences of Koszul duality}

For a $G_{\RR}$-orbit $O^{\RR}_{\l}\subset X$, recall from \refl{TRT} that $\bT$ carries a canonical real form $\s_{\l}$ depending only on $\l$. Let $\theta_{\l}: \bT\to \bT$ be the corresponding Cartan involution.

For a $\dK$-orbit $O^{\dK}_{\dl}\subset \dX$, choosing a Borel subgroup $\dB\in O^{\dK}_{\dl}$, and let $\dbT_{\dl}\subset \dbT$ be the image of the projection $\dK\cap \dB\to \dbT$, which is independent of the choice of $\dB$. As in \refl{TRT}, by choosing a $\dth$-stable maximal torus $\dT\subset \dB$, $\dbT$ is identified with $\dT$ and $\dbT_{\dl}$ is identified with $\dT^{\dth}$. This induces an involution $\dth_{\dl}$ on $\dbT$ that depends only on $\dl$, such that $\dbT_{\dl}=\dbT^{\dth_{\dl}}$.

The following is a known property of Vogan duality \cite{V2}; we give a geometric argument using Koszul duality.

\cor{} Let  $\cF\in \dcD^{\dua}_{\dK}$ be a simple perverse sheaf supported on the $\dK$-orbit $O^{\dK}_{\dl}$ for some $\dl\in \dI$. Let $\cT_{\l,\chi}=\ka(\cF)$. Then under the canonical identification
\begin{equation}\label{XXT}
\XX_{*}(\bT)\cong \XX^{*}(\dbT),
\end{equation}
the involution $\theta_{\l}$ corresponds to the involution $-\dth_{\dl}$.
\ecor{}
\prf Combining the two isomorphisms in \refp{gr eq}(2) and \reft{KD derived}(2), both applied to $\cF_{1}=\cF_{2}=\D^{gr}_{\cF}$, and using \reft{KD derived}\eqref{KD st cst}, we get an isomorphism of $\cR$-modules
\begin{equation}\label{self ext st}
\Ext^{\bu}(\D_{\cF}, \D_{\cF})\ot_{\cR^{\bu}}\cR\cong \Ext^{\bu}(\wt\D_{\l,\chi}, \wt\D_{\l,\chi}).
\end{equation}
Note that $\Ext^{\bu}(\D_{\cF}, \D_{\cF})=H^{*}_{\dbT_{\dl}}(\pt,\bfk)\cong \Sym(\XX^{*}(\dbT_{\dl})\ot_{\Z}\bfk[-2])$ as an $\cR^{\bu}$-module. On the other hand, by \refl{comp DA}, $\Ext^{\bu}(\wt\D_{\l,\chi}, \wt\D_{\l,\chi})\cong \ol\cR_{\l}$,  which is the completion of $\bfk[\XX_{*}(\bT/\bT_{\l})]$ at the augmentation ideal. The isomorphism \eqref{self ext st} as $\cR$-modules imply that $\XX^{*}(\dbT_{\dl})_{\QQ}=\XX_{*}(\bT/\bT_{\l})_{\QQ}$ as quotients of $\XX^{*}(\dbT)=\XX_{*}(\dbT)_{\QQ}$. This implies that under \eqref{XXT}, the $-1$ eigenspace of $\dth_{\dl}$ on $\XX^{*}(\dbT)_{\QQ}$ coincides with the $+1$ eigenspace of $\theta_{\l}$ on $\XX_{*}(\bT)_{\QQ}$. Since both $\theta_{\l}$ and $\dth_{\dl}$ preserve the Killing forms, which are intertwined under \eqref{XXT}, we conclude that $-\dth_{\dl}$ is intertwined with $\theta_{\l}$ under  \eqref{XXT}.
\epr

Recall the Lusztig-Vogan polynomials for the stalks of $IC$ sheaves on $\dK\bs \dX$, see \cite[Theorems 1.11, 1.12]{LV}. For simple perverse sheaves $\cF,\cF'\in \dcD_{\dK}$, let
\begin{equation*}
P_{\cF',\cF}(q)=\sum_{i\ge0}\dim_{\bfk}q^{i}\Hom(\cF, \Na_{\cF'}[2i]).
\end{equation*}

\cor{LV poly} For an indecomposable \fmo tilting sheaf $\cT_{\l,\chi}$  in $\CM^{a}_{G_{\RR}}$, denote by $\cF_{\l,\chi}\in \dcD^{\dua}_{\dK}$ the  simple perverse sheaf that  corresponds to $\cT_{\l,\chi}$ under the bijection $|\ka|$ in \reft{KD}\eqref{KD IC to tilting}. Then for any $\l'\in I$, we have an isomorphism
\begin{equation*}
\wt i_{\l'}^{*}\cT_{\l,\chi}\cong \bigoplus_{(\l',\chi')\in \wt I^{a}}\cL_{\l',\chi'}^{\op P_{\cF_{\l',\chi'},\cF_{\l,\chi}}(1)}[p_{\l'}].
\end{equation*}
\ecor
\prf Let $\cF'=\cF_{\l',\chi'}$. Let $m_{\l',\chi'}$ be the multiplicity of $\cL_{\l',\chi'}[p_{\l'}]$ in $\wt i_{\l'}^{*}\cT_{\l,\chi}$. Combining the two isomorphisms in \refp{gr eq}(2) and \reft{KD derived}(2), both applied to $\cF_{1}=\cF\{0\}$ and $\cF_{2}=\Na^{gr}_{\cF'}$, and using \reft{KD derived}\eqref{KD st cst}, we get an isomorphism of $\cR$-modules
\begin{equation*}
\Ext^{\bu}(\cF,\Na_{\cF'})\ot_{\cR^{\bu}}\cR\cong \Ext^{\bu}(\cT_{\l,\chi}, \wt\Na_{\l',\chi'}).
\end{equation*}
We know the right side is a direct sum of $m_{\l',\chi'}$ copies of $\Hom(\wt\D_{\l',\chi'}, \wt\Na_{\l',\chi'})$, which is a free $\ol\cR_{\l'}$-module of rank $m_{\l',\chi'}$. The left side is a direct sum of $P_{\cF',\cF}(1)$ copies of $\Ext^{\bu}(\D_{\cF'}, \Na_{\cF'})\ot_{\cR^{\bu}}\cR$ (using that the stalks of $\cF$ along $O^{\dK}_{\dl'}$ are concentrated in degrees of the same parity), which is a free $H^{*}_{\dbT_{\dl'}}(\pt,\bfk)\ot_{\cR^{\bu}}\cR$-module of rank $P_{\cF',\cF}(1)$ (where $\dl'$ indexes the $\dK$-orbit whose closure is the support of $\cF'$). Comparing the dimension of the reductions to $\bfk$-modules on both sides, we conclude that $m_{\l',\chi'}=P_{\cF',\cF}(1)$.
\epr

\ssec{bigtilting}{Corepresentability of the real Soergel functor}

Recall from \refp{VHk} that the big tilting object $\cT_{w_{0}}\in \CH_{G}$ corepresents the classical Soergel functor $\VV$. The next statement is its analogue for quasi-split real groups.

Let $G_{\RR}$ be quasi-split and nice. For a $\bW_{\l_{0}}$-orbit $a\subset \LS_{\l_{0}}$, let $\cB^{a}=(\prod_{\chi\in a}\CS)^{\bW_{\l_{0}}}\ot_{\cR^{\bW}}\cR$ be the direct factor of $\cB$ corresponding to the block $a$. 


\th{bigtilting} Assume $G_{\RR}$ is quasi-split and nice.
\begin{enumerate}
\item Let $\chi\in\LS_{\l_{0}}$ and let $a$ be the $\bW_{\l_{0}}$-orbit of $\chi$. Then $\cT_{\l_{0},\chi}\star \cT_{w_{0}}$ contains a direct summand $\cT^{a}$ satisfying $\VV^{\sh}_{\RR}(\cT^{a})\cong \cB^{a}$ as $\cB$-modules. The isomorphism class of $\cT^{a}$ only depends on the block $a$.
\item We have an $\cR$-algebra isomorphism $\End(\cT^{a})\cong \cB^{a}$.
\item For any choice of a regular covector $\xi\in i\cN^{*}(\RR)\cap \cO_{\reg}$ and its lifting $\wt\xi$, there is an isomorphism of functors 
$$\VV_{\RR,\wt\xi}\cong \RHom(\op_{a}\cT^{a}, -): \CM_{G_{\RR}}\to D^{b}(\rmod\cR)$$
where the sum is over $\bW_{\l_{0}}$-orbits $a\subset \LS_{\l_{0}}$.   
In particular, the functor $\VV_{\RR,\wt\xi}$ does not depend on the choice of $\wt\xi$, up to isomorphism.

\item Under the functor $\ka$ in \reft{KD}, $\cT^{a}$ is the image of the constant perverse sheaf $\uk[\dim \dX]$.
\end{enumerate}
\eth
\prf
We fix $\wt\xi$ and denote $\VV_{\RR,\wt\xi}$ by $\VV_{\RR}$.

(1) Under $\VV^{\sh}_{\RR}$, $\cT_{\l_{0},\chi}\star \cT_{w_{0}}$ becomes the $\cB^{a}$-module $\CS_{\chi}\ot_{\cR}(\cR\ot_{\cR^{\bW}}\cR)\cong \CS_{\chi}\ot_{\cR^{\bW}}\cR$. Note that $\cB^{a}=\CS^{\bW_{\l_{0},\chi}}\ot_{\cR^{\bW}}\cR$, where $\bW_{\l_{0},\chi}$ is the stabilizer of $\chi$ under the cross action of $\bW_{\l_{0}}$ on $\LS_{\l_{0}}$. Since $\CS^{\bW_{\l_{0},\chi}}$ is a summand of $\CS$ as a $\CS^{\bW_{\l_{0},\chi}}$-module, $\VV^{\sh}_{\RR}(\cT_{\l_{0},\chi}\star \cT_{w_{0}})\cong \CS\ot_{\cR^{\bW}}\cR$ contains a free rank one $\cB^{a}$-module as a  direct summand. Since $\VV^{\sh}_{\RR}$ is fully faithful by \refc{maincor}, we conclude that $\cT_{\l_{0},\chi}\star \cT_{w_{0}}$ contains a summand $\cT^{a}$ with $\VV^{\sh}_{\RR}(\cT^{a})\cong\cB^{a}$ as $\cB$-modules.


For another $\chi'\in a$, the same argument would give another $\cT'$ with the same property $\VV^{\sh}_{\RR}(\cT')\cong \cB^{a}$. Since $\VV^{\sh}_{\RR}$ is fully faithful, we conclude that  $\cT^{a}\cong \cT'$.

(2) By \refc{maincor}, $\End(\cT^{a})\cong \End_{\cB}(\VV^{\sh}_{\RR}(\cT^{a}))\cong\End_{\cB}(\cB^{a})=\cB^{a}$.

(3) Choose a free generator $v_{a}\in \VV_{\RR}(\cT^{a})$ as a right $\cB^{a}$-module. Then we get a map functorial in $\cF\in \CM^{a}_{G_{\RR}}$
\begin{equation*}
\a: \RHom(\op_{a}\cT^{a}, \cF)\xr{\VV_{\RR}(-)} \RHom_{\cR}(\op_{a}\VV_{\RR}(\cT^{a}), \VV_{\RR}(\cF))\xr{\op\ev_{v_{a}}} \VV_{\RR}(\cF)
\end{equation*}
where the last map is evaluation at $v_{a}$. This defines a map between two triangulated functors $\CM_{G_{\RR}}\to D^{b}(\rmod\cR)$. To show it is an isomorphism, it suffices to check on generators. We show that $\a$ is an isomorphism for $\cF\in \TGR$. Indeed, in this case, both sides are concentrated in degree $0$. Using \refc{maincor}, $\a$ becomes
\begin{equation*}
\Hom(\op\cT^{a}, \cF)\cong\Hom_{\cB}(\op\VV^{\sh}_{\RR}(\cT^{a}), \VV^{\sh}_{\RR}(\cF))\xr{\ev_{\sum v_{a}}}\VV^{\sh}_{\RR}(\cF)=\VV_{\RR}(\cF)
\end{equation*}
where $\ev_{\sum v_{a}}$ is an isomorphism because $\sum v_{a}$ is a free generator of $\VV^{\sh}_{\RR}(\op_{a}\cT^{a})$ as a $\cB$-module. This proves (2).

(4) Let $C=\uk[\dim \dX]$. By construction, the global sections functor $\HH^{\dua}$ on $\SSC(\dcD^{\dua}_{\dK})$ is intertwined with $\VV^{\sh}_{\RR}$ on $\Tilt(\CM^{a}_{G_{\RR}})$ in the sense that if $\cF\in \SSC(\dcD^{\dua}_{\dK})$, then 
\begin{equation*}
\HH^{\dua}(\cF)\ot_{\cR^{\bu}}\cR=\Ext^{\bu}(C, \cF)\ot_{\cR^{\bu}}\cR\cong \VV^{\sh}_{\RR}(\ka(\cF))
\end{equation*}
as $\cB^{a}$-modules, where we use the isomorphism $\dcB^{\dua, \bu}\ot_{\cR^{\bu}}\cR=H^{*}(\dK\bs\dX, \cF)\ot_{\cR^{\bu}}\cR\cong \cB^{a}$ proved in \cite[Corollary 2.6]{BV}. From this we see that $\VV^{\sh}_{\RR}(\ka(C))\cong \HH^{\dua}(C)\ot_{\cR^{\bu}}\cR=\dcB^{\dua, \bu}\ot_{\cR^{\bu}}\cR\cong\cB^{a}$. Repeating the argument in (3) we see that $\ka(C)$ also corepresents the functor $\VV_{\RR}$, hence $\ka(C)\cong \cT^{a}$.

\epr


\begin{thebibliography}{9999}

\bibitem{AdC}  Adams, J.; du Cloux, F. Algorithms for representation theory of real reductive groups. J. Inst. Math. Jussieu 8 (2009), no. 2, 209-259.

\bibitem{BBD} Beilinson, A.; Bernstein, J.; Deligne, P. Faisceaux pervers. In {\itshape Analysis and topology on singular spaces I
(Luminy, 1981)}, 5-171, Ast{\'e}risque 100, Soc. Math. France, Paris, 1982.

\bibitem{BBM} Beilinson, A; Bezrukavnikov, R; Mirkovi{\'c}, I. Tilting exercises. Mosc. Math. J.4 (2004), no. 3, 547-557,
782.

\bibitem{BG}  Beilinson, A.; Ginzburg, V. Wall-crossing functors and D-modules. Represent. Theory 3 (1999), 1-31.

\bibitem{BGS} Beilinson, A; Ginzburg, V; Soergel, W. Koszul duality patterns in representation theory. J. Amer. Math.
Soc. 9 (1996), no. 2, 473-527.

\bibitem{BL} Bernstein, J.; Lunts, V. Equivariant sheaves and functors. Lecture Notes in Mathematics, 1578. SpringerVerlag, Berlin, 1994. iv+139 pp.

\bibitem{B} Bezrukavnikov, R., written by Kanstrup, T. Canonical basis and representation categories. MIT course notes, \url{https://math.mit.edu/~bezrukav/old/Course_RT.pdf}, 2013. 

\bibitem{BR} Bezrukavnikov, R.; Riche, S., A topological approach to Springer theory, in
"Interactions between Representation Theory and Algebraic Geometry", Birkh{\"a}user
(2019).

\bibitem{BV} Bezrukavnikov, R.; Vilonen, K. Koszul Duality for Quasi-split Real Groups. Inv. Math., 226,139-193 (2021).

\bibitem{BY} Bezrukavnikov, R.; Yun Z. On Koszul duality for Kac-Moody groups, Represent. Theory
17 (2013), 1-98.

\bibitem{BS}  Borel, A.; Springer, T. A. Rationality properties of linear algebraic groups II, Tohoku Math. J. (2) 20(4): 443-497 (1968).

\bibitem{G} Gaitsgory, D. The notion of category over an algebraic stack, arXiv preprint
math/0507192 (2005).

\bibitem{GKM}
Goresky, M.; Kottwitz, R.; MacPherson, R. Equivariant cohomology, Koszul duality, and the localization theorem. Invent. Math. 131, 25-83 (1997).

\bibitem{HMSW}
Hecht, H.; Mili\u ci\'c, D.; Schmid, W.; Wolf, J.A.
Localization and standard modules for real semisimple Lie groups I,
Invent. Math. 90 (1987), 297-332.

\bibitem{H} Helgason, S. Some results on invariant differential operators on symmetric spaces. Amer. J. Math. 114 (1992), no. 4,
789-811.

\bibitem{KS}
Kashiwara, M.; Schapira, P. Sheaves on manifolds, Grundlehren der Mathematischen
Wissenschaften 292, Springer, 1990.

\bibitem{KSc} Kashiwara, M.; Schmid W.  Quasi-equivariant $\mathcal{D}$-modules, equivariant
derived category, and representations of reductive Lie groups. Lie Th.
and Geom., in Honor of Bertram Kostant, Prog. in Math.,
Birkh\"auser, 457-488 (1994)

\bibitem{LV}
Lusztig, G.; Vogan, D. Singularities of closures of $K$-orbits on flag manifolds. Invent. Math. 71 (1983), no. 2, 365-379.

\bibitem{LY} Lusztig, G.; Yun, Z. Endoscopy for Hecke categories, character sheaves and representations. Forum Math. Pi 8 (2020), e12.


\bibitem{MUV} Mirkovi{\'c}, I.; Uzawa, T.; Vilonen, K. Matsuki correspondence for sheaves. Invent. Math.,
109(2):231-245, 1992.

\bibitem{R} Ringel, C.M. The category of modules with good filtrations over a quasi-hereditary
algebra has almost split sequences. Mathematische Zeitschrift, 208(2):209-223, (1991).

\bibitem{RS} Richardson, R. W.; , Springer, T. A. The Bruhat order on symmetric
varieties, Geom. Dedicata, 35(1-3), 389-436 (1990).

\bibitem{S} Soergel, W. Kategorie O, perverse Garben und Moduln {\"u}ber den Koinvarianten zur Weylgruppe, J. Amer.
Math. Soc. 3 (1990), no. 2, 421-445.

\bibitem{S2} Soergel, W. Langlands' philosophy and Koszul duality. Algebra -- representation theory (Constanta, 2000), 379-414, NATO Sci. Ser. II Math. Phys. Chem., 28,
Kluwer Acad. Publ., Dordrecht, 2001.



\bibitem{Vi} Vilonen, K. Geometric methods in representation theory, Rep. th. of Lie groups (J. Adams and D. Vogan, eds.), IAS/Park City Mathematics Series, vol 8, AMS, 241 - 290 (2000).

\bibitem{V}  Vogan, D. Irreducible characters of semisimple Lie groups. III. Proof of Kazhdan-Lusztig conjecture in the integral case, Invent. Math. 71 (1983), no. 2, 381-417.

\bibitem{V2} Vogan, D. Irreducible characters of semisimple Lie groups IV. Character-multiplicity duality, Duke Math. J., 49, no. 4, 943-1073, 1982.

\bibitem{Y}  Yun, Z. Weights of mixed tilting sheaves and geometric Ringel duality. Selecta Math. (N.S.) 14 (2009), no.
2, 299-320.

\end{thebibliography}
\end{document}